\documentclass[11pt, A4paper, leqno]{amsart}

\usepackage{graphicx}
\usepackage{amssymb,amsmath}
\usepackage{epstopdf}
\usepackage{sfmath}
\usepackage{url}

\usepackage[small, nohug, heads-vee]{diagrams} 
\diagramstyle[labelstyle=\scriptstyle]

\newtheorem{theorem}{Theorem}[section]
\newtheorem{lemma}[theorem]{Lemma}
\newtheorem{corollary}[theorem]{Corollary}
\newtheorem{proposition}[theorem]{Proposition}

\theoremstyle{remark}
\newtheorem{remark}[theorem]{Remark}
\newtheorem{hypothesis}[theorem]{Hypothesis}
\newtheorem{example}[theorem]{Example}

\theoremstyle{definition}
\newtheorem{definition}[theorem]{Definition}

\newcommand\bR{{\mathbb{R}}}
\newcommand\bC{{\mathbb C}}
\newcommand\bZ{{\mathbb Z}}

\newcommand\bV{{\mathbb V}}

\newcommand\bA{{\mathbb{A}}}
\newcommand\bU{{\mathbb{U}}}
\newcommand\Hom{{\rm Hom}}
\newcommand\dev{{\bf dev}}
\newcommand\SI{{\mathbb{S}}}

\newcommand\Bd{{\rm bd}}
\newcommand\clo{{\rm Cl}}
\newcommand\bdd{{\mathbf{d}}}

\newcommand\ra{\rightarrow}

\newcommand\emp{\emptyset}
\newcommand\eps{\epsilon}

\newcommand\Aff{{\mathbf{Aff}}}
\newcommand\ovl{\overline}
\newcommand\Aut{{\mathbf{Aut}}}
\newcommand\Idd{{\rm I}}

\newcommand\bv{{\mathbf{v}}}
\newcommand\bu{{\mathbf{u}}}
\newcommand\CN{{\mathcal{N}}}
\newcommand\Pgl{{\mathrm{PGL}}(n+1, \bR)}
\newcommand\Ag{{\mathrm{Ag}}}

\newcommand\rpn{\mathbb{RP}^n}

\newcommand\rpno{\mathbb{RP}^{n-1}}

\newcommand\SL{{\mathsf{SL}}}

\newcommand\SO{{\mathsf{SO}}}
\newcommand\Ort{{\mathsf{O}}}
\newcommand\PGL{{\mathsf{PGL}}}
\newcommand\SLnp{{\mathsf{SL}}_\pm(n+1, \bR)}
\newcommand\SLn{{\mathsf{SL}}_\pm(n, \bR)}
\newcommand\SLf{{\mathsf{SL}}_\pm(4, \bR)}

\newcommand\GL{{\mathsf{GL}}}
\newcommand\PSO{{\mathsf{PSO}}}
\newcommand\PO{{\mathsf{PO}}}
\newcommand\GLnp{{\mathsf{GL}}(n+1, \bR)}
\newcommand\PGLnp{{\mathsf{PGL}}(n+1, \bR)}

\newcommand\orb{\mathcal{O}} 
\newcommand\torb{\tilde{\mathcal{O}}}
\newcommand\bGamma{{\boldsymbol \Gamma}}
\newcommand\leng{{\mathrm{length}}}

\newcommand{\ranK}{{\mathrm{rank}}}

\newcommand\mx{{\mathrm{max}}}
\newcommand\cwl{{\mathrm{cwl}}}

\begin{document}

\title[Ends of real projective orbifolds]
{The classification of radial and totally geodesic ends of properly convex real projective orbifolds}


\author{Suhyoung Choi}
\address{ Department of Mathematics \\ KAIST \\
Daejeon 305-701, South Korea 
}
\email{schoi@math.kaist.ac.kr}

\date{\today}







\begin{abstract} 
Real projective structures on $n$-orbifolds are useful in understanding the space of 
representations of discrete groups into $\SL(n+1, \bR)$ or $\PGL(n+1, \bR)$. A recent work shows that many hyperbolic manifolds 
deform to manifolds with such structures not projectively equivalent to the original ones. 
The purpose of this paper is to understand 
the structures of ends of real projective $n$-dimensional orbifolds. 
In particular, these have the radial or totally geodesic ends. Hyperbolic manifolds with cusps and hyper-ideal ends are examples.
For this, we will study the natural conditions on eigenvalues of holonomy 
representations of ends when these ends are manageably understandable. 
The main techniques are the theory of Fried and Goldman on affine manifolds, a generalization of 
the work of Goldman, Labourie, and Margulis on flat Lorentzian $3$-manifolds and 
the work on Riemannian foliations by Molino, Carri\`ere, and so on. 
We will show that only the radial or totally geodesic ends of lens type or horospherical ends exist for strongly irreducible properly convex 
real projective orbifolds under the suitable conditions. 

\end{abstract}

\thanks{This work was supported by the National Research Foundation
of Korea (NRF) grant funded by the Korea government (MEST) (No.2010-0027001).} 

\maketitle




\section{Introduction} 


\subsection{Preliminary definitions.} 
\subsubsection{Topology of orbifolds and their ends.}
An {\em orbifold} $\orb$ is a topological space with charts modeling open sets by quotients of Euclidean open sets or half-open sets 
by finite group actions and compatibly patched with one another. 
The boundary $\partial \orb$ of an orbifold is defined as the set of points with only half-open sets as models. 
Orbifolds are stratified by manifolds. 
Let $\orb$ denote an $n$-dimensional orbifold with finitely many ends 
where end-neighborhoods are homeomorphic to closed  $(n-1)$-dimensional orbifolds times an open interval. 
We will require that $\orb$ is {\em strongly tame}; that is, $\orb$ has a compact suborbifold $K$ 
so that $\orb - K$ is a disjoint union of end-neighborhoods homeomorphic to 
closed $(n-1)$-dimensional orbifolds multiplied by open intervals.
Hence $\partial \orb$ is a compact suborbifold. 
This is a strong assumption; however, we note that the mathematicians have 
great difficulty understanding the topology of the ends of manifolds presently. 
(We apologize for going through definitions for a few pages. 
See \cite{Cbook} for an introduction to the geometric orbifold theory.)

An {\em orbifold covering map} is a map so that for a given modeling open set as above, the inverse 
image is a union of modeling open sets that are quotients as above. 
We say that an orbifold is a manifold if it has a subatlas of charts with trivial local groups. 
We will consider good orbifolds only, i.e., covered by a simply connected manifold. 
In this case, the universal covering orbifold $\torb$ is a manifold
with an orbifold covering map $p_\orb: \torb \ra \orb$. 
The group of deck transformations will be denote by $\pi_1(\orb)$ or $\bGamma$.
They act properly discontinuously on $\torb$ but not necessarily freely. 

By strong tameness, $\orb$ has only finitely many ends $E_1, \ldots, E_m$,
and each end has an end-neighborhood diffeomorphic to $\Sigma_{E_i} \times (0, 1)$.
Let $\Sigma_{E_i}$ here denote the compact orbifold diffeomorphism type of the end $E_i$, 
which is uniquely determined. 
Such end-neighborhoods of these types are said to be of the {\em product types}. 

Each end-neighborhood $U$ diffeomorphic to $\Sigma_{\tilde E} \times (0, 1)$ of an end $E$ lifts to a connected open set 
$\tilde U$ in $\torb$ 
where a subgroup of deck transformations $\bGamma_{\tilde U}$ acts on $\tilde U$ where 
$p_{\torb}^{-1}(U) = \bigcup_{g\in \pi_1(\orb)} g(\tilde U)$. Here, each component of 
$\tilde U$ is said to a {\em proper pseudo-end-neighborhood}.
\begin{itemize} 
\item A {\em pseudo-end sequence} is a sequence of proper pseudo-end-neighborhoods 
$U_1 \supset U_2 \supset \cdots $ so that for each compact subset $K$ of $\torb$
there exists an integer $N$ so that 
$K \cap U_i = \emp$ for $i > N$.  
\item Two pseudo-end sequences are {\em compatible} if an element of one sequence is contained 
eventually in the element of the other sequence. 
\item A compatibility class of a pseudo-end sequence is called a {\em pseudo-end} of $\torb$.
Each of these corresponds to an end of $\orb$ under the universal covering map $p_{\orb}$.
\item For a pseudo-end $\tilde E$ of $\torb$, we denote by $\bGamma_{\tilde E}$ the subgroup $\bGamma_{\tilde U}$ where 
$U$ and $\tilde U$ is as above. We call $\bGamma_{\tilde E}$ is called a {\em pseudo-end fundamental group}.
\item A pseudo-end-neighborhood $U$ of a pseudo-end $\tilde E$ is a $\bGamma_{\tilde E}$-invariant open set containing 
a proper pseudo-end-neighborhood of $\tilde E$. 
\end{itemize}
(See Section \ref{sub:ends} for more detail.)





\subsubsection{Real projective structures on orbifolds.} 
Recall the real projective space $\rpn$ 
is defined as 
$\bR^{n+1} - \{O\}$ under the quotient relation $\vec{v} \sim \vec{w}$ iff $\vec{v} = s\vec{w}$ for $s \in \bR -\{O\}$. 
We denote by $[x]$ the equivalence class of a nonzero vector $x$. 
The general linear group $\GLnp$ acts on $\bR^{n+1}$ and $\PGLnp$ acts faithfully on $\rpn$. 
Denote by $\bR_+ =\{ r \in \bR| r > 0\}$.
The {\em real projective sphere} $\SI^n$ is defined as the quotient of $\bR^{n+1} -\{O\}$ under the quotient relation 
$\vec{v} \sim \vec{w}$ iff $\vec{v} = s\vec{w}$ for $s \in \bR_+$. 
We will also use $\SI^n$ as the double cover of $\rpn$ and $\Aut(\SI^n)$, isomorphic to the subgroup $\SLnp$ of $\GLnp$ of 
determinant $\pm 1$, double-covers $\PGLnp$ and acts as a group of projective automorphisms of $\SI^n$. 
A {\em projective map} of a real projective orbifold to another is a map that is projective by charts to $\rpn$. 
Let $\Pi: \bR^{n+1}-\{O\} \ra \bR P^n$ be a projection and let $\Pi':  \bR^{n+1}-\{O\} \ra \SI^n$ denote one for $\SI^n$. 
An infinite subgroup $\Gamma$ of $\PGLnp$ (resp. $\SLnp$) is {\em strongly irreducible} if every finite-index subgroup is irreducible. 
A {\em subspace} $S$ of $\bR P^n$ (resp. $\SI^n$) is the image of a subspace with the origin removed under the projection $\Pi$ (resp. $\Pi'$).

A line in $\rpn$ or $\SI^n$ is an embedded arc in a $1$-dimensional subspace. 
A {\em projective geodesic} is an arc developing into a line in $\rpn$
or to a one-dimensional subspace of $\SI^n$. 
An affine subspace $A^n$ can be identified with the complement of a codimension-one subspace 
$\rpno$ so that the geodesic structures are same up to parameterizations. 
A {\em convex subset} of $\rpn$ is a convex subset of an affine subspace in this paper. 
A {\em properly convex subset} of  $\rpn$ is a precompact convex subset of an affine subspace. 
$\bR^n$ identifies with an open half-space in $\SI^n$ defined by a linear function on $\bR^{n+1}$. 
(In this paper an affine space is either embedded in $\rpn$ or $\SI^n$.)

An {\em $i$-dimensional complete affine subspace} is 
a subset of a projective manifold projectively diffeomorphic to 
an $i$-dimensional affine subspace in some affine subspace $A^n$ of $\rpn$ or $\SI^n$. 

Again an affine subspace in $\SI^n$ is a lift of an affine space in $\rpn$, 
which is the interior of an $n$-hemisphere. 
Convexity and proper convexity in $\SI^n$ are defined in the same way as in $\rpn$. 
 
We will consider an orbifold $\orb$ with a real projective structure: 
This can be expressed as 
\begin{itemize}
\item having a pair $(\dev, h)$ where 
$\dev:\torb \ra \rpn$ is an immersion equivariant with respect to 
\item the homomorphism $h: \pi_1(\orb) \ra \PGLnp$ where 
$\torb$ is the universal cover and $\pi_1(\orb)$ is the group of deck transformations acting on $\torb$. 
\end{itemize}
$(\dev, h)$ is only determined up to an action of $\PGLnp$ 
given by 
\[ g \circ (\dev, h(\cdot)) = (g \circ \dev, g h(\cdot) g^{-1}) \hbox{ for } g \in \PGLnp. \]
We will use only one pair where $\dev$ is an embedding for this paper and hence 
identify $\torb$ with its image. 
A {\em holonomy} is an image of an element under $h$. 
The {\em holonomy group} is the image group $h(\pi_1(\orb))$.  

Let $x_0, x_1, \dots, x_n$ denote the standard coordinates of $\bR^{n+1}$. 
The interior $B$ in $\rpn$  or $\SI^n$ of a standard ball that is the image of the positive cone of 
$x_0^2 > x_1^2 + \dots + x_n^2$ in $\bR^{n+1}$
can be identified with a hyperbolic $n$-space. The group of isometries of the hyperbolic space 
equals the group $\Aut(B)$ of projective automorphisms acting on $B$. 
Thus, a complete hyperbolic manifold carries a unique real projective structure and is denoted by $B/\Gamma$ for 
$\Gamma \subset \Aut(B)$. 

We also have a lift $\dev': \torb \ra \SI^n$ and $h': \pi_1(\orb) \ra \SLnp$, 
which are also called developing maps and holonomy homomorphisms. 
The discussions below apply to $\rpn$ and $\SI^n$ equally. 
This pair also completely determines the real projective structure on $\orb$. 
Fixing $\dev$, we can identify $\torb$ with $\dev(\torb)$ in $\SI^n$ when $\dev$ is an embedding. 
This identifies $\pi_1(\orb)$ with a group of projective automorphisms $\Gamma$ in $\Aut(\SI^n)$.
The image of $h'$ is still called a {\em holonomy group}.

An orbifold $\orb$ is {\em convex} (resp. {\em properly convex} and {\em complete affine}) 
if $\torb$ is a convex domain (resp. a properly convex domain and an affine subspace.). 

A {\em totally geodesic hypersurface} $A$ in $\torb$ or $\orb$ is 
a subset where each point $p$ in $A$ has a neighborhood $U$ projectively 
diffeomorphic to an open or half-open ball where $A$ corresponds to a subspace of 
codimension-one. 


\begin{description} 
\item[Radial ends:]
We will assume that our real projective orbifold
$\orb$ is a strongly tame orbifold and some of the ends 
are {\em radial}. This means that each end has a neighborhood $U$, and each 
component $\tilde U$ of the inverse image $p_\orb^{-1}(U)$
has a foliation by properly embedded projective geodesics ending at a common point $\bv_{\tilde U} \in \rpn$. 
We call such a point a {\em pseudo-end vertex}. 
\begin{itemize} 
\item The {\em space of directions} of oriented projective geodesics through $\bv_{\tilde E}$ forms \hfil\break
an $(n-1)$-dimensional real projective space. 
We denote it by $\SI^{n-1}_{\bv_{\tilde E}}$, called a {\em linking sphere}. 
\item Let $\tilde \Sigma_{\tilde E}$ denote the space of equivalence classes of lines from $\bv_{\tilde E}$ in $\tilde U$
where two lines are regarded equivalent if they are identical near $\bv_{\tilde E}$. 
$\tilde \Sigma_{\tilde E}$ projects to a convex open domain in an affine space in  $\SI^{n-1}_{\bv_E}$
by the convexity of $\torb$. Then by Proposition \ref{prop:projconv} $\tilde \Sigma_{\tilde E}$ is projectively diffeomorphic to
\begin{itemize}
\item either a complex affine space $A^{n-1}$, 
\item a properly convex domain, 
\item or a convex but not properly convex 
and not complete affine domain in $A^{n-1}$. 
\end{itemize} 
\item The subgroup $\bGamma_{\tilde E}$, a pseudo-end fundamental group, of $\bGamma$ 
fixes $\bv_{\tilde E}$ and  acts on 
as a projective automorphism group on $\SI^n _{\bv_E}$. 
Thus, the quotient $\tilde \Sigma_{\tilde E}/\bGamma_{\tilde E}$ admits a real projective 
structure of one-dimension lower. 
\item We denote by $\Sigma_{\tilde E} $ the real projective $(n-1)$-orbifold $\tilde \Sigma_{E}/\bGamma_{E}$. 
Since we can find a transversal orbifold $\Sigma_{\tilde E}$ to the radial foliation in 
a pseudo-end-neighborhood for each pseudo-end $\tilde E$ of $\mathcal{O}$,
it lifts to a transversal surface $\tilde \Sigma_{\tilde E}$ in $\tilde U$. 
\item We say that a radial pseudo-end $\tilde E$ is  {\em convex} (resp. {\em properly convex}, and {\em complete affine}) 
if $\tilde \Sigma_{\tilde E}$ is convex  (resp. properly convex, and complete affine). 
\end{itemize}

Thus, a radial end is either
\begin{description}
\item[CA:] complete affine, 
\item[PC:] properly convex, or 
\item[NPCC:] convex but not properly convex and not complete affine. 
\end{description}

\item[Totally geodesic ends:] 
An {\em end} is totally geodesic if an end-neighborhood $U$ compactifies to an orbifold with boundary in an ambient orbifold
by adding a totally geodesic suborbifold $\Sigma_E$ homeomorphic to $\Sigma_E \times I$ for an interval $I$. 
The choice of the compactification is said to be the {\em totally geodesic end structure}. 
Two compactifications are equivalent if some respective neighborhoods are projectively diffeomorphic. 
(One can see in \cite{cdcr1} two inequivalent ways to compactify for real projective elementary annulus.)
If $\Sigma_E$ is properly convex, then the end is said to be {\em properly convex}. 
\end{description}
Note that the diffeomorphism types of end orbifolds are determined for radial or totally geodesic ends. 
We will now say that a radial end is a {\rm R-end} and a totally geodesic end is a {\rm T-end}. 




\subsubsection{Horospherical domains, lens domains, lens-cones, and so on.} 

If $A$ is a domain of subspace of $\rpn$ or $\SI^n$, we denote by $\Bd A$ the topological boundary 
in the  the subspace. 
The closure $\clo(A)$ of a subset $A$ of $\rpn$ or $\SI^n$ is the topological closure in $\rpn$ or in $\SI^n$. 
Define $\partial A$ for a manifold or orbifold $A$ to be the {\em manifold or orbifold boundary}. 
Also, $A^o$ will denote the manifold or orbifold interior of $A$. 

\begin{definition}\label{defn:join}
Given a convex set $D$ in $\rpn$, we obtain a connected cone $C_D$ in $\bR^{n+1}-\{O\}$ mapping to $D$,
determined up to the antipodal map. For a convex domain $D \subset \SI^n$, we have a unique domain $C_D \subset \bR^{n+1}-\{O\}$. 

A {\em join} of two properly convex subsets $A$ and $B$ in a convex domain $D$ of $\rpn$ or $\SI^n$ is defined 
\[A \ast B := \{[ t x + (1-t) y]| x, y \in C_D,  [x] \in A, [y] \in B, t \in [0, 1] \} \]
where $C_D$ is a cone corresponding to $D$ in $\bR^{n+1}$. The definition is independent of the choice of $C_D$. 
\end{definition} 

\begin{definition}
Let $C_1, \dots, C_m$ be cone respectively in a set of independent vector subspaces $V_1, \dots, V_m$ of $\bR^{n+1}$. 
In general, a {\em sum} of convex sets $C_1, \dots, C_m$ in $\bR^{n+1}$ 
in independent subspaces $V_i$, we define 
\[ C_1+ \dots + C_m := \{v | v = c_1+ \cdots + c_m, c_i \in C_i \}.\]
A {\em strict join} of convex sets $\Omega_i$ in $\SI^n$ (resp. in $\bR P^n$) is given as 
\[\Omega_1 \ast \cdots \ast \Omega_m := \Pi(C_1 + \cdots C_m) \hbox{ (resp. } \Pi'(C_1 + \cdots C_m)  ) \]
where each $C_i-\{O\}$ is a convex cone with image $\Omega_i$ for each $i$. 
\end{definition}
(The join above does depend on the choice of cones.)

In the following, all the sets are required to be inside an affine subspace $A^n$ and its closure either in $\bR P^n$ or $\SI^n$. 
\begin{itemize}
\item A subdomain and a submanifold $K$ of $A^n$ is said to be a {\em horoball} if it is strictly convex, 
and the boundary $\partial K$ is diffeomorphic to $\bR^{n-1}$ and $\Bd K - \partial K$ is a single point. 
The boundary $\partial K$ is said to be a {\em horosphere}. 
\item $K$ is {\em lens-shaped} if it is a convex domain and $\partial K$ is a disjoint union of two smoothly strictly convex embedded 
$(n-1)$-cells $\partial_+ K$ and $\partial K_-$. 
\item A {\em cone} is a domain in $A^n$ whose closure in $\rpn$ has a point in the boundary, called an {\em end vertex}
so that every other point has a segment contained in the domain with endpoint the cone point and itself. 
\item A {\em cone} $\{p\} \ast L$ over a lens-shaped domain $L$ in $A^n$, $p\not\in \clo(L)$ is a convex domain
so that $\{p\}\ast L = \{p\} \ast \partial_+ L$ for one boundary component $\partial_+ L$ of $L$. 
A {\em lens} is the lens-shaped domain $L$, not determined uniquely by the lens-cone itself.
\item We can allow $L$ to have non-smooth boundary that lies in the boundary of $p \ast L$. 
\begin{itemize}
\item One of  two boundary components of $L$ is called a {\em top} or {\em bottom} hypersurfaces 
depending on whether it is further away from $p$ or not. The top component is denoted by $\partial_+ L$
which can be not smooth. $\partial_-L$ is required to be smooth. 
\item A cone over $L$ where $\partial (\{p\} \ast L -\{p\}) = \partial_+ L, p \not\in \clo(L)$ is said to be a {\em generalized lens-cone} 
and $L$ is said to be a {\em generalized lens}.
\item A quasi-lens cone is a properly convex cone of form $p\ast S$ for a strictly convex hypsersurface $S$ 
so that $\partial (\{p\} \ast S -\{p\})= S$ and $p \in \clo(S) - S$
and the space of directions from $p$ to $S$ is a properly convex domain in $\SI^{n-1}_p$.
\end{itemize}
\item A {\em totally-geodesic domain} is a convex domain in a hyperspace. 
A {\em cone-over} a totally-geodesic domain $D$ is a union of all segments with  one end point a point $x$ not in the hyperspace
and the other in $D$. We denote it by $\{x\} \ast D$. 
\end{itemize} 

Let the radial pseudo-end $\tilde E$ have a pseudo-end-neighborhood of form $\{p\} \ast L -\{p\}$ that is 
a generalized lens-cone $p \ast L$ over a generalized lens $L$
where $\partial (p\ast L -\{p\}) = \partial_+L$. 
A {\em concave pseudo-end-neighborhood} of $\tilde E$ is the open pseudo-end-neighborhood in $\torb$ contained in 
a radial pseudo-end-neighborhood in $\tilde{\mathcal{O}}$ that is a component 
of $\{p\}\ast L -\{p\} - L$ containing $p$ in the boundary. 
As it is defined, such a pseudo-end-neighborhood always exists for a generalized lens pseudo-end. 


From now on, we will replace the term ``pseudo-end'' with ``p-end'' everywhere. 

\begin{description}
\item [Horospherical R-end:] A pseudo-R-end of $\torb$ is {\em horospherical} if it 
has a horoball in $\torb$ as a pseudo-end-neighborhood, or equivalently an open pseudo-end-neighborhood $U$ in $\torb$ 
so that $\Bd U \cap \torb = \Bd U - \{v\}$ for a boundary fixed point $v$ 
where the p-end fundamental group properly discontinuously on. 
\item[Lens-shaped R-end:]  An R-end is {\em lens-shaped} (resp. {\em totally geodesic cone-shaped},
{\em generalized lens-shaped}, {\em quasi-lens shaped}) 
 if it has a pseudo-end-neighborhood that is a lens-cone (resp. a cone over a totally-geodesic domain, 
 a concave pseudo-end-neighborhood, or a quasi-lens cone.) 
\item[Lens-shaped T-end:] A pseudo-T-end of $\torb$ is of {\em lens-type} if it has a lens p-end-neighborhood in 
an ambient orbifold of $\torb$. A T-end of $\orb$ is of {\em lens-type} if 
the corresponding pseudo-end is of lens-type. 
\end{description}

\subsection{Main results.} 
We will later see that horospherical end-neighborhoods are projectively diffeomorphic to 
horospherical end-neighborhoods of hyperbolic orbifolds. 
Let $\tilde E$ be a p-end and $\bGamma_{\tilde E}$ the associated p-end fundamental group.
If every subgroup of finite index of a group $\bGamma_{\tilde E} \subset \bGamma$ has a finite center, 
we say that $\bGamma_{\tilde E}$ 
is a {\em virtual center-free group} or  a {\em vcf-group}.
An {\em admissible group} is a finite extension of a finite product group
$\bZ^{k-1} \times \Gamma_1 \times \cdots \times \Gamma_k$ 
for trivial or infinite hyperbolic groups $\Gamma_i$ in the sense of Gromov. 
(See Section \ref{sub:ben} for details. In this paper, we will simply use $\bZ^{k-1}$ and $\Gamma_i$ to denote 
the subgroup in $\bGamma_{\tilde E}$ corresponding to it.)

Let $\Gamma$ be generated by finitely many elements $g_1, \ldots, g_m$. 
The {\em conjugate word length} $\cwl(g)$ of $g \in \pi_1(\tilde E)$ is the minimum of 
the word length of the conjugates of $g$ in $\pi_1(\tilde E)$. 


Let $\Omega$ be a convex domain in an affine space $A$ in $\bR P^n$ or $\SI^n$. 
Let $[o, s, q, p]$ denote the cross ratio of four points as defined by 
\[ \frac{\bar o - \bar q}{\bar s - \bar q} \frac{\bar s - \bar p}{\bar o - \bar p} \] 
where $\bar o, \bar p, \bar q, \bar s$ denote respectively the first coordinates of the homogeneous coordinates  
$o, p, q , s$ in a $1$-dimensional subspace so that the second coordinates equal $1$. 
Define a metric
$d_\Omega(p, q)= \log|[o,s,q,p]|$ where $o$ and $s$ are 
endpoints of the maximal segment in $\Omega$ containing $p, q$
where $o, q$ separated $p, s$. 
The metric is one given by a Finsler metric provided $\Omega$ is properly convex. (See \cite{Kobpaper}.)
Given a properly convex real projective structure on ${\mathcal{O}}$, it carries a Hilbert metric which we denote by $d_{\torb}$ 
on $\tilde{\mathcal{O}}$ and hence on $\tilde {\mathcal{O}}$. 
This induces a metric on ${\mathcal{O}}$. 
(Note that even if $\torb$ is not properly convex, $d_\Omega$ is still a pseudo-metric that is useful.) 

Let $d_K$ denote the Hilbert metric of the interior $K^o$ of a properly convex domain $K$ in $\bR P^n$ or $\SI^n$. 
Suppose that a projective automorphism $g$ acts on $K$. 
Let $\leng_K(g)$ denote the infinum of $\{ d_K(x, g(x))| x \in K^o\}$, compatible with $\cwl(g)$. 

An {\em ellipsoid} is a subset in an affine space defined as a zero locus of a positive definite quadratic polynomial in term of 
the affine coordinates. 
A projective conjugate  ${\mathcal H}_{\bv}$
 of a parabolic subgroup of $\SO( i_0+1, 1)$ acting cocompactly on $E - \{\bv\}$ for an $i_0$-dimensional ellipsoid $E$ 
containing the point $\bv$ 
 is called an {\em $i_0$-dimensional cusp group}. 
If the horospherical neighborhood with the p-R-end vertex $\bv$ 
has the p-end fundamental group that is a discrete cocompact subgroup in ${\mathcal H}_{\bv}$, then we call 
the p-R-end to be of {\em cusp type}. 

Our first main result classifies CA p-R-ends. 

\begin{theorem}\label{thm:mainaffine} 
Let $\mathcal O$ be a properly convex real projective $n$-orbifold with radial or totally geodesic ends. 
Let $\tilde E$  be a p-R-end of its universal cover $\torb$. 
Then $\tilde E$ is a complete affine p-R-end if and only if $\tilde E$ is a cusp p-R-end. 
\end{theorem}
\begin{proof}
Theorem \ref{thm:comphoro} implies that a complete end is of cusp type. 
Since a cusp end is horospherical, 
Proposition \ref{prop:affinehoro} implies the converse. 
\end{proof}


We will learn later that every norm of the eigenvalues $\lambda_i(g) = 1$, $g \in \bGamma_{\tilde E}$ 
if and only if $\tilde E$ is horospherical by Proposition \ref{prop:affinehoro} and Theorem \ref{thm:comphoro}. 
Thus, in these cases, we say that $\tilde E$ satisfies the {\em uniform middle eigenvalue condition} always. 

A subset $A$ of $\bR P^n$ or $\SI^n$ {\em span} a subspace $S$ if $S$ is the smallest subspace containing $A$. 

The following definition applies to properly convex R-ends. However, we will generalize this 
to NPCC ends in Definition \ref{defn:NPCC} in Part III. 

\begin{definition}\label{defn:umec}
Let $\bv_{\tilde E}$ be a p-end vertex of a p-R-end $\tilde E$. 
We assume that $\bGamma_{\tilde E}$ is admissible and the associated real projective orbifold $\Sigma_{\tilde E}$ is properly convex.
We assume that 
$\bGamma_{\tilde E}$ acts on a strict join $\clo(\tilde \Sigma_{\tilde E}) = K := K_1 * \cdots * K_{l_0}$ in $\SI^{n-1}_{\bv_{\tilde E}}$
where $K_j$ is a properly convex compact domain in a projective sphere $\SI^{i_j}$ of dimension $i_j \geq 0$.
Thus, $\bGamma_{\tilde E}$ restricts to a semisimple hyperbolic group $\Gamma_j$ acting on $K_j$ 
for some $j =1, \dots, l_0$ 
and also contains the central abelian group $\bZ^{l_0-1}$. The admissibility implies that $\Gamma_j$ is a hyperbolic group
and 
\[ \bGamma_{\tilde E} \cong \bZ^{l_0-1}\times \Gamma_1 \times \cdots \times \Gamma_{l_0}.\]
Let $\hat K_i$ denote the subspace spanned by $\bv_{\tilde E}$ and the segments from $\bv_{\tilde E}$ in the direction of $K_i$. 
The p-end fundamental group $\bGamma_{\tilde E}$ satisfies the {\em uniform middle-eigenvalue condition} if 
each $g\in \bGamma_{\tilde E}$ satisfies for a uniform  $C> 0$ independent of $g$
\begin{equation}\label{eqn:umec}
C^{-1} \leng_K(g) \leq \log\left(\frac{\bar\lambda(g)}{\lambda_{\bv_{\tilde E}}(g)}\right) \leq C \leng_K(g) , 
\end{equation}
for $\bar \lambda(g)$ equal to 
\begin{itemize}
\item the largest norm of the eigenvalues of $g$ which must occur for a fixed point of $\hat K_i$ if $g \in \Gamma_i$
\end{itemize} 
and the eigenvalue $\lambda_{\bv_{\tilde E}}(g)$ of $g$ at $\bv_{\tilde E}$.

If we require only \[\bar \lambda(g) \geq \lambda_{\bv_{\tilde E}}(g) \hbox{ for } g \in \bGamma_{\tilde E},\] 
and the uniform middle eigenvalue condition for each hyperbolic $\Gamma_i$, 
then we say that $\bGamma_{\tilde E}$ satisfies the {\em weakly uniform middle-eigenvalue conditions}. 
\end{definition}
The definition of course applies to the case when $\bGamma_{\tilde E}$ has the finite index subgroup with the above properties.

We give a dual definition: 
\begin{definition} 
Suppose that $\tilde E$ is a properly convex p-T-end. 
Then let $\bGamma_{\tilde E}^*$ acts on a point $\bv^*_{\tilde E} \in \bR P^{n \ast}$
corresponding to $\tilde \Sigma_{\tilde E}$ with the eigenvalue to be denoted $\lambda_{\bv_{\tilde E}}$.
Let $g^*:\bR^{n+1 \ast} \ra \bR^{n+1 \ast}$ be the dual transformation of $g: \bR^{n+1} \ra \bR^{n+1}$. 
Assume that $\bGamma_{\tilde E}$ acts on a properly convex compact domain $K = \clo(\tilde \Sigma_{\tilde E})$
and $K$ is a strict join $K := K_1 * \cdots * K_{l_0}$. Defining $\Gamma_i$ as above. 
The p-end fundamental group $\bGamma_{\tilde E}$ satisfies the {\em uniform middle-eigenvalue condition}
if it satisfies 
\begin{itemize}
\item the equation \ref{eqn:umec} for the largest norm $\bar \lambda(g)$ 
of the eigenvalues of $g$ which must occur for a fixed point of $K_i$ if $g \in \Gamma_i$ and 
\end{itemize} 
the eigenvalue $\lambda_{\bv_{\tilde E}}(g)$ of $g^*$ in the vector in the direction of $\bv_{\tilde E}^*$.
\end{definition} 

Here $\bGamma_{\tilde E}$ will act on a properly convex domain $K^o$ of lower-dimension
and we will apply the definition here. 
This condition is similar to ones studied by Guichard and Wienhard \cite{GW}, and the results also 
seem similar. Our main tools to understand these questions are in Appendix \ref{app:dual}, and 
the author does not really know the precise relationship here.) 



The condition is an open condition; and hence a ``structurally stable one."
(See Corollary \ref{cor:mideigen}.)

Our second main result is: 
\begin{theorem}\label{thm:secondmain} 
Let $\mathcal{O}$ be a strongly tame properly convex real projective orbifold with radial or totally geodesic ends.
Each end fundamental group is virtually isomorphic to a direct product of hyperbolic groups 
and infinite cyclic groups. Assume that the holonomy group of $\mathcal{O}$ is strongly irreducible.
\begin{itemize} 
\item Let $\tilde E$ be a properly convex  p-R-end. 
\begin{itemize} 
\item Suppose that the p-end holonomy group satisfies the uniform middle-eigenvalue condition
Then $\tilde E$ is of generalized lens-type.
\item Suppose that the p-each end holonomy group satisfies the weakly uniform middle-eigenvalue condition.
Then $\tilde E$ is of generalized lens-type or of quasi-lens-type .
\end{itemize} 
\item If $\orb$ satisfies the triangle condition  or $\tilde E$ is reducible or is a totally geodesic R-end, 
then we can replace the word ``generalized lens-type''
to ``lens-type'' in each of the above statements. 
\item Let $\tilde E$ be a totally geodesic end. 
If $\tilde E$ satisfies the uniform middle-eigenvalue condition, 
then $\tilde E$ is of lens-type. 
\end{itemize} 
\end{theorem}

\begin{theorem}\label{thm:thirdmain} 
Let $\mathcal{O}$ be a strongly tame properly convex real projective orbifold with radial or totally geodesic ends.
Assume that the holonomy group of $\mathcal{O}$ is strongly irreducible.
Each end fundamental group is virtually isomorphic to a direct product of hyperbolic groups.
Let $\tilde E$ be an NPCC p-R-end. 
Suppose that the p-each end holonomy group satisfies the weakly uniform middle-eigenvalue condition.
Then $\tilde E$ is of quasi-joined type p-R-end. 
\end{theorem}

Joined-ends and quasi-joined ends do not satisfy the uniform middle-eigenvalue condition by construction. 
The above three theorems directly imply the following main result of this paper: 
\begin{corollary} \label{cor:main} 
Let $\mathcal{O}$ be a strongly tame properly convex real projective orbifold with radial or totally geodesic ends.
Assume that the holonomy group of $\mathcal{O}$ is strongly irreducible.
Each end fundamental group is virtually isomorphic to a direct product of hyperbolic groups.
Suppose that the each end holonomy group satisfies the uniform middle-eigenvalue condition.
Then each end is a lens-type R-end, an R-end of cusp type, or a lens-type T-end.
\end{corollary}

We will explain the quasi-joined type in Section \ref{sub:quai-lens}
and prove these in Section \ref{sec:secondmain}. 

Our work is a ``classification'' since 
we will show how to construct lens-type R-ends (Theorem \ref{thm:equ}), quasi-lens-type R-ends 
(Propositions \ref{prop:quasilens1},  \ref{prop:quasilens2}), 
lens-type T-ends (Theorem \ref{thm:equ2}), and quasi-joined NPCC R-ends
(Example \ref{exmp:joined},  Theorems \ref{thm:NPCCcase} and \ref{thm:NPCCcase2}) 
in a reasonable sense. (Of course, provided we know how to 
compute certain cohomology groups.)


\begin{remark}
A summary of the deformation spaces of real projective structures on closed orbifolds and surfaces is given 
in \cite{Cbook} and \cite{Choi2004}. See also Marquis \cite{Marquis} for the end theory of $2$-orbifolds. 
The deformation space of real projective structures on an orbifold 
loosely speaking is the space of isotopy equivalent real projective structures on
a given orbifold. (See \cite{conv} also.) 
\end{remark}

Let $\bR P^{n \ast}={\mathcal P}(\bR^{n+1 \ast})$ be the dual real projective space of $\bR P^n$. 
In Section \ref{sec:endth}, we define the projective dual domain $\Omega^*$ in $\bR P^{n \ast}$ 
to a properly convex domain $\Omega$ in $\bR P^n$ where 
the dual group $\Gamma^*$ to $\Gamma$ acts on. 
We show that 
an R-end corresponds to a T-end and vice versa. 
(See Section \ref{sub:dualend}.)


\begin{remark}[(Duality of ends)] \label{rem:duality} 
Above orbifold $\orb=\torb/\Gamma$ has a diffeomorphic dual orbifold $\orb^*$ defined as the quotient of the dual domain 
$\torb^*$ by the dual group of $\Gamma$ by Theorem \ref{thm:dualdiff}.
The ends of $\orb$ and $\orb^*$ are in a one-to-one correspondence. 
Horospherical ends are dual to themselves, i.e., ``self-dual'', 
and properly convex R-ends and T-ends are dual to one another. (See Proposition \ref{prop:dualend}.)
We will see that properly convex R-ends of generalized lens type 
are always dual to T-ends of lens type by Proposition \ref{prop:dualend2}, Theorem \ref{thm:equ}, 
and Theorem \ref{thm:equ2}. 
\end{remark}

\begin{remark}[(Self-dual reducible ends)] \label{rem:red} 
A generalized lens-type reducible properly convex R-end is always 
totally geodesic by Theorem \ref{thm:redtot} and 
Theorem \ref{thm:secondmain}, the R-end is of lens-type always. 
The dual end is totally geodesic of lens type since it satisfies the uniform middle eigenvalue condition. 
The end can be made into a totally geodesic radial one since it fixes a unique point dual to the totally
geodesic ideal boundary component and by taking a cone over that point. 
Thus, the reducible properly convex ends are ``self-dual''. Thus, we consider these the model cases.
\end{remark}

\subsubsection{Some motivations.} 
To motivate why we think that these results are important, we sketch some history: 
It was discovered by D. Cooper, D. Long, and M. Thistlethwaite \cite{Cooper2006}, \cite{CLT} that 
many closed hyperbolic $3$-manifolds deform to real projective $3$-manifolds. 
Later S. Tillmann found an example of a $3$-orbifold obtained from pasting sides of a single ideal hyperbolic tetrahedron 
admitting a complete hyperbolic structure with cusps and a one-parameter family of 
real projective structure deformed from the hyperbolic one (see \cite{conv}).
Also, Craig Hodgson, Gye-Seon Lee, and I found a few other examples: $3$-dimensional ideal hyperbolic Coxeter orbifolds 
without edges of order $3$ has at least $6$-dimensional deformation spaces in \cite{CHL}. 

Crampon and Marquis \cite{CM2} and Cooper, Long, and Tillmann \cite{CLT2} have done similar study with 
the finite volume condition. In this case, only possible ends are horospherical ones. 
The work here studies more general type ends 
but we have benefited from their work. 
We will see that there are examples where horospherical ends deform to lens-type ones and vice versa
( see also Example \ref{exmp:Lee}.)

Our main aim is to understand these phenomena theoretically. It became clear from 
our attempt in \cite{conv} that we need to understand and classify the types of 
ends of the relevant convex real projective orbifolds. We will start with the simplest ones: 
radial type ones. But as Mike Davis observed, there are many other types such 
as ones preserving subspaces of dimension greater than equal to $0$. 
In fact Daryl Cooper found some such an example from 
$S/\SL(3, \bZ)$ for the space $S$ of unimodular positive definite bilinear forms. 
We will not present any of them here; 
however, it seems very likely that many techniques here will be applicable. 

In \cite{conv}, we show that the deformation spaces of real projective structures on orbifolds are locally homeomorphic to 
the spaces of conjugacy classes of representations of their fundamental groups 
where both spaces are restricted by some end conditions. 

It remains how to see for which of these types of real orbifolds, 
nontrivial deformations exist or not for a given example such as a complete hyperbolic manifolds 
and how to compute the deformation space. 
We conjecture that maybe these types of real projective orbifolds with R-ends might be very flexible. 
Of course, we have no real analytical or algebraic means to understand these phenomena yet.
However, we do have some class of examples such as Theorem 1 of \cite{CHL} and results in \cite{Choi2006}. 

\subsection{Outline.} 
There are three parts: 
\begin{itemize}
\item[(I)] The preliminary review and examples. We discuss some parts on duality and finish our work on complete ends. 
\item[(II)] We discuss properly convex R-ends and T-ends.
\item[(III)] We discuss NPCC R-ends.
\end{itemize} 

Part I: 
In Section \ref{sec:prelim}, we go over basic definitions. We discuss ends of orbifolds, convexity, the Benoist theory on convex divisible actions, 
and so on. 

In Section \ref{sec:exend}, we discuss objects associated with ends
and examples of ends; horospherical ones, totally geodesic ones, 
and bendings of ends.

In Section \ref{sec:duality} we discuss the dual orbifolds of a given convex real projective orbifold.

In Section  \ref{sec:horo}, 
we discuss about horospherical ends. First, they are complete ends and 
have holonomy matrices with unit norm eigenvalues only
and their end fundamental groups are virtually abelian. 
Conversely, a complete end in 
a properly convex orbifold has to be a horospherical end. 

We begin the part II:

In Section \ref{sec:endth}, we start to study the end theory. First, we discuss the holonomy representation spaces.
Tubular actions and the dual theory of affine actions are discussed. We show that distanced actions
and asymptotically nice actions are dual. We prove that the uniform middle eigenvalue condition 
implies the existence of the distanced action. 

In Section \ref{sub:dualend}, we show that the dual orbifold is diffeomorphic to the original one
by the Vinberg's work. We obtain a one-to-one correspondence between ends of a dual orbifold
and the ends of the original one. Next, we showed the horospherical ends are dual to horospherical ones. 
Properly convex R-ends are dual to T-ends and vice versa.


In Section \ref{subsec:lens}, 
we discuss the properties of lens-shaped ends. We show that if the holonomy is irreducible, 
the lens shaped ends have concave neighborhoods. If the lens-shaped end is reducible, then 
it can be made into a totally-geodesic R-end of lens type, which is a surprising result in the author's opinion. 

In Section \ref{sec:chlens}, 
we show that the uniform middle-eigenvalue condition of a properly convex end is equivalent to the 
lens-shaped property of the end under some assumptions. In particular, this is true for 
reducible properly convex ends. This is a major section with numerous central lemmas. 

In Section \ref{sec:results}, we prove many results we need in anther paper \cite{conv}, 
not central to this paper. 
We show that the lens shaped property is a stable property under 
the change of holonomy representations. We obtain the exhaustion by a sequence of p-end-neighborhoods 
of $\torb$. 

Now to the final part III: 
In Section \ref{sec:notprop}, we discuss the R-ends that are NPCC. 
First, we show that the end holonomy group for an end $E$ will have an exact sequence 
\[ 1 \ra N \ra h(\pi_1(\tilde E)) \longrightarrow N_K \ra 1\] 
where $N_K$ is in the projective automorphism group $\Aut(K)$ of a properly convex compact set $K$ 
and $N$ is the normal subgroup mapped to the trivial automorphism of $K$
and $K^o/N_K$ is compact. 
We show that $\Sigma_{\tilde E}$ is foliated by complete affine spaces of dimension $\geq 1$. 

In Section \ref{sec:discrete}, we discuss the case when $N_K$ is a discrete. Here, $N$ is virtually abelian 
and is conjugate to a discrete cocompact subgroup of a cusp group. 
We introduce the example of joining of horospherical and concave type ends. 
By computations involving the normalization conditions, 
we show that the above exact sequence is virtually split and we can surprisingly show that 
the R-ends are of join or quasi-join types. 

In Section \ref{sec:indiscrete}. we discuss the case when $N_K$ is not discrete. Here, there is a foliation 
by complete affine spaces as above. The leaf closures are submanifolds $V_l$ 
by the theory of Molino \cite{Mol} on Riemannian foliations. 
We use some estimate to show that each leaf 
is of polynomial growth. This shows that the identity component of the closure of
$N_K$ is abelian and $\pi_1(V_l)$ is solvable using the work of Carri\`ere \cite{Car}. 
One can then take the syndetic closure to 
obtain a bigger group that act transitively on each leaf. 
We find a normal cusp group acting on each leaf transitively. Then we show that 
the end also splits virtually.

Finally for both of these cases, we show that the orbifold has to be reducible
by considering the limit actions of some elements in the joined ends. 
This proves that the joined end does not exist, proving Theorem \ref{thm:secondmain}
in Section \ref{sec:secondmain}.

In Appendix \ref{app:dual}, we show that the affine action of irreducible group $\Gamma$ acting cocompactly 
on a convex domain $\Omega$ in the boundary of the affine space is 
asymptotically nice if $\Gamma$ satisfies the uniform middle-eigenvalue condition. 
This was needed in Section  \ref{sec:notprop}. 

\begin{remark}
Note that the results are stated in the space $\SI^n$ or $\bR P^n$. Often the result for $\SI^n$ implies 
the result for $\bR P^n$. In this case, we only prove for $\SI^n$. In other cases, we can easily modify 
the $\SI^n$-version proof to one for the $\bR P^n$-version proof. We will say this in the proofs. 
\end{remark}

\subsection{Acknowledgements} 
We thank David Fried for helping me understand the issues with distanced nature of the tubular actions and duality
and Yves Carri\`ere with the general approach to study the indiscrete cases for nonproperly convex ends. 
The basic Lie group approach of Riemannian foliations was a key idea here as well as the theory of 
Fried on distal groups. 
We thank Yves Benoist with some initial discussions on this topic, which were very helpful
for Section \ref{sub:holfib} and thank Bill Goldman and Francois Labourie 
for discussions resulting in Appendix \ref{sub:asymnice}.
We thank Daryl Cooper and Stephan Tillmann 
for explaining their work and help and we also thank Micka\"el Crampon and Ludovic Marquis also. 
Their works obviously were influential 
here. The study was begun with a conversation with Tillmann at ``Manifolds at Melbourne 2006" 
and I began to work on this seriously from my sabbatical at Univ. Melbourne from 2008. 
We also thank Craig Hodgson and Gye-Seon Lee for working with me with many examples and their 
insights. The idea of R-ends comes from the cooperation with them.




\part{Preliminaries and the characterization of complete ends} 

\section{Preliminaries} \label{sec:prelim}

In this paper, we will be using the smooth category: that is, we will be using smooth maps and smooth charts and so on. 
We explain the material in the introduction again.  
We will establish that the universal cover $\torb$ of our orbifold $\orb$ is a domain 
in $\SI^n$ with a projective automorphism group $\Gamma \subset \SLnp$ acting on it. 
In this case, $\orb$ is projectively diffeomorphic to $\torb/\Gamma$. 

\subsection{Distances used} 

\begin{definition}\label{defn:Haus} 
Let $\bdd$ denote the standard spherical metric on $\SI^n$ {\rm (}resp. $\bR P^n${\rm )}.
Given two compact subsets $K_1$ and $K_2$ of $\SI^n$ {\rm (}resp. $\bR P^n${\rm ),}
we define the spherical distance $\bdd_H(K_1, K_2)$ between $K_1$ and $K_2$ to be 
\[ \inf\{\eps > 0| K_2 \subset N_\eps(K_1), K_1 \subset N_\eps(K_2)  \}.\] 
The simple distance $\bdd(K_1, K_2)$ is defined as 
\[ \inf\{ \bdd(x, y)| x \in K_1, K_2 \}.\]
\end{definition}

Recall that every sequence of compact sets $\{K_i\}$ in $\SI^n$ (resp. $\bR P^n$) has
a convergent subsequence. 
Also, given a sequence $\{K_i\}$ of compact sets, 
$\{K_i\} \ra K$ for a compact set $K$ if and only if every sequence of points $x_i \in K_i$ has limit points in $K$ only
and every point of $K$ has a sequence of points $x_i \in K_i$ converging to it. 
(These facts can be found in some topology textbooks.) 

\subsection{Real projective structures}

Let $O$ denote the origin of any vector space here.
Given a vector space $V$, we denote by ${\mathcal P}(V)$ the projective space 
$(V -\{O\})/\sim$ where $\vec{v} \sim \vec{w}$ iff $\vec{v} = s \vec{w}$ for $s \in \bR -\{0\}$ 
and we denote by ${\mathcal S}(V)$ the sphere $(V-\{O\})/\sim$ where $\vec{v} \sim \vec{w}$ for $s \in \bR_+$. 
We denote $\rpn = {\mathcal P}(\bR^{n+1})$ and $\SI^n = {\mathcal S}(\bR^{n+1})$. 
A {\em subspace} of ${\mathcal P}(V)$ or ${\mathcal S}(V)$ is the image of a subspace in $V$ with $O$ removed.
Given any linear isomorphism $f: V \ra W$, we denote by ${\mathcal P}(f)$ the induced 
projective isomorphism ${\mathcal P}(V) \ra {\mathcal P}(W)$ and ${\mathcal S}(f)$ the induced map 
${\mathcal S}(V) \ra {\mathcal S}(W)$. These maps are called {\em projective maps}. 

The complement of a codimension-one subspace $W$ in $\rpn$ can be considered an affine 
space $A^n$ by correspondence 
\[[1, x_1, \dots, x_n] \ra (x_1, \dots, x_n)\] for a coordinate system where $W$ is given by $x_0=0$. 
The group $\Aff(A^n)$ of projective automorphisms acting on $A^n$ is identical with 
the group of affine transformations of form 
\[ \vec{x} \mapsto A \vec{x} + \vec{b} \] 
for a linear map $A: \bR^n \ra \bR^n$ and $\vec{b} \in \bR^n$. 
The projective geodesics and the affine geodesics agree up to parametrizations.

A cone $C$ in $\bR^{n+1} -\{O\}$ is a subset so that given a vector $x \in C$, 
$s x \in C$ for every $s \in \bR_+$. 
A {\em convex cone} is a cone that is a convex subset of $\bR^{n+1}$ in the usual sense. 
A {\em proper convex cone} is a convex cone not containing a complete affine line. 


Note that we can double-cover $\rpn$ by $\SI^n$ the unit sphere in $\bR^{n+1}$ 
and this induces a real projective structure on $\SI^n$. 

We can think of $\SI^n$ as ${\mathcal S}(\bR^{n+1})$.
We call this the real projective sphere. The antipodal map
\[\mathcal{A}: \SI^n \ra \SI^n \hbox{ given by } [\vec{v}] \ra [-\vec{v}] \hbox{ for } \vec{v} \in \bR^{n+1} -\{O\}\]
which generates the covering automorphism group of $\SI^n \ra \rpn$.
The group $\Aut(\SI^n)$ of projective automorphisms of $\SI^n$ is isomorphic to $\SLnp$.

A {\em great segment} is a geodesic segment with antipodal end vertices, which is convex but not properly convex. 
A segment has $\bdd$-length $=\pi$ if and only if it is a great segment. 

Given a projective structure where $\dev: \torb \ra \rpn$ is an embedding to a properly convex 
open subset as in this paper, 
$\dev$ lifts to an embedding $\dev': \torb \ra \SI^n$ to an open domain $D$ without any pair of antipodal 
points. $D$ is determined up to $\mathcal{A}$. 

We will identify $\torb$ with $D$ or $\mathcal{A}(D)$ and 
$\pi_1(\orb)$ or $\bGamma$ lifts to a subgroup $\bGamma'$ of $\SLnp$ acting faithfully and discretely on $\torb$.
There is a unique way to lift so that $D/\bGamma'$ is projectively diffeomorphic to $\torb/\bGamma$. 
Thus, we also define the p-end vertices of p-R-ends of $\torb$ as points in
the boundary of $\torb$ in $\SI^n$ from now on. (see  \cite{conv}.)  



\subsubsection{Ends} \label{sub:ends}

Suppose that $\mathcal{O}$ is a strongly tame
properly convex real projective orbifold with radial or totally geodesic ends and a universal cover $\tilde{\mathcal{O}}$
with compact boundary $\partial \orb$ and some ends. (This will be the universal assumption for this paper.)

Consider a sequence of open sets $U_1, U_2, \ldots$ so that $U_i \supset U_{i+1}$ where 
each $U_i$ is a component of the complement of a compact subset in $\orb$ so that $\clo(U_i)$ is not compact, 
and given each compact set $K$ in $\orb$, $U_i \cap K \ne \emp$ for only finitely many $i$. 
Such a sequence is said to be an {\em  end-neighborhood system}. 
Two such sequences $\{U_1, U_2, \ldots \}$ 
and $\{U'_1, U'_2, \ldots \}$ are equivalent if 
for each $U_i$ we find $k$ so that $U'_j \subset U_i$ for $j > k$ 
and conversely for each $U'_i$ we find $k'$ such that $U_j \subset U'_i$ for $j > k'$. 
An equivalence class of end-neighborhoods is said to be an {\em end} of $\orb$.
A {\em neighborhood} of an end is one of the open set in the sequence in the equivalence class
of the end. 


A {\em radial end-neighborhood system} of $\orb$ is the union of end-neighborhoods for all ends disjoint from 
one another where each end-neighborhood is of product type and is radially foliated
compatibly with the product structure. 


Given a component of such a system, we obtain that the inverse image is a disjoint 
union of connected open sets where we have the subgroup acting on one of them $\tilde U$ denoted by 
$\bGamma_{\tilde U}$ so that $\tilde U/\bGamma_{\tilde U}$ is homeomorphic to the product end-neighborhood. 
$\tilde U$ is also foliated by lines ending at a common vertex $\bv_{\tilde U}$. 
$\tilde U$ is said to be a {\em proper p-end-neighborhood}.
Note that any other component $\tilde U'$ is of form $\gamma(\tilde U)$ for $\gamma \in \bGamma - \bGamma_{\tilde U}$ 
and $\bGamma_{\tilde U'} = \gamma \bGamma_{\tilde U} \gamma^{-1}$ and $\bv_{\tilde U'} = \gamma(\bv_{\tilde U})$. 

By an abuse of terminology, 
an open set $\tilde U'$ containing $\tilde U$ as above and where $\bGamma_{\tilde U}$ acts on
will be called a {\em p-end-neighborhood} of the {\em p-end vertex} $\bv_{\tilde U}$. 
$\tilde U'$ is not required to cover an open set in $\mathcal{O}$. 
Here, $\tilde U'$ may be not be a neighborhood in topological sense as in the cases of horospherical ends. 
We call $\bGamma_{\tilde U}$ the p-end fundamental group. 
Up to the $\bGamma$-action, there are only finitely many p-end vertices and p-end fundamental groups. 
For an end $E$, $\bGamma_U$ is well-defined up to conjugation by $\bGamma$ 
and we denote it by $\bGamma_{\tilde E}$ often for suitable choice of $\tilde U$. 
Its conjugacy class is more appropriately denoted $\bGamma_{\tilde E}$. 

From now on, we will use the term ``p-end" instead of the term ``pseudo-end" everywhere. 


\begin{lemma}\label{lem:inv}
Suppose that $\mathcal{O}$ is a strongly tame properly convex real projective orbifold 
with radial or totally geodesic ends and a universal cover $\tilde{\mathcal{O}}$
with $\partial \orb$ compact. Let $U$ be an end-neighborhood. 
Let $\tilde U$ be  the inverse image of the union $U$ of mutually disjoint end-neighborhoods. 
For a given component $U_1$ of $\tilde U$, 
if $\gamma(U_1)\cap U_1 \ne \emp$, then $\gamma(U_1)=U_1$ and 
$\gamma$ lies in the fundamental group $\bGamma_{E'}$ of the p-end $E'$ associated with $U_1$. 
\end{lemma}
\begin{proof} 
This follows since $U_1$ covers an end-neighborhood. 
\end{proof}

\begin{lemma}\label{lem:endcover} 
Let $U$ be a p-R-end-neighborhood of a p-end vertex $\bv_{\tilde E}$
with $\Bd U \cap \torb$ meeting each open great segment with endpoints $\bv_{\tilde E}$ and $\bv_{\tilde E-}$ uniquely. 
Suppose that the boundary $\Bd U \cap \torb$ of $U$ maps into  
an end-neighborhood of $\orb$ under the covering map or equivalently $\clo(U) \cap \torb$ maps into
the end-neighborhood. 
Then $\Bd U \cap \torb$ covers a compact hypersurface homotopy equivalent to the end orbifold
$\Sigma_{\tilde E}$ and its end-neighborhood
and $p_{\orb}(U)$ is homeomorphic to $\Sigma_{\tilde E} \times \bR$. 
\end{lemma} 
\begin{proof} 
Let $V'$ be the end-neighborhood of $\orb$ that $\Bd U \cap \torb$ or $\clo(U) \cap \torb$ maps 
into. Then for a component of the inverse image $V$ of $V'$, we have 
$U \subset V'$. Since $\bGamma_{\tilde E}$ is precisely the set of deck-transformations acting on $V$, 
$U$ covers $p_{\orb}(U)$ in $V'$ with the deck transformation group $\bGamma_{\tilde E}$. 
Also, $\Bd U \cap \torb$ covers the boundary of $p_{\orb}(U)$ in $V'$, 
and hence is a compact hypersurface. Since $V$ is homeomorphic to $\Sigma_{\tilde E} \times \bR$,
the result follows. 
\end{proof}

\begin{proposition}\label{prop:endvertex} 
Suppose that $\mathcal{O}$ is a strongly tame properly convex real projective orbifold with radial or totally geodesic ends, 
and its developing map sends the universal cover $\torb$
to a convex domain. Let $U'$ be an end-neighborhood in $\orb$. 
Let $\tilde U$ be $p_{\orb}^{-1}(U')$ as above with $E'$ the p-end in $\torb$ associated with a component $U$ of $\tilde U$.
Then the closure of each component of  $\tilde U$ 
contains the p-end vertex $\bv_{E'}$ of the leaf of radial foliation in $\tilde U$ lifted from $U$, and 
there exists a unique one for each component $U_1$ of $\tilde U$ associated with a p-R-end $E'$ of $\torb$. 
The subgroup of $h(\pi_1(\mathcal{O}))$ acting on $U_1$ 
or fixing the p-end vertex $\bv_{E'}$ is precisely in the subgroup $\bGamma_{E'}$.
\end{proposition}


\begin{definition}\label{defn:convhull} 
Given a subset $K$ of a convex domain $\Omega$ of an affine space $A^n$ in $\SI^n$ 
{\rm (}resp. $\bR P^n${\rm ),} the {\em convex hull} of 
$K$ is defined as the smallest convex set containing $K$ in $\clo(\Omega) \subset A^n$ where we required $\clo(\Omega) \subset A^n$. 
\end{definition}
The convex hull is {\em well-defined} as long as $\Omega$ is properly convex. Otherwise, it may be not. 
Often when $\Omega$ is a properly convex domain, we will take the closure $\clo(\Omega)$ instead of $\Omega$ usually. 
This does not change the convex hull. (Usually it will be clear what $\Omega$ is by context but we will mention these.) 

We will show later that 
the p-end-neighborhood can be chosen to be properly convex by taking the convex hull of a well-chosen 
p-end-neighborhood in $\torb$. 
However, there is no guarantee that the images of convex ones are disjoint.

\subsection{Convexity and convex domains}\label{subsec:conv}








A complete real line in $\rpn$ is a $1$-dimensional subspace of $\rpn$ with one point removed. 
That is, it is the intersection of a $1$-dimensional subspace by an affine space. 
An {\em affine $i$-dimensional subspace} is a submanifold of $\SI^n$ or $\rpn$ projectively diffeomorphic 
to an $i$-dimensional affine subspace of a complete affine space.
A {\em convex} projective geodesic is a projective geodesic in a real projective orbifold which lifts to 
a projective geodesic, the image of whose composition with a developing map does not contain a complete real line. 
A real projective orbifold is {\em convex} if every path can be homotoped to a convex projective geodesic with endpoints fixed. 
It is {\em properly convex} if it contains no great open segment in the orbifold.

In the double cover $\SI^n$ of $\rpn$, an affine space $A^n$ is the interior of a hemisphere. 
A domain in $\rpn$ or $\SI^n$ is {\em convex} if it lies in some affine subspace and satisfies the convexity 
property above. Note that a convex domain in $\rpn$ lifts to ones in $\SI^n$ up to the antipodal map 
$\mathcal{A}$. A convex domain in $\SI^n$ not containing an antipodal pair maps to one in $\rpn$ homeomorphically. 
(Actually from now on, we will only be interested in convex domains in $\SI^n$.)

\begin{proposition}\label{prop:projconv}
\begin{itemize}
\item A real projective $n$-orbifold is convex if and only if the developing map sends 
the universal cover to a convex domain in $\rpn$ or $\SI^n$. 
\item A real projective $n$-orbifold is properly convex if and only if the developing map sends 
the universal cover to a properly convex open domain in a compact domain in an affine patch of $\rpn$. 
\item If a convex real projective $n$-orbifold is not properly convex and not complete affine, then
its holonomy is virtually reducible in $\PGL(n+1, \bR)$ or $\SLnp$. In this case, $\torb$ is foliated 
by affine subspaces $l$ of dimension $i$ with the common boundary $\clo(l) - l$ equal to 
a fixed subspace  $\SI^{i-1}_{\infty}$ in $\Bd \torb$. 
\end{itemize}
\end{proposition}
\begin{proof} 
The first item is Proposition A.1 of \cite{psconv}. 
The second follows immediately. 
For the final item, a convex subset of $\rpn$ is a convex subset of an affine patch $A^n$, isomorphic to an affine space. 
A convex open domain $D$ in $A^n$ that has a great open segment must 
contain a maximal complete affine subspace.  
Two such complete maximal affine subspaces do not intersect since otherwise a larger complete affine subspace of 
higher dimension is in $D$ by convexity. We showed in \cite{ChCh} that the maximal complete affine subspaces
foliated the domain.  (See also \cite{GV}.) The foliation is preserved under the group action 
since the leaves are lower-dimensional complete affine subspaces in $D$. 
This implies that the boundary of the affine subspaces is a lower dimensional subspace. 
These subspaces are preserved under the group action.
\end{proof} 


\subsubsection{The Benoist theory} \label{sub:ben}

In late 1990s, Benoist more or less completed the theory of the divisible action as started by Benzecri, Vinberg, Koszul, Vey, and so on
in the series of papers \cite{Ben1}, \cite{Ben2}, \cite{Ben3}, \cite{Ben4}, \cite{Ben5}, \cite{Benasym}.
The comprehensive theory will aid us much in this paper. 

\begin{proposition}[(Corollary 2.13 \cite{Ben3})]\label{prop:Benoist}  
Suppose that a discrete subgroup $\Gamma$ of $\SLn$ {\rm (}resp. $\PGL(n, \bR)${\rm )}
acts on a properly convex $(n-1)$-dimensional open domain $\Omega$ in $\SI^{n-1}$ {\rm (}resp, $\bR P^{n-1}${\rm )} 
so that $\Omega/\Gamma$ is compact. Then the following statements are equivalent. 
\begin{itemize} 
\item Every subgroup of finite index of $\Gamma$ has a finite center. 
 \item Every subgroup of finite index of $\Gamma$ has a trivial center. 
\item Every subgroup of finite index of $\Gamma$ is irreducible in $\SLn$. 
That is, $\Gamma$ is strongly irreducible. 
\item The Zariski closure of $\Gamma$ is semisimple. 
\item $\Gamma$ does not contain an infinite nilpotent normal subgroup. 
\item $\Gamma$ does not contain an infinite abelian normal subgroup.
\end{itemize}
\end{proposition}
\begin{proof}
Corollary 2.13 of \cite{Ben2} considers $\PGL(n, \bR)$ and $\bR P^{n-1}$. 
However, the version for $\SI^{n-1}$ follows from this since we can always lift 
a properly convex domain in $\bR P^{n-1}$ to one $\Omega$ in $\SI^{n-1}$ and 
the group to one in $\SLn$ acting on $\Omega$. 
\end{proof}

The group with properties above is said to be the group with a {\em trivial virtual center}. 

\begin{theorem}[(Theorem 1.1 of \cite{Ben3})] \label{thm:Benoist} 
Let $n-1 \geq 1$. 
Suppose that a virtual-center-free discrete subgroup $\Gamma$ of $\SLn$ {\rm (}resp. $\PGL(n, \bR)${\rm )}  acts on 
a properly convex $(n-1)$-dimensional open domain $\Omega \subset \SI^{n-1}$ so 
that $\Omega/\Gamma$ is compact.
Then every representation of a component of $\Hom(\Gamma, \SLn)$ {\rm (}resp. $\Hom(\Gamma, \PGL(n, \bR))${\rm ) }
containing the inclusion representation also acts on a properly convex $(n-1)$-dimensional open domain cocompactly. 
\end{theorem}
(When $\Gamma$ is a hyperbolic group and $n=3$, Inkang Kim \cite{ink} proved this simultaneously.)

We call the group such as above theorem a {\em vcf-group}. By above Proposition \ref{prop:Benoist}, 
we see that every representation of the group acts irreducibly.



\begin{proposition}[(Benoist \cite{Ben2})] \label{prop:Ben2} Assume $n \geq 2$. 
Let $\Sigma$ be a closed $(n-1)$-dimensional properly convex projective orbifold
and let $\Omega$ denote its universal cover in $\SI^{n-1}$ {\rm (}resp. $\bR P^{n-1}${\rm ).}  
Then 
\begin{itemize}
\item $\Omega$ is projectively diffeomorphic to the interior of 
a strict join $K_1 * \cdots * K_{l_0}$ where $K_i$ is a properly convex open domain of dimension $n_i \geq 0$ in 
the subspace $\SI^{n_i}$ in $\SI^n$ {\rm (}resp. $\bR P^{n_i}$ in $\bR P^n${\rm )}
corresponding to a convex cone $C_i \subset \bR^{n_i+1}$. 
\item $\Omega$ is the image of $C_1 \oplus \cdots \oplus C_r$. 
\item The fundamental group 
$\pi_1(\Sigma)$ is virtually isomorphic to $\bZ^{l_0-1} \times \Gamma_1 \times \cdots \Gamma_{l_0}$ for 
$l_0-1 + \sum n_i = n$. 
\item Suppose that each $\Gamma_i$ is hyperbolic or trivial. 
Each $\Gamma_j$ acts on $K_j$ cocompactly and the Zariski closure is the trivial group or an
embedded copy of 
\[\SL(n_i+1, \bR), \SL_{\pm}(n_i+1), \SO(n_i, 1) \hbox{ or } \Ort(n_i, 1)\] 
\[ \hbox{ {\rm (}resp. } \PGL(n_i+1, \bR), \PSO(n_i, 1), \PO(n_i, 1)\hbox{{\rm )}}.\]
in $\SLn$ {\rm (}resp. $\PGL(n, \bR)${\rm )} and acts trivially on $K_m$ for $m \ne j$. 
\item The subgroup corresponding to $\bZ^{l_0-1}$ acts trivially on each $K_j$.
\end{itemize} 
\end{proposition} 
\begin{proof} 
First consider the version for $\SI^n$.
The first four items and the last one are from Theorem 1.1. in \cite{Ben2}, 
where the work is done over $\GL(n, \bR)$. However, we assume that 
our elements are in $\SLn$ and by adding dilatations, we obtain the needed results. 
The Zariski closure part is obtained by 
Theorem 1.1 in \cite{Ben1}, and Theorem 1.3 of \cite{Ben5}.


Let $\hat h: \pi_1(\Sigma_{\tilde E}) \ra \SL_\pm(n, \bR)$ be the homomorphism associated with $\tilde E$.
The first part of the fourth item is also from Theorem 1.1 of \cite{Ben2}. 
By Theorem 1.1 of \cite{Ben3}, the Zariski closure of $\hat h(\pi_1(\tilde E))$ is virtually a product 
$\bR^{l_0-1} \times G_1 \times \cdots \times G_{l_0}$ and $G_j$, $j=1, \dots, l_0$, 
is an irreducible reductive Lie subgroup of $\SL_\pm(V_i)$.  
Suppose $\Gamma_i$ acts nontrivially on $C_k$ for $k \ne i$. Then elements of 
Zariski closures $Z_k^i$ of their images commute in $G_k$ and $G_k$ is the centralizer of products of 
subgroups $Z_k^i$s. Since $G_k$ is irreducible linear algebraic subgroup as listed above, this is absurd. 
(We were helped by Benoist in this argument.)

The proof for the $\bR P^n$-version follows easily again by the lifting arguments. 
\end{proof}

To explain more, $K_i$ could be a point. 
For some $s$, $1 \leq s\leq r$, 
we could obtain a decomposition where each $K_i$ for $i \geq s$ has dimension $\geq 2$ and 
$\Gamma_i$ is a hyperbolic group. 
Then $\Gamma$ is virtually a product of hyperbolic groups and an abelian group that is the center of the group. 


\section{Examples of ends} \label{sec:exend}

We will present some examples here, which we will fully justify later. 

\subsection{Definitions associated with ends}

We will use: 
\begin{definition}\label{defn:endvertex} 
Let $\orb$ denote a strongly tame convex real projective $n$-orbifold with radial ends or totally geodesic ends
with the universal cover $\torb \subset \SI^n$ (resp. $\subset \bR P^n$) 
and the group of deck transformation 
$\bGamma$ acting on $\torb$ projectively.  
Let $\bv_{\tilde E}$ be the p-end vertex in $\SI^n$ (resp. $\bR P^n$) 
corresponding to a p-R-end $\tilde E$ of $\torb$. 
 We need the linking sphere $\SI^{n-1}_{\bv_{\tilde E}}$ of great segments at ${\bv_{\tilde E}}$. 
 The subgroup of projective automorphisms $\SLnp_{\bv_{\tilde E}}$ 
 (resp. $\PGL(n+1, \bR)_{\bv_{\tilde E}}$)
 of $\SI^n$ (resp. $\bR P^n$)  fixing $\bv_{\tilde E}$ acts on it 
 projectivized as $\SL_{\pm}(n, \bR)$ acting on $\SI^{n-1}_{\bv_{\tilde E}}$.

 We can associate an R-end $E$ of $\mathcal O$ with a p-R-end vertex 
 ${\bv_{\tilde E}} \in \SI^n$ (resp. $\in \bR P^n$). 
Thus, $\torb$ has only finitely many orbit types for p-end vertices under $\bGamma$. We recall
from the introduction. 
\begin{itemize} 
\item We denote $\bGamma_{\bv_{\tilde E}}$ by $\pi_1(\tilde E)$ also and is said to be the {\em p-end fundamental group} of $\tilde E$. 
\item  Two segments from $\bv_{\tilde E}$ are {\em equivalent} if they agree in a neighborhood of $\bv_{\tilde E}$.
 Given a p-R-end $\tilde E$ corresponding to $\bv_{\tilde E}$, 
 we denote by $R_{\bv_{\tilde E}}(\torb)$ the space of equivalence classes of lines from 
 ${\bv_{\tilde E}}$ in $\tilde{\mathcal{O}}$. This is a  convex domain  since $\torb$ is. 
\item For a subset $K$ of $U_1$, we denote by $R_{\bv_{\tilde E}}(K)$, the equivalence class space of lines from ${\bv_{\tilde E}}$ ending at $K$, 
 which is a convex set provided $K$ is. 
 \item We have $R_{\bv_{\tilde E}}(\torb), R_{\bv_{\tilde E}}(K) \subset \SI^{n-1}_{\bv_{\tilde E}}$. 
\item  We will denote this by $\tilde \Sigma_{\tilde E}$ as well, a universal cover of the end orbifold $\Sigma_{\tilde E}$. 
 \end{itemize} 
\end{definition} 

\begin{definition} \label{defn:endsurf}
Let $\orb$ denote a strongly tame convex real projective $n$-orbifold with radial ends or totally geodesic ends
with the universal cover $\torb \subset \SI^n$ (resp. $\subset \bR P^n$)
and the group of deck transformation 
$\bGamma$ acting on $\torb$ projectively.  
Given a totally geodesic end $E$, let $S_E$ denote 
the totally geodesic orbifold corresponding to $E$. 
A pseudo-end $\tilde E$ has a corresponding totally geodesic surface $\tilde S_{\tilde E}$, which is a convex domain 
in a hyperspace, covering $S_E = \tilde S_{\tilde E}/\bGamma_{\tilde E}$. 
We will use $S_{\tilde E}$ instead of $S_E$, which we call 
{\em ideal boundary orbifold}. $\tilde S_{\tilde E}$ is called the {\em ideal boundary} of 
$\tilde E$. 
\end{definition} 


A {\em properly convex p-R-end} is a p-R-end $\tilde E$ with a properly convex domain
$R_{\bv_{\tilde E}}(\torb)$ for the p-end vertex $\bv_{\tilde E}$. 
A {\em complete p-R-end} is a p-R-end with $R_{\bv_{\tilde E}}(\torb)$ projectively a complete affine space. 
A {\em properly convex R-end} is an end corresponding to a properly convex p-R-end.
A {\em complete R-end} is an end corresponding to a complete p-R-end.

\begin{proposition}\label{prop:endorbstr} 
Let $\orb$ be a strongly tame convex real projective orbifold with radial or totally geodesic ends. 
Let $\tilde E$ be a p-R-end of $\torb$. 
Then 
\begin{itemize}
\item $R_{\bv_{\tilde E}}(\torb)$ is also a convex domain in an affine subspace of $\SI^{n-1}_{\bv_{\tilde E}}$
for the associated vertex $\bv_{\tilde E}$ for $\tilde E$.  
\item The group $h(\pi_1(\tilde E))$ induces a group $\hat h(\pi_1(\tilde E))$ of  projective transformations of 
$\SI^{n-1}$ acting on $R_{\bv_{\tilde E}}(\torb)$ and $R_{\bv_{\tilde E}}(\torb)/\hat h(\pi_1(\tilde E))$ is diffeomorphic to 
the end orbifold $\Sigma_{\tilde E}$ 
and has the induced projective structure. 
\end{itemize} 
\end{proposition}
\begin{proof}
Straightforward. 
\end{proof}

We remark also that 
for the orbifold with admissible ends and the infinite-index end fundamental group condition, the p-end vertices are infinitely many 
for each equivalence class of vertices. 


\subsubsection{Examples}\label{subsub:examples}

From hyperbolic manifolds, we obtain some examples of ends. 
Let $M$ be a complete hyperbolic manifolds with cusps. 
$M$ is a quotient space of the interior $\Omega$ of an ellipsoid in $\rpn$ or $\SI^n$
under the action of a discrete subgroup $\bGamma$ of $\Aut(\Omega)$. 
Then horoballs are p-end-neighborhoods of the horospherical ends. 

Suppose that $M$ has totally geodesic embedded surfaces $S_1,.., S_m$ homotopic to the ends.  
\begin{itemize}
\item We remove the outside of $S_j$s to obtain a properly convex 
real projective orbifold $M'$ with totally geodesic boundary.
\item Each $S_i$ corresponds to a disjoint union of totally geodesic domains $\bigcup_{j \in J} \tilde S_{i, j}$
for a collection $J$. For each of which $\tilde S_{i, j} \subset \Omega$, a group 
$\Gamma_{i,j}$ acts on it where $\tilde S_{i, j}/\Gamma_{i, j}$ is 
a closed orbifold projectively diffeomorphic to $S_i$. 
\item Then $\Gamma_{i, j}$ fixes a point $p_{i,j}$ outside the ellipsoid by taking the dual 
point of $\tilde S_{i, j}$ outside the ellipsoid. 
\item Hence, we form the cone 
$M_{i, j} := \{p_{i, j}\} \ast \tilde S_{i, j}$. 
\item We obtain the quotient 
$M_{i, j}^o/\Gamma_{i, j}$ of the interior and attach to 
$M'$ to obtain the examples of real projective manifolds with radial ends. 
\end{itemize}
(This orbifold is called the {\em hyper-ideal extension} of the hyperbolic manifolds as real projective manifolds.)



\begin{proposition}\label{prop:lensend}
Suppose that $M$ is a strongly tame properly convex real projective orbifold with radial or totally geodesic ends.
Suppose that 
\begin{itemize}
\item the holonomy group of each end fundamental group is generated by the homotopy classes of closed curves about singularities or
\item  has the holonomy fixing the end vertex with eigenvalues $1$ and 
\item an R-end $E$ has a compact totally geodesic 
properly convex hyperspace in a p-end-neighborhood and not containing the p-end vertex. 
\end{itemize} 
Then the end $E$ is of lens-type.
\end{proposition}
\begin{proof} 
Let $\tilde M$ be the universal cover of $M$ in $\SI^n$. It will be sufficient to prove for this case. 
Let $E$ be an R-end of $M$ with a compact totally geodesic subspace $\Sigma$ in a p-end-neighborhood. Then 
a p-end-neighborhood $U$ of $\tilde E$ corresponding to $E$
contains the universal cover $\tilde \Sigma$ of $\Sigma$. 

Since the end fundamental group $\Gamma_{\tilde E}$
is generated by closed curves about singularities. Then since 
the singularities are of finite order, the eigenvalues of the generators  corresponding to the p-end vertex 
$\bv_{\tilde E}$ equal $1$
and hence every element of the end fundamental group has $1$ as the eigenvalue at
$\bv_{\tilde E}$.
Now assume that the holonomy of the elements of the end fundamental group,
fixes the p-end vertex with eigenvalues equal to $1$. 

Then $U$ can be chosen to be 
the open cone over the totally geodesic domain with vertex $\bv_{\tilde E}$. 
projectively diffeomorphic 
to the interior of a properly convex cone in 
an affine subspace $A^n$ and the end fundamental group acts on 
it as a discrete linear group of determinant $1$. 
The theory of convex cones applies, and using the level sets of 
the Koszul-Vinberg function we obtain 
a smooth convex one-sided neighborhood in $U$ 
(see Lemma 6.5 and 6.6 of Goldman \cite{wmgnote}).
Also, the outer one-sided neighborhood can be obtained
by a reflection about the plane containing $\tilde \Sigma$ and the p-end vertex
and some dilatation action so that it is in $\torb$. 
\end{proof}

Let $S_{3,3,3}$ denote the $2$-orbifold with base space homeomorphic to a $2$-sphere and 
with three cone-points of order $3$. 
\begin{proposition} \label{prop:lensauto}
Let $\mathcal{O}$ be a  convex real projective $3$-orbifold with radial ends where the end orbifolds each of which is homeomorphic 
to a sphere $S_{3,3,3}$ or a disk with three silvered edges and three vertices of orders $3, 3, 3$. 
Then the orbifold has only lens-shaped R-ends or horospherical R-ends. 
\end{proposition}
\begin{proof}
Again, it is sufficient to prove this for the case $\torb \subset \SI^n$. 
Let $\tilde E$ be a p-R-end of type $S_{3, 3, 3}$ for $\torb$. 
It is sufficient to consider only $S_{3,3,3}$ since it double-covers the disk orbifold. 
Since $\bGamma_{\tilde E}$ is generated by finite order elements fixing a p-end vertex $\bv_{\tilde E}$,
every holonomy element has eigenvalue equal to $1$ at $\bv_{\tilde E}$.
Take a finite-index free abelian group $A$ of rank two.
Since $\Sigma_E$ is convex, 
a convex projective torus $T^2$ covers $\Sigma_E$ finitely.
Therefore, $\tilde \Sigma_{\tilde E}$ is projectively diffeomorphic either to a complete affine space or 
to the interior of a properly convex triangle or to a half-space by the classification of 
convex tori found in many places including \cite{wmgnote} and \cite{BenNil} and 
by Proposition \ref{prop:projconv}. 
Since there exists a holonomy automorphism of order $3$ fixing a point of 
$\tilde \Sigma_{\tilde E}$, 
it cannot be a quotient of a half-space with a distinguished foliation by lines.
Thus, the end orbifold admits a complete affine structure or is a quotient of a properly convex triangle

 If $\Sigma_{\tilde E}$ has a complete affine structure, 
 we have a horospherical end for $E$ by Theorem \ref{thm:comphoro}.
Suppose that $\Sigma_{\tilde E}$ has a properly convex triangle as its universal cover. 
$A$ acts with an element $g'$ with an eigenvalue $>1$ and an eigenvalue $<1$ as a transformation 
in $\SL_\pm(3, \bR)$ the group of projective automorphisms at $\SI^2_{\bv}$.
$g'$ fixes $v_1$ and $v_2$ other than $\bv_{\tilde E}$
in directions of the vertices of the triangle in the cone. 
Since the corresponding eigenvalue at $\bv_{\tilde E}$ is $1$ and $g'$ acts on 
a properly convex compact domain, 
$g'$ has four fixed points and 
an invariant subspace $P$ disjoint from $\bv_{\tilde E}$. 
That is, $g'$ is diagonalizable. 
Since elements of $A$ commute with $g'$, 
so does every other $g \in A$. 
The end fundamental group acts on $P$ 
as well. We have a totally geodesic R-end and by Proposition \ref{prop:lensend}, the end is lens-shaped. 
(See also Theorem \ref{thm:redtot}.) 
\end{proof}


\begin{example}[(Lee's example)] \label{exmp:Lee}
Consider the Coxeter orbifold $\hat P$ with the underlying space on a polyhedron $P$ 
with the combinatorics of a cube with all sides mirrored and
all edges given order $3$ but vertices removed. 
By the Mostow-Prasad rigidity and the Andreev theorem, 
the orbifold has a unique complete hyperbolic structure. 
There exists a six-dimensional space of real projective structures on it 
as found in \cite{CHL} where one has a projectively fixed fundamental domain 
in the universal cover.

There are eight ideal vertices of $P$ corresponding to eight ends of $\hat P$. 
Each end orbifold is a $2$-orbifold based on a triangle with edges mirrored 
and vertex orders are all $3$. Thus, each end has a neighborhood homeomorphic 
to the $2$-orbifold multiplied by $(0, 1)$.
We can characterize them by a real-valued invariant.  
Their invariants are related when we are working on the restricted deformation space. 
(They might be independent  in the full deformation space as M. Davis and R. Green observed. )

Then this end can be horospherical or 
be a radial lens type with a totally geodesic realization end orbifold by Proposition \ref{prop:lensauto}.
When $P$ is hyperbolic, the ends are horospherical as 
$P$ has a complete hyperbolic structure. 

Experimentations also suggest that we realize totally geodesic R-ends after deformations. 
This applies to S. Tillman's example. (See my other paper on the mathematics archive \cite{conv} for details)
\end{example}


The following construction is called ``bending'' and was investigated by Johnson and Millson
\cite{JM}. 

\begin{example}\label{exmp:bending} 
Let $\orb$ have the usual assumptions. 
Let $E$ be a totally geodesic R-end with a p-R-end $\tilde E$.
Let  the associated orbifold 
$\Sigma_{E}$ for $E$ of $\orb$ be a closed $2$-orbifold and 
let $c$ be a  simple closed geodesic in $\Sigma_{\tilde E}$. 
Suppose that $E$ has an end-neighborhood $U$ in $\orb$ diffeomorphic to $\Sigma_{E} \times (0,1)$
with totally geodesic $\Bd U$ diffeomorphic to $\Sigma_E$.
Let $\tilde U$ be a p-end-neighborhood in $\torb$ corresponding to $\tilde E$
bounded by $\tilde \Sigma_{\tilde E}$ covering $\Sigma_E$. 

Now a lift $\tilde c$ of $c$ is in an embedded disk $A'$, covering 
an annulus $A$ diffeomorphic to $c \times (0, 1)$, foliated by lines from $\bv_{\tilde E}$.
Let $g_c$ be the deck transformation corresponding to $\tilde c$ and $c$. 
Then the holonomy $g_c$ is conjugate 
to a diagonal matrix with entries $\lambda,  \lambda^{-1}, 1, 1, \lambda > 1$
where the last $1$ corresponds to the vertex $\bv$.  
We take an element $k_b$ of $\SLf$ of form in this system of coordinates
\begin{equation}\label{eqn:bendingm} 
\left(
\begin{array}{cccc}
1           &       0              & 0   & 0  \\
 0          &       1              & 0  & 0 \\ 
 0           &      0              & 1  & 0 \\ 
 0           &      0               & b & 1   
\end{array}
\right)
\end{equation}
where $b \in \bR$. 
$k_b$ commutes with $g_c$. 
Let us just work on the end $E$. 
We can ``bend'' $E$ by $k_b$: 

Then we cut $U$ by $A$ and we obtain two copies $A_1$ and $A_2$ corresponding to a single component 
by completing $U - A$. We take an ambient real projective manifold $U'$ containing 
the completion.  We can find neighborhoods $N_1$ and $N_2$ of $A_1$ and $A_2$ in $U$
diffeomorphic by a projective map $\hat k_b$ induced by $k_b$.

We take a disjoint union $(U - A) \coprod N_1 \coprod N_2$ and 
quotient it by identifying elements of $N_1$ with elements near $A_1$ in $U- A $ by the identity map
and elements of $N_2$ with elements near $A_2$ in $U-  A$ by the identity also.
We then glue back $N_1$ and $N_2$ by $\hat k_b$
the real projective diffeomorphism of a neighborhood of $N_1$ to that of $N_2$. 

For sufficiently small $b$, we see that the end is still of lens type
and it is not a totally geodesic R-end. (This follows since the condition of being 
a generalized-lens type R-end is an open condition. 
See Theorem \ref{thm:qFuch}.)

Since $k_b$ fixes a subspace of dimension $2$ containing $\bv_{\tilde E}$ and the geodesic fixed by $g_c$, 
the totally geodesic subspace is bent. 
We see that $b> 0$, we obtain a boundary of $E$ bent in a positive manner. 
The deformed holonomy group acts on a convex domain obtained by bendings of these types everywhere. 

For the same $c$, 
let $k_s$ be given by 
\begin{equation}\label{eqn:bendingm2} 
\left(
\begin{array}{cccc}
s           &       0              & 0   & 0  \\
 0          &       s              & 0  & 0 \\ 
 0           &      0              & s  & 0 \\ 
 0           &       0            & 0 & 1/s^3   
\end{array}
\right)
\end{equation}
where $s \in \bR_+$. 
These give us bendings of second type. (We talked about this in \cite{conv}.) 
For $s$ sufficiently close to $1$, the property of being of lens-type is preserved 
and being a radial totally geodesic end. 
(However, these will be understood by cohomology.)

If $s \lambda < 1$ for the maximal eigenvalue $\lambda$ of a closed curve $c_1$
meeting $c$ odd number of times, we have that the holonomy along $c_1$ has the attacking 
fixed point at $\bv_{\tilde E}$. This implies that we no longer have lens-type ends if we have 
started with a lens-shaped end.
.
\end{example} 








\section{The duality of real projective orbifolds}\label{sec:duality}

\subsection{The duality} 
We starts from linear duality. Let $\Gamma$ be a group of linear transformations $\GL(n+1, \bR)$. 
Let $\Gamma^*$ be the {\em affine dual group} defined by $\{g^{\ast -1}| g \in \Gamma \}$. 
Suppose that $\Gamma$ acts on a properly convex cone $C$ in $\bR^{n+1}$ with the vertex $O$.

An open convex cone  $C^*$ in $\bR^{n+1, *}$  is {\em dual} to an open convex cone $C $ in $\bR^{n+1}$  if 
$C^* \subset \bR^{n+1 \ast}$ is the set of linear transformations taking positive values on $\clo(C)-\{O\}$.
$C^*$ is a cone with vertex as the origin again. Note $(C^*)^* = C$. 

Now $\Gamma^*$ will acts on $C^*$.
A {\em central dilatational extension} $\Gamma'$ of $\Gamma$ by $\bZ$ is given by adding a dilatation by a scalar 
$s \in \bR_+ -\{1\}$ for the set $\bR_+$ of positive real numbers. 
The dual $\Gamma^{\prime \ast}$ of $\Gamma'$ is a central dilatation extension of $\Gamma^*$. 
 Also, if $\Gamma'$ is cocompact on $C$ if and only if $\Gamma^{\prime *}$ is on $C^*$. 
 (See \cite{wmgnote} for details.)

 Given a subgroup $\Gamma$ in $\PGL(n+1, \bR)$, a {\em lift} in $\GL(n+1, \bR)$ is any subgroup that maps to $\Gamma$ injectively.
 Given a subgroup $\Gamma$ in $\Pgl$, the dual group $\Gamma^*$ is the image in $\Pgl$ of the dual of 
 any linear lift of $\Gamma$. 

A properly convex open domain $\Omega$ in $P(\bR^{n+1})$ is {\em dual} to a properly convex open domain
$\Omega^*$ in $P(\bR^{n+1, \ast})$ if $\Omega$ corresponds to an open convex cone $C$ 
and $\Omega^*$ to its dual $C^*$. We say that $\Omega^*$ is dual to $\Omega$. 
We also have $(\Omega^*)^* = \Omega$ and $\Omega$ is properly convex if and only if so is $\Omega^*$. 

We call $\Gamma$ a {\em divisible group} if a central dilatational extension acts cocompactly on $C$.
$\Gamma$ is divisible if and only if so is $\Gamma^*$. 

Recall $\SI^n := {\mathcal{S}}(\bR^{n+1})$. We define $\SI^{n\ast} := {\mathcal{S}}(\bR^{n+1 \ast})$.

For an open properly convex subset $\Omega$ in $\SI^{n}$, the dual domain is defined as the quotient 
of the dual cone of the cone corresponding to $C_\Omega$ in $\SI^{n\ast}$. The dual set is also open and properly convex
but the dimension may not change.  
Again, we have $(\Omega^*)^* =\Omega$. 

Given a properly convex domain $\Omega$ in $\SI^n$ (resp. $\bR P^n$), 
we define the {\em augmented boundary} of $\Omega$
\[\Bd^{\Ag} \Omega  := \{ (x, h)| x \in \Bd \Omega, h \hbox{ is a supporting hyperplane of } \Omega, 
h \ni x \} .\] 
Each $x \in \Bd \Omega$ has at least one supporting hyperspace, 
a hyperspace is an element of $\bR P^{n \ast}$ since it is represented as a linear functional,   
and an element of $\bR P^n$ represent a hyperspace in $\bR P^{n \ast}$. 

\begin{theorem} \label{thm:Serre} 
Let $A$ be a subset $\Bd \Omega$. Let $A':= \Pi_{\Ag}^{-1}(A)$ be the subset $\Bd^{\Ag}(A)$. 
Then $\Pi_\Ag| A': A' \ra A$ is a Serre fibration. 
\end{theorem} 
\begin{proof}
We take a Euclidean metric on an affine space containing $\clo(\Omega)$. 
The supporting hyperplanes can be identified with unit vectors. 
Each fiber $\Pi_\Ag^{-1}(x)$ is 
a properly convex compact domain in a sphere of unit vectors through $x$. 
We find a continuous unit vector field to $\Bd \Omega$ by taking the center of mass of each fiber with respect 
to the Euclidean metric. This gives a local coordinate system on each fiber by giving origin and each fiber is a convex domain 
containing the origin. Then the Serre-fibration property is clear now.
\end{proof}

\begin{remark} 
We notice that for open properly convex domains $\Omega_1$ and $\Omega_2$ in $\SI^n$ 
(resp. in $\bR P^n$) we have 
\begin{equation}\label{eqn:dualinc}
\Omega_1 \subset \Omega_2 \hbox{ if and only if } \Omega_2^* \subset \Omega_1^*  
\end{equation}
\end{remark}

\begin{lemma}\label{lem:predual}
Let $\Omega^*$ be the dual of a properly convex domain $\Omega$ in $\SI^n$ {\rm (}resp. $\bR P^n${\rm ).} 
Then 
\begin{itemize}
\item $\Bd \Omega$ is $C^1$ at a point $p$ if and only $\Bd \Omega^*$  is $C^1$ at the corresponding point. 
\item $\mathcal{D}$ sends the pair of $p$ and the associated supporting hyperplanes of $\Omega$
 to the pairs of the totally geodesic hyperplane containing $D$ and points of $D$.  
\end{itemize}
\end{lemma}
\begin{proof} 
These are straightforward. 
\end{proof} 

The homeomorphism below will be known as the {\em duality map}. 
\begin{proposition} \label{prop:duality}
Let $\Omega$ and $\Omega^*$ be dual domains in $\SI^{n \ast}$ {\rm (}resp. $\bR P^{n \ast}${\rm ).} 
\begin{itemize}  
\item[(i)] There is a proper quotient map $\Pi_{\Ag}: \Bd^{\Ag} \Omega \ra \Bd \Omega$
given by sending $(x, h)$ to $x$. 
\item[(ii)] Let a projective automorphism group 
$\Gamma$ acts on a properly convex open domain $\Omega$ if and only 
$\Gamma^*$ acts on $\Omega^*$.
\item[(iii)] There exists a duality homeomorphism 
\[ {\mathcal{D}}: \Bd^{\Ag} \Omega \leftrightarrow \Bd^{\Ag} \Omega^* \] 
given by sending $(x, h)$ to $(h, x)$ for each $(x, h) \in \Bd^{\Ag} \Omega$. 
\item[(iv)] Let $A \subset \Bd^{\Ag} \Omega$ be a subspace and $A^*\subset \Bd^{\Ag} \Omega^*$
be the corresponding dual subspace $\mathcal{D}(A)$. If a group $\Gamma$ acts on $A$ so that $A/\Gamma$ is compact 
if and only if $\Gamma^*$ acts on $A^*$ and $A^*/\Gamma^*$ is compact. 
\end{itemize} 
\end{proposition} 
\begin{proof} 
We will prove for $\bR P^n$ but the same proof works for $\SI^n$. 
(i) Each fiber is a closed set of hyperplanes at a point forming a compact set.
The set of supporting hyperplanes at a compact subset of $\Bd \Omega$ is closed. 
The closed set of hyperplanes having a point in a compact subset of $\bR P^{n+1}$ is compact. 
Thus, $\Pi_{\Ag}$ is proper.  Clearly, $\Phi_{\Ag}$ is continuous, and it is an open map
since $\Bd^{\Ag} \Omega$ is given the subspace topology from $\bR P^n \times \bR P^{n \ast}$
with a box topology where $\Phi_{\Ag}$ extends to a projection.

(ii) Straightforward. (See Chapter 6 of \cite{wmgnote}.)

(iii) An element $(x, h)$ is $\Bd^{\Ag} \Omega$ if and only if $x \in \Bd \Omega$ and $h$ is represented 
by a linear functional $\alpha_h$ so that $\alpha_h(y) > 0$ for all $y$ in the open cone $C$ corresponding to $\Omega$ and 
$\alpha_h(v_x) =0$ for a vector $v_x$ representing $x$. 

Since the dual cone $C^*$ consists of all nonzero $1$-form $\alpha$ so that $\alpha(y) > 0$ for all $y \in \clo(C) - \{O\}$. 
Thus, $\alpha(v_x) > 0$ for all $\alpha \in C^*$ and $\alpha_y(v_x) = 0$. 
$\alpha_h \not\in C^*$ since $v_x \in \clo( C)-\{O\}$
but $h \in \Bd \Omega^*$ as we can perturb $\alpha_h$ so that it is in $C^*$. 
Thus, $x$ is a supporting hyperspace at $h \in \Bd \Omega^*$.
Hence we obtain a continuous map ${\mathcal{D}}: \Bd^{\Ag} \Omega \ra \Bd^{\Ag} \Omega^*$. 
The inverse map is obtained in a similar way. 

(iv) The item is clear from (ii) and (iii).  
\end{proof} 

\begin{definition}
The two subgroups $G_1$ of $\Gamma$ and $G_2$ of $\Gamma^*$ are {\em dual} if 
sending $g \ra g^{-1, T}$ gives us a one-to-one map $G_1 \ra G_2$. 
A set in $A \subset \Bd \Omega$ is {\em dual} to a set $B \subset \Bd \Omega^*$ if 
 $\mathcal{D}: \Pi_\Ag^{-1}(A) \ra \Pi_{\Ag}^{-1}(B)$ is a one-to-one and onto map. 
\end{definition}


We have $\orb = \Omega/\Gamma$ for a properly convex domain $\Omega$, 
the dual orbifold $\orb^* = \Omega^*/\Gamma^*$ is a properly convex real projective orbifold 
homotopy equivalent to $\orb$. The dual orbifold is well-defined up to projective diffeomorphisms. 
We call $\orb^*$ a projectively dual orbifold to $\orb$.
Clearly, $\orb$ is projectively dual to $\orb^*$. 

\begin{theorem}[(Vinberg)]  \label{thm:dualdiff} 
The dual orbifold $\orb^*$ is diffeomorphic to $\orb$.
\end{theorem}
\begin{proof} 
Use here the duality diffeomorphism $\torb \ra \torb^*$ of Vinberg (See  \cite{vin63} and Theorem 6.8 in \cite{wmgnote}.) 
\end{proof} 
We call the map the {\em Vinberg duality diffeomorphism}.

\section{Characterization of complete ends}\label{sec:horo}

\subsection{Horospherical ends}




By an {\em exiting sequence} of p-end-neighborhoods $U_i$ of $\torb$, we mean 
a sequence of p-end-neighborhoods $U_i$ so that $p(U_i) \subset \orb$ is so that 
for each compact subset $K$ of $\orb$, $p(U_i) \cap K \ne \emp$ for only finitely many $i$s. 


By the following proposition shows that we can exchange words ``horospherical'' with ``cuspidal''. 
\begin{proposition}\label{prop:affinehoro} 
Let $\mathcal O$ be a properly convex real projective $n$-orbifold with radial or totally geodesic ends. 
Let $\tilde E$  be a horospherical end of its universal cover $\torb$, $\torb \subset \SI^n$ {\rm (}resp. $\subset \bR P^n${\rm )} 
and $\bGamma_{\tilde E}$ denote the p-end fundamental group. 
\begin{itemize}
\item[(i)] The space $\tilde \Sigma_{\tilde E} := R_{\bv_{\tilde E}}(\torb)$ 
of lines from the end point $\bv_{\tilde E}$ forms a complete affine space of dimension $n-1$.
\item[(ii)] The norms of eigenvalues of $g \in \bGamma_{\tilde E}$
are all $1$.
\item[(iii)] $\bGamma_{\tilde E}$ is virtually abelian and a finite index subgroup is in 
a conjugate of a connected parabolic subgroup of $\SO(n, 1)$ of rank $n-1$ in $\SL_\pm(n+1, \bR)$ or $\PGL(n+1, \bR)$
that acts on an ellipsoid in $\clo(\torb) \subset \rpn$. And hence $\tilde E$ has a cusp-type.
\item[(iv)] For any compact set $K'$ inside a horospherical end-neighborhood, 
$\orb$ contains a smooth convex smooth neighborhood disjoint from $K'$. 
\item[(v)] A p-end point of a horospherical p-end cannot be an endpoint of a segment in $\Bd \tilde{\mathcal{O}}$. 
\end{itemize}
\end{proposition}
\begin{proof} 
We will prove for the case $\torb \subset \SI^n$. The $\bR P^n$-version follows from this. 
Let $U$ be a horospherical p-end-neighborhood with the p-end vertex $\bv_{\tilde E}$.
The space of great segments from the p-end vertex passing $U$ forms a convex subset $\tilde \Sigma_{\tilde E}$ 
of a complete affine space $\bR^{n-1} \subset \SI^{n-1}_{\tilde E}$ by Proposition \ref{prop:projconv}
and covers an end orbifold $\Sigma_{\tilde E}$ with the discrete group $\pi_1(\tilde E)$ acting as a discrete subgroup $\Gamma'_{\tilde E}$ of
the projective automorphisms so that $\tilde \Sigma_{\tilde E}/\Gamma'_{\tilde E}$ is projectively isomorphic to $\Sigma_{\tilde E}$.

(i) By Proposition \ref{prop:projconv}, 
\begin{itemize}
\item $\tilde \Sigma_{\tilde E}$ is properly convex,
\item is foliated by complete affine spaces of dimension $i_0$ with the common boundary sphere of dimension $i_0-1$ and 
the space of the leaves forms a properly open convex subset $K^o$ of $\SI^{n-i_0-1}$ or 
\item is a complete affine space. 
\end{itemize}
Then $\bGamma_{\tilde E}$ acts on $K^o$ cocompactly but perhaps not discretely. 
We aim to show the first two cases do not occur. 

Suppose that we are in the second case and 
 $1 \leq i_0 \leq n-2$. This implies that $\tilde \Sigma_{\tilde E}$ is foliated by complete affine spaces of dimension $i_0 \leq n-2$. 

For each element $g$ of $\bGamma_{\tilde E}$, a complex or negative eigenvalue of $g$ in $\bC - \bR_+$ 
cannot have a maximal or 
minimal absolute value different from $1$ since otherwise 
by taking the convex hull in $\torb$ of $\{g^m(x)| m \in \bZ\}$ for a generic point $x$ of $U$, 
we see that $U$ must be not properly convex. Thus, the largest and the smallest absolute value
eigenvalues of $g$ are positive. 

Since $\bGamma_{\tilde E}$ acts on a properly convex subset $K$ of dimension $\geq 1$,  an element 
$g$ has an eigenvalue $>1$ and an eigenvalue $< 1$ by Benoist \cite{Ben1}
as an element of projective automorphism on the minimal great sphere containing $K$.
Hence, we obtain the largest norm of eigenvalues and the smallest one of $g$ in $\Aut(\SI^n)$ both different from $1$. 
Therefore, let $\lambda_1 >1$ be the greatest norm eigenvalue and 
$\lambda_2< 1$ be the smallest norm one of this element $g$. Let $\lambda_{\bv_{\tilde E}} >0$ 
be the eigenvalue of $g$ associated with $\bv_{\tilde E}$. 
These are all positive. 
The possibilities are as follows 
\[ \lambda_1 = \lambda_{\bv_{\tilde E}}> \lambda_2, \quad \lambda_1 > \lambda_{\bv_{\tilde E}} > \lambda_2, \quad
\lambda_1 > \lambda_2 = \lambda_{\bv_{\tilde E}}. \]
In all cases, at least one of the largest norm or the smallest norm is different from $\lambda_1$. 
Thus $g$ fixes a point $x_\infty$ distinct from $\bv_{\tilde E}$ with the distinct eigenvalue from $\lambda_0$. 
We have $x_\infty \in \clo(U)$ since $x_\infty$ is a limit of $g^i(x)$ for $x \in U$, $i \ra \infty$ or $i\ra -\infty$. 
As $x_\infty \not \in U$, we obtain $x_\infty =\bv_{\tilde E}$ by the definition of the horoballs.
This is a contradiction. 

The first possibility is also shown to not occur similarly. Thus, 
 $\tilde \Sigma_{\tilde E}$ is a complete affine space. 

(ii) If $g \in \bGamma_{\tilde E}$ has a norm of eigenvalue different from $1$, then 
we can apply the second and the third paragraphs above to obtain a contradiction. 
We obtain $\lambda_j = 1$ for each norm $\lambda_j$ of eigenvalues of $g$ for every $g \in \bGamma_{\tilde E}$.


(iii) Since $\tilde \Sigma_{\tilde E}$ is a complete affine space, 
$\tilde \Sigma_{\tilde E}/\bGamma_{\tilde E}$ is a complete affine manifold with the norms of eigenvalues holonomy matrices all equal to $1$
where $\bGamma'_{\tilde E}$ denotes the affine transformation group corresponding to $\bGamma_{\tilde E}$. 
(By D. Fried \cite{Fried86}, this implies that $\pi_1(\tilde E)$ is virtually nilpotent.) 
The conclusion follows by Proposition 7.21 of \cite{CM2} (related to Theorem 1.6 of \cite{CM2}): 
By the theorem, we see that $\bGamma_{\tilde E}$ is in a conjugate of $\SO(n, 1)$
and hence acts on an $(n-1)$-dimensional ellipsoid fixing a unique point. 
Since a horosphere has a Euclidean metric invariant under the group action, 
the image group is in a Euclidean isometry group. 
Hence, the group is virtually abelian by the Bieberbach theorem. 

(iv) We can choose an exiting sequence of p-end horoball neighborhoods $U_i$
where a cusp group acts. We can consider the hyperbolic spaces to understand this. 

(v) Suppose that $\Bd \torb$ contains a segment $s$ ending at the p-end vertex $\bv_{\tilde E}$. 
Then $s$ is on an invariant hyperspace of $\bGamma_{\tilde E}$.
Now conjugating $\bGamma_{\tilde E}$ into a parabolic group of $\SO(n,1)$ fixing $(1,-1,0,\dots, 0)$. 
By simple computations, 
we can find a sequence $g_i \in \bGamma_{\tilde E}$ so that $\{g_i(s)\}$ geometrically converges to 
a great segment. This contradicts the proper convexity of $\torb$. 
\end{proof}

\subsection{A complete end is horospherical}\label{sec:comphoro}





We will now show a converse of Proposition \ref{prop:affinehoro}.

The results here  overlap with the results of Crampon-Marquis \cite{CM2} and 
Cooper-Long-Tillman \cite{CLT}. However, the results are more general than theirs 
and were originally conceived before their papers appeared.  We also make use of 
Crampon-Marquis \cite{CM2}. 

This prove Theorem \ref{thm:secondmain} for complete R-ends.

\begin{theorem}\label{thm:comphoro} 
Let $\orb$ be a properly convex $n$-orbifold with radial or totally geodesic ends. 
Suppose that $\tilde E$ is a complete R-end of its universal cover $\torb$ in $\SI^n$ or 
in $\bR P^n$. Let $\bv_{\tilde E} \in \SI^n$  be the p-end vertex with the p-end fundamental group 
$\bGamma_{\tilde E}$. 
Then
\begin{itemize} 
\item[(i)] The eigenvalues of elements of  $\bGamma_{\tilde E}$ have unit norms only. 
\item[(ii)] A nilpotent Lie group fixing $\bv_{\tilde E}$ contains a finite index subgroup of $\bGamma_{\tilde E}$. 
\item[(iii)] $\tilde E$ is horospherical, i.e., cuspidal. 
\end{itemize} 
\end{theorem} 

\begin{proof} 
The proof here is for $\SI^n$ but it implies the $\bR P^n$-version. 

(i) Since $\tilde E$ is complete, $\tilde \Sigma_{\tilde E}$ is identifiable with $\bR^{n-1}$. 
$\bGamma_{\tilde E}$ induces $\Gamma'_{\tilde E}$ in $\Aff(\bR^{n-1})$
that are of form $x \mapsto Mx + b$ where $M$ is a linear map $\bR^{n-1} \ra \bR^{n-1}$
and $b$ is a vector in $\bR^{n-1}$. 
For each $\gamma \in \bGamma_{\tilde E}$, 
we write $\hat L(\gamma)$ this linear part of the affine transformation corresponding to $\gamma$. 

Suppose that one of the norms of relative eigenvalues of ${\hat L}(\gamma)$ for $\gamma \ne \Idd$ is greater than $1$ or less than $1$. 
At least one eigenvalue of ${\hat L}(\gamma)$ is $1$ since $\gamma$ acts without fixed point on $\bR^{n-1}$.
(See \cite{KS}.)
Now, ${\hat L}(\gamma)$ has a maximal vector  subspace $A$ of $\bR^{n-1}$ where all norms of the eigenvalues are $1$.

Suppose that $A$ is a proper subspace of $\bR^{n-1}$. 
Then $\gamma$ acts on the affine space $\bR^{n-1}/A$ 
as an affine transformation with the linear parts without a norm of eigenvalue equal to $1$.
Hence, $\gamma$ has a fixed point in $\bR^{n-1}/A$, and 
$\gamma$ acts on an affine subspace $A'$ parallel to $A$.

A subspace $H$ containing ${\bv_{\tilde E}}$ corresponds to the direction of $A'$ from $\bv_{\tilde E}$.
Let $H^+$ be the open half-space of one dimension higher corresponding to directions in $A$
with $\Bd H^+ \ni {\bv_{\tilde E}}$ so that $H^+$ is invariant under $\gamma$. 
For $\gamma$ as a projective transformation fixing the vertex ${\bv_{\tilde E}}$, 
the eigenvalue of $\gamma$ corresponding to ${\bv_{\tilde E}}$ equals 
the ones for the subspace $H^+$. This equals $\lambda_{\bv_{\tilde E}}$. 

There exists a projective subspace $S$ of dimension $\geq 0$ where 
the points are associated with eigenvalues $\lambda$
where $|\lambda| > \lambda_{\bv_{\tilde E}}$ up to reselecting $\gamma$ to be a nonzero integral power of $\gamma$ 
if necessary. 

Let $S'$ the smallest subspace containing $H$ and $S$. Let $U$ be a p-end-neighborhood of $\tilde E$. 
Let $y_1$ and $y_2$ be generic points of $U \cap S' - H^+$
so that $\ovl{y_1 y_2}$ meets $H$ in its interior.

Then we can choose a subsequence $m_i$, $m_i \ra \infty$, so that 
$\gamma^{m_i}(y_1) \ra f$ and $\gamma^{m_i}(y_2) \ra f_-$ as $i \ra +\infty$
unto relabeling $y_1$ and $y_2$ for a pair of antipodal points $f, f_- \in S$.
This implies $f, f_- \in \clo(\torb)$, 
and $\torb$ is not properly convex. This is a contradiction. 
Therefore,  the norms of eigenvalues of $\hat L(\gamma)$ all equal $\lambda_{\bv_{\tilde E}}$
and $A'$ is the $(n-1)$-dimensional affine subspace $\bR^{n-1}$. 
Thus, the norms of eigenvalues of $\gamma$ all equal to $1$ since the product of
the eigenvalues equal $\pm 1$. 

(ii) Since $\tilde \Sigma_{\tilde E}/\bGamma_{\tilde E}$ is a compact complete affine manifold, 
a finite index subgroup $F$ of $\bGamma_{\tilde E}$ is contained in a nilpotent Lie subgroup
by Theorem 3 in Fried \cite{Fried86} acting on $\tilde \Sigma_{\tilde E}$. 
Now, by Malcev, it follows that the same group is contained in 
a nilpotent group $N$ acting on $\SI^n$ since $F$ is unipotent. 

(iii) 
The dimension of $\CN$ is $n-1 = \dim \tilde \Sigma_{\tilde E}$ by Theorem 3 again. 

Let $U$ be a component of the inverse image of a p-end-neighborhood 
so that ${\bv_{\tilde E}} \in \Bd(U)$. Assume that $U$ is radial. 
Since a finite index subgroup $F$ of $\bGamma_{\tilde E}$ is in $N$ so that $N/F$ is compact by Malcev, 
$N$ will act on a smaller open set covering a p-end-neighborhood
by taking intersections under images of it under $N$ if necessary. 
We let $U$ be this open set from now on. Consequently, $\Bd U \cap \torb$ is smooth. 

We will now show that $U$ is a horospherical p-end-neighborhood:
We identify ${\bv_{\tilde E}}$ with $[1, 0, .., 0]$.
Let $W$ denote the subspace in $\SI^n$ containing ${\bv_{\tilde E}}$ supporting $U$. 
$W$ correspond to the boundary of the direction of $\tilde \Sigma_{\tilde E}$ and hence is unique 
and, thus, $N$-invariant. Also, 
$W \cap \clo(\torb)$ is a properly convex subset of $W$. 

Let $y$ be a point of $U$. 
Suppose that $N$ contains sequence $\{g_i\}$ so that  
\[g_i(y) \ra x_0 \in W \cap \clo(\torb) \hbox{ and } x_0 \ne {\bv_{\tilde E}};\] 
that is, $x_0$ in the boundary direction of $A$ from $\bv_{\tilde E}$. 
The collection of all such $x_0$ has a properly convex, convex hull $U_1$ in $\clo(\torb)$
in a subspace $V$ in $W$. 
The dimension of $V$ is $\geq 1$ as it contains $x_0$.

Again $N$ acts on $V$. 
Now, $V$ is divided into disjoint open hemispheres of various dimensions where $N$ acts on:
By Theorem 3.5.3 of \cite{Var}, $N$ preserves a flag structure 
$V_0 \subset V_1 \subset \dots \subset V_k = V$.
We take components of complement $V_i - V_{i-1}$. 
Let $H_V:=V - V_{k-1}$. 

Suppose that $\dim V = n-1$ for contradiction. 
Then $H_V \cap U_1$ is not empty since otherwise, we would have a smaller dimensional $V$. 
Let $h_V$ be the component of $H_V$ meeting $U_1$.
Since $N$ is orthopotent, $h_V$ has an $N$-invariant metric by Theorem 3 of Fried \cite{Fried86}.

We claim that the orbit of the action of $N$ is of dimension $n-1$ and hence locally transitive on $H_V$: 
If not, then a one-parameter subgroup $N'$ fixes a point of $h_V$.  
This group acts trivially on $h_V$ by the unipotent group contains a trivial orthogonal subgroup. 
Since $N'$ is not trivial, it acts as a group of nontrivial translations on the affine space $H^o$. 
Then $N'(U)$ is not properly convex. Also, an orbit of $N$ is open. 
Thus, $N$ acts transitively on $h_V$ since the orbit of $N$ in $h_V$ is closed by the invariant metric on $h_V$. 

Hence, the orbit $N(y)$ of $N$ for $y \in H_V \cap U_1$ contains a component of $H_V$. 
Since $\bGamma_{\tilde E}(y) \subset \clo(\torb)$
and a convex hull in $\clo(\torb)$ is $N(y)$ where $N(y) \subset H_V$. 
Since $F \bGamma_{\tilde E}  = N$ for a compact subset $F$ of $N$, 
the orbit $\bGamma_{\tilde E}(y)$ is within a bounded distance from every point of $N(y)$. 
Thus, a convex hull in $\clo(H_V)$ is $N(y)$, and 
this contradicts the assumption that $\clo(\torb)$ is properly convex
(compare with arguments in \cite{CM2}.)

Suppose that the dimension of $V$ is $\leq n-2$. 
Let $H$ be a subspace of dimension $1$ bigger than $\dim V$ and containing $V$ and meeting $U$. 
Then $H$ is sent to disjoint subspaces or to itself under $N$. Since $N$ acts transitively on $A$, 
a nilpotent subgroup $N_H$ of $N$ acts on $H$.
Now we are reduced to $\dim V$ by one or more. 
The orbit $N_H(y)$ for a limit point $y \in H_V$ contains a component of $V - V_{k-1}$
as above. Thus, $N_H(y)$ contains the same component, an affine subspace. 
As above, we have a contradiction to the proper convexity. 

Therefore, points such as $x_0 \in W \cap \Bd(\torb)  - \{\bv_{\tilde E}\}$ do not exist. 
Hence for any sequence of elements $g_i \in \bGamma_{\tilde E}$, we have
$g_i(y) \ra \bv_{\tilde E}$. 

Hence, $\Bd U = (\Bd U \cap \torb) \cup \{ \bv_{\tilde E}\}$.  Since the directions from ${\bv_{\tilde E}}$ to $\Bd U \cap \torb$
form $\bR^{n-1}$, $\Bd U$ is $C^1$ at ${\bv_{\tilde E}}$. 
Since $U$ is radial, this means that $U$ is a horospherical p-end-neighborhood. 
Clearly, $\Bd U$ is homeomorphic to an $(n-1)$-sphere.

\end{proof}

\part{Uniform middle eigenvalue conditions and lens-type ends.}

In this part, we will concentrate on ends that are properly convex.

\section{The end theory}
\label{sec:endth}
In this section, we discuss the properties of lens-shaped radial and totally geodesic ends and their duality also.

\subsection{The holonomy homomorphisms of the end fundamental groups: the fibering} \label{sub:holfib}

We will discuss for $\SI^n$ only here but the obvious $\bR P^n$-version exists for the theory. 
Let $\tilde E$ be a p-R-end of $\torb$. 
Let $\SLnp_{\bv_{\tilde E}}$ be the subgroup of $\SLnp$ fixing a point $\bv_{\tilde E} \in \SI^n$.
This group can be understood as follows by letting $\bv_{\tilde E} = [0, \ldots, 0, 1]$ 
as a group of matrices: For $g \in \SLnp_{\bv_{\tilde E}}$, we have 
\[ \left( \begin{array}{cc} 
        \frac{1}{\lambda_{\bv_{\tilde E}}(g)^{1/n}} \hat h(g) & \vec{0} \\ 
        \vec{v}_g                & \lambda_{\bv_{\tilde E}}(g)
        \end{array} \right) \] 
where $\hat h(g) \in \SLn, \vec{v} \in \bR^{n \ast}, \lambda_{\bv_{\tilde E}}(g) \in \bR_+ $, so-called the linear part of $h$.
Here, \[\lambda_{\bv_{\tilde E}}: g \mapsto \lambda_{\bv_{\tilde E}}(g) \hbox{ for } g \in \SLnp_{\bv_{\tilde E}}\] is a homomorphism 
so it is trivial in the commutator group $[\Gamma_{\tilde E}, \Gamma_{\tilde E}]$. 
There is a group homomorphism ${\mathcal L}': \SLnp \ra \SLn \times \bR_+$ 
by sending the above matrix to $(\hat h, \lambda)$ with kernel $\bR^{n \ast}$ a dual space to $\bR^n$. 
Thus, we obtain a diffeomorphism \[\SLnp_{\bv_{\tilde E}} \ra \SLn \times \bR^{n \ast} \times \bR_+.\]
We note the multiplication rules
\[ (A, \vec{v}, \lambda) (B, \vec{w}, \mu) = (AB, \frac{1}{ \mu^{1/n} } \vec{v}B + \lambda \vec{w}, \lambda \mu). \] 

Let $\Sigma_{\tilde E}$ be the end $(n-1)$-orbifold. 
Given a representation $\hat h: \pi_1(\Sigma_{\tilde E}) \ra \SLn$ and $\lambda: \pi_1(\Sigma_{\tilde E}) \ra \bR_+$, 
we denote by $\bR^{n}_{\hat h, \lambda}$
the $\bR$-module with the $\pi_1(\Sigma_{\tilde E})$-action given 
by \[g\cdot \vec v = \frac{1}{\lambda(g)^{1/n}}\hat h(g)(\vec v).\] 
And we denote by $\bR^{n \ast}_{\hat h, \lambda}$ will the dual vector space
with the dual action given by 
\[g\cdot \vec v = \frac{1}{{\lambda(g)^{1/n}}}\hat h(g)^{\ast}(\vec v).\] 
Let $H^1(\pi_1(\tilde E), \bR^{n \ast}_{\hat h, \lambda})$ denote the cohomology 
space of $1$-cocycles $\vec v(g) \in  \bR^{n \ast}_{\hat h, \lambda}$ 


As $\Hom(\pi_1(\Sigma_{\tilde E}), \bR_+)$ equals $H^1(\pi_1(\Sigma_{\tilde E}), \bR)$, we obtain: 

\begin{theorem} \label{thm:defspace}
Let $\orb$ be a strongly tame properly convex real projective orbifold with radial or totally geodesic ends and 
let $\torb$ be its universal cover. 
Let $\Sigma_{\tilde E}$ be the end orbifold associated with a p-R-end $\tilde E$ of $\torb$. 
Then the space of representations 
\[\Hom(\pi_1(\Sigma_{\tilde E}), \SLnp_{\bv_{\tilde E}})/\SLnp_{\bv_{\tilde E}}\] 
is  the fiber space $B$ over 
\[\Hom(\pi_1(\Sigma_{\tilde E}), \SLn)/\SLn \times H^1(\pi_1(\Sigma_{\tilde E}), \bR)\]
with the fiber isomorphic to $H^1(\pi_1(\Sigma_{\tilde E}), \bR^{n \ast}_{\hat h, \lambda}) $ 
for each $([\hat h], \lambda)$. 
\end{theorem}


We remark that we don't really understand the fiber dimensions and their behavior as we change 
the base points. A similar idea is given by Mess  \cite{Mess}. 
In fact the dualizing these matrices gives us
a representation to $\Aff(A^n)$. In particular if we restrict ourselves 
to linear parts to be in $\SO(n, 1)$, then we are exactly in the cases studied by Mess. 
(See the concept of the duality in Subsection \ref{sub:affdualtub} and Appendix \ref{app:dual}.)
Thus, one interesting question that Benoist and we talked about is how to compute the dimension of 
$H^1(\pi_1(\Sigma_{\tilde E}), \bR^{n \ast}_{\hat h, \lambda}) $ under some general conditions on $\hat h$.

\subsubsection{Tubular actions.}

Let us give a pair of antipodal points $\bv$ and $\bv_-$. 
If a group $\Gamma$ of projective automorphisms fixes a pair of fixed points $\bv$ and $\bv_-$, 
then $\Gamma$ is said to be {\em tubular}.
There is a projection $\Pi_{\bv}: \SI^n -\{\bv, \bv_-\} \ra \SI^{n-1}_{\bv}$ given 
by sending every great segments with endpoints $\bv$ and $\bv_-$
to the sphere of directions at $\bv$. 
(We denote by $\bR P^{n-1}_{\bv}$ the quotient of $\SI^{n-1}_\bv$ under the antipodal map 
given by the change of directions. 
We use the same notation $\Pi_{\bv}: \bR P^n -\{\bv\} \ra \bR P^{n-1}_{\bv}$ the induced projection.)

A {\em tube} in $\SI^n$ (resp. in $\bR P^n$) is the closure of the inverse image of a convex domain $\Omega$
in $\SI^{n-1}_{\bv}$ (resp. in $\bR P^{n-1}_{\bv}$).
Given a p-R-end $\tilde E$ of $\torb$, the {\em end domain} is $R_{\bv}(\torb)$. 
If a p-R-end $\tilde E$ has the end domain $\tilde \Sigma_{\tilde E}$, 
$h(\pi_1(\tilde E))$ acts on the tube domain ${\mathcal T}_{\tilde E}$ associated with $\tilde \Sigma_{\tilde E}$. 

We will now discuss for the $\SI^n$-version but the $\bR P^n$ version is obviously clearly obtained from this 
by minor discussions. 

Letting $\bv$ have the coordinates $[0, \dots, 0, 1]$, we obtain 
the matrix of $g$ of $\pi_1(\tilde E)$ of form 
\begin{equation}\label{eqn:bendingm3} 
\left(
\begin{array}{cc}
\frac{1}{\lambda_{\bv}(g)^{\frac{1}{n}}} \hat h(g)          &       0                \\
\vec{b}_g           &      \lambda_{\bv}(g)                  
\end{array}
\right)
\end{equation}
where $\vec{b}_g$ is an $n\times 1$-vector and $\hat h(g)$ is an $n\times n$-matrix of determinant $\pm 1$
and $\lambda_{\bv}(g) $ is a positive constant. 

Note that the representation $\hat h: \pi_1(\tilde E) \ra \SLn$ is given by 
sending $g \mapsto \hat h(g)$. Here we have $\lambda_{\bv}(g) > 0$.  
If $\tilde \Sigma_{\tilde E}$ is properly convex, then the convex tubular domain is {\em properly tubular}
and the action is said to be {\em properly tubular}. 

\subsubsection{Affine actions  dual to tubular actions.}\label{sub:affdualtub}





The automorphism in $\SI^n$ acting on a codimension-one subspace ${\SI^{n-1}_\infty}$ of ${\mathcal S}(\bR^{n+1})$ and 
the components of the complement
acts on an affine space $A^n$, a component of the complement of ${\SI^{n-1}_\infty}$.  
The subgroup of projective automorphisms preserving ${\SI^{n-1}_\infty}$ and the components equals
the affine group $\Aff(A^n)$.

By duality, a great $(n-1)$-sphere ${\SI^{n-1}_\infty}$ corresponds to a point $\bv_{\SI^{n-1}_\infty}$. 
Thus, for a group $\Gamma$ in $\Aff(A^n)$, 
the dual groups $\Gamma^*$ acts on ${\mathcal S}(\bR^{n+1, *})$ fixing $\bv_{\SI^{n-1}_\infty}$.
(See Proposition \ref{prop:duality} also.)
%




Suppose that $\Gamma$ acts on a properly convex open domain $U$ where $\Omega := \Bd U \cap {\SI^{n-1}_\infty}$
is a properly convex domain. 
We call $\Gamma$ a {\em properly convex affine} action.
We claim that the dual group $\Gamma^*$ acts on a properly tubular
domain $B$ with vertices $\bv:= \bv_{\SI^{n-1}_\infty}$ and $\bv_- := \bv_{{\SI^{n-1}_\infty}, -}$.
The domain $\Omega^o$ and domain $R_{\bv}(B)$ in the linking sphere $\SI^{n-1}_{\bv}$ from $\bv$ in direction of $B^o$
are projectively diffeomorphic to a pair of dual domains: 
\begin{itemize}
\item Given $\Omega^o$, we obtain the dual domain $\Omega^{o\ast}$ in $\SI^{n-1 \ast}_\infty$. 
\item A supporting $n-1$-hemisphere in $\SI^{n-1}_\infty$ of $\Omega$ corresponds to 
a point of $\Bd \Omega^{o\ast}$ and vice versa. 
\item An open $n$-hemisphere supporting $\Omega^o$ contains an $n-1$-hemisphere in $\SI^{n-1}$
supporting $\Omega^o$. 
\item The set of $n$-hemispheres containing a fixed supporting $n-1$-hemisphere of $\SI^{n-1}$ and  supporting  $\Omega^o$  
forms a great open segment in $\SI^{n \ast}$ with end points $\bv$ and $\bv_-$. 
\item The set $\Bd \Omega^{o\ast}$ parametrizes the space of such open segments. 
Let $I_x$ for $x \in \Bd \Omega^{o\ast}$ denote such a segment. 
\item  $\bigcup_{x\in  \Bd \Omega^{o\ast}} \clo(I_x)$ is 
the boundary of a convex tube $B:= {\mathcal T}(\Omega^{o\ast})$ with vertices $\bv$ and $\bv_-$. 
\item Thus, there is a one-to-one correspondence between 
the set of open $n$-hemispheres supporting $\Omega^o$ and the set of $\Bd {\mathcal T}(\Omega^{o\ast}) -\{\bv, \bv_-\}$.
(Also, $R_{\bv}(B) = \Omega^{o\ast}$.)
\end{itemize}



Given a convex open subset $U$ of $A^n$, an asymptopic hyperspace $H$ of $U$ at 
a point $x \in \Bd A^n \cap \clo(\Bd U)$ is a hyperspace so that a component of $A^n -H$ contains $U$.

\begin{definition}\label{defn:tubular}
A properly tubular action is said to be {\em distanced} if the tubular domain contains 
a properly convex compact $\Gamma$-invariant subset disjoint from the vertices. 
A properly convex affine action of $\Gamma$ is said to be {\em asymptotically nice} if 
$\Gamma$ acts on a properly convex open domain $U'$ in $A^n$ with boundary in 
$\Omega \subset {\SI^{n-1}_\infty}$ 
so that a supporting hyperspace $H_x$ exists with $H_x$ at each point $x \in \Bd \Omega$ 
is not in ${\SI^{n-1}_\infty}$.
\end{definition}
(Here, we can choose $H_x$ so that $\{H_x| x\in \Bd \Omega\}$ is $\Gamma$-invariant by choosing 
the supporting hyperspaces to $U$ in $A^n$ for each $x$.)



\begin{proposition}\label{prop:dualDA} 
Let $\Gamma$ and $\Gamma^*$ be dual groups where $\Gamma$ has an affine action on $A^n$ and $\Gamma^*$ is tubular with 
the vertex $\bv = \bv_{\SI^{n-1}_\infty}$ dual to the boundary $\SI^{n-1}_\infty$ of $A^n$.
Let $\Gamma= (\Gamma^*)^*$ acts on a convex open domain $\Omega$ with compact $\Omega/\Gamma$.
Then $\Gamma$ acts asymptotically nicely if and only if 
$\Gamma^*$ acts on a properly tubular domain $B$ and is distanced. 
\end{proposition}
\begin{proof} 
For each point $x$ of $\Bd \Omega$, 
an open hemisphere in $\SI^n$ at $x$ supports $\Omega$
uniformly bounded at a distance in the Hausdorff metric $\bdd_H$-sense from the open hemisphere $A^n$ with boundary ${\SI^{n-1}_\infty}$
or $\SI^n - A^n$. We choose the set so that $\Gamma$ acts on it. 
(Otherwise, $U$ would become empty.) 

The dual points of the supporting hyperplanes of $U$ at $\Bd \Omega$ are points on $\Bd B$ for a tube domain $B$
with vertex $\bv$ dual to $\SI^{n-1}_\infty$. 
Since the hyperspheres of form $H$ supporting $U$ at $x \in \Bd \Omega$, 
are bounded at a distance from ${\SI^{n-1}_\infty}$  in the $\bdd_H$-sense, the dual points are uniformly bounded 
at a distance from the vertices $\bv$ and $\bv_-$. Let us call this compact set $K$. Then for every point of 
$\Bd \Omega^*$ the boundary of the dual $\Omega^* \subset \SI^{n-1}_{\bv}$ of $\Omega$, 
we have a point of $K$ in the corresponding great segment from $\bv$ to $\bv_-$.
$K$ is uniformly bounded at a distance from $\bv$ and $\bv_-$  in the $\bdd$-sense.
The convex hull of $K$ in $\clo(\torb)$  is a compact convex set bounded at a uniform distance from $\bv$ and $\bv_-$
since the tube domain is properly convex. 
Since $K$ is $\Gamma$-invariant, so is the convex hull in $\clo(\torb)$. 

Conversely, every compact convex subset $K$ of the tubular domain $B$ bounded away from 
$\bv$ and $\bv_-$ meets a great segment from $\bv$ to $\bv_-$ at a point bounded away from the end points. 
Let $A'$ denote the set $\Bd B - \{\bv, \bv_-\}$. 
Then $K \cap A'$ is a compact convex and $\Gamma$-invariant and bounded away from $\bv, \bv_-$. 
We denote by $K'$ the boundary of a component $A'- K$ containing $\bv$ in its closure. 
Then $K'$ is again $\Gamma$-invariant. 


Each point $x$ of $K'$ is dual to a hypersphere $P$ in $\SI^n$ bounded at a distance from ${\SI^{n-1}_\infty}$ 
since $x$ is bounded at a distance from $\bv, \bv_-$. 
Since $x \in \Bd B$, $P$ must be a supporting plane to the dual of $B$, a convex domain $\Omega$ in ${\SI^{n-1}_\infty}$; and 
$P \cap A^n$ is a complete hyperplane with a point of $\Bd \Omega$ its boundary in $\SI^n$. 
The intersection of the corresponding half-spaces in $A^n$ is not empty and is a properly convex open domain.



\end{proof}

\begin{theorem}\label{thm:distanced}
Let $\Gamma$ be a nontrivial properly convex tubular action at vertex $\bv = \bv_{\SI^{n-1}_\infty}$ on 
$\SI^n$ {\rm (}resp. in $\bR P^n${\rm )} 
and acts on a properly convex tube $B$
and satisfies the uniform middle-eigenvalue conditions. 
Then $\Gamma$ is distanced inside the tube $B$  where $\Gamma$ acts on 
and the minimal distanced $\Gamma$-invariant compact set $K$ in $B$ 
is uniquely determined. Furthermore, $K$ meets each open boundary segment in $\partial B$ 
at unique point. 
\end{theorem} 
\begin{proof}
Let $\bv$ be the vertex of  $B$. 
First assume that $\Gamma$ induces an irreducible action on the link sphere $\SI^{n-1}_{\bv}$. 
The dual group $\Gamma^*$ acts on a properly convex domain $U^* \subset \SI^n$ dual to $U$.  
Then the closure of $U^*$ meets ${\SI^{n-1}_\infty}$ in a domain $\Omega^*$ dual to the convex domain in $\SI^{n-1}_{\bv}$
corresponding to the tube of $\Gamma$.
By Theorem \ref{thm:asymnice}, $\Gamma^*$ is asymptotically nice. 
Proposition \ref{prop:dualDA} implies the result.
The uniqueness part of Theorem \ref{thm:asymnice} 
implies the uniqueness of the minimal set and the last statement.

Suppose that $\Gamma$ acts reducibly on $\SI^{n-1}_{\bv}$. 
Then $\Gamma$ is isomorphic to $\bZ^{l_0-1} \times \Gamma_1 \times \dots \times \Gamma_{l_0}$ where 
$\Gamma_i$ is nontrivial hyperbolic for $i=1, \dots, s$ and trivial for $s+1 \leq i \leq l_0$ where $s \leq  l_0$. 
Each $\Gamma_i$ for $i = 1, \dots, s$ acts on a nontrivial tube $B_i$ with vertices $\bv$ and $\bv_-$ in a subspace. 
By above it is a distanced action and it acts on a $\Gamma_i$-invariant compact convex set $K_i \subset B_i$
disjoint from $\{ \bv, \bv_-\}$. 

For each $i$, $s+1 \leq i \leq r$,  $B_i$ is a great segment with endpoints $\bv$ and $\bv_-$. 
A point $p_i$ corresponds to $B_i$ in $\SI^{n-1}_{\bv}$. 
For $g$ in the center, we have $\lambda_1(g) > \lambda_{\bv}(g)$
for the eigenvalues in $\SI^n$. 
Hence, some central $g$ fixes a point $p'_i \in B_i$ where $\lambda_1(g) > \lambda_{\bv}(g)$ holds. 
By commutativity, $p'_i$ is a fixed point of every element of $\bZ^{s-1}$. 

The convex hull of
\[K_1 \cup \cdots \cup K_s \cup \{p'_{s+1}, \dots, p'_{l_0}\}\] 
in $\clo(B)$
is a distanced $\Gamma$-invariant compact convex set. 

Suppose that $\Gamma$ satisfies the uniform middle-eigenvalue conditions. 
Then $K_i$ for $i=1, \dots, s$ is unique by Theorem \ref{thm:asymnice} and Proposition \ref{prop:dualDA}. 
Also, $p'_i$ is uniquely determined for $i \geq s+1$. 
\end{proof}

\subsubsection{The duality of T-ends and properly convex R-ends}\label{sub:dualend}

The Vinberg duality diffeomorphism induces a one-to-one correspondence between p-ends of $\torb$ and $\torb^*$ 
by considering the dual relationship $\bGamma_{\tilde E}$ and $\bGamma^*_{\tilde E'}$ for each pair of 
p-ends $\tilde E$ and $\tilde E'$ with dual p-end fundamental groups.



\begin{proposition}\label{prop:dualend} 
Let $\orb$ be a strongly tame properly convex real projective manifold with radial ends or totally geodesic ends.
Then the dual real projective orbifold $\orb^*$ is also strongly tame and has the same number of ends so that 
\begin{itemize} 
\item there exists a one-to-one correspondence $\mathcal{C}$ between the set of ends of $\orb$ and the set of ends of $\orb^*$. 
\item $\mathcal{C}$ restrict to such a one  between the subset of horospherical ends of $\orb$ and the subset of horospherical ones of $\orb^*$.
\item $\mathcal{C}$ restrict to such a one  between the subset of 
totally geodesic ends of $\orb$ with the subset of ends of properly convex radial ones of $\orb^*$.
The ideal boundary $\tilde S_{\tilde E}$ is projectively dual to 
$\tilde \Sigma_{\tilde E^*}$ for the corresponding dual end $\tilde E^*$ of $\tilde E$. 
\item $\mathcal{C}$ restrict to such a one  between 
the subset of all properly convex R-ends of $\orb$ and the subset of all T-ends of $\orb^*$. 
Also, $\tilde \Sigma_{\tilde E}$ is projectively dual to the ideal boundary $\tilde S_{\tilde E^*}$ 
for the corresponding dual end $\tilde E^*$ of $\tilde E$. 
\end{itemize} 
\end{proposition}
\begin{proof}
We prove for the $\SI^n$-version.
By the Vinberg dual diffeomorphism of Theorem \ref{thm:dualdiff}, $\orb^*$ is also strongly tame. 
Let $\torb$ be the universal cover of $\orb$. Let $\torb^*$ be the dual domain. 
The first item is proved above. 
Each p-end vertex $x$ of $\torb$ has a supporting hyperplanes of $C$ whose supporting 
linear functions form a properly convex domain of dimension $n-1$ if $x$ corresponds to a properly convex 
p-R-end. 
Since there is a subgroup of a cusp group acting on $\clo(\torb)$ with $x$ fixed, 
the intersection of the unique supporting hyperspace $h$ with $\clo(\torb)$ is a singleton $\{x\}$. 
The dual subgroup is also a cusp group and acts on $\clo(\torb^*)$ with $h$ fixed.
So the corresponding $\torb^*$ has the dual hyperspace $x^*$ of $x$ as the unique intersection 
at $h^*$ dual to $h$ at $\clo(\torb^*)$. Hence $x^*$ is a horospherical end. 

A  p-T-end $\tilde E$ of $\torb$ corresponds to a p-R-end of $\torb^*$ whose p-end fundamental group
acts fixing a point $\bv$ dual to the hyperplane containing $\tilde S_{\tilde E}$. 
The point is in the boundary of $\torb^*$  by Proposition \ref{prop:duality} 
and hence the corresponding end is radial. 
Each point of $\clo(\tilde S_{\tilde E})-\tilde S_{\tilde E}$ corresponds to a supporting hyperplane of $\torb^*$ at $\bv$. 
By taking a linear coordinate centered at $\bv$ of an affine patch containing $\bv$, 
we obtain that $R_{\bv}(\torb^*)$ is projectively dual to the convex domain
projectively diffeomorphic to $\tilde S_{\tilde E}$ by Proposition \ref{prop:duality}.  
The third item follows. The fourth item follows similarly. 
\end{proof} 
An NPCC p-R-end may have a dual end that is not radial nor totally geodesic.  



\begin{proposition}\label{prop:dualend2}
Let $\orb$ be a strongly tame properly convex real projective manifold with radial ends or totally geodesic ends.
The following conditions are equivalent\,{\rm :} 
\begin{itemize} 
\item[(i)] A properly convex R-end of $\orb$ satisfies the uniform middle-eigenvalue condition.
\item[(ii)] The corresponding totally geodesic end of $\orb^*$ satisfies this condition.
\end{itemize} 
\end{proposition}
\begin{proof}
The items (i) and (ii) are equivalent by considering equation \ref{eqn:bendingm3}.
\end{proof}


\subsection{The properties of lens-shaped ends} \label{subsec:lens}

One of the main result of this section is that a generalized lens-type end has 
a ``concave end-neighborhood'' that actually covers a p-end-neighborhood.

\begin{figure}
\centering
\includegraphics[height=5cm]{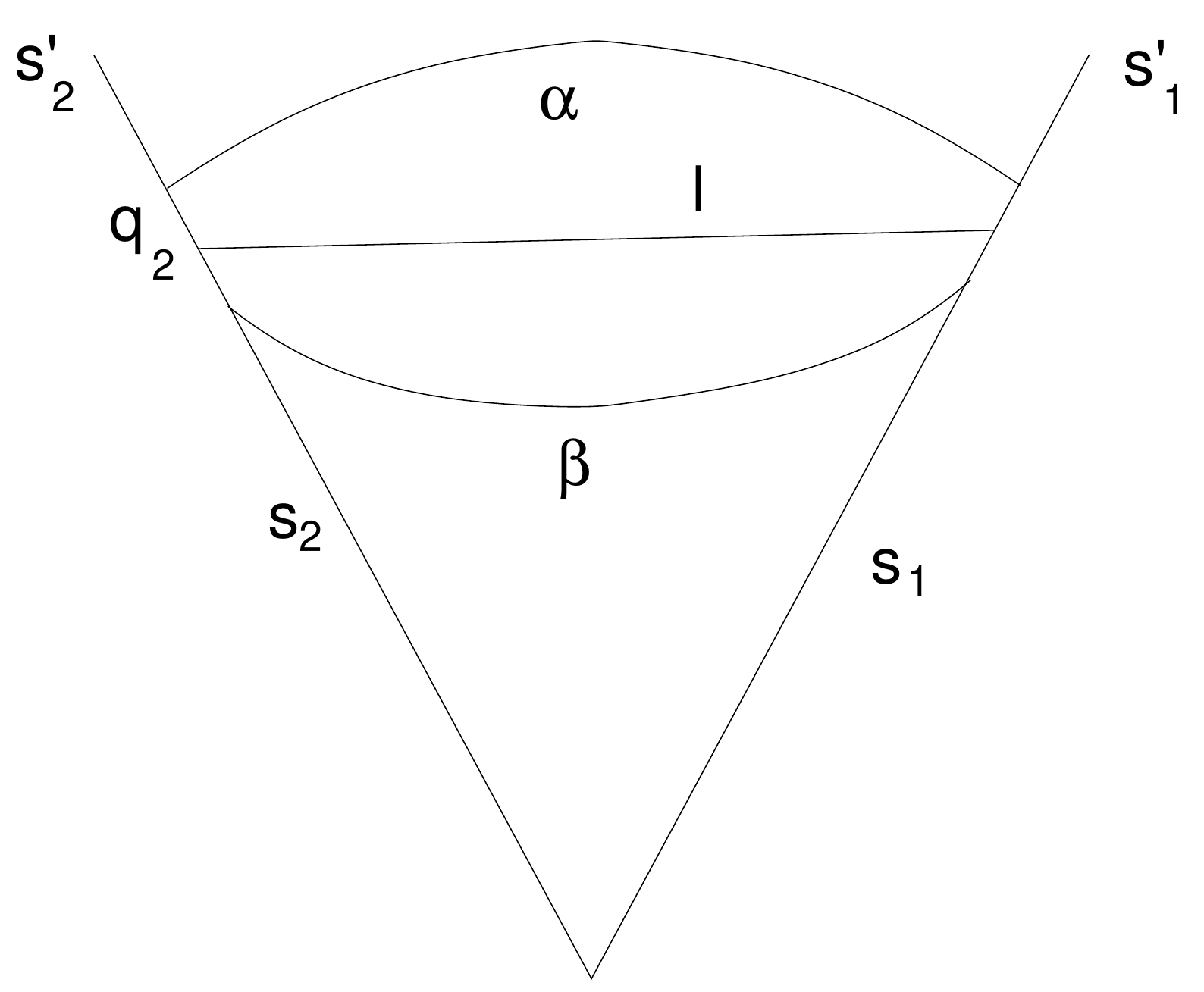}

\caption{The figure for Lemma \ref{lem:recurrent}. }

\label{fig:lenslem}
\end{figure}

Given three sequences of projectively independent points $\{p^{(j)}_i\}$ with $j=1, 2, 3$ 
so that $\{p^{(j)}_i\} \ra p^{(j)}$ where $p^{(1)}, p^{(2)}, p^{(3)}$ are independent points in $\SI^n$. 
Then a simple matrix computation shows that 
a uniformly bounded sequence $\{r_i\}$ of elements of $\Aut(\SI^n)$ or $\PGL(n+1, \bR)$
acts so that $r_i(p^{(j)}_i) = p^{(j)}$ for every $i$ and $j=1, 2, 3$.


A {\em convex arc} is an arc in a totally geodesic subspace where an arc projectively equivalent to 
a graph of a convex function $I \ra \bR$ for a connected interval in $\bR$.

Find the tube $B_{\tilde E}$ of great open segments with end points $\bv_{\tilde E}$ and $\bv_{\tilde E-}$ corresponding to elements of 
$\tilde \Sigma_{\tilde E}$. The union is a convex domain not properly convex with distinguished vertices $\bv_{\tilde E}$ and $\bv_{\tilde E-}$. 

We first need the following technical lemmas on recurrent geodesics.
\begin{lemma}\label{lem:recurrent} 
Let $\mathcal O$ be a properly convex real projective $n$-orbifold with radial or totally geodesic ends. 
Suppose that the universal cover $\torb$ is in $ \SI^n$.
Suppose that $g_i \in \SLnp$ be a sequence of automorphisms
so that $g_i(\bv_{\tilde E})= \bv_{\tilde E}$ for an end vertex ${\bv_{\tilde E}}$ and $l$ is a maximal segment in a generalized lens
with endpoints in $\Bd \tilde{\mathcal{O}}$.
{\rm (}See Figure \ref{fig:lenslem}.{\rm )}
Let $g'_i$ denote the induced projective automorphisms on $\SI^{n-1}_{\bv_{\tilde E}}$. 
$\{g'_i(l') \subset \tilde \Sigma_{\tilde E}\}$ converges geometrically to $l'$
where $l'$ is the projection of $l$ to the linking sphere $\SI^{n-1}_{\bv_{\tilde E}}$ of ${\bv_{\tilde E}}$. 
Let $P$ be the $2$-dimensional subspace containing ${\bv_{\tilde E}}$ and $l$. 
Furthermore, we suppose that 
\begin{itemize}
\item  In $P$, $l$ is in the disk $D$ bounded by two segments $s_1$ and $s_2$ from ${\bv_{\tilde E}}$ and a convex curve $\alpha$  with 
endpoints $q_1$ and $q_2$  that are end points of $s_1$ and $s_2$ respectively. 
\item $\beta$ is another convex curve with $\beta^o \subset D^o$ and endpoints in $s_1-\{\bv_{\tilde E}\}$ and $s_2-\{\bv_{\tilde E}\}$
so that $\alpha$ and $\beta$ and parts of $s_1$ and $s_2$ bound a convex disk in $D$. 
\item There is a sequence of points $\tilde q_i \in \alpha$ converging to $q_1$ and $g_i(\tilde q_i) \in F$ 
for a fixed fundamental domain $F$ of $\tilde{\mathcal{O}}$.
\item The sequences $g_i(D)$, $g_i(\alpha)$, $g_i(\beta)$, $g_i(s_1)$, and $g_i(s_2)$
converge to $D, \alpha, \beta, s_1,$ and $s_2$ respectively. 
\end{itemize} 
Then 
\begin{itemize} 
\item If the end points of $\alpha$ and $\beta$ do not coincide at $s_1$ or $s_2$, 
then $\alpha$ and $\beta$ must be geodesics from $q_1$ or $q_2$. 
\item Suppose that the pairs of endpoints of $\alpha$ and $\beta$ coincide and they are distinct curves.
Then no segment in $\clo(\tilde{\mathcal{O}})$ extends $s_1$ or $s_2$ properly.
\end{itemize} 
\end{lemma} 
\begin{proof}
By the geometric convergence conditions, 
we obtain a bounded sequence of elements $r_i \in \SLnp$ so that $r_i(g_i(s_1))=s_1$ and $r_i(g_i(s_2))=s_2$
and $\{r_i\} \ra \Idd$. Then $r_i \circ g_i$ is represented as an element of $\SL_\pm(3, \bR)$ in the projective plane $P$
containing $D$. Using ${\bv_{\tilde E}}$ and $q_1$ and $q_2$ as standard basis points, $r_i\circ g_i$ is represented as a diagonal matrix.
Moreover $\{r_i \circ g_i(\alpha)\}$ is still converging to $\alpha$ as $\{r_i\} \ra \Idd$. 
Hence, this implies that the diagonal elements of each $r_i \circ g_i$ are of form $\lambda_i, \mu_i, \tau_i$ where 
$\{\lambda_i\} \ra 0, \{\tau_i\} \ra +\infty$ as $i \ra \infty$ and 
$\lambda_i$ is associated with $q_1$ and 
$\mu_i$ is associated with ${\bv_{\tilde E}}$ and $\tau_i$ is associated with $q_2$.
(Thus, $r_i \circ g_i$ is diagonalizable with fixed points $q_1, q_2, \bv_{\tilde E}$.) 

We have that $\{r_i\circ g_i(\beta)\}$ also converges to $\beta$. If the end point of $\beta$ at $s_1$ is different from that of $\alpha$, 
then \[1/C < \{\log|\lambda_i/\mu_i|\} < C \hbox{ for a constant } C >0:\] 
Otherwise, for the endpoint $\delta_1 \beta$, we obtain a contradiction
\[r_i\circ g_i(\delta_1\alpha) \ra q_1 \hbox{ or } \bv_{\tilde E} \hbox{ for } i \ra \infty.\]
Since $r_i \circ g_i(\delta_1 \beta) \ra \delta_1\beta$, it follows that $\lambda_i/\mu_i \ra 1$.
In this case, $\beta$ has to be a geodesic from $q_2$ 
since $\{r_i \circ g_i(\beta)\} \ra \beta$. 
And so is $\alpha$. The similar argument holds for the case involving $s_2$

For the second item, $\{\mu_i/\tau_i\} \ra 0, +\infty$ and $\{\lambda/\mu_i\} \ra 0, +\infty$ also
since otherwise we can show that $\beta$ and $\alpha$ have to be geodesics with distinct endpoints as above. 
If a segment $s'_2$ in $\clo(\tilde \Omega)$ extends $s_2$,
then $\{r_i \circ g_i(s'_2)\}$ converges to a great segment
and so does $\{g_i(s'_2)\}$ as $i \ra \infty$ or $i \ra -\infty$. This contradicts the proper convexity of $\mathcal{O}$. 

\end{proof}

%


%


A {\em trivial one-dimensional cone} is an open half space in $\bR^1$ given by $x > 0$ or $x < 0$. 

Recall that if $\pi_1(\tilde E)$ is an admissible group, then $\pi_1(\tilde E)$ has  a finite index subgroup isomorphic to 
$\bZ^{k-1} \times \bGamma_1 \times \cdots \times \bGamma_k$ for some $k \geq 0$
where each $\bGamma_i$ is hyperbolic or trivial. 

Let us consider $\Sigma_{\tilde E}$ the real projective $(n-1)$-orbifold associated with $\tilde E$ and 
consider $\tilde \Sigma_{\tilde E}$ as a domain 
in $\SI^{n-1}_{\bv_{\tilde E}}$ and $h(\pi_1(\tilde E))$ induces $\hat h: \pi_1(\tilde E) \ra \SL_\pm(n, \bR)$ acting on $\tilde \Sigma_{\tilde E}$. 
We denote by $\Bd \tilde \Sigma_{\tilde E}$ the boundary of $\tilde \Sigma_{\tilde E}$ in $\SI^{n-1}$. 

\begin{definition} \label{defn:sl}
A (generalized) lens-shaped p-R-end with the p-end vertex $\bv_{\tilde E}$ is {\em strictly generalized lens-shaped} 
if we choose a (generalized) lens domain $D$
with the top hypersurfaces $A$ and $B$ so that each great open segment in $\SI^n$ from $\bv_{\tilde E}$ in the direction of 
$\Bd \tilde \Sigma_{\tilde E}$ meets 
$\clo(D) - A - B$ at a unique point. 
\end{definition}
As a consequence $\clo(A) - A = \clo(B) - B$ and $\clo(A) \cup \clo(B) = \Bd D$.

\begin{theorem}\label{thm:lensclass}
Let $\mathcal{O}$ be a strongly tame properly convex $n$-orbifold with radial or totally geodesic ends. 
Assume that the universal cover $\torb$ is a subset of $\SI^n$ {\rm (}resp. in $\bR P^n${\rm )}.
Let $\tilde E$ be a p-R-end of $\tilde{\mathcal{O}}$ with a generalized lens p-end-neighborhood. 
Let $\bv_{\tilde E}$ be an associated with a p-end vertex. 
Assume that $\pi_1(\tilde E)$ is hyperbolic.
\begin{itemize} 
\item[(i)] The complement of the manifold boundary of the generalized lens-shaped domain $D$ is a nowhere dense set
in $\Bd \clo(D)$. Moreover, $\Bd D -\partial D$ is independent of the choice of $D$, and 
$D$ is strictly generalized lens-shaped. Moreover, each element $g \in \bGamma_{\tilde E}$ has an attracting 
fixed point in $\Bd \clo(D)$ in the great segment from $\bv_{\tilde E}$ in $\Bd \tilde \Sigma_{\tilde E}$. 
The set of attracting fixed points is dense in $\Bd \clo(D) - A - B$ for the top and the bottom hypersurfaces. 
\item[(ii)] 
The closure in $\clo(V)$ of a concave p-end-neighborhood $V$ of $\bv_{\tilde E}$ contains 
every segment $l \subset \Bd \tilde{\mathcal{O}}$ 
with $l^o \cap \clo(U) \neq \emp$ for  any concave p-end-neighborhood $U$ of $\bv_{\tilde E}$. 
The set $S(\bv_{\tilde E})$ of maximal segments from $\bv_{\tilde E}$ in $\clo(V)$ is independent of $V$,
and $\bigcup S(\bv_{\tilde E}) = \clo(U) \cap \Bd \tilde{\mathcal{O}}.$
\item[(iii)] $S(g(\bv_{\tilde E}))=g(S(\bv_{\tilde E}))$ for $g \in \pi_1(\tilde E)$. 
\item[(iv)] Any concave p-end-neighborhood $U$ of $\bv_{\tilde E}$ under the covering map $\tilde{\mathcal{O}} \ra \mathcal{O}$
covers the p-end-neighborhood of $\tilde E$ of form $U/\pi_1(\tilde E)$. That is, a concave p-end-neighborhood is a proper p-end-neighborhood. 
\item[(v)] Assume that $w$ is 
the p-end vertex of an irreducible hyperbolic p-R-end.  
Then $S^o(\bv_{\tilde E}) \cap S(w) = \emp$ or $\bv_{\tilde E}=w$ for 
p-end vertices $\bv_{\tilde E}$ and $w$ where we defined $S^o(\bv_{\tilde E})$ to denote 
the relative interior of $\bigcup S(\bv_{\tilde E})$ in $\Bd \tilde{\mathcal{O}}$. 
\end{itemize}
\end{theorem}
\begin{proof}  The proof is done for $\SI^n$ but the result implies the $\bR P^n$-version. 
Here the closure is independent of the ambient spaces. 

(i) By Fait 2.12 \cite{Ben3}, we obtain that $\pi_1(\tilde E)$ is vcf and 
acts irreducibly on a proper convex cone and the cone has to be strictly convex by Theorem 1.1 of \cite{Ben2}. 

Let $C_{\tilde E}$ be a concave end. Since $\Gamma_{\tilde E}$ acts on $C_{\tilde E}$. 
By definition, $C_{\tilde E}$ is a component of the complement of a generalized lens domain $D$ in 
a generalize R-end. 

We have a domain $D$ with boundary components $A$ and $B$ transversal to the lines in $R_{\bv_{\tilde E}}(\torb)$. 
We can assume that $B$ is strictly concave and smooth as we have a concave end-neighborhood. 
$\bGamma_{\tilde E}$ acts on both $A$ and $B$. We define
 $A_1:=\clo(A) - A$ and $B_1:= \clo(B) - B$. 

By Theorem 1.2 of \cite{Ben1}, the geodesic flow on the real projective 
$(n-1)$-orbifold $\tilde \Sigma_{\tilde E}/\bGamma_{\tilde E}$ is topologically mixing, 
i.e., recurrent since $\bGamma_{\tilde E}$ is hyperbolic. 
Thus, each geodesic $l$ in $\tilde \Sigma_{\tilde E} \subset \SI^{n-1}_{\bv_{\tilde E}}$, we can find a sequence $\{g_i \in \bGamma_{\tilde E}\}$ 
that satisfies the conditions of Lemma \ref{lem:recurrent}. 
The two arcs in $\Bd D$ corresponding to $l$ share endpoints. 
Since this is true for all geodesics, we obtain $A_1=B_1$ and $A\cup B$ is dense in $\Bd D$. 


Hence, $\partial D = A \cup B$. 
Thus, $\Bd D - \partial D$ is the closures of the attracting and repelling fixed points of $h(\pi_1(\tilde E))$
since by Theorem 1.1 of \cite{Ben1}, the set of fixed points is dense in $A_1 = B_1$. 
Therefore this set is independent of the choice of $D$. 

(ii) Consider any segment $l$ in $\Bd \tilde{\mathcal O}$ with $l^o$ 
meeting $\clo(U_1)$ for a concave p-end-neighborhood $U_1$ of $\bv_{\tilde E}$. 
This segment is contained in a union of segments from $\bv_{\tilde E}$ since $\clo(\torb)$ is convex. 
These segments are all in the boundary of $\clo(U_1)$ by 
the fact that a segment from $\bv_{\tilde E}$ can pass the interior of $U_1$ or correspond to a geodesic in $\Bd \tilde \Sigma_{\tilde E} \subset \SI^{n-1}_{\bv_{\tilde E}}$. 
We suppose that $l$ is a segment from 
$\bv_{\tilde E}$ containing a segment $l_0$ in $\clo(U_1)\cap \Bd \tilde{\mathcal O}$ from $\bv_{\tilde E}$, 
and we will show that $l$ is in $\clo(U_1) \cap \Bd \torb$. This will be sufficient to prove (ii). 

A point of $\Bd \tilde \Sigma_{\tilde E}$ is a p-end point of a recurrent geodesic 
by Lemma \ref{lem:redergodic}. 
Suppose that the interior of $l$ contains a point $p$ of $\Bd \clo(D) - A - B$ that
is in the direction of a p-end point of a recurrent geodesic $m$ in $\tilde \Sigma_{\tilde E}$. 
Lemma \ref{lem:recurrent} again applies. Thus, $l^o$ does not meet  $\Bd \clo(D) - A - B$. 

Given a segment $l'$ from $\bv_{\tilde E}$ in $\Bd \tilde{\mathcal O}$ not meeting 
 $\Bd\clo(D) - A - B$ in its interior, the maximal segment $l''$ containing $l$ from $\bv_{\tilde E}$ in $\Bd \tilde{\mathcal O}$
 meets $\Bd\clo(D) - A - B$ at the end. 
 Thus, $l \subset \clo(U_1)\cap \Bd \tilde{\mathcal O}$.

Let $U'$ be any p-end-neighborhood associated with $\bv_{\tilde E}$. 
Then since each $g \in \bGamma_{\tilde E}$ has an attracting fixed point and the repelling fixed point 
on $\Bd\clo(D) - A - B$, for any segment $s$ in $\clo(U') \cap \Bd \torb$ from $\bv_{\tilde E}$, $\{g^i(s)\}$ converges 
to an element of $S(\bv _{\tilde E})$. Since the attracting and the repelling fixed points $g \in \bGamma_{\tilde E}$ 
is dense in the directions of $\Bd \tilde \Sigma_{\tilde E}$, 
we have $\bigcup S(\bv_{\tilde E}) \subset \clo(U') \cap \Bd \tilde{\mathcal O}.$

We can form $S'(\bv_{\tilde E})$ as
the set of maximal segments from $\bv_{\tilde E}$ in $\clo(U') \cap \Bd \tilde{\mathcal O}$.
Then no segment $l$ in $S'(\bv_{\tilde E})$ has interior points in $\Bd D - A - B$ as above. 
 Thus, $S(\bv_{\tilde E}) = S'(\bv_{\tilde E})$. 
 
 Also, since every points of $\clo(U') \cap \Bd \torb$ has a segment in the direction of $\Bd \tilde \Sigma_{\tilde E}$, 
 $\bigcup S(\bv_{\tilde E}) = \clo(U') \cap \Bd \torb$. 

(iii) By the proof above, 
we now characterize $S(\bv_{\tilde E})$ as the maximal segments in $\Bd \tilde{\mathcal O}$ from $\bv_{\tilde E}$ 
ending at points of  $\Bd D - A - B$.
 Since $g(D)$ is the generalized lens for the the generalized lens neighborhood $g(U)$ of $g(\bv_{\tilde E})$, 
we obtain $g(S(\bv_{\tilde E}))= S(g(\bv_{\tilde E}))$ for any p-end vertex $\bv_{\tilde E}$.

(iv) Given a concave-end-neighborhood $C_{\tilde E}$ of a p-end vertex $\bv_{\tilde E}$,
we show that \[g(C_{\tilde E})=C_{\tilde E} \hbox{ or } g(C_{\tilde E})\cap C_{\tilde E} = \emp \hbox{  for } g \in \bGamma:\]

Suppose that 
\[g (C_{\tilde E}) \cap C_{\tilde E} \ne \emp, g(C_{\tilde E}) \not \subset C_{\tilde E}, 
\hbox{ and } C_{\tilde E} \not\subset g(C_{\tilde E}).\]



Since $C_{\tilde E}$ is concave, 
each point of $\Bd C_{\tilde E} \cap \torb$ is contained in a totally geodesic hypersurface 
$D$ so that a component $C_{E,1}$ of $C_{\tilde E} - D$ is in $C_{\tilde E}$ 
and $\clo(C_{E,1}) \ni \bv_{C_{\tilde E}}$ for the p-end vertex $\bv_{C_{\tilde E}}$ of $C_{\tilde E}$. 
Similar statements hold for $g(C_{\tilde E})$. 

Since $g(C_{\tilde E}) \cap C_{\tilde E} \ne \emp$, one is not a subset of the other, 
we have $\Bd g(C_{\tilde E}) \cap C_{\tilde E} \ne \emp$ or $g(C_{\tilde E}) \cap \Bd C_{\tilde E} \ne \emp$. Then by above 
a set of form of $C_{E,1}$ for some boundary point of $C_{E,1}$ and $g(C_{E_1})$ meet.
Since $C_{E, 1}$ is a component of a separating hypersurface in $\torb$, 
the interiors of some segments in $S(\bv_{\tilde E})$ meet segments of $S(g(\bv_{\tilde E}))$ or vice versa.
By (ii), this implies that $\bigcup S(\bv_{\tilde E}) \subset \bigcup S(g(\bv_{\tilde E}))$. 
Since $\bv_{\tilde E}$ is a unique point not in the interior of a segment but in the interior of 
$\bigcup S(\bv_{\tilde E})$ and similarly for $g(\bv_{\tilde E})$, we obtain
$\bv_{\tilde E} = g(\bv_{\tilde E})$. Hence, $g \in \bGamma_{\tilde E}$, and thus, $C_{\tilde E} = g(C_{\tilde E})$ 
as $C_{\tilde E}$ is a concave neighborhood. Therefore, this is a contradiction. 
We obtain three possibilities  \[g (C_{\tilde E}) \cap C_{\tilde E} = \emp, g(C_{\tilde E}) \subset C_{\tilde E}
\hbox{ or } C_{\tilde E} \subset g(C_{\tilde E}).\]

In the last two cases, it follows that $g(C_{\tilde E}) = C_{\tilde E}$ since 
$g$ fixes $\bv_{\tilde E}$, i.e., $g \in \bGamma_{\tilde E}$. 
This implies that $C_{\tilde E}$ is a proper p-end-neighborhood. 

(iv) If $S(\bv_{\tilde E})^o\cap S(w) \ne \emp$, then the above argument works with in this situation 
to show that $\bv_{\tilde E} = w$. 


\end{proof} 

\begin{lemma} \label{lem:redergodic} 
The geodesic flow on $\tilde \Sigma_{\tilde E}/\pi_1(\tilde E)$ is recurrent.
\end{lemma}
\begin{proof} 
(The proof here is for the $\SI^n$-version nominally only. )
If $\pi_1(\tilde E)$ is a hyperbolic group, then the conclusion follows from Theorem 1.2 \cite{Ben1}. 
Assume now that $\pi_1(\tilde E)$ is a virtual product of 
the hyperbolic groups or trivial groups and the abelian group in the center that 
acts ergodically on the geodesics in $D^o = \tilde \Sigma_{\tilde E}$, 
and $D$ is a strict join $\clo(D_1) * \cdots \clo(D_k)$ for some $k$, $k \geq 2$
where the virtual center $\cong \bZ^{k-1}$ acts trivially. 

Let $\rho$ denote a geodesic passing the interior of $\tilde \Sigma_{\tilde E}$
where an end point $p_1$ is in the strict join $J_1$ of $\clo(D_{i_1}) * \cdots * \clo(D_{i_{l_0}})$ and 
and the other end point $p_2$ has to be in the strict join $J_2$ of the remaining $D_i$s. 
We can choose $g_i \in \bZ^{l_0-1}$ so that 
\[\{g_i|J_1\} \ra \Idd_{J_1} \hbox{ and } \{g_i|J_2\} \ra \Idd_{J_2}.\]
In this case, $g_i(\rho)$ approximates $\rho$ arbitrarily as $g_i \ra \infty$ in $\bGamma_{\tilde E}$. 
(This exists as $\bZ^{l_0-1}$ acts cocompactly on a properly convex $l_0$-simplex  $\Delta$ in a projective space
with transferred eigenvalues from the strict joins. This follows since in lattice in $\bR^{l_0-1}$ 
any vector is approximated by integer sum of basis vectors by uniformly bounded errors. 
\end{proof}


Now we go to the cases when $\pi_1(\tilde E)$ has more than two nontrivial factors abelian or hyperbolic. 
The following theorem shows that concave ends are totally geodesic and of lens type. 
The author obtained the proof of (i-3) from Benoist. 

 

\begin{theorem}\label{thm:redtot}
Let $\mathcal{O}$ be a strongly tame  $n$-orbifold with radial or totally geodesic ends.
Suppose that the holonomy group $\bGamma$ is not virtually reducible. 
Let $\tilde E$ be a p-R-end of the universal cover $\torb$, $\torb \subset \SI^n$ 
{\rm (}resp. $\subset \bR P^n${\rm ),} with a generalized lens p-end-neighborhood. 
Let $\bv_{\tilde E}$ be the p-end vertex  and $\tilde \Sigma_{\tilde E}$ the p-end domain of $\tilde E$. 
Suppose that the p-end fundamental group $\bGamma_{\tilde E}$ is admissible.
Then the following statements hold:
\begin{itemize} 
\item[(i)] For $\SI^{n-1}_{\bv_{\tilde E}}$, we obtain 
\begin{itemize}
\item[(i-1)]  Under a finite-index subgroup of $\hat h(\pi_1(\tilde E))$, 
 $\bR^{n}$ splits into $V_1 \oplus \cdots \oplus V_{l_0}$ and $\tilde \Sigma_{\tilde E}$ is the quotient of the sum 
$C_1+ \cdots + C_{l_0}$ for properly convex or trivial one-dimensional cones $C_i \subset V_i$ for $i=1, \dots, l_0$
\item[(i-2)] The Zariski closure of a finite index subgroup of $\hat h(\pi_1(\tilde E))$ is isomorphic 
to the product $G = G_1 \times \cdots \times G_{l_0} \times \bR^{l_0-1}$ 
where $G_i$ is a semisimple subgroup of $\Aut({\mathcal S}(V_i))$.  
\item[(i-3)] Let $D_i$ denote the image of $C_i$ in $\SI^{n-1}_{\bv_{\tilde E}}$.
Each hyperbolic group factor of $\pi_1(\tilde E)$ divides
exactly one $D_i$ and acts on trivially on $D_j$ for $j \ne i$.
\item[(i-4)] A finite index subgroup of 
$\pi_1(\tilde E)$ has a rank $l_0-1$ free abelian group center corresponding to $\bZ^{l_0-1}$ in $\bR^{l_0-1}$.
\end{itemize}
\item[(ii)] $g$ in the center is diagonalizable with positive eigenvalues. 
For a nonidentity element $g$  in the center, the eigenvalue $\lambda_{\bv_{\tilde E}}$ 
of $g$ at ${\bv_{\tilde E}}$ is strictly between its largest norm and smallest norm eigenvalues. 
\item[(iii)] The p-R-end is  totally geodesic. 
$D_i \subset \SI^{n-1}_{\bv_{\tilde E}}$ 
is projectively diffeomorphic by the projection $\Pi_{\bv_{\tilde E}}$ 
to totally geodesic convex domain $D'_i$ in $\SI^n$ (resp. in $\bR P^n$)
of dimension $\dim V_i -1$ disjoint from $\bv_{\tilde E}$, and the actions of $\Gamma_i$ 
are conjugate by $\Pi_{\bv_{\tilde E}}$.
\item[(iv)] The p-R-end is strictly lens-shaped, and
each $C_i$ corresponds to a cone $C^*_i$ in  $\SI^n$ (resp. in $\bR P^n$)
over a totally geodesic $(n-1)$-dimensional domain $D'_i$ with $\bv_{\tilde E}$.  
The p-R-end has a p-end-neighborhood, the interior of 
\[\bv_{\tilde E} \ast D \hbox{ for } D:=  \clo(D'_1) \ast \cdots \ast \clo(D'_{l_0})\]
where the interior $D \cap \torb$ 
forms the boundary in $\torb$. 
\item[(v)] The set $S(\bv_{\tilde E})$ of maximal segments in $\Bd \torb$ from $\bv_{\tilde E}$ in the closure of a p-end-neighborhood of 
$\bv_{\tilde E}$ is independent of the p-end-neighborhood. 
$S(\bv_{\tilde E})$ is equal to the set of maximal segments with vertex $\bv_{\tilde E}$ in 
the union $\bigcup_{i=1}^j \bv_{\tilde E} *\clo(D'_1)*\cdots *\clo(D'_{i-1})* \clo(D'_{i+1}) * \cdots * \clo(D'_{l_0})$.
\item[(vi)] 
A concave p-end-neighborhood of $\tilde E$ is a proper p-end-neighborhood. 
Finally, the statements (iii) and (v) of Theorem \ref{thm:lensclass} also hold. 
\end{itemize}
\end{theorem} 
\begin{proof} 
Again the $\SI^n$-version is enough.
(i) Since each hyperbolic factor of $\bGamma_{\tilde E}$ contains no nontrivial normal subgroup, 
it goes to one of $\Gamma_i$ isomorphically to a finite index subgroup. 
Hence, each $\Gamma_i$ is hyperbolic or trivial. 
Now the proof follows from Proposition \ref{prop:Ben2}. 

(ii) If $\lambda_{\bv_{\tilde E}}(g)$ is the largest norm of eigenvalue with multiplicity one, 
then $\{g^n(x)\}$ for a point $x$ of a generalized lens converges to $\bv_{\tilde E}$ as $n \ra \infty$. 
Since the closure of a generalized lens is disjoint from the point, this is a contradiction. 
Therefore, the largest norm $\lambda_1(g)$ of the eigenvalues of $g$ is greater than or equal to 
$\lambda_{\bv_{\tilde E}}(g)$. 

Let $U$ be a concave p-end-neighborhood of $\tilde E$ in $\tilde {\mathcal{O}}$.
Let $S_1,..., S_{l_0}$ be the projective subspaces in general position meeting only at the p-end vertex $\bv_{\tilde E}$
where factor groups $\bGamma_1, ...,\bGamma_{l_0}$ act irreducibly on. 
Let $C_i$ denote the union of great segments from $\bv_{\tilde E}$ corresponding to the invariant cones in $S_i$ 
where $\bGamma_i$ acts irreducibly for each $i$. 
The abelian center isomorphic to $\bZ^{l_0-1}$ acts as the identity on $C_i$ in the projective space $\SI^n_{\bv_{\tilde E}}$. 
Let $g\in \bZ^{l_0-1}$. $g| C_i$ can have more than two eigenvalues or just single eigenvalue. 
In the second case $g|C_i$ could be represented by a matrix with eigenvalues all $1$  fixing $\bv_{\tilde E}$. 
Since a generalized lens $L$ meets it, $g|C_i$ has to be identity by the proper convexity of $\torb$: 
Otherwise, $g^n|C$ will send $x\in L \cap C_i$ to $\bv_{\tilde E}$ and to $\bv_{\tilde E -}$ as $i \ra \pm \infty$. 
This contradicts the proper convexity of $\torb$. 
 
We have one of the two possibilities for $C_i$: 
\begin{itemize} 
\item[(a)] $g|C_i$ fixes each point of a hyperspace $P_i \subset S_i$ not passing through $\bv_{\tilde E}$ 
and $g$ has a representation as a scalar multiplication in the affine subspace $S_i - P_i$ of $S_i$. 
Since $g$ commutes with every element of $\bGamma_i$ acting on $C_i$, 
$\bGamma_i$ acts on $P_i$ as well.  
\item[(b)] $g|C_i$ is an identity. 
\end{itemize}
We denote $I_1:=\{ i| \exists g \in \bZ^{l_0-1}, g|C_i \ne \Idd\} $ and 
 $I_2:= \{i| \forall g \in \bZ^{l_0-1}, g|C_i = \Idd\}.$

Suppose that $I_2 \ne \emp$. 
For each $C_i$, we can find $g_i \in \bZ^{l_0-1}$ with the largest norm eigenvalue associated with it. 
By multiplying with some other element of the virtual center, we can show that 
if $i \in I_1$, then $C_i \cap P_i$ has a sequence $\{g_{i, j}\}$ with $i$ fixed so that the premises of Lemma \ref{lem:decjoin} and 
if $i \in I_2$, then $C_i$ has such a sequence $\{g_{i, j}\}$.
By Lemma \ref{lem:decjoin}, this implies that $\clo(\torb)$ is a strict join. 
This contradicts the assumption that $\bGamma$ is not virtually reducible by Lemma \ref{lem:joinred}. 
Thus, $I_2 = \emp$. 

(iii)  By (ii), for all $C_i$, every $g\in\bZ^{l_0-1}-\{\Idd\}$ acts as nonidentity. 
Then the strict join of all $P_i$ gives us a hyperspace $P$ disjoint from $\bv_{\tilde E}$. 
We will show that it forms a p-T-end for $\tilde E$:

From above, we obtain that every nontrivial $g \in \bZ^{l_0-1}$ 
is clearly diagonalizable with positive eigenvalues associated with $P_i$ and $\bv_{\tilde E}$
and the eigenvalue at $\bv_{\tilde E}$ is smaller than the maximal ones at $P_i$. 


Let us choose $C_i$. 
We can find at least one $g'\in \bZ^{l_0-1}$ 
so that $g'$ has the largest norm eigenvalue $\lambda_1(g'_i)$ 
with respect to $C_i$ as an automorphism of $\SI^{n-1}_{\bv_{\tilde E}}$. 
We have $\lambda_1(g') > \lambda_{\bv_{\tilde E}}(g')$. 

Each $C'_i \cap P_i$ has an attracting fixed point of some $g_i \in \bGamma_i$ restricted to $P_i$ if $\bGamma_i$ is hyperbolic. 
We can choose $g_i$ so that the largest norm eigenvalue $\lambda_i$ of $g_i| P_i$ is sufficiently large. 
This follows since $\bGamma_i$ is linear on $S_i-P_i$ where we know that this is true 
for strictly convex cones by the theories of Koszul and so on. 
If $\bGamma_i$ is a trivial group, then we choose $g_i$ to be the identity. 
Then by taking $k$ sufficiently large, $g'^k g_i$ has an attracting fixed point in $C'_i \cap P_i$. 
This must be in $\clo(\tilde{\mathcal{O}})$. 
Since the set of attracting fixed points in $C'_i$ is dense in $\Bd C'_i \cap P_i$ by Benoist \cite{Ben1},
we obtain $C'_i \cap P_i \subset \clo(\tilde{\mathcal{O}})$.

Let $D'_i$ denote $C_i \cap P_i$. 
Then the strict join $D'$ of $\clo(D'_1), .., \clo(D'_{l_0})$ equals $P \cap \clo(\tilde{\mathcal{O}})$, which is $h(\pi_1(\tilde E))$-invariant.
And $D^{\prime o}$ is a properly convex subset. If any point of $D^{\prime o}$ is in $\Bd \tilde{\mathcal{O}}$, 
then $D'$ is a subset of  $\Bd \tilde{\mathcal{O}}$ by Lemma \ref{lem:simplexbd}. 
Then $\torb$ is a contained in $\bv_{\tilde E} \ast D'$ where $D'$ is a strict join of $D'_1, \dots, D'_{l_0}$. 
Some or none could be $0$-dimensional. Then $\bGamma$ acts on a strict join. 
By Lemma \ref{lem:joinred}, $\bGamma$ is virtually reducible, a contradiction. 
Therefore, $D^{\prime o} \subset \tilde{\mathcal{O}}$, and $\tilde E$ is a totally geodesic end.

(iv) Let $P$ be the minimal totally geodesic subspace containing all of $P_1, \dots, P_{l_0}$. 
The hyperspace $P$ separates $\tilde{\mathcal{O}}$ into two parts, ones in the p-end-neighborhood $U$ and the subspace outside it. 
Clearly $U$ covers $\Sigma_{\tilde E}$ times an interval 
by the action of $h(\pi_1(\tilde E))$ and the boundary of $U$ goes to a compact orbifold 
projectively diffeomorphic to $\Sigma_{\tilde E}$. 

Since each nontrivial $g \in \bZ^{l_0-1}$ has $\bv_{\tilde E}$ with the eigenvalue strictly between the largest norm
and the smallest norm ones, each point of $\Bd L - \partial L$  is a limit point of 
some sequence $g_i(x)$ for $x \in D^o$ for a generalized lens. Therefore $\Bd L - \partial L$ is exactly the boundary of
the top hypersurface and the bottom one of an open cell neighborhood of $L$ by Lemma \ref{lem:decjoin}.
We can use Proposition \ref{prop:convhull2} replacing its third condition by Lemma \ref{lem:centerprojection}.
Hence, $\tilde E$ is of strict lens-type. The rest follows by the proof of (iii). 

(v) Let $U$ be the end-neighborhood of $\bv_{\tilde E}$ obtained in (iv). 
For each $i$, we can find a sequence $g_j$ in the virtual center so that 
\[g_j|  \clo(D'_1)*\cdots *\clo(D'_{i-1})*\clo(D'_{i+1})* \cdots * \clo(D'_{l_0})\] converges to the identity. 
Then \[\bv_{\tilde E} * \clo(D'_1)*\cdots *\clo(D'_{i-1})*\clo(D'_{i+1})* \cdots * \clo(D'_{l_0})= \Bd \torb \cap \clo(U)\] 
by the eigenvalue conditions of the virtual center obtained in (iii) and Lemma \ref{lem:centerprojection}.
Hence, (v) follows easily now.

(vi) follows by argument similar to the proof of Theorem \ref{thm:lensclass}. 

\end{proof}


\subsubsection{Technical lemmas}

\begin{lemma} \label{lem:joinred} 
If a group $G$ of projective automorphisms 
acts on a strict join $A= A_1 \ast A_2$ for two compact convex sets $A_1$ and $A_2$, then $G$ is virtually reducible. 
\end{lemma} 
\begin{proof} 
We prove for $\SI^n$.
Let $x_1, \dots, x_{n+1}$ denote the homogeneous coordinates. 
There is a maximal number collection of compact convex sets $A'_1, \dots, A'_m$ so that 
$A = A'_1 \ast \cdots \ast A'_m$ where $A'_i \subset S_i$ for a subspace $S_i$ corresponding to 
a subspace $V_i \subset \bR^{n+1}$ that form independent set of subspaces. 
Then $g \in G$ permutes the collection $\{A'_1, \dots, A'_m\}$. 

Suppose not. We give coordinates so that $A'_i$ satisfies $x_j = 0$ for $j \in I_i$ for some indices
and $x_i \geq 0$ for elements of $A$. 
Then we form a new collection of nonempty sets $J':= \{ A'_i \cap g(A'_j)| 0 \leq i, j \leq n, g \in G\}$
with more elements. 
 Since $A = g(A) = g(A'_1) \ast \cdots \ast g(A'_n)$, using coordinates we can show 
that each $A'_i$ is a strict join of nonempty sets in $J'_i := \{ A'_i \cap g(A'_j)| 0 \leq j \leq n, g \in G\}.$
 $A$ is a strict join of the collection of the sets in $J'$, a contraction to the maximal property. 
 
Hence, by taking a finite index subgroup $G'$ of $G$ acting trivially on the collection, 
$G'$ is reducible. 
\end{proof}

\begin{lemma} \label{lem:decjoin} 
 Suppose that a set $G$ of projective automorphisms  in $\SI^n$ {\rm (}resp. in $\bR P^n${\rm )}
 acts on subspaces $S_1, \dots, S_{l_0}$ and a properly convex domain $\Omega \subset \SI^n$ {\rm (}resp.  $\subset \bR P^n${\rm )}, corresponding 
 to subspaces $V_1, \dots, V_{l_0}$ so that $V_i \cap V_j =\{0\}$ for $i \ne j$ 
 and $V_1 \oplus \cdots \oplus V_{l_0} = \bR^{n+1}$. Let $\Omega_i : = \clo(\Omega) \cap S_i$. 
 We assume that  
 \begin{itemize}
\item for each $S_i$, $G_i := \{g|S_i| g \in G\}$ form a bounded set of automorphisms and 
\item for each $S_i$, there exists a sequence $\{g_{i, j} \in G\}$ with largest norm eigenvalue $\lambda_{i, j}$ restricted at $S_i$
 has the property $\{\lambda_{i, j}\} \ra \infty$.  
 \end{itemize}
 Then we obtain $\clo(\Omega) = \Omega_1 * \cdots * \Omega_{l_0}$ for $\Omega_j \ne \emp, j=1, \dots, l_0$. 
 \end{lemma} 
 \begin{proof}
  We will prove for $\SI^n$ but the proof for $\bR P^n$ is identical.  
 First, $\Omega_i \subset \clo(\Omega)$ by definition. 
 Since the element of a strict join has a vector that is a linear combination of elements of 
 the vectors in direction of $\Omega_1, \dots, \Omega_{l_0}$, 
 Hence, we obtain \[\Omega_1 \ast \cdots \ast \Omega_{l_0} \subset \clo(\Omega)\] since
 $\clo(\Omega)$ is convex. 
 
 Let $z = [ \vec{v}_z]$ for a vector $\vec{v}_z$ in $\bR^{n+1}$.
 We write $\vec{v}_z= \vec{v}_1 + \cdots + \vec{v}_{l_0}$, $\vec{v}_j \in V_j$ for each $j$, $j=1, \dots, l_0$, which is a unique sum. 
Then $z$ determines $z_i = [v_i]$ uniquely. 

Let $z $ be any point.  
 We choose a subsequence of $\{g_{i, j}\}$ so that $\{g_{i, j}|S_i\}$ converges to a projective automorphism 
 $g_{i, \infty}: S_i \ra S_i$ and $\lambda_{i, j} \ra \infty$ as $j \ra \infty$. 
  Then $g_{i, \infty}$ also acts on $\Omega_i$. 
And $g_{i, j}(z_i) \ra g_{i, \infty}(z_i) = z_{i, \infty}$  for a point $z_{i, \infty} \in S_i$.
We also have
 \begin{equation} \label{eqn:lim}
 z_i = g_{i, \infty}^{-1}(g_{i, \infty}(z_i)) = g_{i, \infty}^{-1}(\lim_j g_{i, j}(z_i)) = g_{i, \infty}^{-1}(z_{i, \infty}).
 \end{equation}

Now suppose $z \in \clo(\Omega)$. 
We have   $g_{i, j}(z) \ra z_{i, \infty}$ by the eigenvalue condition. Thus, 
we obtain $z_{i, \infty} \in \Omega_i$ as $z_{i, \infty}$ is the limit of a sequence of orbit points of $z$. 
 Hence we also obtain $z_i \in \Omega_i$ by equation \ref{eqn:lim} 
 and $\Omega_i \ne \emp$.
 This shows that $\clo(\Omega)$ is a strict join. 

 
 \end{proof}

\begin{lemma} \label{lem:centerprojection}
Suppose that $\orb$ is a strongly tame properly convex real projective orbifold with radial or totally geodesic ends 
and the holonomy group is not virtually reducible. Assume $\torb \subset \SI^n$ {\rm (} resp. $\torb \subset \bR P^n${\rm )}. 
Suppose that $\tilde E$ is a lens-type p-R-end.
Then 
for every sequence $\{g_j\}$ of distinct elements of the virtual center $\bZ^{l_0-1}$, 
we have \[\frac{\lambda_1(g_j)}{\lambda_{\bv_{\tilde E}}(g_j)} \ra \infty.\] 
\end{lemma}
\begin{proof} 
Suppose that for a sequence $g_j$ of $\bZ^l - \{\Idd\}$, 
we have $\{\lambda_1(g_j)/\lambda_{\bv_{\tilde E}}(g_j)\}$ is bounded. 
Since $\lambda_1(g_j) >  \lambda_{\bv_{\tilde E}}(g_j)$, 
we assume without loss of generality that $\lambda_1(g_j)$ occurs for a fixed collection $C'_i$, $i \in I$, by taking a subsequence of 
$\{g_j\}$ if necessary.  
Then $\{g_j\}$ acts as a bounded set of projective automorphisms of $\ast_{i \in I} C'_i$. 
Since $g_j$ acts trivially on each $D'_j$ for each $j$ for all $j \ne i$ by Theorem \ref{thm:redtot}(i). 
Again by Lemma \ref{lem:decjoin}, $\clo(\Omega)$ is a nontrivial strict join
and this leads to contradiction by Lemma \ref{lem:joinred}. 
 \end{proof}
 


\section{The characterization of lens-shaped representations} \label{sec:chlens}

The main purpose of this section is to characterize the lens-shaped representations
in terms of eigenvalues. This is a major result of this paper and is needed 
for understanding the duality of the ends.

First, we prove the inequality result for irreducible cases of ends. 
Next, we give the definition of uniform middle-eigenvalue conditions. 
We show that the uniform middle-eigenvalue conditions imply the existence of limits. 
Finally, we show the equivalence of the lens condition 
and the uniform middle-eigenvalue condition in Theorem \ref{thm:equ}
for both radial type ends and totally geodesic ends under very general conditions. 

We will now be working on $\SI^n$. However, the arguments easily apply to $\bR P^n$ versions,
which we omit. 

\subsection{The uniform middle-eigenvalue conditions}

Let $\orb$ be a properly convex real projective orbifold with radial ends
and $\torb$ be the universal cover in $\SI^n$. 
Let $\tilde E$ be a p-R-end of $\torb$ and $\bv_{\tilde E}$ be the p-end vertex. 
Let $h: \pi_1(\tilde E) \ra \SLnp_{\bv_{\tilde E}}$ be a homomorphism and suppose that $\pi_1(\tilde E)$ is hyperbolic. 
Assume that for each nonidentity element of $\pi_1(\tilde E)$, 
the eigenvalue of $g$ at the vertex ${\bv_{\tilde E}}$ of $\tilde E$ has a norm strictly between the maximal
and the minimal norms of eigenvalues of $g$ ---(*). We say that $h$ satisfies 
the {\em middle-eigenvalue condition}.
We denote by the norms of eigenvalues of $g$ by
\[\lambda_1(g), \ldots , \lambda_n(g), \lambda_{\bv_{\tilde E}}(g), \hbox{where } \lambda_1(g) \cdots \lambda_n(g) \lambda_{\bv_{\tilde E}}(g)= \pm 1. \]

Recall ${\mathcal L}_1$ from the beginning of Section \ref{sec:endth}. 
We denote by $\hat h: \pi_1(\tilde E) \ra \SLn$ the homomorphism 
${\mathcal L}_1 \circ h$. Since $\hat h$ is a holonomy of a closed convex real projective $(n-1)$-orbifold,
and $\Sigma_{\tilde E}$ is assumed to be properly convex, 
$\hat h(\pi_1(\tilde E))$ divides a properly convex domain $\tilde \Sigma_{\tilde E}$ in $\SI^{n-1}_{\bv_{\tilde E}}$.

We denote by $\tilde \lambda_1(g), ..., \tilde \lambda_n(g)$ the norms of eigenvalues of 
$\hat h(g)$ so that 
\[\tilde \lambda_1(g) \geq  \ldots \geq \tilde \lambda_n(g), \tilde \lambda_1(g)   \ldots \tilde \lambda_n(g) = \pm 1\] hold.
These are called the {\em relative norms of eigenvalues} of $g$.
We have $\lambda_i(g) = \tilde \lambda_i(g)/ \lambda_{\bv_{\tilde E}}(g)^{1/n}$ for $i=1, .., n$.  
Since ${\bv_{\tilde E}}$ is fixed, the norm of the corresponding eigenvalue $\lambda_{\bv_{\tilde E}}(g)$ is one of these. 

Note here that eigenvalues corresponding to 
$\lambda_1(g), \tilde \lambda_1(g), \lambda_n(g), \tilde \lambda_n(g), \lambda_{\bv_{\tilde E}}(g)$ are all 
positive, mainly by the work of Benoist (see \cite{Benasym}), which really goes back to Kuiper, Koszul, and so on.  

We define the $\leng(g)$ to be $\log(\tilde\lambda_1(g)/\tilde \lambda_n(g)) = \log(\lambda_1(g)/\lambda_n(g))$,
which is the infimum of the Hilbert metric lengths of the associated closed curves in $\tilde \Sigma_{\tilde E}/\hat h(\pi_1(\tilde E))$.

We recall the results in \cite{Benasym} and \cite{Ben5}.
\begin{definition} \label{defn:pro}
Each element $g \in \SLnp$ 
\begin{itemize}
\item that has the largest and smallest norms of the eigenvalues 
which are distinct and
\item the largest or the smallest norm correspond to the eigenvectors with positive eigenvalues respectively
\end{itemize} 
is said to be {\em bi-semiproximal}.
Each element $g \in \SLnp$ 
\begin{itemize}
\item that has the largest and smallest norms of the eigenvalues 
which are distinct and
\item the largest or the smallest norm correspond to the unique eigenvector respectively
\end{itemize} 
is said to be {\em biproximal}. 
\end{definition}  
All infinite order elements of $\hat h(\pi_1(\tilde E))$ are bi-semiproximal
and a finite index subgroup has only bi-semiproximal elements and the identity.
Note also an element is {\em semiproximal} if and only if it is bi-semiproximal 
when 
$\Gamma$ acts on a properly convex domain divisibly. 


When $\pi_1(\tilde E)$ is hyperbolic, 
all infinite order elements of $\hat h(\pi_1(\tilde E))$ are biproximal
and a finite index subgroup has only biproximal elements and the identity.
Note also an element is {\em proximal} if and only if it is biproximal 
when $\bGamma_{\tilde E}$ is a hyperbolic group. 

Assume that $\bGamma_{\tilde E}$ is hyperbolic. 
Suppose that $g \in \bGamma_{\tilde E}$ is proximal. 
We define 
\begin{equation}\label{eqn:alphabetag}
\alpha_g := \frac{\log \tilde \lambda_1(g)- \log \tilde \lambda_n(g)}{\log \tilde \lambda_1(g) - \log \tilde \lambda_{n-1}(g)}, 
\beta_g :=   \frac{\log \tilde \lambda_1(g)- \log \tilde \lambda_n(g)}{\log \tilde \lambda_1(g) - \log \tilde \lambda_{2}(g)},
\end{equation} 
and denote by $\bGamma_{\tilde E}^p$ the set of proximal elements. We define
\[\beta_{\bGamma_{\tilde E}} := \sup_{g \in \bGamma_{\tilde E}^p} \beta_g, 
\alpha_{\bGamma_{\tilde E}} := \inf_{g\in \bGamma_{\tilde E}^p} \alpha_g. \]
Proposition 20 of \cite{Guichard} shows that  
we have 
\begin{equation}\label{eqn:betabound}
1 < \alpha_{\tilde \Sigma_{\tilde E}} \leq \alpha_\Gamma \leq 2 \leq \beta_\Gamma \leq \beta_{\tilde \Sigma_{\tilde E}} < \infty 
\end{equation}
for constants $\alpha_{\tilde \Sigma_{\tilde E}}$ and $\beta_{\tilde \Sigma_{\tilde E}}$ depending only on $\tilde \Sigma_{\tilde E}$
since $\tilde \Sigma_{\tilde E}$ is properly and strictly convex.

Here, it follows that $\alpha_{\bGamma_{\tilde E}}, \beta_{\bGamma_{\tilde E}}$
depends on $\hat h$ and form a function of convex divisible part of 
$\Hom(\pi_1(\tilde E), \SLnp)/\SLnp$ with algebraic convergence topology.




\begin{theorem}\label{thm:eignlem} 
Let $\orb$ be a strongly tame properly convex real projective manifold with radial ends or totally geodesic ends.
Let $\tilde E$ be a properly convex p-R-end of the universal cover $\torb$, $\torb \subset \SI^n$, $n \geq 2$.
Let $\bGamma_{\tilde E}$ be a hyperbolic group. 
Then 
\[ \frac{1}{n}\left(1+ \frac{n-2}{\beta_{\bGamma_{\tilde E}}}\right) \leng(g)
 \leq \log \tilde \lambda_1(g)   \leq  \frac{1}{n}\left(1+ \frac{n-2}{\alpha_{\bGamma_{\tilde E}}}\right) \leng(g)\]
for every proximal element $g \in \hat h(\pi_1(\tilde E))$.
\end{theorem}
\begin{proof} 
Since there is a biproximal subgroup of finite index, we concentrate on biproximal elements only.
We obtain from above that 
\[ \frac{\log \frac{\tilde \lambda_1(g)}{\tilde \lambda_n(g)}}{\log \frac{\tilde \lambda_1(g)}{\tilde \lambda_2(g)}} \leq \beta_{\tilde \Sigma_{\tilde E}}.\] 
We deduce that 
\begin{equation}\label{eqn:eigratio} 
\frac{\tilde \lambda_1(g)}{\tilde \lambda_2(g)} \geq \left( \frac{\lambda_1(g)}{\lambda_n(g)} \right)^{1/\beta_{\tilde \Sigma_{\tilde E}}}
=  \left( \frac{\tilde \lambda_1(g)}{\tilde \lambda_n(g)} \right)^{1/\beta_{\Omega}} = \exp(\frac{\leng(g)}{\beta_{\tilde \Sigma_{\tilde E}}}).
\end{equation}
Since we have $\tilde \lambda_i \leq \tilde \lambda_2 $ for $i\geq 2$, we obtain
\begin{equation}\label{eqn:betab} 
\frac{\tilde \lambda_1(g)}{\tilde \lambda_i(g)} \geq \left( \frac{\lambda_1}{\lambda_n} \right)^{1/\beta_{\tilde \Sigma_{\tilde E}}}
\end{equation}
and since $\tilde \lambda_1 \cdots \tilde \lambda_n = 1$, 
we have 
\[ \tilde \lambda_1(g)^n = \frac{\tilde \lambda_1(g)}{\tilde \lambda_2(g)} \cdots  \frac{\tilde \lambda_1(g)}{\tilde \lambda_{n-1}(g)}
 \frac{\tilde \lambda_1(g)}{\tilde \lambda_n(g)} \geq \left(  \frac{\tilde \lambda_1(g)}{\tilde \lambda_n(g)} \right)^{\frac{n-2}{\beta} + 1}.\] 
 We obtain 
 \begin{equation}\label{eqn:betabd}
  \log \tilde \lambda_1(g) \geq \frac{1}{n}\left(1+ \frac{n-2}{\beta_{\bGamma_{\tilde E}}}\right) \leng(g).
  \end{equation}
By similar reasoning, we also obtain 
\begin{equation}\label{eqn:alphabd}
\log \tilde \lambda_1(g) \leq \frac{1}{n}\left(1+ \frac{n-2}{\alpha_{\bGamma_{\tilde E}}}\right) \leng(g).
\end{equation} 

\end{proof}

\begin{remark} 
Under the assumption of Theorem \ref{thm:eignlem}, if we do not assume that $\pi_1(\tilde E)$ is hyperbolic, then 
we obtain 
\[ \frac{1}{n} \leng(g) \leq \log \tilde \lambda_1(g)   \leq C \frac{n-1}{n} \leng(g)\]
for every semiproximal element $g \in \hat h(\pi_1(\tilde E))$.
\end{remark} 
\begin{proof} 
Let $\tilde \lambda_i(g)$ denote the norms of $\hat h(g)$ for $i=1, 2, \dots, n$. 
\[\log \tilde \lambda_1(g) \geq  \ldots \geq \log \tilde \lambda_n(g), \log \tilde \lambda_1(g)  + \cdots + \log \tilde \lambda_n(g) = 0\] 
hold.
We deduce 
\begin{alignat}{3} 
\log \tilde \lambda_n(g) &=& -\log \lambda_1 - \cdots - \log \tilde \lambda_{n-1}(g) \nonumber \\
& \geq & -(n-1) \log \tilde \lambda_1 \nonumber\\ 
\log \tilde \lambda_1(g) & \geq & -\frac{1}{n-1} \log \tilde \lambda_n(g) \nonumber\\ 
\left(1+ \frac{1}{n-1}\right) \log \tilde \lambda_1(g) & \geq & \frac{1}{n-1} \log \frac{\tilde \lambda_1(g)}{\tilde \lambda_n(g)}\nonumber \\ 
\log \tilde \lambda_1(g) & \geq & \frac{1}{n} \leng(g).
\end{alignat}
We also deduce 
\begin{alignat}{3} 
-\log \tilde \lambda_1(g) & = & \log \tilde \lambda_2(g) + \cdots + \log \tilde \lambda_{n}(g) \nonumber \\
 & \geq & (n-1) \log \tilde \lambda_{n}(g) \nonumber \\ 
-(n-1) \log \tilde \lambda_{n}(g) & \geq & \log \tilde \lambda_1(g) \nonumber \\ 
(n-1) \log \frac{\tilde \lambda_1(g)}{\tilde \lambda_{n}(g)} & \geq & n \log \tilde \lambda_1(g) \nonumber \\ 
\frac{n-1}{n} \leng(g) & \geq & \log \tilde \lambda_1(g).
\end{alignat} 
\end{proof}

\begin{remark}
We cannot show that the middle-eigenvalue condition implies 
the uniform middle-eigenvalue condition. This could be false.
For example, we  could obtain a sequence of elements $g_i \in \Gamma$ so that 
$\lambda_1(g_i)/ \lambda_{\bv_{\tilde E}}(g_i) \ra 1$ while $\Gamma$ satisfies the middle-eigenvalue 
condition. Certainly, we could have an element $g$ where 
$\lambda_1(g) = \lambda_{\bv_{\tilde E}}(g)$. 
However, even if there is no such element, we might still have 
a counter-example. 
For example, suppose that we might have 
\[\frac{\log (\lambda_1(g_i)/\lambda_{\bv_{\tilde E}}(g_i))}{\leng(g)} \ra 0.\] 
(Such assignments are not really understood
but see Benoist \cite{Benasym}. Also, an analogous phenomenon seems to happen 
with the Margulis space-time and diffused Margulis invariants as investigated by Charette, Drumm, Goldman, Labourie, and Margulis 
recently.)
\end{remark}

\subsubsection{The uniform middle-eigenvalue conditions and the orbits} 

Let $\tilde E$ be a p-R-end of the universal cover $\torb$ of 
 a properly convex real projective orbifold $\orb$ with radial ends. 
Assume that $\bGamma_{\tilde E}$ satisfies the uniform middle-eigenvalue condition. 
There exists a $\bGamma_{\tilde E}$-invariant convex set $K$ distanced from $\{\bv_{\tilde E}, \bv_{\tilde E-}\}$
by Theorem \ref{thm:distanced}. 
For the corresponding tube ${\mathcal T}_{\bv_{\tilde E}}$, $K \cap \Bd {\mathcal T}_{\bv_{\tilde E}}$ is a compact 
subset distanced from $\{\bv_{\tilde E}, \bv_{\tilde E-}\}$ .
We call $K$ the {\em $\bGamma_{\tilde E}$-invariant boundary distanced set}. 
Let $C_1$ denote the convex hull of $K$ in the tube ${\mathcal T}_{\bf_{\tilde E}}$ obtained by 
Theorem \ref{thm:distanced}.
Then $C_1$ is a $\bGamma_{\tilde E}$-invariant 
subset of ${\mathcal T}_{\bv_{\tilde E}}$. 

Also, $K \cap \Bd {\mathcal T}_{\bv_{\tilde E}}$ contains all attracting and repelling 
fixed points of $\gamma \in \bGamma_{\tilde E}$ by invariance and the middle-eigenvalue condition. 

\begin{theorem}\label{thm:attracting} 
Let $\orb$ be a strongly tame properly convex real projective manifold with radial ends or totally geodesic ends.
Let $\tilde E$ be a properly convex p-R-end of the universal cover $\torb$, $\torb \subset \SI^n$. 
Assume that $\bGamma_{\tilde E}$ is irreducible and hyperbolic and satisfies the uniform middle eigenvalue conditions. 
\begin{itemize}
\item Suppose that $\gamma_i$ is a sequence of elements of $\bGamma_{\tilde E}$ acting on ${\mathcal T}_{\bv_{\tilde E}}$. 
\item The sequence of attracting fixed points $a_i$ and the sequence of  repelling fixed points $b_i$ are so that 
$a_i \ra a_\infty$ and $b_i \ra b_\infty$ where $a_\infty, b_\infty$ are not in $\{ \bv_{\tilde E}, \bv_{\tilde E-}\}$
for $a_\infty \ne b_\infty$. 
\item Suppose that the sequence $\{\lambda_i\}$ of eigenvalues where 
$\lambda_i$ corresponds to $a_i$ converges to $+\infty$. 
\end{itemize} 
Then for 
\[M := {\mathcal T}_{\bv_{\tilde E}} - \clo(\bigcup_{i=1}^\infty \ovl{b_i\bv_{\tilde E}} \cup \ovl{b_i\bv_{\tilde E-}}), \]
the point $a_\infty$ is the limit of $\{\gamma_i(K)\}$ for any compact subset $K \subset M$. 
\end{theorem} 
\begin{proof} 
There exists a totally geodesic sphere $\SI^{n-1}_i$ at $b_i$ supporting ${\mathcal T}_{\bv_{\tilde E}}$. 
$a_i$ is uniformly bounded away from $\SI^{n-1}_i$ for $i$ sufficiently large.
$\SI^{n-1}_i$ bounds an open hemisphere $H_i$ containing $a_i$ where $a_i$ 
is the attracting fixed point so that for a euclidean metric $d_{E, i}$, 
$\gamma_i| H_i: H_i \ra H_i$ is a contraction by the inverse $k_i$ of the factor 
\[\min \left\{\frac{\tilde \lambda_1(\gamma_i)}{\tilde \lambda_2(\gamma_i)}, 
\frac{\tilde \lambda_1(\gamma_i)}{\lambda_{\bv_{\tilde E}}(\gamma_i)^{\frac{n+1}{n}}}\right\}.\]
Also, $k_i \ra 0$ by the uniform middle eigenvalue condition and 
and by equation \ref{eqn:eigratio}. 
Note that $\{\clo(H_i)\}$ converges geometrically to $\clo(H)$ for an open hemisphere containing $a$ in
the interior. 

Actually, we can choose a Euclidean metric $d_{E, i}$ on $H_i^o$ 
so that $\{d_{E, i}| J \times J \}$ is uniformly convergent for any compact subset $J$ of 
$H_\infty$.
This implies that since $\{a_i \} \ra a$,  if $d_{E_i}(a_i, K) \leq \eps$ for sufficiently small $\eps > 0$, then 
$\bdd(a_i, K) \leq C' \eps$ for a positive constant $C'$. 

Any compact subset $K$ of $M$, we gave $K \subset H_\infty$ and 
  the distance $\bdd(K, \Bd H_i)$ is uniformly bounded by a constant $\delta$. 
$\bdd(K, \Bd H_i) > \delta$ implies that 
$d_{E_i}(a_i, K) \leq C/\delta$ for a positive constant $C> 0$
Acting by $g_i$, we obtain 
$d_{E_i}(g_i(K), a_i) \leq k_i C/\delta$, which implies  
$\bdd(g_i(K_i), a_i) \leq C' k_i C/\delta.$
Since $\{k_i\} \ra 0$ and $\{a_i\} \ra a$ imply that $\{g_i(K)\}$ geometrically converges to $a$. 
 \end{proof}

For the following, we don't assume that $\bGamma_{\tilde E}$ acts irreducibly.
\begin{proposition}\label{prop:orbit}
Let $\orb$ be a strongly tame properly convex real projective manifold with radial ends or totally geodesic ends.
Let $\tilde E$ be a properly convex p-R-end of the universal cover $\torb$, $\torb \subset \SI^n$. 
Assume that $\bGamma_{\tilde E}$ satisfies the uniform middle eigenvalue condition. 
Let $\bv_{\tilde E}$ be the R-end vertex
and  $z \in {\mathcal T}^o_{\bv_{\tilde E}}$. Then
\begin{itemize} 
\item[(i)] Limit points of orbit elements of $z$ are in the $\bGamma_{\tilde E}$-invariant boundary distanced set $K$.
\item[(ii)] Each point of $K$ is a limit of $g_i(x)$ for a point $x \in \torb$ for a sequence $g_i \in \bGamma_{\tilde E}$. 
\item[(iii)] For each segment $s$ in $S(\bv_{\tilde E})$, the great segment meets $K$ 
at the end point of $s$ other than $\bv_{\tilde E}$. 
That is, there is a one-to-one correspondence between $S(\bv_{\tilde E})$ and $K$ and hence also with 
$\Bd \tilde \Sigma_{\tilde E}$ in $\SI^{n-1}_{\bv_{\tilde E}}$. 
\end{itemize} 
\end{proposition} 
\begin{proof} 
(i) Consider first the irreducible hyperbolic $\bGamma_{\tilde E}$.
Given $z \in  {\mathcal T}^o_{\bv_{\tilde E}}$, we
let $\gamma_i$ be any sequence in $\bGamma_{\tilde E}$ so that the corresponding sequence of $\gamma_i(z)$
in $\tilde \Sigma_{\tilde E}$ converges to a point $z'$ in $\Bd \tilde \Sigma_{\tilde E}$. 
Let $z_\infty$ denote the point of $K$ corresponding to $z'$.

Clearly, a fixed point of $g$ in $\Bd {\mathcal T}_{\bv_{\tilde E}}  - \{\bv_{\tilde E}, \bv_{\tilde E -}\}$ is in $K_b$ since $g$ has a unique fixed point 
on each open segment in the boundary. 
We can assume that for the attracting fixed points $a_i$ and $r_i$ of $\gamma_i$, 
we have $\{a_i\} \ra a$ and $\{r_i\} \ra r$ for $a_i, r_i, a, r \in K$
where $a, r \in K$ by the closedness of $K$. 
Assume $a \ne r$ first.
By Theorem \ref{thm:attracting}, we have $\{\gamma_i(z)\} \ra a$ and hence $z_\infty = a$. 

However, it could be that $a = r$. In this case, we choose $\gamma_0 \in \bGamma_{\tilde E}$
so that $\gamma_0(a) \ne r$. Then $\gamma_0\gamma_i$ has the attracting fixed point $a'_i$ 
so that we obtain $\{a'_i\} \ra \gamma_0(a)$ 
and repelling fixed points $r'_i$ so that $\{r'_i \}\ra r$ holds
by Lemma \ref{lem:gatt}.

Then as above $\{\gamma_0 \gamma_i(z) \} \ra \gamma_0(a)$
and we need to multiply by $\gamma_0^{-1}$ now
to show $\{\gamma_i(z) \} \ra a$. 



Suppose that $\bGamma_E$ is reducible. 
Then a totally geodesic hyperspace $H$ is disjoint from $\{\bv_{\tilde E}, \bv_{\tilde E-}\}$ 
and meets $\torb$. Then for any sequence $g_i$ so that 
$g_i(x) \ra x_0$, let $x'$ denote the corresponding point of $\tilde \Sigma_{\tilde E}$.
Then $g_i(x')$ converges to a point $y \in \SI^{n-1}_{\bv_{\tilde E}}$. 
Let $\vec x$ be the vector in the direction of $x'$. 
We write $\vec x = \vec x_E + \vec x_H$ where $\vec x_H$ is in the direction of $H$ and $\vec x_E$ is in the direction of $\bv_{\tilde E}$. 
By the uniform middle eigenvalue condition, we obtain
$g_i(x') \ra x''$ for $x'' \in H$. Hence, $x'' \in H \cap K$. 

(ii) We can take any $a \in \Bd \tilde \Sigma_{\tilde E}$, and 
find a sequence $g_i \in \bGamma_{\tilde E}$ 
so that $g_i(z) \ra a$ for $z \in \tilde \Sigma_{\tilde E}$. 
We can use a point $z'$ in $U$ in direction of $z$ and this implies the result. 

(iii) $K$ meets at a unique point the every segment in $\Bd {\mathcal T}_{\bv_{\tilde E}}$ from $\bv_{\tilde E}$ to $\bv_{\tilde E -}$ by Theorem \ref{thm:distanced}.

\end{proof}

\begin{lemma} \label{lem:gatt}
Let $\{g_i\}$ be a sequence of projective automorphisms 
acting on a strictly convex domain $\Omega$ in $\SI^n$ {\rm (}resp. $\bR P^n${\rm ).} 
Suppose that the sequence of attracting fixed points 
$\{a_i \in \Bd \Omega\} \ra a$ and the sequence of 
repelling fixed points $\{r_i \in \Bd \Omega\} \ra r$. 
Assume that the corresponding sequence of eigenvalues of 
$a_i$ limits to $+\infty$ and that of $r_i$ limits to $0$. 
Let $g$ be any projective automorphism of $\Omega$. 
Then $\{gg_i\}$ has the sequence of attracting fixed points 
$\{a'_i\}$ converging to $g(a)$ and the sequence of repelling 
fixed points converging to $r$. 
\end{lemma}
\begin{proof}
The new attracting fixed points $a'_i \ra g(a_\ast)$ 
since a compact ball $B$ in $\Bd \Omega$ 
so that $g(a_i) \in B^o, r_\infty \not\in B$, 
becomes a small ball near $g(a_i)$ under $gg_i$ 
so that $gg_i(B) \subset B$. 
Thus, the fixed point $a'_i$ of $g g_i$ lies in the disk $gg_i(B)$. 
Also $gg_i| K$ for $K \in \clo(\Omega) -\{r_\ast\}$ converges to $g(a_\ast)$. 
Thus, the repelling fixed point $r'_i$ of $gg_i$ converges to $r$ also. 

We can also argue using $g_i^{-1}g^{-1}$ for repelling points. 
\end{proof}

\subsubsection{Convex compact actions of the p-end fundamental groups.} \label{subsec:redlens} 

In this section, we will prove Proposition \ref{prop:convhull2} obtaining a lens when we have a {\em convex cocompact action} of the end 
fundamental group as defined by the premise of the proposition. 

\begin{lemma}\label{lem:simplexbd}
Let $\orb$ be a strongly tame properly convex real projective manifold with radial ends or totally geodesic ends.
Let $\tilde E$ be a properly convex p-R-end of the universal cover $\torb$,  and $\torb$ is a subset of $\SI^n$. 
Suppose that $\orb$ is properly convex. 
Let $\sigma$ be a convex domain in $\clo(\torb) \cap P$ for a subspace $P$. 
Then either $\sigma \subset  \Bd \torb$ or $\sigma^o$ is in $\torb$. 
\end{lemma} 
\begin{proof} 
Suppose that $\sigma^o$ meets 
$\Bd \tilde {\mathcal{O}}$ and is not contained in it entirely.  
We can find a segment $s \subset \sigma^o$ with a point $z$ so that a component $s_1$ of $s-\{z\}$ 
is in $ \Bd \tilde {\mathcal{O}}$ and the other component $s_2$ is disjoint from it. 
We may perturb $s$ in the subspace containing $s$ and $\bv_{\tilde E}$ so that the new segment $s'$ meets $\Bd \tilde {\mathcal {O}}$ only in its interior. 
This contradicts the fact that $\tilde {\mathcal{O}}$ is convex by Theorem A.2 of \cite{psconv}. 
\end{proof}


\begin{proposition}\label{prop:convhull2} 
Let $\orb$ be a strongly tame  properly convex real projective orbifold with radial or totally geodesic ends 
and with admissible end fundamental groups. 
Assume that the universal cover $\torb$ is a subset of $\SI^n$.
\begin{itemize} 
\item Let $\bGamma_{\tilde E}$ be the holonomy group of a properly convex p-R-end $\tilde E$, 
\item Let ${\mathcal T}_{\bv_{\tilde E}}$ be an open tube corresponding to $R(\bv_{\tilde E})$.
\item Suppose that $\bGamma_{\tilde E}$ satisfies the uniform middle eigenvalue condition, 
and acts on a distanced compact convex set 
$K$ in $\clo({\mathcal T}_{\bv_{\tilde E}})$  
with $K \cap {\mathcal T}_{\bv_{\tilde E}} \subset \torb$. 
\end{itemize} 
Then any p-end-neighborhood containing $K \cap \torb$ contains a lens-cone p-end-neighborhood of 
the p-R-end $\tilde E$. 
\end{proposition}
\begin{proof} 
By assumption, $\torb - K$ has two components. 
Let $K^b$ denote $\Bd {\mathcal T}_{\bv_{\tilde E}} \cap K$. 
Let us choose finitely many points $z_1, \dots, z_m \in \torb - K$ in the two components. 
We can cut off $B$ by hyperspaces the vertices and obtain a properly convex domain. 
Then it contains a convex hull $C_2:= C(\bGamma_{\tilde E}(\{z_1, \dots, z_m\}, K)$. 
Proposition \ref{prop:orbit} shows that the orbits of $z_i$ for each $i$ accumulate to points of $K^b$ only. 
Hence, a totally geodesic hypersphere separates $\bv_{\tilde E}$ with these orbit points
and another one separates $\bv_{\tilde E-}$ and the orbit points. 
Thus, $C_2$ is a compact convex set disjoint from $\bv_{\tilde E}$ and $\bv_{\tilde E-}$ 
and $C_2 \cap \Bd {\mathcal T}_{\tilde E} = K'$. 

Continuing to assume as above. 
\begin{lemma}\label{lem:push} 
We are given a distanced compact convex set $K$ in $\clo({\mathcal T}_{\bv_{\tilde E}})$  where $\bGamma_{\tilde E}$ acts on, 
where $(K\cap{\mathcal T}^o_{\bv_{\tilde E}})/\bGamma_{\tilde E}$ is compact. 
Then we can choose 
$z_1, \dots, z_m$ in $\torb$ so that 
for $C_2:= C(\bGamma_{\tilde E}(\{z_1, \dots, z_m\}, K))$, 
$\Bd C'_2 \cap \torb$ is disjoint from $K$
and $C_2 \subset \torb$. 
\end{lemma}
\begin{proof} 
We can cover a compact fundamental domain of 
$\Bd K \cap {\mathcal T}_{\bv_{\tilde E}}$ by the interior of $n$-balls in $\torb$ that are convex hulls of finite set of points in $U$. 
Then $K$ will satisfy the properties. 
\end{proof}


We continue:
\begin{lemma} \label{lem:infiniteline} 
Let $C$ be another $\bGamma_{\tilde E}$-invariant 
distanced compact convex set with boundary in ${\mathcal{T}}_{\tilde E}$
where $(C \cap {\mathcal{T}^o_{\tilde E}})/\bGamma_{\tilde E}$ is compact. 
There are two components $A$ and $B$ of
$\Bd C \cap {\mathcal T}^o_{\bv_{\tilde E}}$ meeting every great segment in ${\mathcal T}^o_{\bv_{\tilde E}}$. 
Suppose that $A$ and $B$ be  disjoint from $C$.
Then $A \cap B$ contains no line ending in $\Bd \torb$. 
\end{lemma} 
 \begin{proof} 
Suppose that there exists 
a line $l$ in  $A \cup B$ ending at a point of $\Bd {\mathcal T}_{\bv_{\tilde E}}$. Assume $l \subset A$. 
The line $l$ project to a line $l'$ in ${\tilde E}$. 

Let $C_1 = C \cap {\mathcal T}_{\bv_{\tilde E}}$. 
Since $A/\bGamma_{\tilde E}$ and $B/\bGamma_{\tilde E}$ are both compact, 
and there exists a fibration $C_1/\bGamma_{\tilde E} \ra A/\bGamma_{\tilde E}$ 
induced from $C_1 \ra A$ using the foliation by great segments from $\bv_{\tilde E}$. 

Since $C_1/\bGamma_{\tilde E}$ is compact,  
we choose a compact fundamental domain $F$ in $C_1$ and 
choose a sequence $\{x_i \in l\}_{i=1, 2, \dots }$ converging to the endpoint of $l'$ in $\Bd \tilde \Sigma_{\tilde E}$. 
We choose $\gamma_i \in \bGamma_{\bv_{\tilde E}}$ so that $\gamma_i(x_i) \in F$ 
where $\{\gamma_i(l')\}$ converges to a segment $l'_\infty$ with both endpoints in $\Bd \tilde \Sigma_{\tilde E}$. 
Hence, $\{\gamma_i(l)\}$ converges to a segment $l_\infty$ in $A$.
We can assume that for the endpoint $z$ of $l$ in $A$, $\gamma_i(z) $ converges to the endpoint $p_1$. 
Proposition \ref{prop:orbit} implies that the endpoint $p_1$ of $l_\infty$ is in $K^b$ also. 
Let $t$ be the endpoint of $l$ not equal to $z$. Then $t \in K^b$
since $t $ is in the boundary $A$ with limit points in $K^b$ by Proposition \ref{prop:orbit}. 
Thus, $\gamma_i(t)$ converges to a point of $K^b$ and both end points of $l_\infty$ is in $K^b$
and hence $l_\infty \subset C_1$.
$l \subset A$ implies that $l_\infty \subset A$. As $A$ is disjoint from $C_1$, 
this is a contradiction.  
\end{proof} 

Since $A$ and analogously $B$ do not contain any geodesic ending at $\Bd \torb$, 
$\Bd C'_1 - \Bd {\mathcal T}_{\bv_{\tilde E}}$ is a union of compact $n-1$-dimensional simplices
meeting one another in strictly convex dihedral angles. 
By choosing $\{z_1, \dots, z_m\}$ sufficiently close to $\Bd C_1$, we may assume 
that $\Bd C'_1 - \Bd {\mathcal T}_{\bv_{\tilde E}}$ is in $\torb$. 
Now by smoothing 
we obtain two boundary components of a lens. 
(Actually the condition can replace the definition of the lens condition.)

When $\bGamma_{\tilde E}$ is reducible, we do the above arguments to show that each factor $\bGamma_i$ acts on 
a compact convex set $K'_i$ in $B(K_i)$ distanced from $\bv_{\tilde E}$ and $\bv_{\tilde E-}$. 
We obtain a strict join $K'_1 * \cdots * K'_{l_0}$ where $\bGamma_{\tilde E}$ acts on naturally. 

Similar to the reasoning in the proof of Theorem \ref{thm:redtot}, 
$\tilde E$ is totally geodesic: 
Let $z \in {\mathcal T}_{\bv_{\tilde E}}$. 
(Basically, each element $g\in \bZ^{l_0-1}$ acts fixing each point of 
totally geodesic plane meeting $B(K_i)$ and $K'_i$ has to be the intersection
since $\lambda_1(g) > \lambda_{\bv_{\tilde E}}(g)$ for $g$ fixing the subspace corresponding to $K_i$ by the uniform middle eigenvalue condition. 
The strict join of $K'_1, \dots, K'_{l_0}$ is totally geodesic compact convex set where 
$\bGamma_{\tilde E}$ acts.)

Again, we need to do argument similar to above to find a lens and to complete the proof:
that is, we find the convex hull $C'_1$ as above and the top and the bottom 
hypersurface boundary components $A$ and $B$. 
By the uniform middle-eigenvalue condition and Proposition \ref{prop:orbit}, 
the orbits of $z$ limit to points of $K^b$ only.

\end{proof}


\subsection{The uniform middle-eigenvalue conditions and the lens-shaped ends.} 




A {\em radially foliated end-neighborhood system} of $\mathcal{O}$
is a collection of end-neighborhoods of $\mathcal{O}$ that are radially foliated 
where each great segment from the end vertex meets the boundary of the end-neighborhoods uniquely 
and the complement is a compact suborbifold with the boundary the union of 
boundary components of the end-neighborhoods. 

We say that $\mathcal{O}$ satisfies the {\em triangle condition} if 
for $\tilde {\mathcal{O}}$, the interior of every triangle $T$ with $\partial T$ in 
$\Bd \tilde {\mathcal{O}}$ is a subset of a radially foliated p-end-neighborhood $U$ in $\tilde {\mathcal{O}}$
from a fixed radially foliated end-neighborhood system of $\mathcal{O}$. 

In \cite{conv}, we will show that this condition is satisfied if $\pi_1(\mathcal{O})$ is 
relatively hyperbolic with respect to the end fundamental groups. 
 We will prove this in \cite{conv} since it is a global result and 
not a result on ends only. 


\begin{theorem}\label{thm:equ} 
Let $\orb$ a strongly tame  properly convex real projective orbifold with radial or totally geodesic ends 
and with admissible end fundamental groups.
Assume that the holonomy group is strongly irreducible.
Assume the following conditions. 
\begin{itemize}
\item The universal cover $\torb$ is a subset of $\SI^n$ {\rm (}resp. in $\bR P^n${\rm ).}
\item The holonomy group $\bGamma$ is strongly irreducible. 
\item ${\mathcal{O}}$ satisfies the triangle condition or, alternatively, assume that $\tilde E$ is reducible.
\end{itemize} 
Let $\bGamma_{\tilde E}$ be the holonomy group of a properly convex R-end $\tilde E$.
Then {\rm (i)} and {\rm (ii)} are equivalent. 
\begin{itemize} 
\item[(i)] $\bGamma_{\tilde E}$ is of lens-type. 
\item[(ii)] $\bGamma_{\tilde E}$ satisfies the uniform middle-eigenvalue condition.
\end{itemize}
Now not assuming the triangle condition nor reducibility, 
$\bGamma_{\tilde E}$ is of generalized lens-type if and only if  $\bGamma_{\tilde E}$ satisfies the uniform middle-eigenvalue condition.
\end{theorem}
\begin{proof} 
By Lemma \ref{lem:coneseq}, every triangle $T$ with  $\partial T$ in 
$\Bd \tilde {\mathcal{O}}$ that is a subset of a p-end-neighborhood in $\tilde {\mathcal{O}}$
in a p-end-neighborhood system does not have the corresponding p-end vertex as its vertex. 

(ii)  $\Rightarrow$ (i): 
Let ${\mathcal T}_{\tilde E}$ denote the tube domain with the p-end vertex $\bv_{\tilde E}$ and $\bv_{E -}$. 
Let $K^b$ denote the intersection of $\Bd \clo( {\mathcal T}_{\tilde E})$ with the distanced compact 
$\bGamma_{\tilde E}$-invariant convex set $K'$ by Theorem \ref{thm:distanced}. 
In the reducible case, there is a lens by Theorem \ref{thm:redtot}. 

Now assume the triangle condition.
Let $C_1$ be the convex hull of $K$ in the tube domain ${\mathcal T}_{\tilde E}$ cut off at vertices. 
Suppose that $\Bd C_1 - K$ meets $\Bd \tilde {\mathcal{O}}$. 
Then by the above discussions, 
$\Bd C_1 - K$ contains a line $l$ with endpoints $x, y$ in $K$ 
completely contained in $\Bd \tilde {\mathcal{O}}$.
There exists a triangle $T$ with vertices $x, y, \bv_{\tilde E}$ with $\partial T \subset \Bd \tilde {\mathcal{O}}$. 
$T^o$ is now a subset of a p-end-neighborhood by assumption. 
Lemma \ref{lem:coneseq} contradicts this.

Therefore, we conclude that $\Bd C_1 - K$ is contained in $\tilde{\mathcal{O}}$. 
The convex hull $C_1$ is a $h(\pi_1(\tilde E))$-distanced from $v_{\tilde E}$ and $v_{\tilde E-}$. 
$C_1 \cap {\mathcal T}_{\tilde E}$ is a subset of $\torb$. 

It is standard to show that $\Bd C_1 - K$ is a union of $i$-simplices ($i \leq n$) in the interior ${\mathcal T}_{\tilde E}^o$ 
of the tube with vertices in $K^b$. 

Proposition \ref{prop:convhull2}  implies the result.

(i) $\Rightarrow$ (ii):
First, the generalized lens condition implies that 
$\bGamma_{\tilde E}$ satisfies the middle-eigenvalue condition that $\lambda_1(g)/\lambda_{\bv_{\tilde E}}(g) > 1$ for every $g$
as the proof of Theorem  \ref{thm:redtot} (iii) shows for irreducible $\bGamma_{\tilde E}$ as well. 

There is a map 
\[ \bGamma_{\tilde E} \ra H_1(\bGamma_{\tilde E}, \bR)\]
obtained by taking a homology class. 
The above map $g \ra \log \lambda_{\bv_{\tilde E}}(g)$ 
induces homomorphism
\[\Lambda^h: H_1(\bGamma_{\tilde E}, \bR) \ra \bR \]
that depends on the holonomy homomorphism $h$. 

By the generalized lens-domain, 
there is a lower boundary component $B$ of $D \cap {\mathcal T}_{\tilde E}^o$ closer to $\bv_{\tilde E}$ 
that is strictly convex and transversal to every radial great segment from $\bv_{\tilde E}$ in $\tilde \Sigma_{\tilde E}$. 


If $\bGamma_{\tilde E}$ satisfies the middle-eigenvalue condition, then so does its factors. 
Suppose that $\bGamma_{\tilde E}$ does not satisfy the uniform middle-eigenvalue condition. 
Then there exists a sequence of elements $g_i$ so that 
\[\frac{\log\left(\frac{\lambda^h_1(g_i)}{\lambda^h_{\bv_{\tilde E}}(g_i)}\right)}{ \leng(g_i)} \ra 0 \hbox{ as } i \ra \infty.\] 

Note that we can change $h$ by only changing the homomorphism $\Lambda^h$
and still obtain a representation. 
By a small change of $h$ so that $\Lambda^h(k)$ 
becomes bigger near the limit homology class $k$ 
of the sequence $[g_i]/\leng(g_i)$
fixing the linear part, 
we obtain that 
\[\log\left(\frac{\lambda^h_1(g)}{\lambda_{\bv_{\tilde E}}(g)}\right)< 0  \hbox{ for some } g \in \Gamma.\]
We know that a small perturbation of a lower 
boundary component of a generalized lens-shaped end remains 
strictly convex and in particular distanced
since we are changing the connection by a small amount 
which does not change the strict convexity. 
(See the proof of Theorem \ref{thm:qFuch}.)
We obtain that $\lambda^h_1(g) < \lambda^h_{\bv_{\tilde E}}(g)$ for some $g$
for the largest eigenvalue $\lambda^h_1(g)$ of $h(g)$
and that $\lambda^h_{\bv_{\tilde E}}(g)$ at $\bv_{\tilde E}$.
However, as above the proof of Theorem \ref{thm:redtot} (iii)
contradicts.

\end{proof} 



\begin{lemma} \label{lem:coneseq} 
Suppose that $\mathcal{O}$ is a strongly tame properly convex real projective manifold with radial or totally geodesic ends 
and satisfies the triangle condition. 
Assume that the universal cover $\torb$ is a subset of $\SI^n$.
Then every triangle $T$ with $\partial T \subset \Bd \torb$ and $T^o$ contained in a radially foliated p-end-neighborhood
has no vertex equal to a p-R-end vertex.
\end{lemma} 
\begin{proof} 
Let $\bv_{\tilde E}$ be a p-end vertex. Choose a fixed radially foliated p-end-neighborhood system. 
Suppose that a triangle $T$ with $\partial T \subset \Bd \torb$ contains a vertex equal to a p-end vertex. 
Let $U$ be an inverse image of a radially foliated p-end-neighborhood $\widehat{U}$ in the p-end-neighborhood system corresponding 
to $\tilde E$ with a p-end vertex $\bv_{\tilde E}$. 
The end orbifold $\Sigma_{\tilde E}$  is a properly convex end of an orbifold $\orb$.

Choose a maximal line $l$ in $T$ with endpoints $\bv_{\tilde E}$ and $w$ in the interior of an edge of $T$ not containing $\bv_{\tilde E}$. 
Then this line has to pass a point of the boundary of $U$ and in $T^o$ by definition of 
the radial foliations of the end-neighborhoods. This implies that $T^o$ is not a subset of a p-end-neighborhood. 
This contradicts the assumption. 
\end{proof}

We now prove the dual to Theorem \ref{thm:equ}. For this we do not need the triangle condition or the
reducibility of the end. 

\begin{theorem}\label{thm:equ2}
Let $\orb$ be a strongly tame properly convex real projective manifold with radial ends or totally geodesic ends.
Assume that the holonomy group is strongly irreducible.
Assume that the universal cover $\torb$ is a subset of $\SI^n$ {\rm (}resp. $\bR P^n${\rm ).}
Let $\tilde S_{\tilde E}$ be a totally geodesic ideal boundary of a p-T-end $\tilde E$ of $\torb$. 
Then the following conditions are equivalent: 
\begin{itemize} 
\item[(i)] $\tilde E$ satisfies the uniform middle-eigenvalue condition.
\item[(ii)] $\tilde S_{\tilde E}$ has a lens-neighborhood in an ambient open manifold containing $\torb$
and hence $\tilde E$ has a lens-type p-end-neighborhood in $\torb$. 
\end{itemize} 
\end{theorem}
\begin{proof}
It suffices to proves for $\SI^n$.
Assuming (i), the existence of a lens neighborhood follows from Theorem \ref{thm:lensn}. 

Assuming (ii), we obtain a totally geodesic $(n-1)$-dimensional properly convex domain $\tilde S_{\tilde E}$ in
a subspace $\SI^{n-1}$ where $\bGamma_{\tilde E}$ acts on. 
Let $U$ be the two-sided properly convex neighborhood of it where $\bGamma_{\tilde E}$ acts on. 
Then since $U$ is a two-sided neighborhood, the supporting hemisphere at each point of $\clo(\tilde S_{\tilde E})-\tilde S_{\tilde E}$ is now transversal to 
$\SI^{n-1}$. Considering the dual $U^*$ of $U$ and $\bGamma_E$ action, we apply (i) $\Rightarrow$ (ii) part of Theorem \ref{thm:equ}. 

\end{proof}

\subsection{The characterization of quasi-lens p-R-end-neighborhoods} \label{sub:quai-lens} 

This is the last remaining case for the properly convex ends with weak uniform middle eigenvalue conditions. 
We will only prove for $\SI^n$. 

\begin{definition}\label{defn:quasilens}
Let $U$ be a totally geodesic lens cone p-end-neighborhood of a p-R-end 
in a subspace $\SI^{n-1}$ with vertex $\bv$. Let $G$ denote the p-end fundamental group
satisfying the weak uniform middle eigenvalue condition. 
\begin{itemize}
\item Let $D$ be the totally geodesic $n-2$-dimensional domain so that $U = D \ast \bv$. 
\item Let $\SI^1$ be a great circle meeting $\SI^{n-1}$ at $\bv$. 
\item Extend $G$ to act on $\SI^1$ as a nondiagonalizable transformation fixing $\bv$. 
\item Let $\zeta$ be a projective automorphism 
acting on $U$ and $\SI^1$ so that $\zeta$ commutes with $G$ and restrict to a diagonalizable transformation on $\clo(D)$
and act as a nondiagonalizable transformation on $\SI^1$ fixing $\bv$ also. 
\end{itemize}
Every element of $G$ and $ \zeta$ can be written as a matrix
\begin{equation}
\left( \begin{array}{c|c}
S(g) & 0 \\
\hline
0 &  \begin{array}{cc}
\lambda_{\bv}(g) & \lambda_{\bv}(g)v(g)\\
0 & \lambda_{\bv}(g) \end{array} 
\end{array}\right) \label{eqn:qj}
\end{equation} 
where $\bv =[0, \dots, 1]$. 
Note that $g \mapsto v(g) \in \bR$ is a well-defined map inducing a homomorphism 
\[ \bGamma_{\tilde E} \ra H_1(\bGamma_{\tilde E}) \ra \bR\] 
and hence 
\[ |v(g)| \leq C \cwl(g)\] for a positive constant $C$. 


\begin{description}
\item[Positive translation condition:] We choose an affine coordinate on a component $I$ of $\SI^1 -\{\bv, \bv_-\}$.
We assume that for each $g \in \langle G, \zeta \rangle$,
if $\lambda_{\bv}(g) > \lambda_2(g)$ for the largest eigenvalue $\lambda_2$ associated with $\clo(D)$, 
then $v(g) > 0$ in equation \ref{eqn:qj}, and 
\[ \frac{v(g)}{\log \frac{\lambda_{\bv}(g)}{\lambda_2(g)}} > c_1 > 0 \]
for a constant $c_1$. 
\end{description}
\end{definition}


\begin{proposition}\label{prop:quasilens1} 
Suppose that $\bGamma_{\tilde E}$ satisfies the positive translation condition. 
Then  the above $U$ is in the boundary of a properly convex p-end open neighborhood $V$ of $\bv$ 
and $\langle G, \zeta \rangle$ acts on $V$.
\end{proposition}
\begin{proof}
Let $I$ be the segment in $\SI^1$ bounded by $\bv$ and $\bv_-$. 
Take $D\ast I$ is a tube with vertices $\bv$ and $\bv_-$. 

Let $x$ be an interior point of the tube. 
Given a sequence $g_i \in G$, then $g_i(x)$ accumulates to points of $D\ast \bv$ 
by the positive translation conditions as we can show by using estimates. 

Given any sequence $g_i \in \langle G, \zeta \rangle$, we write as $g_i = \zeta^{j_i} g'_i$ for $g'_i \in G$. 
We write 
\begin{align} 
& x = [v], v = v_1 + v_2, [v_1] \in D, [v_2] \in I -\{\bv\} \subset \SI^1, \nonumber \\
& g_i(x) = [g_i(v_1) + g_i(v_2)].
\end{align} 
If $\lambda_{\bv}(g_i)/\lambda_2(g_i) \ra \infty$, then $||g_i(v_1)||/||g_i(v_2)|| \ra 0$  
and $g_i(x)$ converges to the limit of $[g_i(v_2)]$, i.e., $\bv$, since 
$v(g_i) \ra \infty$. 
If $\lambda_{\bv}(g_i)/\lambda_2(g_i)$ is uniformly bounded from $0$ and $\infty$, 
then $|v(g_i)| < C'$ for a constant. This implies $g_i(x)$ lies in a $(\pi-\eps)$-$\bdd$-neighborhood of 
$\bv_{\tilde E}$ for a uniform constant $\eps$.
If $\lambda_{\bv}(g_i)/\lambda_2(g_i) \ra 0$, then 
$||g_i(v_2)||/||g_i(v_1)|| \ra 0$ and $g_i(x)$ has accumulation points in $D$ only. 
Since these points are inside the properly convex tube and outside a small ball at $\bv_{\tilde E-}$, 
 the interior of the convex hull of the orbit of $x$ is a properly convex open domain as desired above. 
\end{proof}

This generalizes the quasi-hyperbolic annulus discussed in \cite{cdcr2}. 
We give a more concise condition at the end of the subsection. 



Conversely, we obtain:

\begin{proposition} \label{prop:quasilens2} 
Let $\orb$ be a strongly tame properly convex real projective manifold with radial ends or totally geodesic ends.
Assume that the universal cover $\torb$ is a subset of $\SI^n$ {\rm (}resp. $\bR P^n${\rm ).}
Suppose that $\pi_1(\orb)$ is strongly irreducible. 
Let $\tilde E$ be a properly convex radial end satisfying the weak uniform middle eigenvalue conditions
but not the uniform middle eigenvalue condition.
Then $\tilde E$ has a quasi-lens type p-end-neighborhood. 
\end{proposition}
\begin{proof}
If $\tilde E$ is irreducible, then it satisfies the uniform middle eigenvalue condition by definition.  
We recall a part of the proof of Theorem \ref{thm:redtot}. 

Let $U$ be a concave p-end-neighborhood of $\tilde E$ in $\tilde {\mathcal{O}}$.
Let $S_1,..., S_{l_0}$ be the projective subspaces in general position meeting only at the p-end vertex $\bv_{\tilde E}$
where factor groups $\bGamma_1, ...,\bGamma_{l_0}$ act irreducibly on. 
Let $C_i$ denote the union of great segments from $\bv_{\tilde E}$ corresponding to the invariant cones in $S_i$ 
where $\bGamma_i$ acts irreducibly for each $i$. 
The abelian center isomorphic to $\bZ^{l_0-1}$ acts as the identity on $C_i$ in the projective space $\SI^n_{\bv_{\tilde E}}$. 
Let $g\in \bZ^{l_0-1}$. $g| C_i$ can have more than two eigenvalues or just single eigenvalue. 
In the second case $g|C_i$ could be represented by a matrix with eigenvalues all $1$  fixing $\bv_{\tilde E}$. 
\begin{itemize} 
\item[(a)] $g|C_i$ fixes each point of a hyperspace $P_i \subset S_i$ not passing through $\bv_{\tilde E}$ 
and $g$ has a representation as a scalar multiplication in the affine subspace $S_i - P_i$ of $S_i$. 
Since $g$ commutes with every element of $\bGamma_i$ acting on $C_i$, 
$\bGamma_i$ acts on $P_i$ as well.  We let $D'_i = C_i \cap P_i$.
\item[(b)] $g|C_i$ is represented by a matrix with eigenvalues all $1$  fixing $\bv_{\tilde E}$. 
\end{itemize}
We denote $I_1:=\{ i| \exists g \in \bZ^{l_0-1}, g|C_i \ne \Idd\} $ and 
 \[I_2:= \{i| \forall g \in \bZ^{l_0-1}, g|C_i  \hbox{ has only one eigenvalue} \}.\]
 
 Let $D_i \subset \SI^{n-1}_{\bv}$ denote the convex compact domain that is the space of great segments in $C_i$ from $\bv_{\tilde E}$ to 
 $\bv_{\tilde E -}$. Then \[\tilde \Sigma_{\tilde E}=D_1 \ast \cdots \ast D_{l_0}\] by Theorem \ref{thm:redtot}. 
 Also, $D'_i$ is projectively diffeomorphic to $D_i$ by projection for $i \in I_1$. 
 
 If hyperbolic $\bGamma_i$ acts on $C_i$, then it satisfies the uniform middle eigenvalue condition by Definition \ref{defn:umec}.
 Hence by Theorem \ref{thm:equ}, $\bGamma_i$ acts on a lens domain $D_i$. 
 If $i \in I_2$, then $g|C_i$ must be identity; otherwise, we again obtain a violation of the proper convexity 
 in the proof of Theorem \ref{thm:redtot}. 
 
 We know $l_2$ is not empty since otherwise $\tilde E$ satisfies a uniform middle eigenvalue condition.


For $i \in I_2$, $\bGamma_i$ is not hyperbolic as above and hence must be a trivial group 
and $C_i$ is a segment. 
Consider $C_{I_2}:= \ast_{i \in I_2} C_i$. Then $g|C_i$ for $g \in \bZ^{l_0-1}$ 
has only eigenvalue $\lambda_{\bv}$ associated with it
and $g|C_i$ is a translation in an affine coordinate system. 
Therefore, $\bZ^{l_0-1}$ acts trivially on the space of great segments in $C_{I_2}$. 
Thus, $\dim C_{I_2} = 1$ since otherwise we cannot obtain the compact quotient 
$\tilde \Sigma_{\tilde E}/\bGamma_{\tilde E}$. 


Therefore, we obtain $D= \ast_{i =1}^{n-1}D_i$ is a totally geodesic plane disjoint from $\bv_{\tilde E}$.  
Let $\bv_{\tilde E} = [0, \dots, 0, 1]\in \SI^n$. Let $l'_2 =\{n\}=I_2$. 
We write $g \in \bGamma_{\tilde E}$ in coordinates as: 
\[ g = \left( \begin{array}{c|c} 
S_g & 0 \\
\hline
0 & \begin{array}{cc}
\lambda_{\bv}(g) & \lambda_{\bv}(g)v(g)\\
0 & \lambda_{\bv}(g) \end{array} 
\end{array}\right)\]
where $S_g$ is a $n-1\times n-1$-matrix representing coordinates $\{1, \dots, n-1\}$. 
Then $V: g \in \bZ^{l_0} \ra v(g) \in \bR$ is a linear function. 
The proper convexity of $\torb$ implies that $v(g) \geq 0$ if 
$\lambda_{\bv}(g_i)/\lambda_2(g_i) > 1$ since otherwise we obtain a great segment in 
$\SI^1$ by a limit of $g_i(s)$ for a segment $s \subset U$ from $\bv$.  

Suppose that we have a sequence $g_i$ so that $\lambda_{\bv}(g_i)/\lambda_2(g_i) \ra \infty$, 
and $v(g_i) < C, v(g_i) \geq 0$ for a uniform constant $C$.
Given a segment $s \subset U$ with an endpoint $\bv$, $g_i(s)$ then converges to a segment $s_\infty$ in
$\SI^1 \cap \clo(\torb)$. Now $v(g) = 0$ for all $g \in \bGamma_{\tilde E}$ since otherwise 
we can apply $g^i(s)$ to obtain a great segment in the limit for $i \ra \pm \infty$. 
We obtain an element $\eta_i$ so that $\lambda_{\bv}(\eta_i)/\lambda_2(\eta_i) \ra \infty$ and 
$\eta_i| D$ is uniformly bounded using the fact that $\tilde \Sigma_{\tilde E}$ is projectively
diffeomorphic to the interior of the cone $\{p\} \ast D$. We have $v(\eta_i) =0$ for all $i$.  
Then we can apply Lemmas \ref{lem:decjoin} and \ref{lem:joinred}
to obtain a contradiction to the strong irreducibility of $\bGamma$. 

Since every element $g$ is of form $\eta^i g'$ for $\lambda_{\bv}(g')/\lambda_2(g)$ uniformly bounded 
above and $\eta$ with $\lambda_{\bv}(\eta) > \lambda_2(\eta)$, 
we can verify the positive translation condition. 
By Proposition \ref{prop:quasilens1}, we obtain a quasi-lens p-end-neighborhood.  
\end{proof}

\begin{remark} 
To explain the positive translation condition more, $\log \lambda_{\bv_{\tilde E}}(g)$ and $v(g)$ give
us homomorphisms  $\log \lambda_{\bv}, V : H_1(\bGamma_{\tilde E}) \ra \bR$. 
Restricted to $\bZ^{l_0-1} \subset H_1(\bGamma_{\tilde E}) $, 
we obtain $\log \lambda_i: \bZ^{l_0-1} \ra \bR$ given by taking the log of the eigenvalues restricted to 
$D_i$ above. 
The condition restricts to the positivity of $V$ on 
the cone $C$ in $\bZ^{l_0-1}$ defined by \[\log \lambda_{\bv_{\tilde E}}([g])  > \log \lambda_i([g]), i=1, \dots, l_0-1.\] 
Since $\lambda_{\bv_{\tilde E}}(g)$ is less than largest norm of the eigenvalues in $\clo(D)$ for $g \in \bGamma_i - \{\Idd\}, i < l_0$
by the uniform middle eigenvalue conditions, this cone condition is equivalent to the full conditions. 
\end{remark}


\section{The results needed later} \label{sec:results}

We will list a number of properties that we will need later \cite{conv}. 
We show the openness of the lens properties, i.e., the stability for properly convex radial ends
and totally geodesic ends. We show that we can find an increasing sequence of 
horoball p-end-neighborhoods, lens-type p-end-neighborhoods for radial or totally geodesic p-ends that exhausts 
$\torb$. We also show that the p-end-neighborhood always contains a horoball p-end-neighborhood 
or a concave neighborhood. 

\subsection{The openness of lens properties}


A {\em radial affine connection} is an affine connection on $\bR^{n+1} -\{O\}$ invariant under 
the radial dilatation $S_t: \vec{v} \ra t\vec{v}$ for every $t > 0$. 


For representations of $\pi_1(\tilde E)$, being a generalized lens-shaped one and being just a lens-shaped one 
are not different conditions. Given a representation of $\pi_1(\tilde E)$ that has a generalized lens-shaped cone neighborhood,  
the holonomy group satisfies the uniform middle eigenvalue condition by Theorem \ref{thm:equ}. 
We can find a lens cone by choosing our orbifold to be ${\mathcal T}_{\bv_{\tilde E}}/\pi_1(\tilde E)$ 
and using the last step of Theorem \ref{thm:equ}. 

\begin{theorem}\label{thm:qFuch}
Let $\orb$ be a strongly tame properly convex real projective manifold with radial ends or totally geodesic ends.
Assume that the holonomy group is strongly irreducible.
Assume that the universal cover $\torb$ is a subset of $\SI^n$ {\rm (}resp.\, $\bR P^n${\rm ).} 
Let $\tilde E$ be a properly convex p-R-end of the universal cover $\torb$. 
Let $\Hom_E(\pi_1(\tilde E), \SLnp)$ {\rm (}resp. $\Hom_E(\pi_1(\tilde E), \PGL(n+1, \bR)${\rm )} be the space of representations of the 
fundamental group of an $n$-orbifold $\Sigma_{\tilde E}$ with an admissible fundamental group. 
Then 
\begin{itemize}
\item[(i)] $\tilde E$ is a generalized lens-type R-end if and only if $\tilde E$ is a strictly generalized lens-type R-end.
\item[(ii)] The subspace of generalized lens-shaped  representations of an R-end is open. 
\end{itemize}
Finally, if $\orb$ satisfies the triangle condition or every end is reducible, then we can replace the word generalized lens-type
to lens-type in each of the above statements. 
\end{theorem}
\begin{proof} 

(i) If $\pi_1(\tilde E)$ is hyperbolic, then the equivalence is in Theorem \ref{thm:lensclass} (i) and 
if $\pi_1(\tilde E)$ is a virtual product of hyperbolic groups and abelian groups, then it is in Theorem \ref{thm:redtot} (iv).

(ii) Let $\mu$ be a representation $\pi_1(\tilde E) \ra \SLnp$ acting on a convex $n$-domain $K$
bounded by two open strictly convex $(n-1)$-cells $A$ and $B$ and $\Bd K - A - B$ is a nowhere 
dense set. We assume that $A$ and $B$ are smooth and strictly convex.

We note that $K/\mu(\pi_1(\tilde E))$ is a compact manifold with boundary equal to the union of two closed $n$-orbifold components
$A/\mu(\pi_1(\tilde E)) \cup B/\mu(\pi_1(\tilde E))$.
Thus,  $A$ and $B$ are strictly convex hypersurfaces.
By using the theory of deformations of geometric structures on compact orbifolds,
we obtain a manifold $N'$ diffeomorphic to $K/\mu(\pi_1(\tilde E))$. 
$\tilde N'$ is a manifold with two boundary components $A'$ and $B'$
and developing into $\SI^{n-1}$.
Suppose that $\mu'$ is sufficiently near $\mu$. Then $\mu'$ must act on $A'$ and $B'$ sufficiently near 
in the compact open $C^r$-topology, $r \geq 2$. 

Given $K$, we can find a convex $(n+1)$-domain $K' \subset K^o$ bounded by two smooth open 
$n$-cells $A'$ and $B'$ in $K^o$. We may also assume that $K'$ is strictly convex. 

Since $K$ is properly convex, we choose $K'$ as above. 
The linear cone $C(K)\subset \bR^{n+1}$ over $K$ has a smooth strictly convex hessian function $V$ 
by Vey's work \cite{Vey}. Let $C(K')$ denote the linear cone over $K'$.
For the fundamental domain $F$ of $C(K')$ under the action of $\mu(\pi_1(\tilde E))$ extended 
by a transformation $\gamma: \vec{v} \mapsto 2\vec{v}$, the hessian restricted to $F \cap C(K')$ has a lower bound. 
Also, the boundary $\partial C(K')$ is strictly convex in any affine coordinates 
in any transversal subspace to the radial directions at any point.

Let $M$ be $C(K')/\langle \mu(\pi_1(\tilde E)), \gamma \rangle$, a compact orbifold. Note that $S_t$, $t \in \bR_+$, 
becomes an action of circle on $M$.
The change of representation $\mu$ to $\mu': \pi_1(\tilde E) \ra \Aut(\SI^n)_{\bv_{\tilde E}}$ 
is realized by a change of holonomy representations of $M$ and hence by
a change of affine connections on $C(K)$. Since $S_t$ commutes with the images of $\mu$ and $\mu'$, 
$S_t$ still gives us a circle action on $M$ with a different affine connection. 
We may assume without loss of generality 
that the circle action is fixed and $M$ is invariant under this action.

If we change $C(K')$ to to a cone $C(K'')$ of $K''$ by 
a sufficiently small change in the radial affine connection which does not 
change the radial directions locally, the positive definiteness of 
the hessian in the fundamental domain and the boundary transversal strict convexity is preserved. 
Thus $K''$ is also a properly convex domain by Koszul's work \cite{Kos}. 

Thus the perturbed $K''$ is a properly convex domain with strictly convex boundary $A''$ and $B''$. 
The complement  $\Lambda =\clo(K'') - A'' - B''$ is a closed subset. 
Then by Theorems \ref{thm:lensclass} and \ref{thm:redtot}, we obtain that the end is also strictly lens-shaped. 

The final statement follows as in the proof of the part (ii)  $\Rightarrow$ (i) of Theorem \ref{thm:equ} 

\end{proof} 

A {\em strict lens p-end-neighborhood} of a p-T-end $\tilde E$ is a lens p-end-neighborhood 
so that for its boundary component $A$, $\clo(A) - A$ is a subset of $\Bd S_{\tilde E} $ for 
the ideal boundary component $S_{\tilde E}$ of $\tilde E$. 

\begin{theorem}\label{thm:qFuch2}
Let $\orb$ be a strongly tame properly convex real projective orbfold with radial ends or totally geodesic ends.
Assume that the holonomy group is strongly irreducible.
Assume that the universal cover $\torb$ is a subset of $\SI^n$ {\rm (}resp. of $\bR P^n${\rm ).}
Let $\tilde E$ be a p-T-end of the universal cover $\torb$. 
Let $\Hom_E(\pi_1(\tilde E), \SLnp)$  {\rm (}resp. $\Hom_E(\pi_1(\tilde E), \PGL(n+1, \bR)${\rm )} be the space of representations of the 
fundamental group of an $n$-orbifold $\Sigma_{\tilde E}$ with an admissible fundamental group. 
Then 
\begin{itemize}
\item[(i)] $\tilde E$ is a  lens-type p-T-end if and only if $\tilde E$ is a strictly  lens-type p-T-end.
\item[(ii)] the subspace of  lens-shaped  representations of a p-T-end is open. 
\end{itemize}
\end{theorem}
\begin{proof} 
We are proving for $\SI^n$ only. 
First assume that $\bGamma_{\tilde E}$ is irreducible. 

(i) Let $L$ be a lens p-end-neighborhood of the totally geodesic domain $\tilde S_{\tilde E}$ corresponding to $\tilde E$. 
The set of attracting limit points 
lies in $\clo(\tilde S_{\tilde E})-\tilde S_{\tilde E}$ by the uniform middle eigenvalue condition. 
The set is dense by the results of Benoist \cite{Ben1}. 
Each hemisphere supporting the lens neighborhood at a point of $\clo(\tilde S_{\tilde E})-\tilde S_{\tilde E}$ is transversal to the hypersphere containing 
$\tilde S_{\tilde E}$. By Lemma \ref{lem:inde}, the hyperspheres are uniformly bounded away from $\SI^{n-1}$. 
Similarly to the proof of Proposition \ref{prop:orbit}, we obtain the result using Lemma \ref{lem:attracting2}. 

(ii) follows as in the proof of Theorem \ref{thm:qFuch} for the irreducible $\bGamma_E$ 
since the lens neighborhood of the end totally geodesic 
orbifold in the ambient space containing $\orb$ has smooth strictly convex boundary.

In the reducible case for $\bGamma_{\tilde E}$,  
$\bGamma_E$ is dual to the holonomy group of a totally geodesic R-end 
acting on a hyperspace $S$ by Theorem \ref{thm:redtot}. 
Thus, $\bGamma_{\tilde E}$ acts fixing a point $S^*$ dual to $S$. 
Then the proof of Theorem \ref{thm:qFuch} for reducible $\bGamma_E$ applies for this case.
\end{proof} 


\begin{lemma}\label{lem:attracting2} 
Let $\tilde E$ be a p-T-end of the universal cover $\torb$, 
a subset of $\SI^n$ {\rm (}resp. $\bR P^n${\rm ),} of a strongly tame properly convex real projective orbifold $\orb$
with radial ends or totally geodesic ends. 
Assume that $\bGamma_{\tilde E}$ is irreducible and hyperbolic. 
\begin{itemize}
\item Suppose that $\gamma_i$ is a sequence of elements of $\bGamma_{\tilde E}$ acting on $S_{\tilde E}$. 
\item The sequence of attracting fixed points $a_i$ and the sequence of  repelling fixed points $b_i$ are so that 
$a_i \ra a_\infty$ and $b_i \ra b_\infty$ where $a_\infty, b_\infty$ are in $K$. 
\item Suppose that the sequence $\{\lambda_i\}$ of eigenvalues where 
$\lambda_i$ corresponds to $a_i$ converges to $+\infty$. 
\end{itemize} 
Then for the maximal domain $U$ of the affine action
the point $\{a_\infty\}$ is the limit of $\{\gamma_i(J)\}$ for any compact subset $J \subset U$. 
\end{lemma} 
\begin{proof} 
The proof is similar to that of Theorem \ref{thm:attracting}. Here we can use the fact that the supporting hyperspheres 
are at uniformly bounded distances from the hypersphere containing $S_{\tilde E}$. 
\end{proof}


\begin{corollary}\label{cor:mideigen}
We are given a properly convex end $\tilde E$ of a strongly tame properly convex orbifold $\orb$ with radial or totally geodesic  ends. 
Assume that $\torb \subset \SI^n$ {\rm (}resp.\, $\torb \subset \bR P^n${\rm ).}  
Then the subset of 
\[\Hom_E(\pi_1(\tilde E), \SL_\pm(n+1, \bR)) \hbox{ {\rm (}resp.}\, \Hom_E(\pi_1(\tilde E), \PGL(n+1, \bR)) ). \]
consisting of  representations satisfying the uniform middle-eigenvalue condition is open.
\end{corollary} 
\begin{proof} 
For p-R-ends, this follows by Theorems \ref{thm:equ} and \ref{thm:qFuch}.
For p-T-ends, this follows by dual results: Theorem \ref{thm:equ2} and Theorems \ref{thm:qFuch2}. 
\end{proof}


\subsection{The end and the limit sets}

\begin{definition} \label{defn:limitset}
Define the {\em limit set} $\Lambda(\tilde E)$ of a p-R-end $\tilde E$ with a generalized p-end-neighborhood to be 
$\Bd D - \partial D$ for a generalized lens $D$ of $\tilde E$ in $\SI^n$ {\rm (}resp. $\bR P^n${\rm )}
and the {\em limit set} $\Lambda(\tilde E)$ of a p-T-end $\tilde E$ of lens type to be 
$\clo(\tilde S_{\tilde E})-\tilde S_{\tilde E}$ for
the ideal totally geodesic boundary component $\tilde S_{\tilde E}$ of $\tilde E$. 
\end{definition} 

\begin{corollary}\label{cor:independence} 
Let $\mathcal{O}$ be a noncompact strongly tame $n$-orbifold with radial or totally geodesic ends 
and the holonomy group is strongly irreducible.
Let $\tilde E$ be a generalized lens-type p-R-end of $\tilde{\mathcal{O}}$ associated with a p-end vertex $\bv_{\tilde E}$,
and let $U$ be a p-end-neighborhood of $\tilde E$ where $\tilde E$ is a p-T-end or p-R-end. 
Then $\clo(U) \cap \Bd \torb$ is independent of the choice of $U$ and so is the limit set $\Lambda(\tilde E)$ of $\tilde E$.
\end{corollary}
\begin{proof} 
Let $\tilde E$ be a generalized lens-type p-R-end. Then by Theorem \ref{thm:equ}, $\tilde E$ satisfies the uniform middle eigenvalue condition. 
Suppose that $\pi_1(\tilde E)$ acts irreducible. 
Let $K^b$ denote $\Bd {\mathcal T}_{\bv_{\tilde E}} \cap K$ for a distanced minimal compact convex set $K$ where $\bGamma_{\tilde E}$ acts on.
Proposition \ref{prop:orbit} shows that the 
limit set is determined by a set $K^b$ in $\bigcup S(v_{\tilde E})$ since $S(v_{\tilde E})$ is an $h(\pi_1(\tilde E))$-invariant set. 
We deduce that $\clo(U) \cap \Bd \torb = \bigcup S(v_{\tilde E})$. 

Also, $\Lambda(\tilde E) \supset K^b$ since $\Lambda(\tilde E)$ is a $\pi_1(\tilde E)$-invariant compact set in 
$\Bd T_{\tilde \Sigma_{\tilde E}} - \{v_{\tilde E}, v_{\tilde E-}\}$. 
By Proposition \ref{prop:orbit}, each point of $K^b$ is a limit of some $g_i(x)$ for $x \in D$ for a generalized lens. 
Since $D$ is  $\pi_1(\tilde E)$-invariant compact set, $K^b \subset \Lambda(\tilde E)$. 

Suppose now that $\pi_1(\tilde E)$ acts reducibly. Then by Theorem \ref{thm:redtot}, $\tilde E$ is a totally geodesic p-R-end. 
Proposition \ref{prop:orbit} again implies the result. 

Let $\tilde E$ be a p-T-end. By Theorem \ref{thm:qFuch2}(i), $\clo(U) \cap \Bd \torb$  equals 
$\clo(\tilde S_{\tilde E})$ since $\clo(A) - A $ is a subset of this set for $A= \Bd L \cap \torb$
for a lens neighborhood $L$ by the strictness of the lens.

\end{proof} 




\subsection{Expansion and shrinking of admissible p-end-neighborhoods} 


\begin{lemma}\label{lem:expand}  
Let $\mathcal{O}$ have a noncompact strongly tame SPC-structure $\mu$ with admissible ends. 
Assume that the holonomy group is strongly irreducible.
Let $U_1$ be a lens cone p-neighborhood of a horospherical or a lens-type p-R-end $\tilde E$ with the p-end vertex $v$ in 
$\tilde{\mathcal{O}}$ that is foliated by segments from $v${\rm ;} 
or $U_1$ is a lens neighborhood of a totally geodesic p-end $\tilde E$.
Let $\bGamma_{\tilde E}$ denote the p-end fundamental group corresponding to $\tilde E$. 
Then the following holds\,{\rm :} 
\begin{itemize} 
\item Given a compact subset of $\tilde{\mathcal{O}}$, there exists an integer $i_0$ such that 
$U_i$ for $i > i_0$ contains it. 
\item The Hausdorff distance between $U_i$ and $\tilde{\mathcal{O}}$ can be made as small as possible, i.e., 
\[ \forall \eps > 0, \exists \delta, \delta > 0, \hbox{ so that }  \bdd_H (U_i, \torb) < \epsilon. \]
\item There exists a sequence of convex open neighborhoods $U_i$ of $U_1$ in $\tilde{\mathcal{O}}$ 
so that $(U_i - U_j)/\bGamma_{\tilde E}$ for a fixed $j$ and $i> j$ is homeomorphic to a product of an open interval with 
the end orbifold. 
\item We can choose $U_i$ so that $\Bd U_i \cap \torb$ is smoothly embedded and strictly convex with 
$\clo(\Bd U_i) - \torb \subset \Lambda$ where $\Lambda$ is the limit set contained in $\bigcup S(v)$ 
if $v$ is the p-end vertex when $\tilde E$ is radial and in $\clo(\tilde  S_{\tilde E}) - \tilde S_{\tilde E}$ if $\tilde E$ is totally geodesic. 
\end{itemize}
\end{lemma}
\begin{proof} 
First, we study the p-R-end case. 
The p-end-neighborhood $U_1$ is foliated by segments from $v$. 
The foliation leaves are geodesics concurrently ending at a vertex $v$ corresponding to the p-end of $U_1$.
We follow the foliation outward from $U_1$ and take a union of finitely many geodesic leaves $L$ from $\bv_{\tilde E}$
of finite length outside $U_1$ and 
take the convex hull of $U_1$ and $\bGamma_{\tilde E}(L)$ in $\tilde{\mathcal{O}}$. 

Suppose that $U_1$ is horospherical. Then the convex hull is again horospherical. 
We can smooth the boundary to be strictly convex. 
Call the set $\Omega_t$ where $t$ is a parameter $\ra \infty$ measuring the distance from $U_1$. 
By taking $L$ sufficiently densely, we can choose a sequence $\Omega_i$ of strictly convex horospherical open sets at $v$ 
so that eventually any compact subset of $\tilde{\mathcal{O}}$ is in it for sufficiently large $i$.

Let $U_1$ be a  lens-cone now.
Take a union of finitely many geodesic leaves $L$ from $\bv_{\tilde E}$ in $\torb$ 
of $d_{\torb}$-length $t$ outside the lens-cone $U_1$ and
take the convex hull of $U_1$ and $\bGamma_{\tilde E}(L)$ in $\tilde{\mathcal{O}}$. 
Denote the result by $\Omega_t$. Thus, the end points of $L$ not equal to $\bv_{\tilde E}$ are in $\torb$.

We claim that $\Bd \Omega_t \cap \tilde{\mathcal{O}}$ is a connected $(n-1)$-cell and 
$\Bd \Omega_t \cap \tilde{\mathcal{O}}/\bGamma_{\tilde E}$ is 
a compact $(n-1)$-orbifold homeomorphic to $\Sigma_{\tilde E}$ and  $\Bd U_1 \cap \torb$ bounds 
a compact orbifold homeomorphic to the product of a closed interval with  
$(\Bd \Omega_t \cap \tilde{\mathcal{O}})/\bGamma_{\tilde E}$: 
First, each leaf of $g(l), g\in \bGamma_{\tilde E}$ for $l$ in $L$ is so that any converging subsequence of 
$\{g_i(l)\}, g_i\in \bGamma_{\tilde E}$, converges to a segment in $S(v)$ for an infinite collection of $g_i$. 
This follows since a limit is a segment in $\Bd \tilde{\mathcal{O}}$ with a p-endpoint $v$ 
and must belong to $S(v)$ by Proposition \ref{prop:affinehoro}. 

Let $S_1$ be the set of segments with end points in $\bGamma_{\tilde E}(L) \cup \bigcup S(v)$
and  define inductively $S_i$ be the set of simplices with sides in $S_{i-1}$. 
Then the convex hull of $\bGamma_{\tilde E}(L)$ in $\clo(\torb)$ is a union of $S_1 \cup \cdots \cup S_n$. 
We claim that for each maximal segment $s$ from $v$ not in $S(v)$, $s^o$ meets $\Bd \Omega_t \cap \torb$ at a unique point: 
Suppose not. 
Then let $v'$ be its other end point of $s$ in  $\Bd \tilde{\mathcal{O}}$ not passing $\Bd \Omega_t \cap \torb$ in the interior. 
Now, $v'$ is contained in the interior of a simplex $\sigma$ in $S_i$ for some $i$.
Since $\sigma^o \cap \Bd \torb \ne \emp$, $\sigma\subset \Bd \torb$ by Lemma \ref{lem:simplexbd}.
Since the end points $\bGamma_{\tilde E}(L)$ are in $\torb$, the only possibility is that 
the vertices of $\sigma$ are in $\bigcup S(v)$. 
Since $U_1$ is convex and contains $\bigcup S(v)$ in its boundary, 
$\sigma$ is in $\clo(U_1)$ and
and is in the interior of the lens-cone, and no interior point 
of $\sigma$ is in $\Bd \tilde{\mathcal{O}}$, a contradiction. 
Therefore, each maximal segment $s$ from $v$ meets the boundary 
$\Bd \Omega_t \cap \tilde{\mathcal{O}}$ exactly once. 

As in Lemma \ref{lem:infiniteline}, $\Bd \Omega_t \cap \torb$ contains no line segment ending in $\Bd \torb$.  
The strictness of convexity of $\Bd \Omega_t$ follows as 
by smoothing as in the proof of Proposition \ref{prop:convhull2}. 
By taking sufficiently many leaves for $L$ with $\bdd_{\torb}$-lengths $t$ sufficiently large, we can show that any compact subset is
inside $\Omega_t$. From this, the final item follows. 
The first three items now follow if $\tilde E$ is an R-end. 


Suppose now that $\tilde E$ is totally geodesic. Now we use the dual domain $\torb^*$ and the group $\bGamma_{\tilde E}^*$. 
Let $\bv_{\tilde E^*}$ denote the vertex dual to $S_{\tilde E}$. 
By the homeomorphism induced by great segments with end points $\bv_{\tilde E}^*$, we obtain 
 \[(\Bd \torb^* - \bigcup S(\bv_{\tilde E^*}))/\bGamma_{\tilde E}^* \cong \Sigma_{\tilde E}/\bGamma_{\tilde E}^*,\] 
a compact orbifold.
Then we obtain $U_i$ containing $\torb^*$ in ${\mathcal{T}}_{\tilde E}$
by taking finitely many hypersphere outside $F_{i}$ disjoint from $\torb^*$ but meeting
${\mathcal{T}}_{\tilde E}$. Let $H_{i}$ be the open hemisphere containing $\torb^*$ bounded by $F_{i}$. 
Then we form $U_1 := \bigcap_{g\in \Gamma_{\tilde E}} g(H_i)$. 
By taking more hyperspheres, we obtain a sequence 
\[U_1 \supset U_2 \supset \cdots \supset U_i \supset U_{i+1} \supset \cdots \supset \torb^* \]
so that $\clo(U_{i+1}) \subset U_i$ and 
\[\bigcap_i \clo(U_i) =\clo(\torb^*) - \bigcup S(\bv_{\tilde E^*}).\]
That is for sufficiently large hyperplanes, we can make 
$U_i$ disjoint from any compact subset disjoint from $\clo(\torb^*) \cup \mathcal{A}(\torb)$.
Now taking the dual $U_i^*$ of $U_i$ and by equation \ref{eqn:dualinc} we obtain
\[ U_1^* \subset U_2^* \subset \cdots \subset U_i^* \subset U_{i+1}^* \subset \cdots \subset \torb.\]
Then $U_{i}^* \subset \torb$ is an increasing sequence eventually containing all compact subset of $\torb$. 
This completes the proof for the first three items.

The fourth item follows from Corollary \ref{cor:independence}.
\end{proof} 


We now discuss the ``shrinking'' of p-end-neighborhoods. These repeat some results. 
\begin{corollary} \label{cor:shrink} 
Suppose that $\orb$ is a strongly tame properly convex real projective orbifold with radial or totally geodesic ends
and let $\torb$ be a properly convex domain in $\SI^n$ {\rm (}resp. $\bR P^n${\rm )} covering $\orb$. 
Assume that the holonomy group is strongly irreducible.
Then the following statements hold\,{\rm :} 
\begin{itemize} 
\item[(i)] If $\tilde E$ is a horospherical p-R-end, 
every p-end-neighborhood of $\tilde E$ contains a horospherical p-end-neighborhood.
\item[(ii)] If $\tilde E$ is a lens-shaped p-R-end or satisfies the uniform middle eigenvalue condition,
 every p-end-neighborhood $V$ where $(\Bd V \cap \torb)/\pi_1(\tilde E)$ is a compact orbifold and 
and  $V^o \supset I \cap \torb$ for the convex hull $I$ of $\bigcup S(\bv_{\tilde E})$ 
of the p-end vertex $\bv_{\tilde E}$ contains a lens-shaped p-end-neighborhood. 
\item[(iii)] If $\tilde E$ is a generalized lens-shaped p-R-end or satisfies the uniform middle eigenvalue condition, 
every p-end-neighborhood of $\tilde E$ contains a concave p-end-neighborhood.
\item[(iv)] Suppose that $\tilde E$ is a p-T-end of lens type or satisfies the uniform middle eigenvalue condition.
Then every p-end-neighborhood contains a convex p-end-neighborhood $L$ with strictly convex boundary in $\torb$. 
\end{itemize} 
\end{corollary} 
\begin{proof}
Let us prove for $\SI^n$. 

(i) Let $v_{\tilde E}$ denote the p-R-end vertex corresponding to $\tilde E$. 
By Theorem \ref{thm:comphoro}, we obtain a conjugate of a parabolic subgroup of $\SO(n, 1)$ as the finite index subgroup of $h(\pi_1(\tilde E))$
acting on $U$, a p-end-neighborhood of $\tilde E$. We can choose an ellipsoid of $\bdd$-diameter $\leq \eps$ for any $\eps > 0$ 
in $U$ fixing $v_{\tilde E}$.    

(ii) This follows from Proposition \ref{prop:convhull2} since 
the convex hull of $\partial \bigcup S(\bv_{\tilde E})$ has the right properties. 

(iii) Suppose that we have a lens-cone $V$
that is a p-end-neighborhood equal to $L \ast v_{\tilde E}$ where $L$ is a generalized lens bounded away from $v_{\tilde E}$. 
Let \[\eps := \sup \{d_{\torb}(x, L)| x \in \Bd U \cap \torb \}.\]
If $\Bd U \cap \torb \subset L$, then $\Bd (V - L) \cap \torb \subset U$ holds. 
Since a point of $V-L$ near $v_{\tilde E}$ is in $U$, and $V-L$ is connected,
the concave p-end-neighborhood $V - L$ is  a subset of $U$ and we are done. 

Now suppose that $\Bd U \cap \torb$ is not a subset of $L$. 
By taking smaller $U$ if necessary, we may assume that $U$ and $L$ are disjoint. 
Since $\Bd U/h(\pi_1(\tilde E))$ and $L/h(\pi_1(\tilde E))$ are compact, $\eps > 0$. 
Let $L' := \{x \in L | d_{V}(x, L) \leq \eps\}$. Then we can show that $L'$ is a generalized ens
since a lower component of $\partial L'$ is strictly convex by Lemma 1.8 of \cite{CLT2}. 
(Given $u, v \in N_\eps$, we find 
\[w, t \in \Omega \hbox{ so that } d_V(u, w) < \eps, d_V(v, t) < \eps.\]
Then $\ovl{uv}$ is within $\eps$ of $\ovl{wt} \subset \Omega$ in the $d_V$-sense.) 
Clearly, $h(\pi_1(\tilde E))$ acts on $L'$. 

Let $V$ be the subspace $V - L'$. Then we choose sufficiently large $\eps'$ so that
 $\Bd U \cap \torb \subset L'$, and hence $V-L'\subset U$ form a concave p-end-neighborhood as above. 


(iv) This follows from Theorem \ref{thm:lensn}.

\end{proof}

 \part{The classification of NPCC ends} 

\section{The uniform middle eigenvalue conditions for NPCC ends} \label{sec:notprop}

We will now study the ends where the transverse real projective structures are not properly convex but not projectively diffeomorphic to 
a complete affine subspace.
Let $\tilde E$ be a p-R-end of $\orb$ and let $U$ the corresponding p-end-neighborhood in $\torb$
with the p-end vertex $\bv_{\tilde E}$. 

The closure $\clo(\tilde \Sigma_{\tilde E})$ contains 
a great $(i_0-1)$-dimensional sphere and $\tilde \Sigma_{\tilde E}$ is foliated by $i_0$-dimensional hemispheres 
with this boundary. 
Let $\SI^{i_0-1}_\infty$ denote the great $(i_0-1)$-dimensional sphere in $\SI^{n-1}_{\bv_{\tilde E}}$ of $\tilde \Sigma_{\tilde E}$. 
The space of $i_0$-dimensional hemispheres in $\SI^{n-1}_{\bv_{\tilde E}}$ with boundary $\SI^{i_0-1}_\infty$ form 
a projective sphere $\SI^{n-i_0-1}$. 
The projection \[\SI^{n-1}_{\bv_{\tilde E}} - \SI^{i_0-1}_\infty \ra \SI^{n-i_0-1} \] 
gives us an image of $\tilde \Sigma_{\tilde E}$ 
that is the interior of a  properly convex compact set $K$. (See \cite{ChCh} for details. See also \cite{GV}.)

Let $\SI^{i_0}_\infty$ be a great $i_0$-dimensional sphere containing $\bv_{\tilde E}$ corresponding to the directions
of $\SI^{i_0-1}_\infty$ from $\bv_{\tilde E}$. 
The space of  $(i_0+1)$-dimensional hemispheres 
with boundary $\SI^{i_0}_\infty$ again has the structure of the projective sphere $\SI^{n-i_0-1}$, 
identifiable with the above one. 
Denote by $\Aut_{\SI^{i_0}_\infty} (\SI^n)$ the group of projective automorphisms of $\SI^n$
acting on $\SI^{i_0}_\infty$ and fixing $\bv_{\tilde E}$.
We also have the projection \[\Pi_K: \SI^{n} - \SI^{i_0}_\infty \ra \SI^{n-i_0-1} \] 
giving us the image $K^o$ of a p-end-neighborhood $U$. 


Each $i_0$-dimensional 
hemisphere $H^{i_0}$ in $\SI^{n-1}_{\bv_{\tilde E}}$ with $\Bd H^{i_0} = \SI^{i_0-1}_\infty$ corresponds to an $(i_0+1)$-dimensional hemisphere 
$H^{i_0+1}$ in $\SI^n$ with common boundary $\SI^{i_0}_\infty$ that contains $\bv_{\tilde E}$. 

Let $\SL_\pm(n+1, \bR)_{\SI^{i_0}_\infty, \bv_{\tilde E}}$ 
denote the subgroup of $\Aut(\SI^n)$ acting on 
$\SI^{i_0}_\infty$ and $\bv_{\infty}$.  
The projection $\Pi_K$ induces a homomorphism 
\[\Pi_K^*: \SL_\pm(n+1, \bR)_{\SI^{i_0}_\infty, \bv_{\tilde E}} 
\ra \SL_\pm( n, \bR).\]

Suppose that $\SI^{i_0}_\infty$ is $h(\pi_1(\tilde E))$-invariant. 
We let $N$ be the subgroup of $h(\pi_1(\tilde E))$ of elements inducing trivial actions on $\SI^{n-i_0-1}$. 
The above exact sequence 
\[ 1 \ra N \ra h(\pi_1(\tilde E)) \stackrel{\Pi^*_K}{\longrightarrow} N_K \ra 1\] 
is so that the kernel normal subgroup $N$ acts trivially on $\SI^{n-i_0-1}$ but acts on each hemisphere with 
boundary equal to $\SI^{i_0}_\infty$
and $N_K$ acts faithfully by the action induced from $\Pi^*_K$.

Here $N_K$ is a subgroup of the group $\Aut(K)$ of the group of projective automorphisms of $K$
is called the {\em semisimple quotient } of $h(\pi_1(\tilde E))$ or $\bGamma_{\tilde E}$. 

\begin{theorem}\label{thm:folaff}
Let $\Sigma_{\tilde E}$ be the end orbifold of an NPCC p-R-end $\tilde E$ of a strongly tame  
properly convex $n$-orbifold $\orb$ with radial or totally geodesic ends. Let $\torb$ be the universal cover in $\SI^n$. 
We consider the induced action of $h(\pi_1(\tilde E))$ 
on $\Aut(\SI^{n-1}_{\bv_{\tilde E}})$ for the corresponding end vertex $\bv_{\tilde E}$. 
Then 
\begin{itemize} 
\item $\Sigma_{\tilde E}$ is foliated by complete affine subspaces of dimension $i_0$, $i_0 > 0$.
\item $h(\pi_1(\tilde E))$ fixes the great sphere $\SI^{i_0-1}_\infty$ of dimension $i_0-1$ in $\SI^{n-1}_{\bv_{\tilde E}}$. 
\item There exists an exact sequence 
\[ 1 \ra N \ra \pi_1(\tilde E) \stackrel{\Pi^*_K}{\longrightarrow} N_K \ra 1 \] 
where $N$ acts trivially on quotient great sphere $\SI^{n-i_0-1}$ and 
$N_K$ acts faithfully on a properly convex domain $K^o$ in $\SI^{n-i_0-1}$ isometrically 
with respect to the Hilbert metric $d_K$. 
\end{itemize} 
\end{theorem}
\begin{proof} 
These follow from Section 1.4 of \cite{ChCh}. (See also \cite{GV}.) 
\end{proof} 

We denote by $\mathcal{F}$ the foliations on $\Sigma_{\tilde E}$ or the corresponding one in $\tilde \Sigma_{\tilde E}$. 

\subsubsection{The main eigenvalue estimations}

We denote by $\bGamma_{\tilde E}$ the p-end fundamental group acting on $U$ fixing $\bv_{\tilde E}$. 
Denote the induced foliations on $\Sigma_{\tilde E}$ and $\tilde \Sigma_{\tilde E}$ by ${\mathcal F}_{\tilde E}$.
For each element $g \in \bGamma_{\tilde E} $, we define $\leng_K(g)$ to 
be $\inf \{ d_K(x, g(x))| x \in K^o \}$.

\begin{definition}\label{defn:jordan}
Given an eigenvalue $\lambda$ of an element $g \in \SLnp$, 
a $\bC$-eigenvector $\vec v$ is a nonzero vector in 
\[\bR E_{\lambda}(g) := \bR^{n+1} \cap (\ker (g - \lambda I) + \ker (g - \bar \lambda I)), \lambda \ne 0, {\mathrm{Im}} \lambda \geq 0\]
A $\bC$-fixed point is a direction of $\bC$-eigenvector. 

Any element of $g$ has a Jordan decomposition. An irreducible Jordan-block corresponds to a unique subspace
in $\bC^{n+1}$, called an {\em elementary Jordan subspace}. We denote by $J_{\mu, i} \subset \bC^{n+1}$ 
for an eigenvalue $\mu \in \bC$
for $i$ in an index set.
A {\em real elementary Jordan subspace} is defined as 
\[ R_{\mu, i} := \bR^{n+1} \cap (J_{\mu, i} + J_{\bar \mu, i}),
\mu \ne 0, {\mathrm{Im}} \mu \geq 0\]
of Jordan subspaces with 
 $\ovl{J_{\mu, i}} = J_{\bar \mu, i}$ in $\bC^{n+1}.$
We define the  {\em real sum} of elementary 
Jordan-block subspaces is defined to be 
\[\bigoplus_{i \in I}R_{\mu, i}\]
for a finite collection $I$.

A point $[\vec v], \vec v \in \bR^{n+1},$ is {\em affiliated} with a norm $\mu$ of an eigenvalue if 
$\vec v \in \bigoplus_{|\lambda| = \mu, i \in I_{\lambda}} R_{\lambda, i}$ 
for a sum of all real elementary Jordan subspaces $R_{\lambda, i}$, $\mu = |\lambda|$. 
\end{definition}

Let $V^{i+1}_\infty$ denote the subspace of $\bR^{n+1}$ corresponding 
to $\SI^i_\infty$. 
By invariance of $\SI^i_\infty$, 
if \[\oplus_{(\mu, i) \in J} R_{\mu, i} \cap V^{i+1}_\infty \ne \emp\] for some finite collection $J$, 
then $\oplus_{(\mu, i) \in J} R_{\mu, i} \cap V^{i+1}_\infty$ always contains a $\bC$-eigenvector. 

\begin{definition}\label{defn:eig} 
Let $\Sigma_{\tilde E}$ be the end orbifold of a nonproperly convex and p-R-end $\tilde E$ of a 
strongly tame properly convex $n$-orbifold $\orb$ with radial or totally geodesic ends. Let $\bGamma_{\tilde E}$ 
be the p-end fundamental group. 
We fix a choice of a Jordan decomposition of $g$ for each $g \in \bGamma_{\tilde E}$. 
\begin{itemize}
\item Let $\lambda_1(g)$ denote the largest norm of the eigenvalue of $g \in \bGamma_{\tilde E}$ affiliated
with $\vec{v} \ne 0$, $[\vec{v}] \in \SI^n - \SI^{i_0}_\infty$,
i.e., \[\vec v \in \bigoplus_{(\mu_1(g), i) \in J} R_{\mu_1(g), i} -  V^{i_0+1}_\infty, |\mu_1| = \lambda_1(g)\]
where $J$ indexes all elementary Jordan subspaces of $\lambda_1(g)$.
\item Also, let $\lambda_{n+1}(g)$ denote the smallest one affiliated with a nonzero vector $\vec v$, $[\vec{v}] \in \SI^n - \SI^{i_0}_\infty$,
i.e., \[\vec v \in \bigoplus_{(\mu_1(g), i) \in  J'} R_{\mu_{n+1}(g), i} - V^{i_0+1}_\infty, |\mu_{n+1}|=\lambda_{n+1}(g)\]
where $J'$ indexes all real elementary Jordan subspaces of $\lambda_{n+1}(g)$.
\item Let $\lambda(g)$ be the largest of the norm of the eigenvalue of $g$ with a $\bC$-eigenvector $\vec v$, $[\vec v] \in \SI^{i_0}_\infty$
and $\lambda'(g)$ the smallest such one. 
\end{itemize} 
\end{definition}

Then for some $l' \geq 1$, $\bGamma_{\tilde E}$ is isomorphic to 
$\bZ^{l'-1} \times \bGamma_1 \times \cdots \times \bGamma_{l'} $ up to finite index
where each $\bGamma_i$ acts irreducibly on $K_i$ for each $i = 1, \dots, r$. 
We assume that $\bGamma_i$ is torsion-free hyperbolic for $i=1, \dots, s$ and 
$\bGamma_i = \{\Idd\}$ for $s+1 \leq  i \leq l'$.  
We will use the notation $\bGamma_i$ the corresponding subgroup in $\bGamma_{\tilde E}$. 

Suppose that $K$ has a decomposition into $K_1 * \cdots * K_{l_0}$ for properly convex domains 
$K_i$, $i =1, \dots, l_0$. 
Let $K_i, i=1, \dots, s$, be the ones with dimension $\geq 2$.
$N_K$ is virtually isomorphic to the product \[ \bZ^{l_0-1} \times \Gamma_1 \times \dots \times \Gamma_{s}\]
where $\Gamma_i$ is obtained from $N_K$ by restricting to $K_i$ and 
$A$ is a free abelian group of finite rank. 

$\bGamma_i \cap N$ is a normal subgroup of $\bGamma_i$. 
If $\bGamma_i$ is hyperbolic, then this group has to be trivial. 
Therefore, we obtain that each $\bGamma_i$ for $i = 1, \dots, s$, is mapped isomorphic to some $\Gamma_i$ 
in $N_K$ provided $\bGamma_i$ is hyperbolic. Thus, each $K_i$ is a strictly convex domain or a point
by the results in \cite{Ben1}. 


The following definition generalizes Definition \ref{defn:umec}. 
The two definitions can be stated in the same manner but we avoid doing so here. 
\begin{definition} \label{defn:NPCC}
We also assume that the {\em uniform middle-eigenvalue condition} relative to $N$: 
\[N_K \cong \bZ^l \times \Gamma_1 \times \dots \times \Gamma_k, l \geq k-1\] acts on
\[\clo(K) = p_{1} \ast \cdots \ast p_{l-k-1} \ast K_1 \ast \cdots \ast K_k \]
where $\Gamma_i$ is a hyperbolic group acting on properly convex domain $K_i$ for each $i$, $i=1, \dots, k$,
 and each $p_j$ is a singleton for $j=1, \dots l-k-1$ with following conditions.
 Let $\hat K_i$ denote the subspace spanned by $\Pi_K^{-1}(K_i) \cup \SI^{i_0}_\infty$. 
\begin{itemize}
\item there exists a constant $C >0$ independent of $g \in \bGamma_{\tilde E}$ such that
\begin{equation} \label{eqn:mec}
C^{-1}  \leng_{K}(g) \leq \log \frac{\bar \lambda(g)}{\lambda_{\bv_{\tilde E}}(g)} \leq C \leng_{K}(g)
\end{equation}
for $\bar \lambda(g)$ equal to 
\begin{itemize}
\item the largest norm of the eigenvalues of $g$ which must occur for a fixed point of $\hat K_i$ if $g \in \Gamma_i$
\end{itemize} 
and the eigenvalue $\lambda_{\bv_{\tilde E}}(g)$ of $g$ at $\bv_{\tilde E}$.

\end{itemize}



If we require only $\bar \lambda(g) \geq \lambda_{\bv_{\tilde E}}(g)$ for $g \in \bGamma_{\tilde E}$, 
and the uniform middle eigenvalue condition for each hyperbolic group $\bGamma_i$, 
then we say that $\bGamma_{\tilde E}$ satisfies the {\em weakly uniform middle-eigenvalue conditions}. 
\end{definition}
This definition is simply an extension of one for 
properly convex end ones since we could have used
the degenerate metric on $\tilde \Sigma_{\tilde E}$ 
based on cross-ratios and the definition would agree. 
In fact, for complete ends, the definition agrees by 
Proposition \ref{prop:affinehoro} and Theorem \ref{thm:comphoro}. 









The following proposition is very important in this part III showing that 
$\lambda_1(g)$ and $\lambda_{n+1}(g)$ are true largest and smallest norms of the eigenvalues of $g$. 
We will sharpen the following to inequality in the discrete and indiscrete cases.  
\begin{proposition}\label{prop:eigSI}  
Let $\Sigma_{\tilde E}$ be the end orbifold of a nonproperly convex p-R-end $\tilde E$ of a strongly tame 
properly convex $n$-orbifold $\orb$ with radial or totally geodesic ends. 
Suppose that $\torb$ in $\SI^n$ (resp. $\bR P^n$) 
covers $\orb$ as a universal cover. 
Let $\bGamma_{\tilde E}$ be the p-end fundamental group satisfying the weak uniform middle-eigenvalue condition.
Let $g \in \bGamma_{\tilde E}$. 
Then
\[\lambda_1(g) \geq \lambda(g) \geq  \lambda'(g) \geq \lambda_{n+1}(g)\]
holds. 
\end{proposition} 
\begin{proof}
We may assume that $g$ is of infinite order. 
Suppose that $\lambda(g) > \lambda_1(g)$. 
We have $\lambda(g) \geq  \lambda_{\bv_{\tilde E}}(g)$,
and $\lambda(g)$ is the largest norm of the eigenvalues of $g$. 
If $\lambda(g) = \lambda_{\bv_{\tilde E}}(g)$, then $\lambda_{\bv_{\tilde E}}(g) > \lambda_1(g)$ contradicts 
the weak uniform middle-eigenvalue condition. Thus, $\lambda(g) > \lambda_{\bv_{\tilde E}}(g) .$

Now, $\lambda_1(g) < \lambda(g)$ implies that
\[R_{\lambda(g)}:= \bigoplus_J R_{\mu, i}(g),  |\mu|=\lambda(g)\]
is a subspace of $V^{i_0+1}_\infty$ and corresponds to a great sphere $\SI^j$. 
Hence, a great sphere $\SI^j$, $j \geq 0$, in $\SI^{i_0}_\infty$ is disjoint from 
$\{\bv_{\tilde E}, \bv_{\tilde E-}\}.$ 
Since $\bv_{\tilde E} \in \SI^{i_0}_\infty$ is not contained in $\SI^j$, we obtain $j+1 \leq i_0$.

There exists a great sphere $C_1$ disjoint from $\SI^j$ 
whose vector space $V_1$ corresponds the real sum of Jordan-block subspaces where 
$g$ has strictly smaller norm eigenvalues and is complementary to $R_{\lambda(g)}$. 
Then $C_1$ contains $\bv_{\tilde E}$
and $C_1$ is of complementary dimension to $S^j$, i.e., $\dim C_1 = n - j-1$. 


Since $C_1$ is complementary to $\SI^j \subset \SI^{i_0}_\infty$, a complementary subspace 
$C'_1$ to $\SI^{i_0}_\infty$ of dimension $n-i_0-1$ is in $C_1$. 
Considering the sphere $\SI^{n-1}_{\bv_{\tilde E}}$ at $\bv_{\tilde E}$, 
it follows that $C'_1$ goes to a $n-i_0-1$-dimensional subspace $C''_1$ in 
$\SI^{n-1}_{\bv_{\tilde E}}$ disjoint from $\partial l$ for any complete affine leaf $l$. 
Each complete affine leaf $l$ of $\tilde \Sigma_{\tilde E}$ has the dimension $i_0$ 
and meets $C''_1$ in $\SI^{n-1}_{\bv_{\tilde E}}$ by the vector space consideration. 


Hence, a small ball $B'$ in $U$ meets $C_1$. 
\begin{align}
& \hbox{For any } [v] \in B', v \in \bR^{n+1}, v = v_1 + v_2 \hbox{ where } [v_1] \in C_1 \hbox{ and } [v_2] \in \SI^j. \nonumber \\
& \hbox{We obtain } g^k([v]) = [ g^k(v_1) + g^k(v_2)].
\end{align}
By the real Jordan decomposition consideration, 
the action of $g^k$ as $k \ra \infty$ makes the former vectors very small compared to the latter ones, i.e., 
\[ ||g^k(v_1)||/||g^k(v_2)|| \ra 0 \hbox{ as } k \ra \infty.\]
Hence, $g^k([v])$ converges to the limit of $g^k([v_2])$ if it exists. 

Now choose $[w]$ in $C_1 \cap B'$ and $v, [v]\in \SI^j$. 
We let $w_1 = [w + \eps v]$ and $w_2 =[w -\eps v]$ in $B'$ for small $\eps'> 0$. 
Choose a subsequence $\{k_i\}$ so that $g^{k_i}(w_1)$ converges to a point of $\SI^n$. 
The above estimations shows that $\{g^{k_i}(w_1)\}$ and $\{g^{k_i}(w_2)\}$ converge to 
an antipodal pair of points in $\clo(U)$ respectively. 
This contradicts the proper convexity of $U$ 
as $g^k(B'') \subset U$ and the geometric limit is in $\clo(U)$. 

Also the consideration of $g^{-1}$ completes the inequality.

\end{proof}



\section{The discrete case } \label{sec:discrete}

Now, we will be working on projective sphere $\SI^n$ only for while. 
Suppose that the semisimple quotient group $N_K$ is a discrete subgroup of $\Aut(\SI^{n-i_0-1})$, 
which is a much simpler case to start. 
(Actually $N_K$ virtually equals the group 
\[\bZ^{l_0-1} \times \Gamma_1 \times \cdots \times \Gamma_{l_0}\] 
since each factor $\Gamma_i$ commutes with the other factors 
and acts trivially on $K_j$ for $j \ne i$ as was shown in the proof of Theorem \ref{thm:redtot}
and $N_K$ acts cocompactly on $K$.)

We have a corresponding fibration 
\begin{eqnarray}
l/N & \ra & \tilde \Sigma_{\tilde E}/\bGamma_{\tilde E} \nonumber \\ 
     &       & \downarrow \nonumber \\ 
   &         & K^o/N_K 
\end{eqnarray}
where the fiber and the quotients are compact orbifolds
since $\Sigma_{\tilde E}$ is compact. 

Since $N$ acts on each leaf $l$ of ${\mathcal F}_{\tilde E}$ in $\tilde \Sigma_{\tilde E}$, 
it also acts on a properly convex domain $\torb$ and $\bv_{\tilde E}$ in a subspace $\SI^{i_0+1}_l$ in $\SI^n$ 
corresponding to $l$. $l/N \times \bR$ is an open real projective orbifold diffeomorphic to 
$(H^{i_0+1}_l \cap \torb)/N$
for an open hemisphere $H^{i_0+1}_l$ corresponding to $l$. 
As above, Proposition \ref{prop:affinehoro} and Theorem \ref{thm:comphoro} show that 
\begin{itemize} 
\item $l$ corresponds to a horospherical end of $(\SI^{i_0+1}_l \cap \torb)/N$
and 
\item $N$ is virtually unipotent and $N$ is virtually a cocompact subgroup of 
a unipotent group, conjugate to a parabolic subgroup of
$\SO(i_0+1, 1)$ in $\Aut(\SI^{i_0+1}_l)$ 
and acting on an ellipsoid of dimension $i_0$ in $H^{i_0+1}_l$. 
\end{itemize} 


By Malcev, the Zariski closure $Z(N)$ of $N$ is a virtually nilpotent Lie group with finitely many components 
and  $Z(N)/N$ is compact. 
Let $\CN$ denote the identity component of the Zariski closure of $N$
so that $\CN/(\CN \cap N)$ is compact.
$\CN \cap N$ acts on the great sphere $\SI^{i_0+1}_l$ containing $\bv_{\tilde E}$ and corresponding to $l$.
Since $N$ acts on a horoball in $\SI^{i_0+1}_l$ and $\CN/\CN\cap N$ is compact, we can modify the horoball 
to be invariant under $\CN$ by taking the convex hull of images of it under $\CN$.

Let $V^{i_0+1}$ denote the subspace corresponding to $\SI^{i_0}_\infty$ containing $\bv_{\tilde E}$
and $V^{i_0+2}$ the subspace corresponding to $\SI^{i_0+1}_l$.
We choose the coordinate system so that 
\[\bv_{\tilde E} = \underbrace{[0, \cdots, 0, 1]}_{n+1}\]
and points of $V^{i_0+1}$ and those of $V^{i_0+2}$ respectively correspond to 
\[ \overbrace{[0, \dots, 0}^{n-i_0}, \ast, \cdots, \ast], \quad \overbrace{[0, \dots, 0}^{n-i_0-1}, \ast, \cdots, \ast].\]
We can write each element  $g \in \CN$ as an $(n+1)\times (n+1)$-matrix
\begin{equation} \label{eqn:nilmat}
 \left( \begin{array}{ccc}                \Idd_{n-i_0-1} &  0 &   0 \\ 
                                                       \vec{0}      & 1  &  0 \\
                                                     C_g          & *   & U_g 
                                    \end{array} \right)
 \end{equation}
where $C_g > 0$ is an $(i_0+1)\times (n-i_0-1)$-matrix, 
$U_g$ is a unipotent $(i_0+1) \times (i_0+1)$-matrix, 
$0$ indicates various zero row or column vectors, 
$\vec{0}$ denotes the zero row-vector of dimension $n-i_0-1$, and 
$\Idd_{n-i_0-1}$ is  the $(n-i_0-1) \times (n-i_0-1)$ identity-matrix. 
This follows since $g$ acts trivially on $\bR^{n+1}/V^{i_0+1}$ 
and $g$ acts as a unipotent matrix on the subspace $V^{i_0+2}$.  
We choose an arbitrary point $w \in \SI^{i_0+1}_l \cap U$
and choose the coordinates so that  
\[ w= [\underbrace{\overbrace{0, \dots, 1}^{n-i_0}, \dots, 0]}_{n+1},\]
and one can think of $w$ as the origin of the affine space $A$ so that 
$U \subset A$ and $\SI^{i_0}_\infty \subset \Bd A$.


Since $\CN$ is abelian and acts on an ellipsoid with a complete Euclidean metric of dimension $i$, 
each element of $\CN$ is a translation in some affine coordinates of $\tilde \Sigma_{\tilde E}$. 
Considering the great $2$-sphere 
containing the translation subspace and $\bv_{\tilde E}$ as a coordinate subspace, 
we write an element of $\CN$ 
on a coordinate system of $\SI^{i_0+1}$ as 
\begin{equation} \label{eqn:generator}
\left( \begin{array}{ccc}                
                                                         1 &  \vec{0}  & 0 \\
                                                        \vec{v}^T  & \Idd_{i_0-1} & \vec{0}^T \\ 
                                                \frac{||\vec{v}||^2}{2}  &  \vec{v}       & 1 
                                    \end{array} \right) \hbox{ for } \vec{v} \in \bR^{i_0}. 
 \end{equation}
 (see \cite{CM2} for details.)
 We can make each translation direction of generators of $\CN$ in $\tilde \Sigma_{\tilde E}$ to be one of the standard vector. 
Therefore, we can find 
a coordinate system of $V^{i_0+2}$ so that the generators are of $(i_0+2) \times (i_0+2)$-matrix forms 
\begin{equation} \label{eqn:generator2}
\hat \CN_j :=  \left( \begin{array}{cccc}                1 &   \vec{0} & 0 \\ 
                                                        \vec{e}^T_j  & \Idd_{i_0}    & 0 \\
                                                       \frac{1}{2} &  \vec{e}_j     & 1 
                                    \end{array} \right)
 \end{equation}
 where $(\vec{e}_{j})_{k} = \delta_{jk}$ a row $i$-vector for $j=1, \dots, i$.  
Hence, the generator $\CN_j$ of $\CN$ is of form 
\begin{equation} \label{eqn:nilmat2}
\CN_j:= \left( \begin{array}{cccc}         \Idd_{n-i_0-1} & 0 &  0 &0 \\ 
                                                       \vec{0}  & 1  &  0  & 0\\
                                                      C_j    & \vec{e}_j^T & \Idd & 0 \\ 
                                                      c'_j    & \frac{1}{2} & \vec{e}_j & 1 
                                    \end{array} \right)
 \end{equation}
where we used coordinates so that $N'_j$ is a lower-triangular form in an $(i_0+1)\times (i_0+1)$-matrix
and $v_j$ is a column vector of dimension $i_0+1$ and is not all zero and $C_j$ is an $(n-i_0-1)\times i_0$-matrix
and $c'_j$ is an $(n-i_0-1)\times 1$-matrix. 
(We remark that $N \cap \CN := \CN(L)$ for a lattice $L$ in $\bR^{i_0}$.) 

Let $g$ be an element of $\Gamma$ mapping to a nontrivial element of $N_K$ 
also acting on $\SI^{i_0}_\infty$. Let $g' = g| \SI^{i_0}_\infty$ and let $U_g$ denote the corresponding 
matrix for $V^{i_0+1}$. Then $U_g N' U_g^{-1}$ is still an element of $N'$ restricted to $V^{i_0+1}$. 
Thus, $U_g$ belongs to a normalizer of the restriction of $N'$ in $V^{i_0+1}$. 
The matrix of $g$ can be written as 
\begin{equation} \label{eqn:gamma}
\left( \begin{array}{cc} S_g          &   0 \\
                                    C_g        & U_g 
                                \end{array} \right)
 \end{equation}
where $U_g$ is an $(i_0+1)\times (i_0+1)$ normalizing matrix and $S_g$ is an $(n-i_0) \times (n-i_0)$ semisimple matrix 
and $C_g$ is an $(n-i_0) \times (i_0+1)$-matrix. 
We call the subgroup 
\[\left\{\frac{1}{|\det(S_g)|^{\frac{1}{n-i_0}}} S_g| g \in h(\pi_1(\tilde E)) \right\} \subset \SL_\pm(n-i_0, \bR)\] the semisimple part of 
$h(\pi_1(\tilde E))$ since it acts on a compact convex subset $K$ discretely and properly discontinuously on a properly convex domain $K^o$
with a compact quotient $K^o/h(\pi_1(\tilde E))$ and has to be semisimple by the main results of \cite{Ben1}.


\begin{example} \label{exmp:joined} 
Let us consider two ends $E_1$, a generalized lens-type radial one, 
with the p-end-neighborhood $U_1$ in the universal cover of a real projective orbifold $\mathcal{O}_1$ in $\SI^{n-i_0-1}$
of dimension $n-i_0-1$ with the p-end vertex $\bv_1$, 
and $E_2$  the p-end-neighborhood $U_2$ , a horospherical type one, in the universal cover
of a real projective orbifold $\mathcal{O}_2$ of dimension $i_0+1$
with the p-end vertex $\bv_2$. 
\begin{itemize} 
\item Let $\bGamma_1$ denote the projective automorphism group in $\Aut(\SI^{n-i_0-1})$ acting on $U_1$ 
corresponding to $E_1$. 
\item We assume that $\bGamma_1$ acts on a great sphere $\SI^{n-i_0-2} \subset \SI^{n-i_0-1}$ disjoint from $\bv_1$.
There exists a properly convex open domain $K'$ in $\SI^{n-i_0-2}$ where $\bGamma_1$ acts cocompactly
but not necessarily freely. 
We change $U_1$ to be the interior of the join of $K'$ and $\bv_1$. 
\item Let $\bGamma_2$ denote the one in $\Aut(\SI^{i_0+1})$ acting on $U_2$
unipotently and hence it is a cusp action. 
\item We embed $\SI^{n-i_0-1}$ and $\SI^{i_0+1}$ in $\SI^n$ meeting tranversally at $\bv = \bv_1 = \bv_2$. 
\item We embed $U_2$ in $\SI^{i_0+1}$ and $\bGamma_2$ in $\Aut(\SI^n)$ fixing each point of 
$\SI^{n-i_0-1}$. 
\item We can embed $U_1$ in $\SI^{n-i_0-1}$ and $\bGamma_1$ in $\Aut(\SI^n)$ acting 
on the embedded $U_1$ so that 
$\bGamma_1$ acts on $\SI^{i_0-1}$ normalizing $\bGamma_2$ and acting on $U_1$. 
One can find some such embeddings by finding an arbitrary 
homomorphism $\rho: \bGamma_1 \ra N(\bGamma_2)$ for a normalizer $N(\bGamma_2)$ of 
$\bGamma_2$ in $\Aut(\SI^n)$. 
\end{itemize}

We find an element $\zeta \in \Aut(\SI^n)$ fixing each point of $\SI^{n-i_0-2}$ and 
acting on $\SI^{i_0+1}$ as an unipotent element normalizing $\bGamma_2$ 
so that the corresponding matrix has only two norms of eigenvalues. 
Then $\zeta$ centralizes $\bGamma_1| \SI^{n-i_0-2}$ and normalizes $\bGamma_2$. 
Let $U$ be the join of $U_1$ and $U_2$, a properly convex domain. 
$U/\langle \bGamma_1, \bGamma_2, \zeta \rangle$ gives us a p-R-end of dimension $n$
diffeomorphic to $\Sigma_{E_1} \times \Sigma_{E_2} \times \SI^1 \times \bR$
and the transversal real projective manifold is diffeomorphic to $\Sigma_{E_1} \times \Sigma_{E_2} \times \SI^1$. 
We call the results the {\em joined} end and the joined end-neighborhoods. 
Those ends with end-neighborhoods finitely covered by these are also called 
{\em joined} end. The generated group $\langle \bGamma_1, \bGamma_2, \zeta \rangle$ is 
called a {\em joined group}. 

Now we generalize this construction slightly: 
Suppose that $\bGamma_1$ and $\bGamma_2$ are Lie groups and they have compact stabilizers at points of $U_1$ and $U_2$ respectively, 
and we have a parameter of $\zeta^t$ for $ t\in \bR$ centralizing $\bGamma_1| \SI^{n-i_0-2}$ and 
normalizing $\bGamma_2$ and restricting to a unipotent action on $\SI^{i_0}$ acting on $U_2$. 
The other conditions remain the same. We obtain 
a {\em joined homogeneous action} of the semisimple and cusp actions. 
Let $U$ be the properly convex open subset obtained as above as a join of 
$U_1$ and $U_2$.  Let $G$ denote the constructed Lie group by taking the embeddings of 
$\bGamma_1$ and $\bGamma_2$ as above. $G$ also has a compact stabilizers on $U$. 
Given a discrete cocompact subgroup of $G$, we obtained a p-end-neighborhood of a {\em joined p-end}
by taking the quotient of $U$. Those ends with end-neighborhoods finitely covered by such a one 
are also called a {\em joined end}. 

Later we will show this case cannot occur. 
We will generalize this construction to quasi-joined end. 
Here, $\Gamma_2$ is not required to act on $U_2$. 
\end{example}

\begin{example}\label{exmp:nonexmp}
Let $N$ be as in equation \ref{eqn:nilmat}.
In fact, we let $C_1 =0$ to simplify arguments and let $N$ be a nilpotent group in 
conjugate to $\SO(i_0+1, 1)$ acting on an $i_0$-dimensional ellipsoid in $\SI^{i_0+1}$. 

We find a closed properly convex real projective 
orbifold $\Sigma$ of dimension $n-i_0-2$ and find a homomorphism from $\pi_1(\Sigma)$
to a subgroup of $\Aut(\SI^{i_0+1})$ normalizing $N$ or even $N$ itself. 
(We will use  a trivial one to begin with. )
Using this, we obtain a group $\Gamma$ so that 
\[ 1 \ra N \ra \Gamma \ra \pi_1(\Sigma) \ra 1. \] 
Actually, we assume that this is ``split'', 
i.e., $\pi_1(\Sigma)$ acts trivially on $N$.

We now consider an example where $i_0 = 1$. 
Let $N$ be $1$-dimensional and be generated by $N_1$ as in Equation \ref{eqn:nilmat2}. 
\renewcommand{\arraystretch}{1.2}
\begin{equation} \label{eqn:nilmat3}
\newcommand*{\temp}{\multicolumn{1}{r|}{}}
N_1:= \left( \begin{array}{ccccc}         \Idd_{n-i_0-1} & 0 & \temp &  0 & 0 \\ 
                                                       \vec{0}           & 1  & \temp &  0  & 0\\
                                                       \cline{1-5}
                                                       \vec{0}         & 1 & \temp & 1   & 0 \\
                                                       \vec{0}       &\frac{1}{2} & \temp & 1  & 1 
                                    \end{array} \right)
 \end{equation}
 where $i_0 = 1$ and we set $C_1=0$.

We take a discrete faithful proximal 
representation $\tilde h: \pi_1(\Sigma) \ra \GL(n-i_0, \bR)$
acting on a convex cone $C_\Sigma$ in $\bR^{n-i_0}$.
We define $h: \pi_1(\Sigma) \ra \GL(n+1, \bR)$ by matrices
\begin{equation} \label{eqn:gammaJp}
h(g):= \left( \begin{array}{ccc}  \tilde h(g)         & 0                          & 0\\ 
                                                 \vec{d}_1(g)    & a_1(g)           & 0 \\
                                                 \vec{d}_2(g)    & c(g)  & \lambda_{\bv_{\tilde E}}(g)    
                                                               \end{array} \right)
\end{equation}
where $\vec{d}_1(g)$ and $\vec{d}_2(g)$ are $n-i_0$-vectors
and $g \mapsto \lambda_{\bv_{\tilde E}}(g)$ is a homomorphism as defined above
for the p-end vertex and $\det \tilde h(g) a_1(g) \lambda_{\bv_{\tilde E}}(g) = 1$.
 \begin{equation} \label{eqn:gammaJpi}
h(g^{-1}):= \left( \begin{array}{cc}  \tilde h(g)^{-1}         &  \begin{array}{cc} 0  & 0 \end{array}  \\ 
                             -   \left(   \begin{array}{cc} \frac{\vec{d}_1(g)}{a_1(g)} \\ \frac{-c(g)\vec{d}_1(g)}{a_1(g) \lambda_{\bv_{\tilde E}}(g)} 
                                + \frac{\vec{d}_2(g)}{\lambda_{\bv_{\tilde E}}(g)}
                                                                 \end{array}   \right) \tilde h(g)^{-1} 

                                          & \begin{array}{cc} \frac{1}{a_1(g)} & 0 \\ 
                                          \frac{-c(g)}{a_1(g) \lambda_{\bv_{\tilde E}}(g)} &  \frac{1}{\lambda_{\bv_{\tilde E}}(g)} \end{array} \\ 
                         
                                                               \end{array} \right).
\end{equation}
 
Then the conjugation of $N_1$ by $h(g)$ gives us 
 \begin{equation} \label{eqn:nilmat4}
 \left( \begin{array}{cc}         \Idd_{n-i_0}  & \begin{array}{cc} 0  & 0 \end{array} \\ 
                                                      \left(\begin{array}{cc} \vec{0} & a_1(g) \\ \vec{\ast} & \ast \end{array} \right) \tilde h(g)^{-1}&
                                                      \begin{array}{cc} 1   & 0 \\ \frac{\lambda_{\bv_{\tilde E}}(g)}{a_1(g)}  & 1 \end{array} 
                                                      \end{array}
                                                      \right).
 \end{equation}
 Our condition on the form of $N_1$ shows that 
 $(0,0,\dots,0,1)$ has to be a common eigenvector by $\tilde h(\pi_1(\tilde E))$
 and we also assume that $a_1(g) = \lambda_{\bv_{\tilde E}}(g)$ for the reasons to be justified later
 and the last row of $\tilde h(g)$ equals 
 $(\vec{0}, \lambda_{\bv_{\tilde E}}(g))$. Thus, the semisimple part of $h(\pi_1(\tilde E))$ is reducible. 

Some further computations show that 
we can take any $h: \pi_1(\tilde E) \ra \SL(n-i_0, \bR)$ with matrices of form 
\renewcommand{\arraystretch}{1.2}
\begin{equation} \label{eqn:nilmat5}
\newcommand*{\temp}{\multicolumn{1}{r|}{}}
h(g):= \left( \begin{array}{ccccc}   S_{n-i_0-1}(g) & 0 & \temp &  0 & 0 \\ 
                                                       \vec{0}           & \lambda_{\bv_{\tilde E}}(g)  &   \temp &  0  & 0\\
                                                        \cline{1-5}
                                                       \vec{0}         & 0 & \temp & \lambda_{\bv_{\tilde E}}(g)   & 0 \\
                                                       \vec{0}          & 0 & \temp & 0  & \lambda_{\bv_{\tilde E}}(g) 
                                    \end{array} \right)
 \end{equation}
 for $g \in \pi_1(\tilde E) - N$ by a choice of coordinates by the semisimple property of the $(n-i_0) \times (n-i_0)$-upper left part of $h(g)$. 
 (Of course, these are not all example we wish to consider but we will modify later to quasi-joined ends.)
 
 Since $\tilde h(\pi_1(\tilde E))$ has a common eigenvector, 
 Benoist \cite{Ben2} shows that the open convex domain $K$ 
 that is the image of $\Pi_K$ in this case is decomposable and 
 $N_K = N'_K \times \bZ$ for another subgroup $N'_1$ 
and the image of the homomorphism $g \in N'_K \ra S_{n-i_0-1}(g)$ can be assumed to give 
a discrete projective automorphism group acting properly discontinuously on a properly convex subset 
$K'$ in $\SI^{n-i_0-2}$ with a compact quotient. 
Furthermore the generator of $\bZ$ is central. 

Let $\mathcal{E}$ be the one-dimensional ellipsoid where lower right $3\times 3$-matrix of $N_K$ acts on. 
From this, the end is of the join form 
$K^{\prime o}/N'_K \times \SI^1 \times {\mathcal{E}}/\bZ$ by taking a double cover if necessary 
and $\pi_1(\tilde E)$ is isomorphic to $N'_K \times \bZ \times \bZ$
up to taking an index two subgroups. 
(In this case, $N_K$ centralizes $\bZ \subset N'_K$
and the second $\bZ$ is in the centralizer of $\Gamma$. )

We can think of this as the join of $K^{\prime o}/N'_K$ with $\mathcal{E}/\bZ$ as
$K'$ and $\mathcal{E}$ are on disjoint complementary projective spaces
of respective dimensions $n-3$ and $2$ to be denoted $S(K')$ and $S(\mathcal{E})$ respectively. 


\end{example}


\subsection{Some preliminary results } 

For results in this subsection, we do not use a discreteness assumption on the semisimple quotient group $N_K$. 

\subsubsection{The standard quadric in $\bR^{i_0+1}$ and the group acting on it.} \label{subsub:quadric}

Let us consider an affine subspace $A^{i_0+1}$ of $\SI^{i_0+1}$ with 
coordinates $x_0, x_1, \dots, x_{i_0+1}$ given by $x_0 > 0$.
The standard quadric in $A^{i_0+1}$ is given by 
\[ x_{i_0+1} = x_1^2 + \cdots + x_{i_0}^2. \] 
Clearly the group of the orthogonal maps $O(i_0)$ acting on the planes given by $x_{i_0+1} = const$ acts on
the quadric also. Also, the matrices of the form 
\[
\left(
\begin{array}{ccc}
 1                                 & 0                    & 0  \\
 \vec{v}^T                    & \Idd_{i_0}        & 0   \\
\frac{||\vec{v}||^2}{2}  & \vec{v}   & 1  
\end{array}
\right)
\]
induce and preserve the quadric. 
They are called the standard cusp group. 

The group of affine transformations that acts on the quadric is exactly the Lie group
generated by the cusp group and $O(i_0)$. The action is transitive and each of the stabilizer is a conjugate of
$O(i_0)$ by elements of the cusp group. 

The proof of this fact is simply that the such an affine transformation is conjugate to an element a parabolic group in
the $i_0+1$-dimensional complete hyperbolic space $H$ where the ideal fixed point is identified with 
$[0, \dots, 0, 1] \in \SI^{i_0+1}$ and with $\Bd H$ tangent to $\Bd A^{i_0}$.

\subsubsection{Technical lemmas}


Using coordinates as in the previous subsection, 
recall $\CN_j$ from Equation \ref{eqn:nilmat2}.
For $\vec{v} \in \bR^{i_0}$, we define 
\renewcommand{\arraystretch}{1.2}
\begin{equation} \label{eqn:nilmatstd}
\newcommand*{\temp}{\multicolumn{1}{r|}{}}
\CN(\vec{v}):= \left( \begin{array}{ccccccc}         \Idd_{n-i_0-1} & 0 &\temp &  0 & 0& \dots & 0 \\ 
                                                       \vec{0}           & 1  &\temp &  0  & 0&  \dots & 0\\
                                                       \cline{1-7}
                                                       \vec{c}_1(\vec{v})    & {v}_1 &\temp & 1   & 0 &  \dots & 0 \\
                                                       \vec{c}_2(\vec{v})   & {v}_2 &\temp & 0   & 1 & \dots & 0\\
                                                       \vdots   & \vdots &\temp & \vdots & \vdots & \ddots & \vdots \\
                                                       \vec{c}_{i_0+1}(\vec{v})  & \frac{1}{2}||\vec{v}||^2& \temp & {v}_1 & v_2 & \dots  & 1
                                                        
                                    \end{array} \right)
 \end{equation}
where $||v||$ is the norm of $\vec{v} = (v_1, \cdots, v_i) \in \bR^{i_0}$. 
The elements of our nilpotent group $\CN$ are of this form since $\CN(\vec{v})$ is the product  
$\prod_{j=1}^{i_0} \CN(e_j)^{v_j}$.  
By the way we defined this, 
for each $k$, $k=1, \dots, i_0$, $\vec{c}_k:\bR^{i_0} \ra \bR^{n-i_0-1}$ are  linear functions of $\vec{v}$ 
defined as $\vec{c}_k(\vec{v}) = \sum_{j=1}^{i_0} \vec c_{kj} v_j$ for $\vec{v} = (v_1, v_2, \dots, v_{i_0})$
so that we form a group. (We do not need the property of $\vec{c}_{i_0+1}$ at the moment.)

We denote by $C_1(\vec{v})$ the $(n-i_0-1) \times i_0$-matrix given by 
the matrix with rows $\vec{c}_j(\vec{v})$ for $j= 1, \dots, i_0$
and by $c_2(\vec{v})$ the row $(n-i_0-1)$-vector $\vec{c}_{i_0+1}(\vec{v})$. 
The lower-right $(i_0+2)\times (i_0+2)$-matrix is form is called the {\em standard cusp matrix form}.

\begin{hypothesis} \label{h:norm}
The assumptions for this subsection are as follows: 
\begin{itemize} 
\item We assume that $\CN$ acts on an NPCC
p-R-end-neighborhood of $\bv_{\tilde E}$ 
and acts on each hemisphere with boundary $\SI^i_{\infty}$, and fixes
$\bv_{\tilde E} \in \SI^i_\infty$. 
\item The p-end fundamental group $\bGamma_{\tilde E}$ normalizes 
$\CN$. 
\end{itemize}
\end{hypothesis}

Since $\SI^{i_0}_\infty$ is invariant, 
$g$, $g\in \Gamma$, is of {\em standard} form 
\renewcommand{\arraystretch}{1.2}
\begin{equation}\label{eqn:matstd}
\newcommand*{\temp}{\multicolumn{1}{r|}{}}
\left( \begin{array}{ccccccc} 
S(g) & \temp & s_1(g) & \temp & 0 & \temp & 0 \\ 
 \cline{1-7}
s_2(g) &\temp & a_1(g) &\temp & 0 &\temp & 0 \\ 
 \cline{1-7}
C_1(g) &\temp & a_4(g) &\temp & A_5(g) &\temp & a_6(g) \\ 
 \cline{1-7}
c_2(g) &\temp & a_7(g) &\temp & a_8(g) &\temp & a_9(g) 
\end{array} 
\right)
\end{equation}
where $S(g)$ is an $(n-i_0-1)\times (n-i_0-1)$-matrix
and $s_1(g)$ is an $(n-i_0-1)$-column vector, 
$s_2(g)$ and $c_2(g)$ are $(n-i_0-1)$-row vectors, 
$C_1(g)$ is an $i_0\times (n-i_0-1)$-matrix, 
$a_4(g)$ and $a_6(g)$ are  $i_0$-column vectors, 
$A_5(g)$ is an $i_0\times i_0$-matrix, 
$a_8(g)$ is an $i_0$-row vector, 
and $a_1(g), a_7(g)$, and $ a_9(g)$ are scalars. 
(We show $a_6(g) = 0$ for any standard form of $g$ soon.)

Denote 
\[ \hat S(g) =
\left( \begin{array}{cc} 
S(g) & s_1(g)\\ 
s_2(g) & a_1(g)
\end{array} \right). \]

\begin{lemma}\label{lem:bdhoro} 
Assume Hypothesis \ref{h:norm}. 
Let $l'$ in $\Bd K$ be a fixed point $l'$ of $\hat S(g)$ for an infinite order $g$, $g \in \bGamma_{\tilde E}$. 
Suppose that the leaf $l'$ corresponds to a hemisphere $H^{i_0+1}_{l'}$ where we assume
\[(\SI^{i_0+1}_{l'} - \SI^{i_0+1}_\infty) \cap \clo(U) \ne \emp.\] 
Then $\CN$ acts on the open ball $U_{l'}$ in $\clo(U)$ bounded by an ellipsoid in 
a component of $H^{i_0+1}_{l'} - \SI^{i_0+1}_\infty$. 
\end{lemma}
\begin{proof} 
The existence of the hemisphere is clear since $\SI^{n-i_0-1}$ is considered 
the space of $(i_0+1)$-dimensional hemispheres. 

Only one component
$\SI^{i_0+1}_{l'} - \SI^{i_0+1}_\infty$ meets $\clo(U)$ since otherwise, 
$K$ contains a pair of antipodal points. 
Since $g$ is semisimple in $\SI^{n-i_0-1}$ and the fixed point has a distinct eigenvalue 
from those of $\SI^{i_0+1}$, we have a corresponding fixed point in
a component above.

Let $A_{l'}$ denote this component. Let $J_{l'}:= A_{l'} \cap \clo(U) \ne \emp$. 
Each $g \in \CN$ then has the form in $H_{l'}^{i_0+1}$ as 
\[
\left(
\begin{array}{ccc}
1              & 0         &  0 \\
L(\vec{v}^T)  & \Idd_{i_0}  & 0  \\
\kappa(\vec{v})  & \vec{v}  & 1  
\end{array}
\right)
\]
since the $\SI^{i_0}_\infty$-part, i.e., the last $i_0+1$ coordinates, is not changed from one for 
equation \ref{eqn:nilmatstd}
where $L: \bR^{i_0} \ra \bR^{i_0}$ is a linear map. The linearity of $L$ is the consequence of the group property. 
$\kappa: \bR^{i_0} \ra \bR$ is some function. We consider $L$ as an $i_0\times i_0$-matrix. 

If there exists a kernel $K_1$ of $L$, then we use $t\vec{v} \in K_1-\{O\}$ and 
as $t \ra \infty$, we can show that $\CN(J_{l'})$ cannot be properly convex. 

Also, since $\CN$ is abelian, the computations shows that $\vec v L \vec w^T = \vec w L \vec v^T$
for every pair of vectors $\vec v$ and $\vec w$ in $\bR^{i_0}$. 
Thus, $L$ is a symmetric matrix. 

We may obtain new coordinates $x_{n-i_0+1}, \dots, x_n$ by taking 
linear combinations of these. 
Since $L$ hence is nonsingular, we can find 
new coordinates $x_{n-i_0+1}, \dots, x_n$ so that $\CN$ is now of standard form: 
We conjugate $\CN$ by 
\[
\left(
\begin{array}{ccc}
1              & 0         &  0 \\
0 & A  & 0  \\
0 &  0 & 1  
\end{array}
\right)
\]
for nonsingular $A$. 
We obtain
\[
\left(
\begin{array}{ccc}
1              & 0         &  0 \\
AL\vec{v}^T  & \Idd_{i_0}  & 0  \\
\kappa(\vec{v})  & \vec{v} A^{-1}  & 1  
\end{array}
\right).
\]
We thus need to solve for $A^{-1} A^{-1 T} = L$, which can be done. 

We can factorize each element of $\CN$ into forms 
\[
\left(
\begin{array}{ccc}
1              & 0         &  0 \\
0  & \Idd_{i_0}  & 0  \\
\kappa(\vec{v}) - \frac{||\vec{v}||^2}{2}  & 0  & 1  
\end{array}
\right)
\left(
\begin{array}{ccc}
1              & 0         &  0 \\
\vec{v}^T  & \Idd_{i_0}  & 0  \\
 \frac{||\vec{v}||^2}{2}   & \vec{v}  & 1  
\end{array}
\right).
\]
Again, by the group property, $\alpha_7(\vec{v}) : = \kappa(\vec{v}) - \frac{||\vec{v}||^2}{2}$ 
gives us a linear function $\alpha_7: \bR^{i_0} \ra \bR$. 
Hence $\alpha_7(\vec{v}) = \kappa_\alpha \cdot \vec{v}$
for $\kappa_\alpha \in \bR^{i_0}$. 
Now, we conjugate $\CN$ by 
the matrix 
\[
\left(
\begin{array}{ccc}
 1 & 0  & 0  \\
 0 & \Idd_{i_0}  & 0  \\
0  & -\kappa_\alpha  & 1  
\end{array}
\right)
\]
and this will put $\CN$ into the standard form. 

Now it is clear that the orbit of $\CN(x_0)$ for a point $x_0$ of $J_{l'}$ is an ellipsoid with a point removed. 
as we can conjugate so that the first column entries from the second one to the $(i_0+1)$-th one equals those of the last row.
Since $\clo(U)$ is $\CN$-invariant, we obtain that $\CN(x_0) \subset J_{l'}$.

\end{proof} 


\begin{lemma}\label{lem:eigSI2} 
For $g \in \bGamma_{\tilde E} - \CN$ going to an infinite order element in $N_K$, 
we have $\lambda_1(g) \geq \lambda(g)$ and 
$\SI^{i_0+1}_{l'} - \SI^{i_0}_\infty$ contains an accumulation point of 
 $\{g^j(x)\}_{j \in \bZ}$  for $x \in U$, and 
$l'$ is a leaf corresponding  to an attracting or repelling fixed point of $\Pi^*_K(g)$.
Thus, we have \[(\SI^{i_0+1}_{l'} - \SI^{i_0}_\infty) \cap \clo(U) \ne \emp.\]
\end{lemma}
\begin{proof}
By Proposition \ref{prop:eigSI},  we have $\lambda_1(g) \geq \lambda(g)$. 
Suppose that $\lambda_1(g) = \lambda(g)$. 
We can assume that $\lambda(g)$ is distinct from $\lambda_{\bv_{\tilde E}}$ by 
the weakly uniform middle-eigenvalue condition and taking $g^{-1}$ instead of $g$ if necessary.  
Let $l'$ denote the fixed point of $g$ in $K$ and hence the corresponding 
sphere $\SI^{i_0+1}_{l'}$ exists in $\SI^n$. 

Suppose that there is no point in 
$\SI^{i_0+1}_{l'} - \SI^{i_0}_\infty$ 
where $\{g^j(x)| j \geq 0 \}$ accumulates. 
Then we obtain a fixed point $y$ of 
$g$ in $\SI^{i_0}_\infty$ corresponding to $\lambda(g)$ where
$\{g^j(x)| j \geq 0 \}$ accumulates to $y \in \SI^{i_0}_\infty$ for $x \in U$ and 
$y$ is distinct from $\bv_{\tilde E}$ as $\lambda_1(g) > \lambda_{\bv_{\tilde E}}(g)$ 
by the weak uniform middle-eigenvalue condition. 
The convex hull of $\CN(y)$ in $\clo(\torb)$ is not properly convex in $\SI^{i_0}_\infty$ as we can prove
from the matrix form of $\CN| \SI^{i_0}_\infty$ of form of equation \ref{eqn:nilmatstd}
and some computations as 
$y \ne \bv_{\tilde E}, \bv_{\tilde E -}$.  
Since this is a subset of $\clo(U)$, we obtained a contradiction to the proper convexity of $\clo(\torb)$. 

Hence, there is a point $y$ of  $\SI^{i_0+1}_{l'} - \SI^{i_0}_\infty$  where $\{g^j(x)| j \geq 0 \}$ accumulates. 

Suppose now that $\lambda_1(g) > \lambda(g)$. 
In this case, $\{g^j(x)| j \geq 0 \}$, $x \in U$, 
accumulates to a point $y$ in $\SI^{i_0+1}_{l'} - \SI^{i_0}_\infty$. 
\end{proof}

Let $a_5(g)$ denote $| \det(A^5_g)|^{\frac{1}{i_0}}$.
Define $\mu_g:= \frac{a_5(g)}{a_1(g)} = \frac{a_9(g)}{a_5(g)}$ for $g \in \bGamma_{\tilde E}$.

\begin{lemma} \label{lem:similarity}
Assume Hypothesis \ref{h:norm}. 
Let $\CN$ be an $i_0$-dimensional nilpotent Lie group with elements of form of equation \ref{eqn:nilmatstd}
and $K^o/N_K$ compact. 
Then 
any element $g \in \bGamma_{\tilde E}$ of form of 
the equation \ref{eqn:matstd} normalizing $\CN$ 
and acting on $\SI^{i_0}_\infty$ induces an $(i_0\times i_0)$-matrix $M_g$ given by
\[g \CN(\vec{v}) g^{-1} = \CN(\vec{v}M_g) \hbox{ where } \] 
\[M_g = \frac{1}{a_1(g)} (A_5(g))^{-1} = \mu_g O_5(g)^{-1} \]
for $O_5(g)$ in a compact Lie group $G_{\tilde E}$, and 
the following hold. 
\begin{itemize} 
\item $(a_5(g))^2 = a_1(g) a_9(g)$ or equivalently $\frac{a_5(g)}{a_1(g)}= \frac{a_1(g)}{a_5(g)}$.
\item Finally, $a_1(g), a_5(g),$ and $a_9(g)$ are all nonzero and $a_6(g) =0$.
\end{itemize}  
This is true as long as $g$ is in the standard form. 
\end{lemma} 
\begin{proof}
Since the conjugation by $g$ sends elements of $\CN$ to itself in a one-to-one manner, 
the correspondence between the set of $\vec{v}$ for $\CN$ and $\vec{v'}$ is one-to-one.

Since we have $g \CN({\vec{v}}) = \CN({\vec{v}'}) g$ for vectors $\vec{v}$ and $\vec{v'}$ 
in $\bR^{i_0}$ by Hypothesis \ref{h:norm},
we consider
\begin{equation} \label{eqn:firstm}
\newcommand*{\temp}{\multicolumn{1}{r|}{}}
\left( \begin{array}{ccccccc} 
S(g) & \temp & s_1(g) & \temp & 0 & \temp & 0 \\ 
 \cline{1-7}
s_2(g) &\temp & a_1(g) &\temp & 0 &\temp & 0 \\ 
 \cline{1-7}
C_1(g) &\temp & a_4(g) &\temp & A_5(g) &\temp & a_6(g) \\ 
 \cline{1-7}
c_2(g) &\temp & a_7(g) &\temp & a_8(g) &\temp & a_9(g) 
\end{array} 
\right)
\left( \begin{array}{ccccccc} 
\Idd_{n-i_0-1} & \temp & 0 & \temp & 0 & \temp & 0 \\ 
 \cline{1-7}
0 &\temp & 1 &\temp & 0 &\temp & 0 \\ 
 \cline{1-7}
C_1({\vec{v}}) &\temp & \vec{v}^T &\temp & \Idd_{i_0} &\temp & 0 \\ 
 \cline{1-7}
c_2({\vec{v}}) &\temp & \frac{||\vec{v}||^2}{2} &\temp & \vec{v} &\temp & 1 
\end{array} 
\right)
\end{equation}
where $C_1({\vec{v}})$ is an $(n-i_0-1)\times i_0$-matrix where
each row is a linear function of $\vec{v}$, 
$c_2({\vec{v}})$ is a $(n-i_0-1)$-row vector,
$\vec{v}$ is an $i_0$-row vector, and $s$ is a scalar. 
This must equal the following matrix
for some $\vec{v'}\in \bR$
\begin{equation} \label{eqn:secondm}
\newcommand*{\temp}{\multicolumn{1}{r|}{}}
\left( \begin{array}{ccccccc} 
\Idd_{n-i_0-1} & \temp & 0 & \temp & 0 & \temp & 0 \\ 
 \cline{1-7}
0 &\temp & 1 &\temp & 0 &\temp & 0 \\ 
 \cline{1-7}
C_1({\vec{v'}}) &\temp & \vec{v'}^T &\temp & \Idd_{i_0} &\temp & 0 \\ 
 \cline{1-7}
c_2({\vec{v'}}) &\temp & \frac{||\vec{v'}||^2 }{2} &\temp & \vec{v'} &\temp & 1 
\end{array} 
\right)
\left( \begin{array}{ccccccc} 
S(g) & \temp & s_1(g) & \temp & 0 & \temp & 0 \\ 
 \cline{1-7}
s_2(g) &\temp & a_1(g) &\temp & 0 &\temp & 0 \\ 
 \cline{1-7}
C_1(g) &\temp & a_4(g) &\temp & A_5(g) &\temp & a_6(g) \\ 
 \cline{1-7}
c_2(g) &\temp & a_7(g) &\temp & a_8(g) &\temp & a_9(g) 
\end{array} 
\right).
\end{equation}
From equation \ref{eqn:firstm}, 
we compute the $(4, 3)$-block of the result 
to be $a_8(g) + a_9(g) \vec{v}$. 
From Equation \ref{eqn:secondm}, 
the $(4, 3)$-block is
$\vec{v'} A_5(g) + a_8(g)$. We obtain the relation
$a_9(g) \vec{v} = \vec{v'} A_5(g)$ for every $\vec{v}$. 
Since the correspondence between $\vec{v}$ and $\vec{v'}$ is one-to-one, 
we obtain 
\begin{equation}\label{eqn:vp1}
\vec{v'} = a_9(g) \vec{v} (A_5(g))^{-1}
\end{equation}
for the $i_0\times i_0$-matrix $A_5(g)$ and we also infer $a_9(g) \ne 0$
and $\det(A_5(g)) \ne 0$. 
The $(3, 2)$-block of the result of Equation \ref{eqn:firstm} 
equals 
\[a_4(g) + A_5(g) \vec{v}^T + \frac{1}{2} ||\vec{v}||^2 a_6(g).\]  
The $(3, 2)$-block of the result of Equation \ref{eqn:secondm} 
equals
\begin{equation}
C_1({\vec{v}'}) s_1(g) + a_1(g) \vec{v}^{\prime T} + a_4(g). 
\end{equation}
Since the $||v||^2$-term is only term that is quadratic, we obtain $a_6(g)=0$ 
and 
\begin{equation} \label{eqn:sim0}
A_5(g) \vec{v}^T =  C_1({\vec{v}'}) s_1(g)  + a_1(g) \vec{v}^{\prime T}.
\end{equation} 
We obtain 
\begin{equation} \label{eqn:a6}
a_6(g) = 0
\end{equation} 
since $||\vec{v}||^2$ is the only quadratic term. 
We will assume this from now on. 


For each $g$, we can choose a coordinate system so that $s_1(g) = 0, a_6(g)=0$ as $\hat S(g)$ is semisimple, 
which involves the coordinate changes of the first $n-i_0$ coordinate functions only. 
We may also assume that $U$ satisfies $x_{n-i_0} > 0$ since $U$ is convex.



Since $\CN$ acts on $\SI^{i_0+1}_{l'}$ for some leaf $l'$ as a cusp group by Lemma \ref{lem:bdhoro},
there exists a coordinate change involving the last $(i_0+1)$-coordinates 
$x_{n-i_0+1}, \dots, x_n, x_{n+1}$
so that the matrix form of the lower-right $(i_0+2)\times(i_0+2)$-matrix of each element $\CN$ is of the standard cusp form.
This will not affect $s_1(g) = 0, a_6(g) =0$ as we can check from the proof of Lemma \ref{lem:bdhoro}. 
Denote this coordinate system by $\Phi_{g, l'}$. 

Let us use $\Phi_{g, l'}$ for a while
using primes for new set of coordinates functions. 
Now $A'_5(g)$ is conjugate to $A_5(g)$ as we can check in the proof of Lemma \ref{lem:bdhoro}. 
Under this coordinate system for given $g$, 
we obtain $a'_1(g) \ne 0$ and we can recompute to show that 
$a'_9(g) \vec{v} = \vec{v'} A'_5(g)$ for every $\vec{v}$ as in equation \ref{eqn:vp1}. 
By equation \ref{eqn:sim0} for this case, we obtain 
\begin{equation}\label{eqn:vp2}
\vec{v'} = \frac{1}{a'_1(g)} \vec{v} (A'_5(g))^T
 \end{equation}
 as $s'_1(g) = 0$ here since we are using the coordinate system $\Phi_{g, l'}$.
 Since this holds for every $\vec{v} \in \bR^{i_0}$, 
 we obtain 
 \[a'_9(g) (A'_5(g))^{-1} = \frac{1}{a'_1(g)} (A'_5(g))^T.\] 
 Hence $\frac{1}{ |\det(A'_5(g))|^{1/i_0}} A'_5(g) \in O(i_0)$. 
 Also, \[\frac{a'_9(g)}{a'_5(g)} = \frac{a'_5(g)}{a'_1(g)}.\] 
 Here, $A'_5(g)$ is a conjugate of the original matrix $A_5(g)$ by linear coordinate changes 
 as we can see from the above processes to obtain the new coordinate system.
 (Check the proof of Lemma \ref{lem:bdhoro} and that the coordinate changes involved for 
 getting $s_1(g)=0$ do not change $A_5(g)$.)  
 
 This implies that the original matrix $A_5(g)$ is conjugate to an orthogonal matrix multiplied by a positive scalar for every $g$. 
The set of matrices $ \{ A_5(g)| g \in \bGamma_{\tilde E}\}$ forms a group since every $g$ is of a standard matrix form 
(see equation \ref{eqn:matstd}) where $a_6(g) = 0$ for every $g$. 
  Given such a group of matrices normalized to have determinant $\pm 1$, we obtain 
 a compact group 
 \[G_{\tilde E}:= \left\{ \frac{1}{|\det A_5(g)|^{\frac{1}{i_0}}} A_5(g) \left| g \in \Gamma_{\tilde E}\right. \right\}\]
  by Lemma \ref{lem:cpt}. 
This group has a coordinate where every element is orthogonal by coordinate changes
of coordinates $x_{n-i_0+1}, \dots, x_n$. 

 
\end{proof} 


\begin{lemma}\label{lem:cpt} 
Suppose that $G$ is a closed subgroup of a linear group $\GL(i_0, \bR)$ where each 
element is conjugate to an orthogonal element. Then $G$ is a compact group. 
\end{lemma}
\begin{proof} 
Clearly, the norms of eigenvalues of $g \in G$ are all $1$. 
$G$ is virtually an orthopotent group by \cite{CG2} or \cite{Moore}. 
Since the group is linear and for each element $g$, 
$\{g^n| n \in \bZ\}$ is a bounded collection of matrices, 
$G$ is a subgroup of an orthogonal group under a coordinate system. 
\end{proof}


We denote by $(C_1({\vec{v}}), \vec{v}^T)$ the matrix obtained 
from $C_1(\vec{v})$ by adding a column vector $\vec{v}^T$.


\begin{lemma} \label{lem:conedecomp1}
Assume as in Lemma \ref{lem:similarity}. Then the following hold{\rm :} 
\begin{itemize} 
\item $K$ is a cone over a totally geodesic 
$(n-i_0-2)$-dimensional domain $K''$. 
\item The rows of $(C_1({\vec{v}}), \vec{v}^T)$ are proportional to a single vector and 
we can find a coordinate system where $C_1({\vec v}) = 0$
not changing any entries of the lower-right $(i_0+2)\times (i_0+2)$-submatrices for all $\vec v \in \bR^{i_0}$. 
\item We can find a common coordinate system where 
\begin{equation}\label{eqn:O5coor}
O_5(g)^{-1} = O_5(g)^T, O_5(g) \in O(i_0), s_1(g) = s_2(g) = 0 \hbox{ for all } g.
\end{equation} 
\item In this coordinate system, we have
\begin{equation} \label{eqn:conedecomp1}
 a_9(g) c_2({\vec{v}})  
 = c_2({\mu_g\vec{v}O_5(g)^{-1}}) S(g)  + \mu_g \vec{v} O_5(g)^{-1} C_1(g).
\end{equation} 


\end{itemize}
\end{lemma}
\begin{proof} 

The assumption implies that $M_g =  \mu_g O_5(g)^{-1}$ by Lemma \ref{lem:similarity}.
We consider the equation 
\begin{equation} \label{eqn:conj0}
g \CN(\vec{v}) g^{-1} = \CN(\mu_g \vec{v} O_5(g)^{-1}).
\end{equation}


For the second, we consider 
\[g \CN(\vec v) = \CN(\mu_g \vec v O_5(g)) g\]
and consider the lower left $(n-i_0) \times (i_0+1)$-matrix of the left side,  
we obtain 
\begin{equation} 
\left(\begin{array}{cc} 
C_1(g) & a_4(g) \\ c_2(g) & a_7(g) \end{array} \right) 
+ 
\left( \begin{array}{cc} 
a_5(g) O_5(g) C_1({\vec{v}}) & a_5(g) O_5(g) \vec{v}\\ 
a_8(g) C_1({\vec{v}}) + a_9 c_2({\vec{v}})  
& a_8(g) {\cdot} \vec{v}^T + a_9(g) \vec{v}\cdot\vec{v}/2
\end{array} \right) \end{equation}
where the entry sizes are clear. 
From the right side, we obtain
\begin{equation} 
\left( \begin{array}{cc} 
C_1({\mu_g \vec{v} O_5(g)^{-1}})& \mu_g O_5(g)^{-1, T}\vec{v}^T \\ 
c_2({\mu_g\vec{v}}O_5(g)^{-1})     & \vec{v}\cdot\vec{v}/2
\end{array} \right) \hat S(g)
+ \left(\begin{array}{cc} 
C_1(g) & a_4(g) \\ \vec{v'} \cdot C_1(g) + c_2(g) & a_7(g) + \vec{v'} \cdot a_4(g) 
\end{array} \right).
\end{equation}
From the top row, we obtain that 
\begin{alignat}{1}\label{eqn:C1} 
& (a_5(g) O_5(g) C_1({\vec{v}}) , a_5(g) O_5(g) \vec{v}^T) 
=  (\mu_g C_1({\vec{v}}  O_5(g)^{-1} ),  \mu_g O_5(g)^{-1, T} \vec{v}^T) \hat S(g). \\ 
&(a_5(g) C_1({\vec{v}}) , a_5(g) \vec{v}^T) \hat S(g^{-1})
=  (\mu_g O_5(g)^{-1} C_1({\vec{v}}  O_5(g)^{-1} ),  \mu_g O_5(g)^{-1}O_5(g)^{-1, T} \vec{v}^T) 
\end{alignat}
since $C_1$ is linear where we multiplied the both sides by $O_5(g)^{-1}$. 
Let us form the subspace $V_C$ in the dual sphere $\bR^{n-i_0 \ast}$ spanned by 
row vectors of $(C_1({\vec{v}}), \vec{v}^T)$. Let $\SI_C^\ast$ denote the corresponding subspace 
in $\SI^{n-i_0-1 \ast}$. Then 
\[\left\{\frac{1}{\det \hat S(g)^{\frac{1}{n-i_0-1}}}\hat S(g) \vert g \in \bGamma_{\tilde E}\right\}\] 
acts on $V_C$ as a group of bounded 
linear automorphisms since $O_5(g) \in G$ for a compact group $G$, 
and hence $\{\hat S(g)| g\in \bGamma_{\tilde E}\}$ on $\SI_C^\ast$ is in a compact group of 
projective automorphisms. 

We recall that the dual group $N_K^*$ of $N_K$ acts on the properly convex dual domain $K^*$ of $K$ by Theorem \ref{thm:dualdiff}.
Then $g$ acts as an element of a compact group on $\SI_C^\ast$. Thus, $N_K^*$ is reducible. 

We claim that $\dim(\SI_C) = 0$. 
Let $\SI_M$ be the maximal invariant subspace where each $g\in N_K^*$ acts orthogonally containing $\SI_C$.  
Since $N_K^*$ is semisimple by Benoist \cite{Ben3}, 
$N_K^*$  acts on 
a complementary subspace $\SI_C$. 
By the theory of Benoist \cite{Ben3}, $K^*$ has an invariant subspace $K^*_1$ and $K^*_2$ 
so that $K^* = K^*_1 \ast K^*_2$, a strict join
so that $\dim K^*_1 = \dim \SI_M$ and $\dim K^*_2 = \dim \SI_C$. 
We assume that $K^*_2 = K^* \cap \SI_M$ and $K^*_2 \cap \SI_C$. 
Also, by the theory of Benoist \cite{Ben3}, $N_K^*$ is isomorphic to 
$N_{K, 1} \times N_{K, 2} \times A$ where $A$ is a subgroup of $\bR$ 
and $N_{K, i}$ acts on a properly convex domain that is
the interior of $K^*_i$ properly and cocompactly for $i=1,2$.
But since $N_{K, 1}$ acts orthogonally on $\SI_M$, 
the only possibility is that $\dim \SI_M = 0$.  

Therefore this shows that the rows of $(C_1({\vec{v}}), \vec{v}^T)$ are proportional to a single row vector.

Since $(C_1({\vec{e}_j}), \vec{e}_j^T)$ has $0$ is the last column except 
for the $j$th one, only the $j$th row is nonzero
and moreover, it equals to a scalar multiple of a common vector 
$(C_1({1, \vec{e}_1}), 1)$ for every $j$ where $C_1({1, \vec{e}_1})$ is the first 
row of $C_1({\vec{e}_1})$. 
Now we can choose coordinates of $\bR^{n-i_0 \ast}$ so that 
this row vector now has a coordinate $(0, \dots, 0, 1)$. 
We can also choose so that $K^*_1$ is given by setting the last coordinate be zero.  
With this change, we need to do conjugation by matrices 
with the top left $(n-i_0-1)\times (n-i_0-1)$-submatrix being different from $\Idd$ and 
the rest of the entries staying the same. 
This will not affect the expressions of matrices of 
 lower right $(i_0+2)\times (i_0+2)$-matrices involved here. 
Thus, $C_1({\vec{v}}) =0$ in this coordinate for all $\vec{v} \in \bR^{i_0}$ and $g \in \bGamma_{\tilde E} - N$. 

And the in the above coordinate system, we obtain
$s_1(g) =0, s_2(g) =0$ 
and that $K$ is a strict join of a point 
\[k =\overbrace{ [0, \dots, 0, 1]}^{n-i_0}\] 
and a domain $K''$ given by
setting $x_{n-i_0} = 0$ in a totally geodesic 
sphere of dimension $n-i_0-2$ by duality. 

For the final item we have: 
\begin{equation} \label{eqn:form1}
g = \left( \begin{array}{cccc} 
 S(g) & 0 & 0 & 0 \\ 
 0 & a_1(g) & 0 & 0 \\ 
C_1(g) & a_4(g) & a_5(g) O_5(g) & 0  \\
c_2(g) & a_7(g) & a_8(g) & a_9(g) 
\end{array} \right), \quad
 \CN(\vec{v}) = \left( \begin{array}{cccc} 
 \Idd & 0 & 0 & 0 \\ 
 0 &    1 & 0 & 0 \\ 
0 & \vec{v}^T & \Idd & 0 \\ 
 c_2({\vec{v}}) & \frac{1}{2} ||\vec{v}||^2 & \vec{v} & 1 
 \end{array} \right). 
 \end{equation} 
 The normalization of $\CN$ shows as in the proof of Lemma \ref{lem:similarity} that $O_5(g)$ is orthogonal now.
 (See equations  \ref{eqn:vp1} and  \ref{eqn:sim0}.)
 We consider the lower-right $(i_0+1)\times (n-i_0)$-submatrices of 
 $g\CN(\vec{v})$ and $\CN(\vec{v'})g$. 
 For the first one, we obtain 
 \[
 \left( \begin{array}{cc} C_1(g) & a_4(g) \\ c_2(g) & a_7(g) \end{array} \right ) 
 + \left( \begin{array}{cc} a_5(g) O_5(g) & 0 \\ a_8(g) & a_9(g) \end{array} \right) 
 \left( \begin{array}{cc} 0 & \vec{v}^T \\ c_2({\vec{v}}) &  \frac{1}{2} ||\vec{v}||^2 \end{array} \right) 
 \]
 For $\CN(\vec{v'})g$, we obtain 
  \[
  \left( \begin{array}{cc} 0 & \vec{v'}^T \\ c_2({\vec{v'}}) & \frac{1}{2} ||\vec{v}'||^2 \end{array} \right) 
 \left( \begin{array}{cc} S(g) & 0 \\ 0 & a_1(g) \end{array} \right ) 
 + \left( \begin{array}{cc}  \Idd & 0 \\ \vec{v'} & 1 \end{array} \right) 
 \left( \begin{array}{cc} C_1(g) & a_4(g) \\ c_2(g) & a_9(g) \end{array} \right).
 \]
Considering $(2, 1)$-blocks, we obtain 
\[ c_2(g) + a_9(g) c_2({\vec{v}}) = c_2({\vec{v'}}) S(g) 
+ \vec{v'} C_1(g) 
+ c_2(g).\]

\end{proof}


\begin{lemma} \label{lem:matrix}
Assume as in Lemma \ref{lem:similarity}. 
Suppose additionally that every $g \in \Gamma \ra M_g$ is so that $M_g$ is in a fixed compact group $O(i_0)$
or equivalently $\mu_g = 1$. 
Then we can find coordinates so that the following holds for all $g$: 
\begin{align} 
O_5(g)^{-1} a_4(g) &= (a_8(g))^T \hbox{ or } a_4(g)^T O_5(g) = a_8(g),\\
a_1(g) = a_9(g)  &= \lambda_{\bv_{\tilde E}}(g) \hbox{ and } 
A_5(g) = \lambda_{\bv_{\tilde E}}(g) O_5(g). 
\end{align}
\end{lemma} 
\begin{proof} 
Since we have $\mu_g = 1$, we obtain $a_1(g) = a_9(g) = a_5(g) = \lambda_{\bv_{\tilde E}}(g)$ and
$A_5(g) = \lambda_{\bv_{\tilde E}}(g) O_5(g)$ by Lemma \ref{lem:similarity}. 

Again, we use equations \ref{eqn:firstm} and \ref{eqn:secondm}. 
We need to only consider lower right $(i_0+2)\times (i_0+2)$-matrices. 
\begin{align} 
&\left( \begin{array}{ccc} 
a_1(g) & 0 & 0 \\ a_4(g) & a_5(g) O_5(g)  & 0 \\ a_7(g) & a_8(g) & a_9(g) \end{array} \right) 
\left(\begin{array}{ccc} 
1 & 0 & 0 \\ \vec{v}^T & \Idd & 0 \\ \frac{1}{2}||\vec{v}||^2 & \vec{v} & 1 \end{array} \right) \\
& = \left( \begin{array}{ccc} 
a_1(g) & 0 & 0 \\ a_4(g) + a_5(g)  O_5(g) \vec{v}^T & a_5(g) O_5(g) & 0 
\\ a_7(g) + a_8(g) \vec{v}^T + \frac{a_9(g)}{2} ||\vec{v}||^2 & a_8(g) + a_9(g) \vec{v} & a_9(g) 
\end{array} \right). 
\end{align}
This equals 
\begin{align}
& \left( \begin{array}{ccc} 
1 & 0 & 0 \\ \vec{v'}^T & \Idd & 0 \\ \frac{1}{2}||\vec{v'}||^2 & \vec{v'} & 1 \end{array} \right) 
\left( \begin{array}{ccc} 
a_1(g) & 0 & 0 \\ a_4(g) & a_5(g) O^5_g & 0 \\ a_7(g) & a_8(g) & a_9(g) \end{array} \right) \\
& = \left(\begin{array}{ccc} 
a_1(g) & 0 & 0 \\ a_1(g) \vec{v'}^T + a_4(g) & a_5(g) O_5(g) & 0 \\ 
\frac{a_1(g)}{2} ||\vec{v'}||^2 + \vec{v'} a_4(g) + a_7(g) & 
a_5(g) \vec{v'} O_5(g) + a_8(g) & a_9(g) \end{array} \right).
\end{align}
Then by comparing the $(3, 2)$-blocks, 
we obtain $a_8(g) + a_9(g) \vec{v} = a_8(g) + a_5(g) \vec{v'} O_5(g) $.
Thus, $\vec{v} =  \vec{v'} O_5(g)$

From the $(3, 1)$-blocks, we obtain 
\[ a_1(g) \vec{v'}\cdot\vec{v'}/2 + \vec{v'}a_4(g) = a_8(g)\vec{v}^T + a_9(g) \vec{v}\cdot\vec{v}/2. \]
Since the quadratic forms have to equal each other, we obtain 
$\vec{v} O_5(g)^{-1} \cdot a_4(g) = \vec{v} \cdot a_8(g)$ for all $\vec{v}$. 
Thus, $(O_5(g)^T a_4(g))^T = a_8(g)^T$.
\end{proof}

Thus, in cases that we are concerned in this section,  and by 
taking a finite index subgroup of $\Gamma$, 
we conclude that each $g \in \Gamma - N$ has the form 
\begin{equation} \label{eqn:formg}
\newcommand*{\temp}{\multicolumn{1}{r|}{}}
\left( \begin{array}{ccccccc} 
S(g) & \temp & 0 & \temp & 0 & \temp & 0 \\ 
 \cline{1-7}
0 &\temp & \lambda_{\bv_{\tilde E}}(g) &\temp & 0 &\temp & 0 \\ 
 \cline{1-7}
C_1(g) &\temp & \lambda_{\bv_{\tilde E}}(g)\vec{v}^T_g &\temp & \lambda_{\bv_{\tilde E}}(g) O_5(g) &\temp & 0 \\ 
 \cline{1-7}
c_2(g) &\temp & a_7(g) &\temp & \lambda_{\bv_{\tilde E}}(g) \vec{v}_g O_5(g) &\temp & \lambda_{\bv_{\tilde E}}(g)
\end{array} 
\right).
\end{equation}



\begin{corollary}\label{cor:formg2} 
If $g$ of form of equation \ref{eqn:formg} centralizes a Zariski dense subset $A'$ of $\CN$, 
then $O_5(g) = \Idd_{i_0}$. 
\end{corollary} 
\begin{proof} 
Note that the subset $A''$ of $\bR^i$ corresponding to $A'$ is also Zariski dense in $\bR^i$. 
$g \CN(\vec{v}) = \CN(\vec{v}) g$ shows that 
$\vec{v} = \vec{v} O_5(g)$ for all $\vec{v} \in A''$. 
Hence $O_5(g) =\Idd$. 
\end{proof}

\subsubsection{The discrete case of $N_K$} 

\begin{proposition} \label{prop:decomposition}
Assume as in Lemma \ref{lem:similarity}. 
Suppose additionally the following{\rm :} 
\begin{itemize} 
\item every $g \in \Gamma \ra M_g$ is  so that $M_g$ is in a fixed compact group $O(i_0)$.
\item $\bGamma_{\tilde E}$ satisfies the weakly uniform middle-eigenvalue conditions
and it normalizes and virtually centralizes $\CN$.
\item $N_K$ acts properly discontinuously, discretely and cocompactly on $K^o$. Also, $K = \{k\} \ast K''$ as above. 
\end{itemize} 
Then $K''$ embeds projectively in the closure of $\Bd \torb$
invariant under $\bGamma_{\tilde E}$, and 
one can find a coordinate system so that for every $\CN(\vec v)$ and each element $g$ of 
$\bGamma_{\tilde E}$ is written so that
\begin{itemize}
\item $C_1(\vec v)=0, c_2({\vec v})=0$, and 
\item $C_1(g)=0$ and $c_2(g) = 0$.
\end{itemize}
\end{proposition}
\begin{proof}
Let $\bGamma'_{\tilde E}$ denote the finite index subgroup of $\bGamma_{\tilde E}$ centralizing $\CN$
and a product of cyclic and hyperbolic groups. 


By Lemma \ref{lem:conedecomp1}, $K$ is a strict join $\{k\} * K''$ for a point $k \in K$ and 
an open convex domain $K''$ of dimension $n-i_0-2$
in $\Bd K$. 
Since $K^o/N_K$ is compact, $K$ has a compact set $F$ which every orbit meets.
$K$ is foliated by open lines from a point $k \in K$ to points of $K''$.
Call these $k$-radial lines. 
Take such a line $l$ and a sequence of points 
$\{k_m\}$ so that 
\[k_m \ra  k_\infty \in K^{\prime \prime o} \hbox{ as } m \ra \infty.\] 
Hence, $\Gamma$ contains a sequence $\{\gamma_m\}$
so that $\gamma_m(k_m) \in F$ and $\gamma_m(l)$ is a line passing $F$
and $\gamma_m(\partial_1 l) \ra k_\infty$ for the endpoint $\partial_1 l$ of $l$ in $K''$. 
Since $K''$ is properly convex, $\{\gamma_m| K''\}$ is a bounded sequence of 
transformations and hence $\gamma_m$ is of form: 
\renewcommand{\arraystretch}{1.2}
\begin{equation}\label{eqn:gamman}
\newcommand*{\temp}{\multicolumn{1}{r|}{}}
\left( \begin{array}{ccccccc} 
\delta_m O_m & \temp & 0 & \temp & 0 & \temp & 0 \\ 
 \cline{1-7}
0 &\temp & \lambda_m &\temp & 0 &\temp & 0 \\ 
 \cline{1-7}
 C_{1,m} &\temp & \lambda_m \vec{v}^T_m &\temp & \lambda_m O_5(\gamma_m)  &\temp & 0 \\ 
 \cline{1-7}
c_{2,m} &\temp & a_7(\gamma_m) &\temp & \lambda_m \vec{v}_m O_5(\gamma_m) &\temp & \lambda_m 
\end{array} 
\right)
\end{equation}
where $\{O_m\}$ is a bounded sequence of matrices in $\Aut(K'')$ in 
$\SL_{\pm}(n-i_0-1, \bR)$
since the set of projective automorphisms of 
$K''$ moving interior points uniformly bounded distances is bounded. 

We note that $\delta_m^{n-i_0-1} \lambda_m^{i_0+2} = 1$
and $\delta_m/\lambda_m \ra 0$ as $\gamma_m|l$ pushes the points towards
the vertex $k$ of $K$. 
For the chosen single element $\gamma_m$, we can find the subspace 
$\SI^{n-i_0-1}$ containing $K''$ and $\bv_{\tilde E}, \bv_{E -}$ where $\gamma_m$ acts on by the $(n-i_0)\times (n-i_0)$-matrix of form 
\[\frac{1}{(\lambda_m \delta_m^{n-i_0-1})^{\frac{1}{n-i_0}}}
\left( \begin{array}{cc} \delta_m O_m & 0 \\ 0 & \lambda_m \end{array} \right)\in \SL_{\pm}(n-i_0, \bR).\]

We choose $m$ so that the norms of eigenvalues of $\delta_m O_m$ are strictly much smaller 
than the norm of $\lambda_m$, the unique norm of the eigenvalues of the lower-right $(i_0+2)\times(i_0+2)$-matrix. 
We fix one such $m_0$. 
Let $S(K''_{m_0})$ denote the $\gamma_{m_0}$-invariant subspace corresponding to subspaces 
associated with the real sum of the real Jordan-block subspaces with norms of eigenvalues $< \lambda_{m_0}$. 
We choose a coordinate system of $\SI^n$ so that $\gamma_{m_0}$ is of form 
so that $C_{1, m_0} =0, c_{2, m_0}=0$. 
Then a compact proper convex domain $K''_{m_0}$ in $S(K''_{m_0})$ maps to $K''$ under 
under the projection $\Pi_K: \SI^n - \SI^{i_0}_\infty \ra \SI^{n-i_0-1}$.

We will now show that $K''_{m_0}$ is invariant under $\bGamma_{\tilde E}$: 
let $n_1$ be an element of $\CN$. Then $\gamma_{m_0} n_1 = n_1 \gamma_{m_0}$. 
Since $n_1 \gamma_{m_0}(S(K''_{m_0})) = \gamma_{m_0} n_1(S(K''_{m_0}))$, we have 
\[n_1(S(K''_{m_0})) = \gamma_{m_0} n_1(S(K''_{m_0})).\]
Since the form of $\gamma_{m_0}$ determines $S(K''_{m_0})$
using the span of real Jordan-block subspaces, 
and $n_1(S(K''_{m_0}))$ is a $\gamma_{m_0}$-invariant subspace, 
and $\CN$ is connected, 
we obtain $n_1(S(K''_{m_0})) = S(K''_{m_0})$
for all $n_1 \in \CN$. Now we obtain 
\[n_1(K''_{m_0}) = K''_{m_0} \hbox{ for all } n_1 \in \CN\]
since they map to $K_{m_0}$ under $\Pi_K$. 
Hence, $C_1(\vec v) = 0$ and $c_2(\vec v) = 0$ 
for every $\vec v \in \bR^i$ in this system of coordinates. 

Let $B(K''_{m_0})$ denote the tube that is a union of great segment
passing $K''_{m_0}$.
Let $S(B(K''_{m_0}))$ denote the minimal subspace of $\SI^n$ containing $B(K''_{m_0})$. 
Let $B(\clo(\tilde \Sigma_{\tilde E}))$ denote the tube with vertex  $\bv_{\tilde E}$ and $\bv_{\tilde E-}$ 
corresponding to directions of $\clo(\tilde \Sigma_{\tilde E})$. 
We note that $B(K''_{m_0}) = B(\clo(\tilde \Sigma_{\tilde E})) \cap S(B(K''_{m_0}))$. 

Since $C_1(\vec v) = 0$ and $c_2(\vec v) = 0$,  
$S(B(K''_{m_0}))$ is the unique subspace of fixed points of $\CN(\vec v)$ 
for all $\vec v$ according to the equation \ref{eqn:nilmatstd}. 

For any element $g \in \bGamma'_{\tilde E}$, 
we also have $\CN(\vec v) g = g \CN(\vec v)$ for all $\vec v \in \bR^i$. 
Again we have $\CN(\vec v) g(x) = g(x)$ for all 
$x \in S(B(K_{m_0}))$ and $\vec v \in \bR^i$. 
This implies that 
\[g(K''_{m_0}) \subset S(B(K''_{m_0})) \cap B(\clo(\tilde \Sigma_{\tilde E})) = B(K''_{m_0})\] 
since $S(K''_{m_0})$ is the unique subspace of fixed points of $\CN$. 
Therefore, $\bGamma'_{\tilde E}$ acts on $B(K''_{m_0})$. 

Each irreducible hyperbolic factor group 
$\bGamma_i \cap \bGamma'_{\tilde E}$ satisfies the uniform 
middle-eigenvalue conditions. Thus, it acts on an invariant set $K_i \subset \clo(K'')\subset \SI^{n-i_0-1}$.  
(Here, we don't consider one corresponding to the vertex $k$ in $K$
and hence the irreducible invariant subspaces are all in the complimentary subspace of $k$ in $K \subset \SI^{n-i_0-1}$. )
Thus, $\bGamma_i$ acts on a set $K'_i$ in $B(K''_{m_0})$ distanced 
from $\bv_{\tilde E}$ and $\bv_{\tilde E-}$, which is unique if $\bGamma_i$ is hyperbolic by Theorem \ref{thm:distanced}.

If $\bGamma_i$ is not hyperbolic, then it is a trivial group. 
$\bGamma'_{\tilde E}$ acts each segment $B(K_i)$ corresponding 
to the $0$-dimensional $K_i$ either
\begin{itemize}
\item trivially or 
\item with a unique common 
fixed point in the interior or 
\item without fixed point in 
the interior but fixing the vertices of $B(K_i)$. 
\end{itemize}
Since $\gamma_{m_0}$ acts on $S(K''_{m_0})$, and 
$S(K''_{m_0}) \cap B(K_i)$ is a point in the interior, 
only the second case is possible. 
(See the proof of Theorem \ref{thm:distanced}.)
We choose the unique fixed point or the arbitrary fixed point $K'_i$ in the  interior of $B(K_i)$.
The strict join of $K'_1, \dots, K'_{l_0}$ in $B(K''_{m_0})$ is 
totally geodesic and $\bGamma'_{\tilde E}$-invariant. 
Therefore, $K''_{m_0}$ where $\gamma_{m_0}$ acts
must be this strict join. 

Since $\bGamma_{\tilde E}/\bGamma'_{\tilde E}$ is finite, 
we obtain finitely many sets of form $g(K''_{m_0})$ for $g \in \bGamma_{\tilde E}$. 
If they are not identical, at least one $g'$ satisfies
$g'(K''_{m_0}) \ne K''_{m_0}$. Then $\gamma_{m_0}^i(g'(K''_{m_0}))$ then produces 
infinitely many distinct sets of form $g(K''_{m_0})$, which is a contradiction. 
Hence $g(K''_{m_0}) = K''_{m_0}$ for all $g \in \bGamma_{\tilde E}$. 
This implies that $C_1(g)=0$ and $c_2(g) = 0$.

\end{proof}

We remark that Propositions \ref{prop:decomposition} and \ref{prop:decomposition2}
have very similar proofs. The first one is much simpler, and so we wrote both proofs. 
It seems worth repeating the proof for convincing the readers. 


\subsection{Joins and quasi-joined ends} \label{subsec:join}

We will now discuss about joins and their generalizations in depth in this subsection.


\begin{hypothesis}\label{h:join} 
Let $G$ be a p-end fundamental group. 
We continue to assume as in Lemma \ref{lem:similarity} for $G$. 
\begin{itemize} 
\item Every $g \in \Gamma \ra M_g$ is 
so that $M_g$ is in a fixed compact group $O(i_0)$. Thus, $\mu_g = 1$ identically. 
\item $G$ acts on the subspace $\SI^{i_0}_\infty$ containing 
$\bv_{\tilde E}$ and the properly convex 
domain $K_{m_0}''$ in the subspace $\SI^{n-i_0-2}$ 
disjoint from $\SI^{i_0}_\infty$. 
\item $\CN$ acts on these two subspaces fixing every points 
of $\SI^{n-i_0-2}$.
\end{itemize} 
\end{hypothesis} 



We assume $\bv_{\tilde E}$ to have coordinates $[0, \dots, 0, 1]$.
$\SI^{n-i_0-2}$ contains the standard points $[e_i]$ for $i=1, \dots, n-i_0 -1$ 
and $\SI^{i_0+1}$ contains $[e_i]$ for $i=n-i_0, \dots, n+1$. 
Let $H$ be the open $n$-hemisphere defined by $x_{n-i_0} > 0$. Then by convexity of $U$, 
we can choose $H$ so that $K'' \subset H$ and $\SI^{i_0}_\infty \subset \clo(H)$. 

By Hypothesis \ref{h:join}, 
elements of $\CN$ have the form of equation \ref{eqn:nilmatstd} with 
\[C_1(\vec{v})=0, c_2(\vec{v})=0 \hbox{ for all } \vec{v} \in \bR^{i_0}\] 
and the group $G$ of form of equation \ref{eqn:formg} with 
\[s_1(g) =0, s_2(g) = 0, C_1(g) = 0, \hbox{ and } c_2(g) = 0.\] 
We assume further that $O_5(g) = \Idd_{i_0}$. 


Again we recall the projection $\Pi_K: \SI^n - \SI^{i_0}_\infty \ra \SI^{n-i_0-1}$. 
$G$ has an induced action on $\SI^{n-i_0 -1}$ and
acts on a properly convex set $K''$ in $\SI^{n-i_0-1}$ so that 
$K$ equals a strict join $k* K''$ for $k$ corresponding to $\SI^{i_0+1}$. 
(Recall the projection $\SI^n - \SI^{i_0}_\infty$ to $\SI^{n-i_0 -1}$. )


\begin{figure}
\centering
\includegraphics[height=6cm]{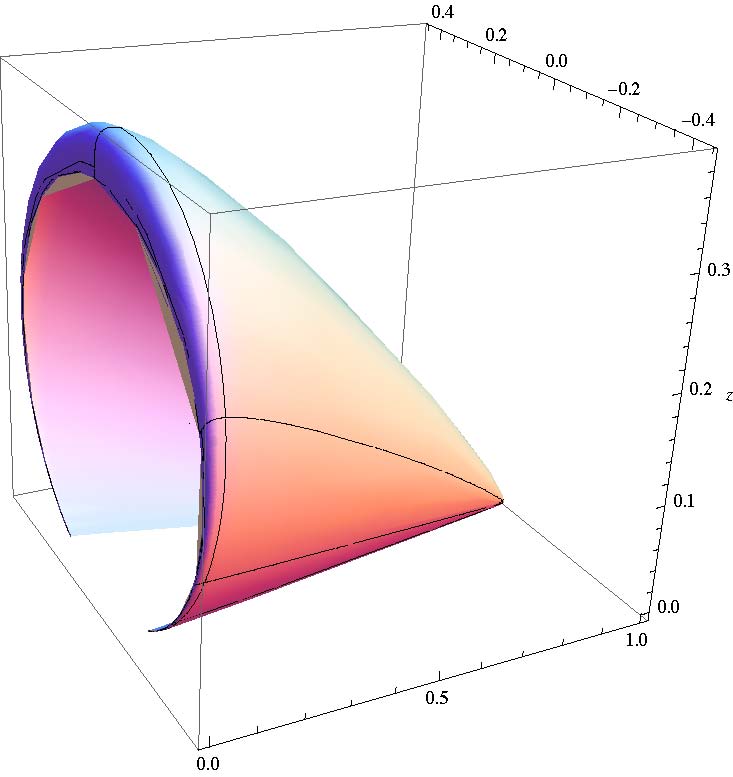}
\caption{A figure of quasi-joined p-R-end-neighborhood}
\label{fig:quasi-j}
\end{figure}

We define invariants from the form of equation \ref{eqn:formg}
\[\alpha_7(g):= \frac{a_7(g)}{\lambda_{\bv_{\tilde E}}(g)} - \frac{||\vec{v}_g||^2}{2} \]
for every $g \in G$. 
\[\alpha_7(g^n) = n \alpha_7(g) \hbox{ and }
\alpha_7(gh) = \alpha_7(g) + \alpha_7(h), \hbox{ whenever }g, h, gh \in G.\] 

Here $\alpha_7(g)$ is determined by factoring  
the matrix of $g$ into commuting matrices of form
\renewcommand{\arraystretch}{1.2}
\begin{equation}\label{eqn:kernel}
\newcommand*{\temp}{\multicolumn{1}{r|}{}}
\left( \begin{array}{ccccccc}
\Idd_{n-i_0-1} & \temp & 0 & \temp & 0 & \temp & 0 \\   
 \cline{1-7}
0                    &\temp & 1 & \temp & 0 & \temp & 0  \\ 
 \cline{1-7}
0                    & \temp &   0 &\temp & \Idd_{i_0} &\temp & 0 \\ 
 \cline{1-7}
0                    &\temp &  \alpha_7(g) &\temp & \vec{0} &\temp & 1 \\ 
\end{array} 
\right)
\left( \begin{array}{ccccccc}
S_g & \temp & 0 & \temp & 0 &\temp & 0 \\ 
\cline{1-7}
0 & \temp & \lambda_{\bv_{\tilde E}}(g) & \temp & 0 & \temp & 0  \\ 
 \cline{1-7}
0& \temp & \lambda_{\bv_{\tilde E}}(g) \vec{v}_g &\temp & \lambda_{\bv_{\tilde E}}(g) O_5(g) &\temp & 0 \\ 
 \cline{1-7}
0& \temp & \lambda_{\bv_{\tilde E}}(g)  \frac{||\vec{v}||^2}{2}  
&\temp & \lambda_{\bv_{\tilde E}}(g) \vec{v}_g O_5(g) &\temp & \lambda_{\bv_{\tilde E}}(g) \\ 
\end{array} 
\right).
\end{equation}

\begin{remark} \label{rem:alpha7}
We give a bit more explanations. 
Recall that the space of segments in a hemisphere $H^{i_0+1}$ with the vertices $\bv_{\tilde E}, \bv_{\tilde E-}$ 
forms an affine space $A^i$ one-dimension lower, and the group $\Aut(H^{i_0+1})_{\bv_{\tilde E}}$ of projective automorphism 
of the hemisphere fixing $\bv_{\tilde E}$ maps to $\Aff(A^{i_0})$ with kernel $K$ equal to 
transformations of an $(i_0+2)\times (i_0+2)$-matrix form
\renewcommand{\arraystretch}{1.2}
\begin{equation}\label{eqn:kernel2}
\newcommand*{\temp}{\multicolumn{1}{r|}{}}
\left( \begin{array}{ccccc} 
1 & \temp & 0 & \temp & 0  \\ 
 \cline{1-5}
0 &\temp & \Idd_{i_0} &\temp & 0 \\ 
 \cline{1-5}
b &\temp & \vec{0} &\temp & 1 \\ 
\end{array} 
\right)
\end{equation}
where $\bv_{\tilde E}$ is given coordinates $[0, 0, \dots, 1]$ and 
a center point of $H^{i_0+1}_l$ the coordinates $[1, 0, \dots, 0]$. 
In other words the transformations are of form 
\begin{align}\label{eqn:temp}
\left[
\begin{array}{c}
 1\\
 x_1  \\
 \vdots \\ 
 x_{i_0} \\ 
 x_{i_0+1}   
\end{array}
\right]
\mapsto 
\left[
\begin{array}{c}
 1 \\
 x_1\\
 \vdots \\ 
 x_{i_0} \\ 
 x_{i_0+1}+b
\end{array}
\right]
\end{align}
and hence $b$ determines the kernel element. 
Hence $\alpha_7(g)$ indicates the translation towards $\bv_{\tilde E}=[0,\dots, 1]$. 
\end{remark}

We define $G_+$ to be a subset of $G$ consisting of elements $g$ so that 
the largest norm $\lambda_1(g)$ of the eigenvalue occurs at the vertex $k$. 
Then since $\mu_g=1$, we necessarily have $\lambda_1(g) = \lambda_{\bv_{\tilde E}}(g)$
with all other norms of the eigenvalues occurring at $K''$ is strictly less than $\lambda_{\bv_{\tilde E}}(g)$.  
The second largest norm $\lambda_2(g)$ of the eigenvalue occurs at the complementary subspace $K''$ of $k$ in $K$.
Thus, $G_+$ is a semigroup.      
The condition that $\alpha_7(g) \geq 0$ for $g \in G_+$ is said to be the 
{\em positive translation condition}.


Again, we define \[\mu_7(g) : = \frac{\alpha_7(g)}{\log\frac{\lambda_{\bv_{\tilde E}}(g)}{\lambda_2(g)}}\] where 
$\lambda_2(g)$ denote the second largest norm of the  eigenvalues of $g$ and 
we restrict $g \in G_+$. 
The condition $\mu_7(g) > C_0, g \in  G_+$ for a uniform constant $C_0$
is called the {\em uniform positive translation condition}. 

Suppose that $G$ is a p-end fundamental group.



For this proposition, we do not assume $N_K$ is discrete. 

\begin{proposition} \label{prop:qjoin}
Let $\Sigma_{\tilde E}$ be the end orbifold of a nonproperly convex R-end $\tilde E$ of a strongly tame 
$n$-orbifold $\orb$ with radial or totally geodesic ends. Let $G$ be the p-end fundamental group. 
Let $\tilde E$ be an NPCC p-R-end
and $G$ acts on a p-end-neighborhood 
$U$ fixing $\bv_{\tilde E}$. Let $K, K'', \SI^{i_0}_\infty,$ and $\SI^{i_0+1}$ be as above.
We assume that $K^o/G$ is compact, $K= K''* k$ in $\SI^{n-i_0}$ with $k$ 
corresponding to $\SI^{i_0+1}$ under the projection $\Pi_K$. 
Assume that 
\begin{itemize} 
\item $G$ satisfies the weakly uniform middle-eigenvalue condition. 
\item Elements of $G$ and $\CN$ are of form of equation \ref{eqn:form1}
with \[C_1(\vec{v}) = 0, c_2(\vec{v})=0, C_1(g) = 0, c_2(g) =0\]
for every $\vec{v} \in \bR^{i_0}$ and $g \in G$.
\item $G$ normalizes $\CN$, and $\CN$ acts on $U$ 
and each leaf of $\mathcal{F}$ of $\tilde \Sigma_{\tilde E}$. 
\end{itemize} 
Then
\begin{itemize} 
\item[(i)] The condition $\alpha_7 \geq 0$ is a necessary condition that $G$ acts on a properly convex domain in $H$.
\item[(ii)] The uniform positive translation condition is equivalent to the existence of
properly convex p-end-neighborhood $U'$ whose closure meets $\SI^{i_0+1}_k$ at $\bv_{\tilde E}$ only.
\item[(iii)] $\alpha_7$ is identically zero if and only if $U$ is a join and $U$ is properly convex. 
\end{itemize}
\end{proposition}
\begin{proof} 
We projectively identify the smallest open tube containing $U$ 
as product of a bounded convex set equivalent to $K^o$ 
multiplied by a complete affine space of dimension $i_0+1$ 
in an affine space given by $H^o$. 
Each of $E_l := H_l \cap U$ is given by \[x_{n+1} > x_{n-i_0+1}^2 + \cdots + x_{n}^2 + C_l \] since $\CN$ acts on each
where $C_l$ is a constant depending on $l$ and $U'$. (See Section \ref{subsub:quadric}.)

Let $A^n$ be an affine space containing $U$ with $K'', \bv_{\tilde E}, \SI^{i_0}_\infty \subset \Bd A^n$.
This gives global coordinate functions $x_{n-i_0+1}, \dots, x_{n+1}$ on $A^n$ where we set $x_{n-i_0} = 1$.

Let $\Pi_{i_0}: U \ra \bR^{i_0}$ be the projection to the last 
$i_0+1$ coordinates $x_{n-i_0+1}, \dots, x_{n+1}$. 
We obtain a commutative diagram and an induced $L_g$ 
\begin{alignat}{3}
H_l  & \stackrel{g}{\ra} & g(H_l) \nonumber \\
\Pi_{i_0} \downarrow &  & \Pi_{i_0} \downarrow \nonumber\\ 
\bR^{i_0} & \stackrel{L_g}{\ra}  & \bR^{i_0} 
\end{alignat}
By Equation \ref{eqn:kernel}, 
$L_g$ preserves the quadric above in
the form of the projection up to translations in the $x_{n+1}$-axis direction.

Suppose that $G$ acts with a uniform positive translation condition.
Given a point $x = [\vec v] \in U'\subset \SI^n$ where 
$\vec v = \vec v_s + \vec v_h$ 
where $\vec v_s$ is in the direction of $K''$ and 
$\vec v_h$ is in one of $H^{i_0+1}$. 
If $g \in G_+$, then $g[\vec v] = [g \vec v_s + g \vec v_h]$ where $[g \vec v_s] \in K''$ 
and $[g \vec v_h] \in H_k$. 
The Euclidean length of $g \vec v_s$ is decreased compared to $v_s$. 
and that of $g \vec v_h$ is increased as we can deduce from the form of $g$ in Equation \ref{eqn:form1}
with $C_1(g)= 0, c_2(g)=0$.

(i) Suppose that $\alpha_7(g) < 0$ for some $g \in G_+$.
Let $k'\in K^o$.  
Then the action by $g$ gives us that $\{g^n(E_{k'})\}$ 
converges geometrically to an $(i_0+1)$-dimensional hemisphere
since $\alpha_7(g^n) \ra -\infty$ as $n \ra \infty$ implies that 
$g$ translates the affine space $H^o_{k'}$ a component  to $H^o_{g^n(k')}$
down toward $[-1,0,\dots, 0]$ in the above coordinate system. 
Thus, $G$ cannot act on a properly convex domain. 

(ii) 
Let $x \in U$. 
Choose an element $g \in G_+$ so that $\lambda_1(g) > \lambda_2(g)$
and let $F'$ be the fundamental domain in $K^o$ with respect to $\langle g \rangle$. 
This corresponds to 
a radial subset $F$ from $\bv_{\tilde E}$ 
bounded away at a distance from $k$ and $K''$ in $U$.
The set $F$ has the property that 
$|\log \frac{\lambda_1(g')}{\lambda_2(g')}| < C_F$ for a positive constant 
whenever $g'(x_0) \in F$ for a fixed $x_0 \in F$.

Let $G_F := \{g \in G| g'(x_0) \in F\}$. For $g \in G_F$, 
$ |\log \frac{\lambda_1(g)}{\lambda_2(g)}| < C$ where $C_F$ is a number depending of $F$ only.  
Hence, $\alpha_7(g)$ is bounded below by some negative number for 
$g \in G_F$ by the uniform positive translation condition. 
In the above affine coordinates for $k'\in F$, 
there is a lower bound on values of linear function obtained from $x_{n+1}$. 
Thus, the convex hull $D_F$ of $\bigcup_{g' \in G_F} g(H^o_{k'} \cap U)$ in $\clo(\torb)$ is a properly convex set because 
this is a union of lower bounded horoballs meeting $\bv_{\tilde E}$ as we can see from 
using the above coordinates. 

Since $\alpha_7(g^i) = i \alpha_7(g) \ra +\infty$ as $i \ra \infty$, we obtain that
$\{g^i(D_F)\} \ra \{ \bv_{\tilde E}\}$ for $i \ra \infty$ and 
$\{g^i(D_F)\}$ geometrically converges to a subset of  $K'' * \bv_{\tilde E}$ for $i \ra -\infty$. 
Thus, using the above coordinates, 
the convex hull of these sets is properly convex also since they are uniformly bounded from below. 


Let $U'$ be a p-end-neighborhood of $\bv_{\tilde E}$ that is the interior of 
the convex hull of $\{g_i(D_F)\}$. By the boundedness from $\bv$ of at most distance $\pi - C$ for some $C> 0$, 
the convex hull is properly convex. Then the above paragraph implies that 
$\clo(U') \cap \SI^{i_0+1}_k =\{ \bv_{\tilde E}\}$ holds. 



Conversely, suppose that $G$ acts on a properly convex p-end-neighborhood $U'$.

Let $z_1$ be a real valued projective function on $K$ with $z_1 =0$ on $K''$ and $z_1 = \infty$ on the vertex $k$. 
This induces a real valued projective function on $U$ also by precomposing with $\Pi_K$.
For an element $g \in G_+$, we take the radial fundamental domain $F$ of $U$ 
for $g$  containing $x_0$. Assume $z_1(x_0) =1$. 
Points of $F$ satisfy $1/C_F < z_1 < C_F$ for a positive constant $C_F$.  
Let $G_{C_F}$ denote this subset of $G_+$ of elements $h$ so that 
\[ 1/C_F < z_1(h(x_0)) < C_F \hbox{ and also } 1/C_F < z_1(h^{-1}(x_0)) < C_F.\]

Suppose that $G_{C_F}$ contains a sequence $g_i $ with above property so that
$\alpha_7(g_i) \ra - \infty$. Then $\{g_i(E_l)\}$ becomes larger and larger and every convergent 
subsequence converges to a hemisphere geometrically for a choice of $\pm 1$. 
Since $U$ is properly convex, this cannot happen. 
Thus, $\{\alpha_7(g)| g \in G_{C_F}\}$ is bounded below by a uniform constant. 
Similarly $\{\alpha_7(g)| g \in G_{C_F}\}$ is bounded above 
as we can use $\alpha_7(g^{-1})$. 


Suppose that $\alpha_7(g) =0$ for some $g\in G_+$. 
Then $g_i(\clo(U) \cap H_l)$ for a leaf $l$ geometrically converges to a horoball $B$ at $H_k$. 
Then $\alpha_7(h) =0$ for all $h \in G$ since otherwise $g^i(B)$ converges to a hemisphere for 
$i \ra \pm \infty$ for $h$ with $\alpha_7(h) \ne 0$. 
Now we obtain a sequence $h_i \in G_+$ with 
$\lambda_{\bv_{\tilde E}}(h_i)/\lambda_2(h_i) \ra \infty$ and $h_i| K''$ is uniformly bounded
since $\tilde \Sigma_{\tilde E}/\bGamma_{\tilde E}$ is compact. 
Since $\alpha_7(h_i)=0$, we obtain a contradiction by Lemmas \ref{lem:decjoin} and \ref{lem:joinred}. 
Therefore, $\mu_7(h) > 0$ for $h \in G_+$ by (i). ---(*)

%

Let $\lambda_{K''}(g)$ denote the largest eigenvalue of $g$ restricted to the subspace spanned by $K''$. 
Suppose that $\mu_7(g'_i) \ra 0$ for a sequence $g'_i \in G_+$. 
Then by above discussion,
we can obtain a sequence $g_i  \in \bigcup_{n \in J} g^nG_{C_F}$ for 
a fixed finite set $J$ and $\mu_7(g_i) \ra 0$.
We can assume that $ \lambda_1(g_i)/\lambda_{K''}(g_i) > 1+\eps$ for a positive constant $\eps > 0$ since we can take powers 
of $g_i$ not changing $\mu_7$. Here, $\lambda_2(g_i) = \lambda_{K''}(g_i)$. 

We obtain a sequence $n_i$, $n_i > 0$, by carefully choosing a slow growing one, so that 
\[\alpha_7(g_i^{n_i})= n_i\alpha_7(g_i) \ra 0 \hbox{ and }
 \lambda_1(g_i^{n_i})/\lambda_{K''}(g_i^{n_i}) \ra \infty.\]   
 However, from such a sequence 
 we obtain that $\{g_i^{n_i}(l \cap U)\}$ converges geometrically
 to a non-degenerate horoball at the boundary of $U$ 
 corresponding to $k$. This again implies that $U$ is a join
 and hence is a contradiction. 
 Hence $\mu_7(g) > C$ for all $g \in G_+$ and a uniform constant 
 $C > 0$. 
 This proves the converse part of (ii). 

Hence, if $h \in G_+$ with $i$ not sufficiently large, this shows that 
$\mu_7(g)$ is positive and hence bounded below. 

(i) and (*) in the proof (ii) proves (iii). 


\end{proof} 

\begin{definition}
The second case of Proposition \ref{prop:qjoin}, 
$\tilde E$ is said to be a {\em quasi-joined p-R-end} and $G$ now is called a {\em quasi-joined end group}.
An end with an end-neighborhood that is covered by a p-end-neighborhood of 
such a p-R-ends is also called a {\em quasi-joined p-R-end}. 
\end{definition}


\subsubsection{Splitting the ends}


\begin{theorem}\label{thm:NPCCcase} 
Let $\Sigma_{\tilde E}$ be the end orbifold of an NPCC p-R-end $\tilde E$ of a strongly tame  
properly convex $n$-orbifold $\orb$ with radial or totally geodesic ends. 
Assume that the holonomy group is strongly irreducible.
Let $\bGamma_{\tilde E}$ be the p-end fundamental group,
and it satisfies the weakly uniform middle-eigenvalue condition. Assume also that the semisimple quotient $N_K$ is discrete. 
Then
$\tilde E$ is a quasi-join of a T-end and a cusp type R-end. 
\end{theorem}
\begin{proof} 
We will continue to use the notation developed above in this proof. 
By Lemma \ref{lem:similarity}, 
$h(g) \CN(\vec{v}) h(g)^{-1} = \CN( \vec{v} M_g)$ where 
$M_g$ is a scalar multiplied by an element of a copy of an orthogonal group $O(i_0)$. 

Since $N \subset \CN$ is a discrete cocompact, $N$ is virtually isomorphic to $\bZ^{i_0}$ as we recall from the beginning of Section 
\ref{sec:discrete}.
Without loss of generality, we assume that $N$ is a cocompact subgroup of $\CN$. 
$h(g)N h(g)^{-1} = N$. Since $N$ corresponds to a lattice $L \subset \bR^n$ by the map $\CN$, 
and the conjugation by $h(g)$ is 
to a map given by right multiplication $M_g: L \ra L$
by Lemma \ref{lem:similarity}.
Thus, $M_g: L \ra L$ is conjugate to an element of 
$\SL_\pm(i_0, \bZ)$ and 
$\{M_g| g \in \bGamma_{\tilde E}\}$ 
is a compact group as their determinant is $\pm 1$. 
Hence, 
the image of the homomorphism given by $g \in h(\pi_1(\tilde E)) \mapsto M_g \in \SL_\pm(i_0, \bZ))$ is a finite order group. 
Thus, $\bGamma_{\tilde E}$ has a finite index group $\bGamma'_{\tilde E}$ centralizing $\CN$. 

We find a kernel $K_1$ of this map and take $\Sigma_{E'}$ to be the corresponding cover of $\Sigma_{\tilde E}$. 
By Proposition \ref{prop:decomposition}, we have the result needed to apply Proposition \ref{prop:qjoin}. 
Now $\gamma_m$ is in the virtual center since $M_g =\Idd$. 
Finally, Proposition \ref{prop:qjoin}(i) and (ii) imply that $\bGamma_{\tilde E}$ virtually is either a join or a quasi-joined group.
The proof of Theorem \ref{thm:NPCCcase2} shows that a joined end cannot occur.  
\end{proof}



\section{The indiscrete case} \label{sec:indiscrete}

Let $\Sigma_{\tilde E}$ be the end orbifold of an NPCC R-end $\tilde E$ of a strongly tame 
properly convex $n$-orbifold $\orb$ with radial or totally geodesic ends. 
Let $\bGamma_{\tilde E}$ be the p-end fundamental group. 
Let $U$ be a p-end-neighborhood in $\torb$ corresponding to a p-end vertex $\bv_{\tilde E}$. 

We can assume that $\partial U$ is smooth by smoothing if necessary. 



Recall the exact sequence 
\[ 1 \ra N \ra \pi_1(\tilde E) \stackrel{\Pi^*_K}{\longrightarrow} N_K \ra 1 \] 



An element $g \in \bGamma_{\tilde E}$ is of form: 
\begin{equation} \label{eqn:g}
\newcommand*{\temp}{\multicolumn{1}{r|}{}}
g = \left( \begin{array}{ccc} 
K(g) & \temp &  0 \\ 
 \cline{1-3}
* &\temp & U(g)
\end{array} 
\right).
\end{equation}
Here $K(g)$ is an $(n-i_0)\times (n-i_0)$-matrix and 
$U(g)$ is  an $(i_0+1)\times (i_0+1)$-matrix acting on $\SI^{i_0}_\infty$. 
We note $\det K(g) \det U(g) = 1$. 


\subsection{Estimations with $KA \bU$.}

Let $\bU$ denote a maximal nilpotent subgroup of $\Aut(\SI^n)_{\SI^{i_0}_\infty}$ given 
by lower triangular matrices with diagonal entries equal to $1$.

\begin{lemma}\label{lem:orth} 
The matrix of $g \in \Aut(\SI^n)$ can be written under a coordinate system orthogonal 
at $V^{i_0+1}_\infty$ as $k(g) a(g) n(g)$ where 
$k(g)$ is an element of $O(n+1)$, $a(g)$ is a diagonal element, and $n(g)$ is in the group  $\bU$ of unipotent 
lower triangular matrices. 
Also, diagonal elements of $a(g)$ are the norms of eigenvalues of $g$ as elements of $\Aut(\SI^n)$. 
\end{lemma} 
\begin{proof} 
%
Let $\vec{v}_1, \dots, \vec{v}_{i_0+1}, \vec{v}_{i_0+2}, \dots, \vec{v}_{n+1}$ denote the basis vectors of $\bR^{n+1}$ 
that are chosen from the real Jordan-block subspaces of $g$ with the same norms of eigenvalues
where $\vec{v}_j \in V^{i_0+1}_\infty$ for $j =1, \dots, i_0+1$. 
We require $[\vec{v}_1] = \bv_{\tilde E}$. 

Now we fix a Euclidean metric on $\bR^{n+1}$. 
We obtain vectors 
\[\vec{v}_1', \dots, \vec{v}_{i_0+1}', \vec{v}_{i_0+2}', \dots, \vec{v}_{n+1}'\]
by the Gram-Schmidt orthogonalization process using the corresponding Euclidean metric on $\bR^{n+1}$.
Then the desired result follows by writing the matrix of $g$ in terms of coordinates
given by letting the basis vectors $\vec{v}'_i = \vec{u}_{n+1 -i}$. 
(See also Proposition 2.1 of Kostant \cite{BK}.)

\end{proof}

We define 
\[ \bU' := \bigcup_{ k \in O(n+1)} k\bU k^{-1}.  \]

\begin{corollary}\label{cor:bdunip} 
Suppose that we have for a positive constant $C_1$, and $g \in \bGamma_{\tilde E}$, 
\[ \frac{1}{C_1} \leq \lambda_{n+1}(g),\lambda_1(g) \leq C_1.\]
Then $g$ is in a bounded distance from $\bU'$ with the bound depending only on $C_1$.
\end{corollary} 
\begin{proof} 
By Lemma \ref{lem:orth}, we can find an element $k\in O(n+1)$ so that 
\[ g  = k k(g) k^{-1} k a(g) k^{-1} k n(g) k^{-1}\] as above. 
Then $k k(g) k^{-1} \in O(n+1)$ and $k a(g) k^{-1}$ is uniformly bounded from $\Idd$ by 
a constant depending only on $C_1$ by Proposition \ref{prop:eigSI}. 
Finally, we obtain $k n(g) k^{-1} \in \bU'$. 
\end{proof} 

A subset of a Lie group is of {\em polynomial growth} if the volume of the ball $B_R(\Idd)$ radius $R$ is less than or
equal to a polynomial of $R$. 
As usual, the metric is given by the standard positive definite left-invariant bilinear form that is invariant 
under the conjugations by the compact group $O(n+1)$. 

\begin{lemma} \label{lem:bU} 
$\bU'$ is of polynomial growth in terms of the distance from $\Idd$. 
\end{lemma}
\begin{proof}
Let $\Aut(\SI^n)$ have a left-invariant Riemannian metric. 
Clearly $\bU$ is of polynomial growth since $\bU$ is nilpotent by Gromov \cite{Gr}.
Given $g \in O(n+1)$, the distance between $gug^{-1}$ and $u$ for $u \in \bU'$ 
is proportional to a constant multiplied by $\bdd(u, \Idd)$: 
Choose $u \in \bU'$ which is unipotent. We can write $u(s) = \exp(s \vec{u})$ where 
$\vec{u}$ is a nilpotent matrix of unit norm. $g(t) := \exp(t \vec{x})$ for $\vec{x}$ in the Lie 
algebra of $O(n+1)$ of unit norm. 
For a family of $g(t) \in O(n+1)$, we define 
\begin{equation}\label{eqn:conj}
u(t, s) = g(t) u(s) g(t)^{-1} =  \exp(s Ad_{g(t)} \vec{u}).
\end{equation}
We compute 
\[ u(t, s)^{-1} \frac{d u(t, s)}{d t} :=u(t, s)^{-1} (\vec{x} u(t, s) - u(t, s) \vec{x}) = (Ad_{u(t, s)^{-1}} - \Idd)( \vec{x} ).\]
Since $\vec{u}$ is nilpotent, $Ad_{u(t, s)^{-1}} - \Idd$ is a polynomial of variables $t, s$. 
The norm of $d u(t, s)/ dt$ is bounded above by a polynomial in $s$ and $t$. 
The conjugation orbits of $O(n+1)$ in $\Aut(\SI^n)$ are compact.
Also, the conjugation by $O(n+1)$ preserves the distances of elements from $\Idd$
since the left-invariant metric $\mu$ is preserved by conjugation at $\Idd$  
and geodesics from $\Idd$ go to geodesics from $\Idd$ of same $\mu$-lengths under 
the conjugations by equation \ref{eqn:conj}.
Hence, we obtain a parametrization of $\bU'$ by $\bU$ and $O(n+1)$ where 
the volume of each orbit of $O(n+1)$ grows polynomially. 
Since $\bU$ is of polynomial growth, $\bU'$ is of polynomial growth
in terms of the distance from $\Idd$. 
\end{proof}


%



\subsection{Closures of leaves} \label{subsec:leaves}

Given a subgroup $G$ of an algebraic Lie group, 
the {\em syndetic hull} $S(G)$ of $G$ is a connected Lie group so that 
$S(G)/G$ is compact. (See Fried and Goldman \cite{FG} and D. Witte \cite{DW}.)

The properly convex open set $K, K \subset \SI^{n-i_0}$ has a Hilbert metric. Also the group $\Aut(K)$ of 
projective automorphisms of $K$ in $\SL_\pm(n-i_0+1, \bR)$ is a closed Lie group. 

\begin{lemma}\label{lem:invmet} 
Let $D$ be a properly convex domain with the closed locally compact group $\Aut(D)$ of smooth automorphisms of $D$.
Given a group $G$ acting isometrically on an open domain $D$ faithfully so that 
$G \rightarrow \Aut(D)$ is an embedding. 
Suppose that $D/G$ is compact. 
Then the closure $\bar G$ of $G$ is a Lie subgroup 
acting on $D$ properly, and there exists a smooth Riemannian metric on $D$ that is 
$\bar G$-invariant. 
\end{lemma} 
\begin{proof} 
Since $\bar G$ is in the isometry group, 
$\bar G$ is a Lie subgroup acting on $D$ properly. 

One can construct a Riemannian metric $\mu$ with bounded entries. 
Let $\phi$ be a function supported on a compact set containing a fundamental domain $F$ of $D/G$ where $\phi|F > 0$. 
We can assume  that the derivatives of the entries of $\phi \mu$ up to the $m$-th order are uniformly bounded above.

Then $\{g^*\phi\mu| g \in \bar G\}$ is an equicontinuous family up to any order. 
Thus the integral 
\[ \int_{g \in \bar G} g^* \phi\mu d \mu \]
of $g^* \phi\mu$ for $g \in \bar G$ is a $C^\infty$-Riemannian metric and that is positive definite. 
This bestows us a $C^\infty$-Riemannian metric $\mu_D$ on $D$ invariant under $\bar G$-action.
\end{proof} 

The foliation on $\tilde \Sigma_{\tilde E}$ given by fibers of $\Pi_K$
has leaves that are $i_0$-dimensional complete affine spaces.
Then $K^o$ admits a smooth Riemannian metric $\mu_{0, 1}$ invariant under $N_K$ by Lemma \ref{lem:invmet}. 
Since $N_K$ is not discrete, a component $N_{K, 0}$ of the closure of $N_K$ in $\Aut(K)$
is a Lie group of dimension $\geq 1$. 
We consider the orthogonal frame bundle $FK^o$ over $K^o$. 
A metric on each fiber of $FK^o$ is induces from $\mu_K$.
Since the action of $N_{K, 0}$ is isometric on $FK^o$ with trivial stabilizers, 
we find that $N_{K, 0}$ acts on a smooth orbit submanifold $\Sigma_{1, 0}$ of $FK^o$ transitively 
with trivial stabilizers. 
(See Lemma 3.4.11 in \cite{Thbook}.)

There exists a bundle $F\tilde \Sigma_{\tilde E}$ from pulling back $FK^o$ by the projection map. 
Since $\bGamma_{\tilde E}$ acts isometrically on $FK^o$, 
the quotient space $F\tilde \Sigma_{\tilde E}/\bGamma_{\tilde E}$ is a bundle $F\Sigma_{\tilde E}$ over 
$\Sigma_{\tilde E}$ with a subbundle with compact fibers isomorphic to the orthogonal group of dimension $n-i_0$. 
Also, $F\tilde \Sigma_{\tilde E}$ is foliated by $i_0$-dimensional affine 
spaces pulled-back from the $i_0$-dimensional leaves on the foliation $\tilde \Sigma_{\tilde E}$. 
Also, $F\tilde \Sigma_{\tilde E}$ covers $F\Sigma_{\tilde E}$. 

Each leaf $l$ of $F\tilde \Sigma_{\tilde E}$ goes to a point $x$ of $FK^o$. 
Here $N_{K, 0}(x)$ is as an orbit in $FK^o$, whose inverse image 
becomes a smooth submanifold $\tilde V_l$ covering a compact submanifold $V_l$ in $F\Sigma_{\tilde E}$
by the work of Molino \cite{Mo1}.
Here $l$ maps to a dense leaf in $V_l$. 

\begin{lemma} \label{lem:polynomial}
Each leaf $l$ is of polynomial growth. That is, each ball $B_R(x)$ in $l$ of radius $R$ has an area 
less than equal to $f(R)$ for a polynomial $f$ where we are using an arbitrary Riemannian metric on 
$F\tilde \Sigma_{\tilde E}$ and $F\Sigma_{\tilde E}$ so that the covering map $F\tilde \Sigma_{\tilde E} \ra F\Sigma_{\tilde E}$ is a local isometry. 
\end{lemma}
\begin{proof} 
Let us choose a fundamental domain $F$ of $F\Sigma_{\tilde E}$. Let $\hat F$ denote the image of $F$ in $FK^o$. 
Then $l$ is a union of $g_i(D_i)$  $i \in I_l$ for the intersection $D_i$ of a leaf with $F$ where $g_i \in \bGamma_{\tilde E}$
for some index set $I_l$. 
We have that $D_i \subset D'_i$ where $D'_i$ is an $\eps$-neighborhood of $D_i$ in the leaf. 
Then \[\{g_i(D'_i)| i \in I_l \}\] cover $l$ in a locally finite manner. 
The subset $G(l):= \{g_i\in \Gamma| i \in I_l\}$ is a discrete subset. 

Choose an arbitrary point $d_i \in D_i$ for every $i \in I_l$. 
The set $\{g_i(d_i)| i \in I_l \}$ and $l$ is quasi-isometric: 
a map from $G(l)$ to $l$ is given by $f_1: g_i \mapsto g_i(d_i)$ 
and the multivalued map $f_2$ from $l$ to $G(l)$ given by sending each point $x \in l$ 
to one of finitely many $g_i$ such that $g_i(D'_i) \ni x$. 
Let $\bGamma_{\tilde E}$ be given the Cayley metric and $\tilde \Sigma_{\tilde E}$ a metric induced 
from $\Sigma_{\tilde E}$. 
Let $\tilde \Sigma_{\tilde E}/\bGamma_{\tilde E}$ have a Riemannian metric and the induced on on $\tilde \Sigma_{\tilde E}$. 
Both maps are quasi-isometries since 
these maps are restrictions of quasi-isometries $\bGamma_{\tilde E} \ra \tilde \Sigma_{\tilde E}$
and $\tilde \Sigma_{\tilde E} \ra \bGamma_{\tilde E}$ defined in an analogous manner.  

The action of $g_i$ in $K$ is bounded since it sends some points of $\hat F$ to ones of $\hat F$. 
Thus, $\Pi^*_K(g_i)$ goes to a bounded subset of $\Aut(K)$. 
Hence 
\[K(g_i) = \det(K(g_i))^{1/(n-i_0)} \hat K(g_i) \hbox{ where } \hat K(g_i) \in \SL_{\pm}(n-i_0, \bR).\]
Let $\tilde \lambda_1(g_i)$ and $\tilde \lambda_n(g_i)$ denote the largest norm and the smallest norm of
eigenvalues of $\hat K(g_i)$. These are bounded by two positive real numbers. 
The largest and the smallest eigenvalues of $g_i$ equal
\[\lambda_1(g) = \det(K(g_i))^{1/(n-i_0)}\tilde \lambda_1(g_i) \hbox{ and } 
\lambda_{n+1}(g) = \det(K(g_i))^{1/(n-i_0)}\tilde \lambda_n(g_i)\]
Denote by $a_j(g_i), j=1, \dots, i_0+1$, the norms of eigenvalues associated with $\SI^{i_0}_\infty$.  
Since 
\[\det(K(g_i)) a_1(g_i) \dots a_{i_0+1}(g_i) = 1,\] 
if $|\det(K(g_i))| \ra 0$ or $\infty$, then 
the equation in Proposition \ref{prop:eigSI} cannot hold. 
Therefore, we obtain 
\[1/C < |\det(K(g_i))| < C\] for a positive constant $C$. 
We deduce that  
the largest norm and the smallest norm of eigenvalues of $g_i$
\[\det(K(g_i))^{1/(n-i_0)}\tilde \lambda_1(g_i) \hbox{ and } 
\det(K(g_i))^{1/(n-i_0)}\tilde \lambda_n(g_i)\]
are bounded above and below by two positive numbers. 
Hence, $\lambda_1(g_i)$ and $\lambda_n(g_i)$ and  
the components of $a(g_i)$ are all bounded above and below by a fixed set of positive numbers. 

By Corollary \ref{cor:bdunip}, $\{g_i\}$ is of bounded distance from $\bU'$.
Let $N_c(\bU')$ be a $c$-neighborhood of $\bU'$. 
Then \[G(l)  \subset N_c(\bU').\] 

Let $d$ denote the left-invariant metric on $\Aut(\SI^n)$. 
By the discreteness of $\bGamma_{\tilde E}$, the set $G(l)$ is discrete 
and there exists a lower  
bound to \[\{d(g_i, g_j)| g_i, g_j \in G(l), i \ne j\}.\]
Also given any $g_i \in G(l)$, there exists 
an element $g_j \in G(l)$ so that 
$d(g_i, g_j) < C$ for a uniform constant $C$.
(We need to choose $g_j$ so that $g_j(F)$ is 
adjacent to $g_i(F)$.)
Let $B_R(\Idd)$ denote the ball in $\SL(n+1, \bR)$ of radius $R$ with the center $\Idd$. 
Then $B_R(\Idd) \cap N_c(\bU')$ is of polynomial growth with respect to $R$, and so is $G(l) \cap B_R(\Idd)$. 
Since the $\{g_i(D'_i)| g_i \in G(l)\}$ 
of uniformly bounded balls  cover $l$ in a locally finite manner, 
$l$ is of polynomial grow as well. 
\end{proof}


\subsection{The orthopotency of $N$} \label{sub:uniN}
Let $p_{\Sigma_{\tilde E}}: F\tilde \Sigma_{\tilde E} \ra F \Sigma_{\tilde E}$ be the covering map
induced from $\tilde \Sigma_{\tilde E} \ra  \Sigma_{\tilde E}$.
The foliation on $\tilde \Sigma_{\tilde E}$ gives us a foliation of $F\tilde \Sigma_{\tilde E}$. 
Recall the fibration \[\Pi_K:  \tilde \Sigma_{\tilde E} \ra K^o
\hbox{ which induces } \tilde \Pi_K: F \tilde \Sigma_{\tilde E} \ra F K^o.\]
Since $N_K$ acts as isometries of Riemannian metric on 
$K^o$, we can obtain a metric on $\Sigma_{\tilde E}$ 
so that the foliation is a Riemannian foliation. 


Let $l$ be a leaf of $F\tilde \Sigma_{\tilde E}$,
and  $p$ be the image of $l$ in $FK^o$. 
Since $l$ maps to a polynomial growth leaf in $F\Sigma_{\tilde E}$ by Lemma \ref{lem:polynomial},
by the work of Carri\`ere \cite{Car}, a connected nilpotent Lie group
$A_l$ in the closure of $N_K$ acts on $FK^o$ so that we have a submanifold
\begin{alignat}{3} 
\tilde \Pi_K^{-1}(A_l(p))  : = \tilde V & \hookrightarrow & F\tilde \Sigma_{\tilde E} \nonumber \\
\downarrow  &                     &p_{\Sigma_{\tilde E}} \downarrow \nonumber \\
V_l      & \hookrightarrow & F \Sigma_{\tilde E}
\end{alignat} 
for a compact submanifold $V_l := \overline{p_{\Sigma_{\tilde E}}(l)}$ in $F \Sigma_{\tilde E}$.
Here $A_l$ is the component of the closure of $N_K$ the image of $\bGamma_{\tilde E}$ in $\Aut(K)$. 
Clearly $A_l$ is an algebraic group. 
Hence, $\bGamma_{\tilde E}$ is in a Lie group 
\[\bR^l \times Z(\Gamma_1) \times \dots \times Z(\Gamma_k), l \geq k-1\]
for the Zariski closure $Z(\Gamma_i)$ of $\Gamma_i$. 

Since $A_l$ is in the product group, we can project to each 
$\Gamma_i$-factor or the central $\bR^{l_0-1}$.
One case is that the image of $A_l$ is $Z(\Gamma_j)$ is not discrete in $\Aut(K_j)$ and hence the image equals a union of 
components of copies of $PO(n_j, 1)$ or $SO(n_j, 1)$ in $Z(\gamma_j)$ 
by Theorem 1.1 of \cite{Ben2} and Fait 5.4 of \cite{Ben1}, proved by Benzecri \cite{Benz}. 
The nilpotency implies that the image is a cusp group fixing a unique point in $\Bd K_j$.
Thus, the image is an abelian group since $A_l$ is connected. 
In case, $\Gamma_j$ is discrete, then a nilpotency implies that the image group 
fixes a unique pair of points in $\Bd K_j$ and hence is abelian also.  
Thus, 
 $A_l$ is an abelian group. 

Let $N_l$ be exactly the subgroup of $\pi_1(V_l)$ fixing a leaf $l$ in $FK^o$, 
for each closure $V_l$ of a leaf $l$, the manifold $V_l$ is compact and 
we have an exact sequence 
\[ 1 \ra N_l \ra h(\pi_1(V_l)) \stackrel{\Pi^*_K}{\longrightarrow} A'_l \ra 1.\]
Since the leaf $l$ is dense in $V_l$, it follows that $A'_l$ is dense in $A_l$. 
Each leaf $l'$ of $\tilde \Sigma_{\tilde E}$ has a realization a subset in $\torb$. 
Since $N_l$ fixes every points of $K^o$ and $N$ is in $\pi_1(V_l)$, we obtain $N = N_l$. 
We have the norms of eigenvalues $\lambda_i(g) = 1$ for $g \in N_l$.  
By Proposition \ref{prop:eigSI}, we have that $N=N_l$ is orthopotent
since the norms of eigenvalues equal $1$ identically and $N_l$ is discrete. Then $N$ is easily seen to be virtually solvable 
since it is of polynomial growth as we can deduce 
from the orthopotent flags.
(See the proof of Proposition \ref{prop:affinehoro} also. )


We summarize below: 
\begin{proposition}\label{prop:V_l} 
Let $l$ be a generic fiber of $F\tilde \Sigma_{\tilde E}$ and $p$ be the corresponding point $p$ of $FK^o$. 
Then there exists an algebraic abelian group $A_l$ acting on $FK^o$ so that 
$\tilde \Pi_K^{-1}(A_l(p)) =\tilde V_l$ covers a compact suborbifold $V_l$ in $F\Sigma_{\tilde E}$, and 
the image of the holonomy group of $\tilde V_l$ is a dense subgroup of $A_l$. 
\end{proposition} 


\subsubsection{The nilradical of the syndetic hull $US_{l, 0}$ is normalized by $\bGamma_{\tilde E}$.}\label{subsub:syndetic}



The leaf holonomy acts on $F\tilde \Sigma_{\tilde E}/\mathcal F$ as an abelian killing field group
without any fixed points. 
Hence, each leaf $l$ is in $\tilde V_l$ with a constant dimension. 
Thus, $\mathcal F$ is a foliation with leaf closures of the identical dimensions. 
The leaf closures form another foliation $\bar F$ with compact leaves by Lemma 5.2 of Molino \cite{Mol}. 
We let $F\Sigma_{\tilde E}/\bar F$ denote the space of closures of leaves has an orbifold structure
where the projection $F\Sigma_{\tilde E} \ra F\Sigma_{\tilde E}/\bar F$ is an orbifold morphism
by Proposition 5.2 of \cite{Mol}.
Since $\Sigma_{\tilde E}$ has a geometric structure induced from the transverse real projective structure, 
$\Sigma_{\tilde E}$ is a very good orbifold. We may assume that $\Sigma_{\tilde E}$ is an $n-1$-manifold
and so is $F\Sigma_{\tilde E}$ since we need our results for finite index subgroups only. 
By Lemma 5.2 of \cite{Mol}, $F\Sigma_{\tilde E}/\bar F$ is the quotient space of $F\tilde \Sigma_{\tilde E}/\mathcal F$ by the connected 
abelian group $A_l$ acting properly with trivial stabilizers.  
Thus, it admits a geometric structure induced from the real projective structure 
of $F\tilde \Sigma_{\tilde E}/\mathcal F$. 
There exists a finite regular manifold-cover $M$ of 
$F\Sigma_{\tilde E}/\bar F$
as in Chapter 13 of Thurston \cite{Thnote} (see \cite{Choi2004} also.)

By pulling back the fiber bundle over orbifolds, we consider the fundamental groups.
We obtain  a regular finite cover $F\Sigma_{\tilde E}^f$ of $F\Sigma_{\tilde E}$ and 
a regular fibration 
\begin{alignat}{3}
V_l  &\ra & F\Sigma_{\tilde E}^f & \ra & M  \nonumber \\
\downarrow &  & \downarrow & & \downarrow \nonumber \\ 
V_l & \ra & F\Sigma_{\tilde E} & \ra & F\Sigma_{\tilde E}/\bar F 
\end{alignat}
where $V_l$ is a generic fiber of $F\Sigma_{\tilde E}^f$ for the induced foliation $\bar F^f$ isomorphic to 
a generic fiber of $F\Sigma_{\tilde E}$. 

We obtain an exact sequence 
\[  \pi_1(V_l) \ra \pi_1(F\Sigma_{\tilde E}^f) \stackrel{\pi'_K}{\longrightarrow} \pi_1(M) \ra 1\]
and the image $\pi_1(V_l)$ is a normal subgroup of $\pi_1(F\Sigma_{\tilde E}^f)$.  
Since $F\Sigma_{\tilde E}^f$ is fibered by fibers diffeomorphic to ${\SO}(n-i_0)$ or its cover, 
we have a fibration 
\[\widetilde{\SO}(n-i_0) \ra F\Sigma_{\tilde E}^f \ra \Sigma_{\tilde E}^f\]
where $\Sigma_{\tilde E}^f$ is a regular finite cover of $\Sigma_{\tilde E}$
and $\widetilde{\SO}(n-i_0)$ is a finite cover of $\SO(n-i_0)$. Thus, 
we also have an exact sequence 
\[ \pi_1(\widetilde{\SO}(n-i_0)) \ra \pi_1(F\Sigma_{\tilde E}^f) \ra \pi_1(\Sigma_{\tilde E}^f) \ra 1.\]
Since $\pi_1(\Sigma_{\tilde E}^f)$ is a quotient group of $\pi_1(F\Sigma_{\tilde E}^f)$, 
the image of $\pi_1(V_l)$ is a normal subgroup of $\pi_1(\Sigma_{\tilde E}^f)$
for the generic $l$ so that $V_l$ in $F\Sigma_{\tilde E}^f$ to $V_l$ in $\Sigma_{\tilde E}^F$ is homeomorphic. 
We define $\bGamma_l$ as the image $h(\pi_1(V_l))$.
The above sequence tells us that 
$\bGamma_l$ is a normal subgroup of a finite index subgroup of $\bGamma_{\tilde E}$. 

From now on, we will assume that $\bGamma_l$ is a normal subgroup of $\bGamma_{\tilde E}$ by 
taking a finite cover of the end-neighborhood if necessary. 

Recall that $A_l$ is abelian from Section \ref{subsec:leaves}. By the above paragraph, 
$\Pi^*_K(\bGamma_{\tilde E})$ normalizes $A'_l$ and hence it closure $A_l$.
Since $A_l$ is a subgroup of the product of hyperbolic groups and abelian groups, 
$A_l$ fixes a common set of fixed points on each factor $K_i$. 
Thus, the normalizer $A_{\tilde E} : = \Pi^*_K(\bGamma_{\tilde E})$ commutes with each element of $A_l$ up to finite index also. 
We may pass to the finite index subgroup and we assume that $A_{\tilde E}$ is in the centralizer of $A_l$.

The group
\begin{equation}\label{eqn:gammal}
\bGamma_{l, l} := \bGamma_l \cap N = N_l = N
\end{equation}
is the subgroup of $\pi_1(V_l)$ 
acting on each leaf of $\tilde \Sigma_{\tilde E}$. 
$\bGamma_{l,l}=N$ is orthorpotent 
and hence is solvable. (See \cite{Fried86}.)

\begin{equation}\label{eqn:exactVl} 
\bGamma_{l,l} \ra \bGamma_l \ra A'_l
\end{equation}
is exact. Since $A'_l$ is abelian and $\bGamma_{l,l}$ is 
solvable, $\bGamma_l$ is solvable. 

We let $Z(\bGamma_{\tilde E})$ and $Z(\bGamma_l)$ denote the Zariski closures in $\Aut(\SI^n)$ of 
$\bGamma_{\tilde E}$ and $\bGamma_l$ respectively. 

By Theorem 1.6 of Goldman-Fried \cite{FG}, 
there exists a closed Lie group $S_l$ containing $\bGamma_l$ 
with the following four properties: 
\begin{itemize}
\item $S_l$ has finitely many components.
\item $S_l/\bGamma_l$ is compact. 
\item The Zariski closure $Z(S_l)$ is the same as 
$Z(\bGamma_l)$. 
\item \begin{equation}\label{eqn:rankSl}
{\mathrm{rank}}(S_l) \leq {\mathrm{rank}} (\bGamma_l). 
\end{equation}
\end{itemize}

Since $\bGamma_{\tilde E}$ normalizes $\bGamma_l$ by above, 
$\bGamma_{\tilde E}$ also normalizes 
$Z(\bGamma_l)=Z(S_l)$ but maybe not $S_l$ itself. 




Since $\bGamma_l$ acts on $V_l$ an algebraic set in $F\tilde \Sigma_{\tilde E}$
over the algebraic orbit $A_l(p)$, 
$Z(\bGamma_l) = Z(S_l)$ also acts on $V_l$
and hence so does $S_l$. 
We have 
Since $Z(\bGamma_l) \ra A_l$, 
we have $S_l \ra A_l$ as an onto map.

We summarize: 
\begin{lemma}\label{lem:V} 
$h(\pi_1(V_l))$ is virtually solvable and  is contained in a virtually solvable Lie group $S_l := S(h(\pi_1(V_l))$ 
with finitely many components, and $S_l/h(\pi_1(V_l))$ is compact. 
$S_l$ acts on $\tilde V_l$. 
Furthermore, one can modify a p-end-neighborhood $U$ so that 
$S_l$ acts on it. Also the Zariski closure of $h(\pi_1(V_l))$ is the same as that of $S_l$. 
\end{lemma} 
\begin{proof} 
By above, $Z(S_l) = Z(\bGamma_l)$ acts on $\tilde V_l$. 
We prove about the p-end-neighborhood only. 
Let $F$ be a compact fundamental domain of $S_l$ under the $\Gamma_l$. 
Then we have 
\[ \bigcap_{g\in S_l} g(U) = \bigcap_{g \in F} g(U).\]
Since $F$ is compact, the latter set is still a p-end-neighborhood. 
\end{proof}

Since $S_l$ acts on $U$ as shown in Lemma \ref{lem:V}, 
we have a homomorphism $S_l \ra Aut(K)$.  
We define by $S_{l, 0}$ the kernel of this map. 
Then $S_{l, 0}$ acts on each leaf of $\tilde \Sigma_{\tilde E}$.
$S_{l, 0}$ is the normal subgroup of $S_l$ characterized by 
the condition that it acts on each leaf; that is, 
its element acts on each of the hemispheres of dimension $i_0+1$
with boundary $\SI^{i_0}_\infty$. 

\subsubsection{The form of $US_{l, 0}$.}






From now on, we will let $S_l$ to denote the only the identity component of itself for simplicity
as $S_l$ has a finitely many components to begin with. 
This will be sufficient for our purpose of getting a cusp group normalized by $\bGamma_{\tilde E}$. 

Let $US_l$ denote the unipotent radical of the Zariski closure
$Z(S_l)$ of $S_l$ in $\Aut(\SI^n)$, which is a solvable Lie group. 
Also, $US_{l, 0}$ denote the unipotent radical of the Zariski closure of $S_{l, 0}$. 
Since $S_{l, 0}$ is normalized by $\bGamma_{\tilde E}$, 
so is $Z(S_{l, 0})$. 

\begin{proposition}\label{prop:ZN} 
Let $l$ be a generic fiber so that $A_l$ acts with trivial stabilizers. 
\begin{itemize} 
\item $S_l$ acts on $\tilde V_l$ and on $\tilde \Sigma_{\tilde E}$ and $\partial U$ freely, properly, and transitively
and acts as isometries on these spaces with respect to Riemannian metrics. 
\item $S_{l,0}$ acts transitively on each leaf $l$ with a compact stabilizer
and acts on an $i_0$-dimensional ellipsoid passing $\bv_{\tilde E}$ with an invariant Euclidean metric. 
\item $S_{l, 0}$ is an $i_0$-dimensional cusp group
and the unipotent radical $US_{l, 0}$ is an $i_0$-dimensional abelian group equal to $S_{l, 0}$. 
\item $US_{l, 0}$ acts on $\tilde V_l$ and is normalized by $\bGamma_{\tilde E}$ also. 
\end{itemize}
\end{proposition}  
\begin{proof} 
Since $Z(S_l) = Z(\bGamma_l)$ acts on $\tilde V_l$ as stated above, it follows 
that $S_l$ and $US_l$ both in the group act on $\tilde V_l$. 

(i) A stabilizer $S_{l, x}$ of each point $x \in \tilde V_l$ for $S_l$ is compact: 
let $F$ be the fundamental domain of $S_l$ with $\Gamma_l$ action. 
Let $F'$ be the image $F(x):= \{g (x)| g \in F\}$ in $\tilde V_l$. This is a compact set. 
Define
\[\Gamma_{l, F'} := \{g \in \Gamma_l|  g(F(x)) \cap F(x) \ne \emp\}.\] 
Then $\Gamma_{l, F'}$ is finite by the properness of the 
action of $\Gamma_l$. 
Since an element of $S_{l, x}$ is a product of an element $g'$ of $\Gamma_l$ and $f \in F$, and
$g' f(x) = x$, it follows that $g'F(x) \cap F(x) \ne \emp$ and $g' \in \Gamma_{l, F}$. 
Hence $S_{l, x} \subset \Gamma_{l, F'}F$ and $S_{l, x}$ is compact.
Similarly, $S_l$ acts properly on $\tilde \Sigma_{\tilde E}$. 
Since $\partial U$ is in one-to-one correspondence with $\tilde \Sigma_{\tilde E}$, 
$S_l$ acts on $\partial U$ properly. Hence, these spaces have 
compact stabilizers with respect to $S_l$. 
The invariant metric follows by Lemma \ref{lem:invmet}. 
Hence, the action is proper and the orbit is closed. 
(Since $\tilde V_l/\Gamma_l$ is compact, $\tilde V_l /S_l$ is compact also. )


(ii) We assume that $\bGamma_{\tilde E}$ is torsion-free by taking a finite index subgroup 
since $\Sigma_{\tilde E}$ is a very good orbifold, admitting a geometric structure. 
Now, we show that $S_l$ acts freely on 
$\tilde \Sigma_{\tilde E}$: 
We use the last part of 
Section 1.8 of \cite{FG} where we can replace $H$ there with $S_l$ and $\bR^n$ with $\tilde V_l$ and $\Gamma$ with 
the solvable subgroup $\bGamma_l$, we obtain the results for $\bGamma_l$: 

First, $\bGamma_l$ is solvable and discrete, and hence is polycyclic
and $S_l$ has the same Zariski closure as $\bGamma_l$. 
We work on the projection of $\tilde V_l$ on 
$\tilde \Sigma_{\tilde E}$, a convex but not properly convex open domain in an affine space $A^{n-1}$.  
The proof  
identical with that of Lemma 1.9 of \cite{FG}
shows that the unipotent radical $US_l$ of $Z(S_l)$ acts freely on $\tilde \Sigma_{\tilde E}$.

Being unipotent, $US_l$ is simply connected. The orbit $US_l(x)$ for $x \in \tilde \Sigma_{\tilde E}$ 
is simply connected and invariant under $Z(\bGamma_l)= Z(S_l)$. 
$US_l(x)/\bGamma_l$ is a $K(\bGamma_l, 1)$-space. 
Thus, $\ranK \bGamma_l = cd \bGamma_l \leq \dim US_l$. 
By Lemma 4.36 of \cite{Rag}, $\dim US_l \leq \dim S_l$ and 
by Lemma 1.6 (iv) of \cite{FG}, we have $\dim S_l \leq \ranK \bGamma_l$.
Thus, $\ranK \bGamma_l = \dim S_l$. 

We now show $S_l$ acts freely on $\tilde \Sigma_{\tilde E}$.
We have a fibration sequence 
\[ \bGamma_l \ra S_l \ra S_l/\bGamma_l \] and 
an exact sequence 
\[\pi_1(S_l) \ra \pi_1(S_l/\bGamma_l) \ra \bGamma_l,\]
and hence $\ranK \pi_1(S_l) + \ranK \bGamma_l = \ranK \pi_1(S_l/\bGamma_l) = \dim S_l$
since $S_l$ is solvable and $S_l/\bGamma_l$ is a compact manifold following the argument in Section 
1.8 of \cite{FG}. (See Proposition 3.7 of \cite{Rag} also where we need to take the universal cover of $S_l$.)
Since $\ranK \bGamma_l = \dim S_l$, we have $\ranK \pi_1(S_l) = 0$. 
This means that $\pi_1(S_l)$ is finite. Being solvable, it is trivial. 
Thus, $S_l$ is simply connected. Since $S_l$ is homotopy equivalent to $T^{j_1}$, 
$S_l$ is contractible. 

Since $S_l$ acts transitively on any of its orbits, 
$S_l$ is homotopy equivalent to a bundle over it with 
fiber homeomorphic to a stabilizer. Since $S_l$ is contractible, the stabilizer is trivial. 
Since $S_l$ acts with trivial stabilizers on $\tilde \Sigma_{\tilde E}$, it acts so on $\tilde V_l$. 
We showed that $S_l$ acts freely on $\tilde V_l$. 
(We followed Section 1.8 of \cite{FG} faithfully here.)

(iii) Now, we show that $S_l$ acts transitively on
$\tilde \Sigma_{\tilde E}$:  Choose $x \in \tilde V_l$. 
There is a map $f: S_l/\bGamma_l \ra \tilde V_l/\bGamma_l$ given by
sending each $g \in S_l$ to $g(x) \in \tilde V_l$. 
By the free action property the map is open. 
The image of the map is also closed since $S_l/\bGamma_l$ is compact.
Hence, the map is onto and $S_l$ acts transitively on $\tilde V_l$. 


(iv) Hence, $S_{l, 0}$ acts simply transitively on each $l$;
$S_{l, 0}$ is diffeomorphic to a leaf $l$ and hence is connected
and is a solvable Lie group. 

Since the subset $U_{l}:= U \cap H^{i_0+1}_l$ of $U$ corresponding to $l$ is a strictly convex set containing $v_{\tilde E}$, 
we have $S_{l, 0}$ acting simply transitively on $\partial U_{ l}$. 
As before in the proof of Theorem \ref{thm:comphoro} using the
results of \cite{CM2}, 
$S_{l, 0}$ acts on an $i_0$-dimensional ellipsoid that has to equal $\partial U_{l}$. 
Since one can identify each leaf with an affine space $S_{l, 0}$ is isomorphic to an affine isometry group 
acting simply transitively on an affine space $\bR^i$.
Let ${\mathcal H}_{\bv_{\tilde E}}$ denote the cusp group acting on the ellipsoid. 
An elementary argument using the cocompact subgroup simultaneously in both groups
shows that $S_{l,0}$ and ${\mathcal H}_{\bv_{\tilde E}}$ are identical.

This shows also that $S_{l, 0}$ is nilpotent and we have $US_{l, 0} = S_{l, 0}$ also. 
Finally, this implies that $US_{l, 0}$ is an $i_0$-dimensional abelian Lie group. 

For $g \in \bGamma_{\tilde E}$, 
$S'_l := g S_l g^{-1}$ is a syndetic hull of $\bGamma_{\tilde E}$. 
Then we define $S'_{l, 0}$ as the subgroup acting trivially on the space of leaves. 
Since $S'_{l, 0}$ has to be the cusp group as above by the same proof, it 
follows that $S'_{l, 0} = S_{l, 0} = gS_{l, 0}g^{-1}$. Thus, $S_{l, 0}$ is a normal subgroup. 

\end{proof}


\subsection{The forms of $\bGamma_{\tilde E}$.}

\subsection{The existence of splitting.} 

\subsubsection{Matrix form.} 

We can parametrize $US_{l, 0}$ by $\CN(\vec{v})$ for $\vec{v} \in \bR^{i_0}$ by Proposition \ref{prop:ZN}. 
As above by Lemmas  \ref{lem:similarity} and \ref{lem:conedecomp1}, we have that the matrices are of form. 
\renewcommand{\arraystretch}{1.2}
\begin{equation} 
\newcommand*{\temp}{\multicolumn{1}{r|}{}}
\CN(\vec{v}) = \left( \begin{array}{ccccccc} 
\Idd_{n-i_0-1} & \temp & 0 & \temp & 0 & \temp & 0 \\ 
 \cline{1-7}
0 &\temp & 1 &\temp & 0 &\temp & 0 \\ 
 \cline{1-7}
0 &\temp & \vec{v}^T &\temp & \Idd_{i_0} &\temp & 0 \\ 
 \cline{1-7}
c_2({\vec{v}}) &\temp & ||\vec{v}||^2 /2&\temp & \vec{v} &\temp & 1 
\end{array} 
\right), 
\end{equation} 
\begin{equation}
\newcommand*{\temp}{\multicolumn{1}{r|}{}}
g = \left( \begin{array}{ccccccc} 
S(g) & \temp & 0 & \temp & 0 & \temp & 0 \\ 
 \cline{1-7}
0 &\temp & a_1(g) &\temp & 0 &\temp & 0 \\ 
 \cline{1-7}
C_1(g) &\temp & a_4(g) &\temp & a_5(g) O_5(g) &\temp & a_6(g) \\ 
 \cline{1-7}
c_2(g) &\temp & a_7(g) &\temp & a_8(g) &\temp & a_9(g) 
\end{array} 
\right)
\end{equation}
where $g \in \bGamma_{\tilde E}$. Recall 
$\mu_g = a_5(g)/a_1(g) = a_9(g)/a_5(g)$. 
Since $S_l$ is in $Z(\bGamma_l)$ and the orthogonality of normalized $A_5(g)$
is an algebraic condition, the above form also holds for $g \in S_l$. 

\begin{proposition} \label{prop:mug}
Assume that $\bGamma_{\tilde E}$ is discrete. Then
we have $\mu_g= 1$ for every $g \in N'''$. 
\end{proposition}
\begin{proof} 
First suppose that $\bGamma_{l,l}$ is trivial. 
Then $N$ is trivial and $\bGamma_{\tilde E} \ra \Aut(K)$ is injective. 
The image $N_K$ centralizes the image of $\bGamma_l$ in $\Aut(K)$. 
This implies that $\bGamma_{\tilde E}$ centralizes $\bGamma_l$ also
since $N=\{\Idd\}$. Here, $\bGamma_l$ is abelian and then we can choose 
a syndetic hull $S_l$ where $S_{l,0}$ is not trivial. 
$\bGamma_{\tilde E}$ centralizes the Zariski closure $Z(\bGamma_l)$ of $\bGamma_l$. 
$S_l \subset Z(\bGamma_l)$ implies that $\bGamma_{\tilde E}$ centralizes $S_l$
and hence $US_{l, 0}$. This implies $\mu_g =1$ of course since $M_g = \Idd$ 
for all $g \in \bGamma_{\tilde E}$. 

Suppose that $\bGamma_{l, l}$ is finite. Then $\bGamma_{\tilde E}$ acts as a group of finite automorphisms of $\Gamma_l$ 
since a finite index subgroup of $\bGamma_{\tilde E}$ centralizes $S_l$. 
This implies $\mu_g = 1$ for all $g \in \bGamma_{\tilde E}$. 

Suppose that $\Gamma_{l, l}$ is infinite. 
For some $h \in \Gamma_{l, l} - \{\Idd\}$, we have $h = kN(\vec{v}_h)$
for $k \in O(i_0)$. 
If $\mu_g \ne 1$ for $g \in \bGamma_{\tilde E}$, we may assume without loss of generality that $\mu_g < 1$. 
We obtain that 
\[g^n k N(\vec{v}_h)g^{-n} = g^n k g^{-n} N(\vec{v}_h \mu_g^n O_g^{5, -n}).\] 
We can choose a subsequence that converges to an elliptic element $k$ 
since the off-diagonal elements of $N(\vec{v})$ are linear functions of $\vec{v}$. 
Thus, $\Gamma_{l, l}$ contains elements arbitrarily close to $k$
This contradicts the discreteness of $\Gamma_l$, and hence, $\mu_g = 1$. 


\end{proof}

\begin{proposition} \label{prop:decomposition2}
Let $\orb$ be a strongly tame properly convex real projective orbifold with radial or totally geodesic ends. 
Let $\tilde E$ be a p-R-end of $\orb$ with a p-end-neighborhood $U$ and the end vertex $\bv_{\tilde E}$.
Assume the following{\rm :}
\begin{itemize}
\item $\bGamma_{\tilde E}$ satisfies the weakly uniform middle-eigenvalue conditions.
\item $\bGamma_{\tilde E}$ normalizes the cusp group. 
\end{itemize} 
Then 
\begin{itemize}
\item One can find a coordinate system so that for every element of $\bGamma_{\tilde E}$ is written so that
$C_1(g)=0$ and $c_2(g) = 0$ for $g \in \bGamma_{\tilde E}$. 
\item Our p-R-end $\tilde E$ is a join or quasi-joined type of a cone over a totally geodesic domain $K''$ 
and a cusp p-R-end with a common p-end vertex $\bv_{\tilde E}$. 
\end{itemize}
\end{proposition}
\begin{proof}
Since $\bGamma_{\tilde E}$ normalizes $US_{l, 0}$ as written above, 
Lemmas \ref{lem:similarity} and \ref{lem:conedecomp1} apply to this situation. 

Since $K^o/N_K$ is compact, $K^o$ has a compact set $F$ which every orbit meets.
$K^o$ is foliated by open lines from a point $k \in K$ to points of convex open domain 
in $\clo(K'')$ of dimension $n-i_0-1$. 
That is, $K$ is the interior of the join $k* \clo(K'')$
for a properly convex open domain $K''$ of dimension $n-i_0-1$. 
Call these {\em $k$-radial lines}. 
Take such a line $l$ and a sequence of points $k_m \ra k_\infty \in K''$ as $m \ra \infty$. 
Hence, $\bGamma_{\tilde E}$ contains a sequence $\{\gamma_m\}$ 
so that $\gamma_m(k_m) \in F$ and $\gamma_m(l)$ is a line passing $F$
so that $\gamma_m(\partial_1 l) \ra k_\infty$ for the endpoint $\partial_1 l$ of $l$ in the interior $K''$. 
Since $K''$ is properly convex, this implies that $\{\gamma_m| K''\}$ is a bounded sequence of 
transformations and hence $\gamma_m$ is of form: 
\renewcommand{\arraystretch}{1.2}
\begin{equation}\label{eqn:gammap}
\newcommand*{\temp}{\multicolumn{1}{r|}{}}
\left( \begin{array}{ccccccc} 
\delta_m O_m & \temp & 0 & \temp & 0 & \temp & 0 \\ 
 \cline{1-7}
0 &\temp & \lambda_m &\temp & 0 &\temp & 0 \\ 
 \cline{1-7}
 C_{1,m} &\temp & \lambda_m \vec{v}^T_m &\temp & \lambda_m O_5({\gamma_m}) &\temp & 0 \\ 
 \cline{1-7}
c_{2,m} &\temp & a_7(\gamma_m) &\temp &\lambda_m  \vec{v}_m O_5(\gamma_m) &\temp & \lambda_m 
\end{array} 
\right)
\end{equation}
where $O_m$ is a bounded sequence of matrices in $\Aut(K'')$ in $\SL(n-i_0-1, \bR)$
since the set of projective automorphisms of 
$K''$ moving points only bounded distances in $d_{K''}$ is bounded
as in the proof of Proposition \ref{prop:decomposition}
and we have $\mu_g=1$ identically by Proposition \ref{prop:mug} and hence Lemma \ref{lem:matrix} applies.
Note that $\delta_m^{n-i_0-1} \lambda_m^{i_0+2} = 1$,
and $\delta_m/\lambda_m \ra 0$ as $\gamma_m|l$ pushes the points toward 
the vertex $k$ of $K$. 

We assume by choosing subsequences so that $\{O_m\}$ converges to $O_\infty \in \Aut(K'')$
and $\delta_m \ra 0$ and $\lambda_m \ra \infty$
for the determinant $\pm 1$ representation of $\gamma_m$.


Now we will show that $\bGamma_{\tilde E}$ acts on a tube on a copy of $K''$ in $\Bd \torb$.

Let $S_m$ denote the subspace spanned by the real Jordan-block subspaces corresponding to 
eigenvalues with norms strictly smaller than $\lambda_m$.
Then we obtain $\dim S_m = \dim K''$, and there is a projection $S_m \ra K'' \subset \SI^{n-i_0 -2}$ 
given by $\SI^n - \SI^{i_0+1}_{k}$ where $\SI^{i_0+1}_k$ is the great sphere containing  $H^{i_0+1}_1$. 
Let $K''_m$ denote the corresponding properly convex subset of $S_m$ to $K''$.


We introduce a coordinate on $\SI^n$ respecting the matrix form of equation \ref{eqn:gammap}.
Thus, given a vector $\vec{a} := (\vec{a}_1, a_2, \vec{a}_3, a_4)$ where $\vec{a}_1$ is a $n-i_0-1$-vector in 
a direction to $S_m$ and $(a_2, \vec{a}_3, a_4)$ is in the direction of $H^{i_0+1}_1$.
$a_2$ is a component in the direction of an interior point of $H^{i_0+1}_1$, 
$a_4$ is the component in the direction of $\bv_{\tilde E}$ and
$\vec{a}_3$ is an $i_0$-vector in the direction of the subspace complementary to the sum of 
the one-dimensional subspaces in these directions. 

Let $m_0$ be a large number so that $\delta_{m_0} > \lambda_{m_0}$. 
At least one point $x_{m_0}$ of $K''_{m_0}$ is in $\Bd \torb$ since we can apply $\gamma_{m_0}^i$ to a point of $U$ 
as $i \ra \infty$. $x_{m_0} \ne \bv_{\tilde E}, \bv_{\tilde E-}$ implies that $x_{m_0}$ has coordinates $\vec{a}_1 \ne 0$. 
Recall that
\renewcommand{\arraystretch}{1.2}
\begin{equation} \label{eqn:secondm2}
\newcommand*{\temp}{\multicolumn{1}{r|}{}}
\CN(\vec{v}) = \left( \begin{array}{ccccccc} 
\Idd_{n-i_0-1} & \temp & 0 & \temp & 0 & \temp & 0 \\ 
 \cline{1-7}
0 &\temp & 1 &\temp & 0 &\temp & 0 \\ 
 \cline{1-7}
0 &\temp & \vec{v}^T &\temp & \Idd_{i_0} &\temp & 0 \\ 
 \cline{1-7}
c_2({\vec{v}}) &\temp & \frac{||\vec{v}||^2 }{2}&\temp & \vec{v} &\temp & 1 
\end{array} 
\right).
\end{equation}
In other words, on $K''_{m_0}$, $a_2=0$ and $\vec{a}_3=0$. 
We note that $c_2(\vec{v})$ is linear in $\vec{v}$ in this situation.

In fact a neighborhood of $x_{m_0}$ in $K''_{m_0}$ is in $\Bd \torb$ 
since $\{O_m\}$ is a bounded collection of projective automorphisms. 

By Proposition \ref{prop:ZN}, $\CN(\vec{v})$ acts on $U$ and hence on the properly convex domain 
$\clo(U)$ which contains $K''_{m_0}$.  
Now we choose the point $x'$ in an open set $K''_{m_0} \subset \Bd \torb$ and represent it as the vector 
$\vec{a} = (\vec{a}_1, 0, \vec{0}, a_4)$.
We note that 
\begin{equation} \label{eqn:nonprop}
\CN(n \vec{v}) \vec{a} = ( \vec{a}_1, 0, \vec{0}, n c_2(\vec{v})\cdot \vec{a}_1+ a_4) \hbox{ for } n \in \bZ. 
\end{equation}
Therefore, it must be that $c_2(\vec{v})\cdot  x'= 0$ 
otherwise one can use $n \ra \infty$ and $n \ra -\infty$, we obtain 
two antipodal points, which are in $\clo(U)$. 
Hence $c_2(\vec{v}))=0$ for all $\vec{v} \in \bR^{i_0}$.


By  Equation \ref{eqn:conedecomp1}, 
$C_1(g) =0$ for all $ g\in \bGamma_{\tilde E}$ since $c_2(\vec{v})=0$ for all 
$\vec{v} \in \bR^i$.

Let $\SI^{n-i_0}_{m_0}$ denote the minimal subspace in $\SI^n$ containing $K''_{m_0}$ and $\bv_{\tilde E}, \bv_{\tilde E-}$. 
We define the tube $B(K''_{m_0})$ that is the union of segments passing $K''_{m_0}$ with endpoints 
$\bv_{\tilde E}, \bv_{\tilde E-}$. Since $C_1(g)=0$, $g \in \bGamma_{\tilde E}$ acts on $\SI^{n-i_0}_{m_0}$.
Since $g$ acts on $K''$, it follows that $g$ acts on $B(K''_{m_0})$, the subset of 
$\SI^{n-i_0}_{m_0}$ corresponding to $K'$ under the projection $\Pi_K$. 




Since $\bGamma_{\tilde E}$ acting on $B(K''_{m_0})$ thus satisfies the  weakly uniform middle-eigenvalue condition, 
Theorem \ref{thm:distanced} implies that 
$\bGamma_{\tilde E}$ acts on a compact set $K'$ distanced from $\bv_{\tilde E}$ and $\bv_{\tilde E-}$
as we saw in the last part of the proof of Proposition \ref{prop:decomposition}.

Since $\gamma_m$ acts on $K'$ as well, and we have
$\delta_m/\lambda_m \ra  0$ and $O_m \ra O_\infty \in \Aut(K'')$.
Suppose that $K' \ne K''_{m_0}$. Then $\gamma_{m_0}^i(K')$ has to 
become strictly closer to $K''_{m_0}$ than $K'$ for $i$ sufficiently large by the matrix form of equation \ref{eqn:gammap} 
as $\gamma_{m_0}$ acts on $B(K''_0)$ with only invariant subspaces $K''_m$ or one in $\{\bv_{\tilde E}, \bv_{\tilde E-}\}$. 
However, we have $\gamma_{m_0}^i(K')=K'$, an invariant set. 
Hence, we have $K'= K''_{m_0}$ for sufficiently large $m_0$. 
(Thus, $K''_{m_0}$ is defined independently of $m_0$.)

It also follows that $K'$ is in a totally geodesic hypersurface $H'$ of dimension $n-i_0-1$
and $K' = B(K''_{m_0}) \cap H'$ since $K'$ meets every complete segment in $B(K''_{m_0})$ with 
vertices $\bv_{\tilde E}$ and $\bv_{\tilde E-}$. 

Since $g$ acts on $K''_{m_0}$ and $\bv_{\tilde E}, \bv_{\tilde E-}$, 
we obtain $C_1(g) = 0$ and $c_2(g) =0$ for all $g \in \bGamma_{\tilde E}$
under this system of coordinates. 

Recall that $\CN$ acts on a horosphere $H$ in $\SI^{i_0+1}_1$ with the vertex $\bv_{\tilde E}$ fixed.
If $\alpha_7(g) = 0$ for all $g \in \bGamma_{\tilde E}$, 
then the join of a horosphere in $H$ and $K'$ form the join p-end-neighborhood that we wished to obtain by Proposition \ref{prop:qjoin}.
If $\alpha_7(g) > 0$ for $g \in \bGamma_{E, +}$, we obtained a quasi-joined p-R-end again  by Proposition \ref{prop:qjoin}.
If $\alpha_7(g) < 0$, we do not have a properly convex p-end-neighborhood of the p-R-end. 
\end{proof}



\subsection{The non-existence of split joined cases as well} 
\label{sub:proofNPCC}

\begin{theorem}\label{thm:NPCCcase2} 
Let $\Sigma_{\tilde E}$ be the end orbifold of an NPCC p-R-end $\tilde E$ of a strongly tame  
properly convex $n$-orbifold $\orb$ with radial or totally geodesic ends. 
Assume that the holonomy group is strongly irreducible.
Let $\bGamma_{\tilde E}$ be the p-end fundamental group,
and it satisfies the weakly uniform middle-eigenvalue condition.  
Then there exists a finite cover $\Sigma_{E'}$ of $\Sigma_{\tilde E}$ so that 
$E'$ is a quasi-join of a totally geodesic R-ends and a cusp type R-end. 
\end{theorem}
\begin{proof}
Proposition \ref{prop:decomposition2}
show that we have a joined or quasi-joined end. 

We will now show that the joined end does not occur. 

By Proposition \ref{prop:mug}, we know $\mu_g = 1$ for all $g \in \bGamma_{\tilde E}$. 
As in Proposition \ref{prop:decomposition} or \ref{prop:decomposition2}, we obtain a sequence 
 $\gamma_m$ of form: 
\renewcommand{\arraystretch}{1.2}
\begin{equation}\label{eqn:gammao}
\newcommand*{\temp}{\multicolumn{1}{r|}{}}
\left( \begin{array}{ccccccc} 
\delta_m O_m & \temp & 0 & \temp & 0 & \temp & 0 \\ 
 \cline{1-7}
0 &\temp & \lambda_m &\temp & 0 &\temp & 0 \\ 
 \cline{1-7}
 0 &\temp & \lambda_m \vec{v}^T_m &\temp & \lambda_m O_5(\gamma_m) &\temp & 0 \\ 
 \cline{1-7}
0 &\temp & \lambda_m\left(\alpha_7(\gamma_m) + \frac{||\vec{v}_m||^2}{2}\right) &\temp &
\lambda_m \vec{v}_m &\temp & \lambda_m 
\end{array} 
\right)
\end{equation}
as  $C_{1, m} =0$ and $c_{2, m}=0$
where $\lambda_m \ra \infty$ and $\delta_m \ra 0$ and $O_m$ is in a set of bounded matrices in $\SL_{\pm}(n-i_0-1)$
and $\mu_7(\gamma_m ) \ra 0$ by Proposition \ref{prop:qjoin}.
This implies $\alpha_7(\gamma_m) \ra 0$ also by definition. 


Since every element $g\in \bGamma_{\tilde E}$ is in the above matrix form, we denote by 
$\vec{v}_g$ the element obtained by the element $\vec{v}$ in the above matrix.

If $N_K$ is discrete, 
we change the sequence of element $\gamma_m \in \Gamma$ of above form with 
a bounded $\vec{v}_{\gamma_m} \in \bR^{i_0}$ by multiplying by an element of $N\subset \CN$. 



Now suppose that $N_K$ is indiscrete. 
Given a positive constant $C$, $\bGamma_{\tilde E} \cap N_{C}(\CN)$ is a discrete subset of $N_{C}(\CN)$. 
There exists a constant $C_1, C_1 > 0$ so that $N_{C_1}(\bGamma_{\tilde E} \cap N_{C}(\CN))$
contains $\CN$: $\tilde \Sigma_{\tilde E}$ has a compact fundamental domain $F$ under $\bGamma_{\tilde E}$. 
Thus, given any $\vec{v}$, 
$\CN(\vec{v})(x)$ for $x \in F$ is in $g(F)$ for some $g \in \bGamma_{\tilde E}$. 
Then $g^{-1}\CN(\vec{v})(x) \in F$. Since $g(y) = \CN({\vec{v}}) \in g(F)$
for $y \in F$, 
$\Pi^*_K(g)$ has the value 
$d_K(\Pi_K(x), g\Pi_K(x))$ is bounded by a constant $C_F$ depending on $F$. 
$g$ is in a bounded neighborhood of $\CN$ by  
Proposition \ref{prop:eigSI} since $g$ is of form of matrix of equation \ref{eqn:gammao}. 
From the linear block form of $g^{-1}\CN(\vec{v})$ and the fact that  $g^{-1}\CN(\vec{v})(x) \in F$, 
we obtain that the corresponding $\vec{v}_{g^{-1}\CN(\vec{v})}$ can be made uniformly bounded independent of $\vec{v}$. 

For element $\gamma_m$ above, we take its vector $\vec{v}_{\gamma_m}$ and find our $g$ in this way.
We obtain $g^{-1}\gamma_m$. Then the corresponding 
$\vec{v}_{g^{-1}\gamma_m}$ is uniformly bounded as
we can see from the block multiplications. 

Thus, from now on, we assume that $\gamma_m$ 
has $\vec{v}_{\gamma_m}$ uniformly bounded.

Note that the lower-right $(i_0+2)\times (i_0+2)$-matrix of the above matrix
must act on the horosphere $H$. $\CN$ also act transitively on $H$. 
Hence, for any such matrix we can find an element of $\CN$ so that 
the product is in the orthogonal group acting on $H$. 



Denote by $S(K')$ and $S(H)$ the subspaces determined by $K'$ and $H$ and containing them. 
$S(K')$ and $S(U)$ form a pair of complementary subspaces in $\SI^n$.


We have the sequence $\gamma_m$ acting on $K'_{\mx}$ is uniformly bounded and 
$\gamma_m$ acting on $H_{\mx}$ in a uniformly bounded manner 
as $m \ra \infty$ since $\{\vec{v}_{\gamma_m}\}$ is bounded and 
$\{ \alpha_7(\gamma_m)\} \ra 0$. 
Let $H_{\mx}$ denote $S(H) \cap \clo(\torb)$ and $K^{\prime}_{\mx}$ the set $S(K') \cap \clo(\torb)$.
By Lemma \ref{lem:decjoin} for $l = 2$ case, 
$\clo(\torb)$ equals the join of $H_{\mx}$ and $K'_{\mx}$. 






Note that a group of parabolic automorphisms in $\Gamma$ acts on $\partial H_\mx -\{{\bv_{\tilde E}}\}$ cocompactly. 
The only $\bGamma_{\tilde E}$-invariant subset of $\partial H_\mx$ is $\bv_{\tilde E}$ and its complement. 
The maximal join decomposition of $H_\mx$ if it exists, has finitely many compact convex subsets 
and they are permuted by $\bGamma_{\tilde E}$. 

Any maximal join decomposition of $\clo(H)$ then has $P(H)$ as a factor. 
Since $\pi_1(\mathcal{O})$ permutes the factors, $P(H_\mx)=P(H)$ is 
$\pi_1(\mathcal{O})$-invariant. 
This implies that $\bGamma$ is reducible. Hence the joined ends cannot occur. 


\end{proof}


\section{The proof of Theorem \ref{thm:secondmain}} \label{sec:secondmain} 

\begin{proof}[{\rm (}Theorem \ref{thm:secondmain}{\rm )}] 
Suppose that $\tilde E$ is a properly convex domain.
First assume the uniform middle eigenvalue condition.  
Now, Theorem \ref{thm:equ} implies that $\tilde E$ is of generalized 
lens type. If $\orb$ satisfies the triangle condition, Theorem \ref{thm:equ} implies that $\tilde E$ is of lens-type. 
Theorem \ref{thm:redtot} implies that $\tilde E$ is of lens-type if $\tilde E$ is reducible. 
If $\tilde E$ is totally geodesic and hyperbolic, then $\pi_1(\tilde E)$ acts cocompactly 
on an open totally geodesic surface $\tilde \Sigma_{\tilde E}$ in $\clo(\Omega)$. 
By Lemma \ref{lem:simplexbd}, 
$\torb$ contains $\tilde \Sigma_{\tilde E}$. By Proposition \ref{prop:convhull2}, we obtain a lens. 

If we assume the weak uniform middle eigenvalue condition, then 
Proposition \ref{prop:quasilens1} implies the result. 

Now suppose that $\tilde E$ is totally geodesic. Then Theorem \ref{thm:equ2} implies the result. 
\end{proof}

\begin{proof}[{\rm (}Theorem \ref{thm:thirdmain}{\rm )}]
Suppose that $\tilde E$ is an NPCC R-end. Then Theorems \ref{thm:NPCCcase} and \ref{thm:NPCCcase2} prove the result.  
\end{proof}


\appendix 

\section{The affine action dual to the tubular action} \label{app:dual}

Let $\Gamma$ be an affine group acting on the affine space $A^n$ with boundary $\Bd A^n$ in $\SI^n$, 
i.e., an open hemisphere.
Recall that $\Gamma$ is asymptotically nice if 
there exists a properly convex invariant $\Gamma$-invariant domain 
$U'$ with boundary in a properly convex domain 
$\Omega \subset \Bd A^n$ and 
$U'$ is in the intersection of all half-spaces $H$, $H \ne A^n$, supporting $U'$ at all point of $\Bd \Omega$.

In this section, we will work with $\SI^n$ only, while
the $\bR P^n$ versions are clear enough. 

Each element of $g$ is of form 
\begin{equation}\label{eqn:bendingm4} 
\left(
\begin{array}{cc}
\frac{1}{\lambda_{{\tilde E}}(g)^{1/n}} \hat h(g)          &       \vec{b}_g     \\
\vec{0}          &     \lambda_{{\tilde E}}(g)                  
\end{array}
\right)
\end{equation}
where $\vec{b}_g$ is $1\times n$-vector and $\hat h(g)$ is an $n\times n$-matrix of determinant $\pm 1$
and $\lambda_{{\tilde E}}(g) > 0$ is a constant. 
In the affine coordinates, it is of form 
\begin{equation}\label{eqn:affact} 
x \mapsto \frac{1}{\lambda_{{\tilde E}}(g)^{1+ \frac{1}{n}}} \hat h(g) x + \frac{1}{\lambda_{{\tilde E}}(g)} \vec{b}_g. 
\end{equation}
Recall that if there exists a uniform constant $C > 0$ so that 
\[C^{-1} \leng(g) \leq \log \frac{\lambda_1(g)}{\lambda_{{\tilde E}}(g)} \leq C \leng(g), \quad
g \in \bGamma_{\tilde E} -\{\Idd\},\] 
then $\Gamma$ is said to satisfy 
the {\em uniform middle-eigenvalue condition}.

In this appendix, it is sufficient for us to prove when $\Gamma$ is a hyperbolic group. 

\begin{theorem}\label{thm:asymnice}
We assume that $\Gamma$ is a hyperbolic group. 
Let $\Gamma$ have a properly convex affine action on the affine space $A^n$, $A^n \subset \SI^n$,
acting on a properly convex domain $U \subset A^n$ with boundary in the convex domain
$\clo(\Omega)$ for a properly convex domain $\Omega$ in $\Bd A^n$. 
Suppose that $\Omega/\Gamma$ is a closed $n-1$-dimensional orbifold and 
$\Gamma$ satisfies the uniform middle-eigenvalue condition. 
Then $\Gamma$ is asymptotically nice with a properly convex 
open domain $U$ so that $\clo(U) \cap \Bd A^n = \clo(\Omega)$ 
and the asymptotic hyperspace at 
each boundary point of $\Omega$ is uniquely determined and transversal to $\Bd A^n$. 
\end{theorem}

In the case when the linear part of the affine maps are unimodular, 
Theorem 8.2.1 of Labourie \cite{Lab} shows that such a domain $U$ exists but without showing the asymptotic niceness. 
In general, we think that the existence of the domain $U$ can be obtained but the proof 
is much longer. (See Appendix of \cite{CG} in the special case that can be extended here.)

(It is fairly easy to show that this holds also for virtual products of hyperbolic and abelian groups as well.
by Proposition \ref{prop:dualend2} and Theorem \ref{thm:distanced}.)

\subsection{The Anosov flow.}\label{sec:anosov}

We apply the work of Goldman-Labourie-Margulis \cite{GLM}: Assume 
as in the premise of Theorem \ref{thm:asymnice}. 
Since $\Omega$ is properly convex, 
$\Omega$ has a Hilbert metric. 
Let $U\Omega$ denote the unit tangent bundle over $\Omega$.
This has a smooth structure as a quotient space of $T\Omega - O/\sim$ where 
$O$ is the image of the zero-section and $\vec{v} \sim \vec{w}$ if $\vec{v}$ and $\vec{w}$ are over the same point of $\Omega$
and $\vec{v} = s \vec{w}$ for a real number $s > 0$.

Assume $\Gamma$ as above. 
Since $\Sigma:= \Omega/\Gamma$ is a properly convex convex real projective orbifold, 
$U\Sigma := U\Omega/\Gamma$ is a compact smooth orbifold again. 
A geodesic flow on $U\Omega/\Gamma$ is Anosov and hence topologically mixing. Hence, the flow is nonwondering everywhere. (See \cite{Ben1}.)
$\Gamma$ acts irreducibly on $\Omega$, and $\Bd \Omega$ is $C^1$. 

Let $h: \Gamma \ra \Aff(A^n)$ denote the representation 
as described in equation \ref{eqn:bendingm4}. 
We form the product $U\Omega \times A^n$ that is an affine bundle over $U\Omega$. 
We take the quotient $U\Omega \times A^n$ by the diagonal action 
\[g(x, \vec u)= (g(x), h(g) \vec u) \hbox{ for } g \in \Gamma, x \in U\Omega, \vec u \in A^n.\] 
We denote the quotient by $\bA$ fibering over
the smooth orbifold $U\Omega/\Gamma$ with fiber $A^n$. 

Let $V^n$ be the vector space associated with $A^n$.
Then we can form $U\Omega \times V^n$ and take the quotient under 
the diagonal action:
\[g(x, \vec u)= (g(x), {\mathcal L}\circ  h(g) \vec u) \hbox{ for } g \in \Gamma, x\in U\Omega, \vec u \in V^n\]
where $\mathcal L$ is the homomorphism taking the linear part of $g$.  
We denote by $\bV$ the fiber bundle over $U\Omega/\Gamma$ 
with fiber $V^n$. 

Since $U\Omega \times A^n$ is a flat $A^n$-bundle over $U\Omega$, 
we have a flat connection $\nabla^{\bA}$ on the bundle $\bA$ over $U\Sigma$
and a flat linear connection $\nabla^{\bV}$ on the bundle $\bV$ over $U\Sigma$. 
The connections are simply the ones induced by the trivial product structure. 

We give a decomposition of $\bV$ into three parts $\bV_+, \bV_0, \bV_-$: 
For each vector $\vec u \in U\Omega$, we find the maximal geodesic 
$l$ ending at two points $\partial_+ l, \partial_- l$. They correspond to 
the $1$-dimensional vector subspaces $V_+$ and $V_-$. 
There exists a unique pair of supporting hyperspheres $H_+$ and $H_-$ in $\Bd A^n$ at each of $\partial_+ l$ and $\partial_- l$. 
We denote by $H_0 = H_+ \cap H_-$. It is a codimension $2$ great sphere in $\Bd A^n$
and corresponds to a vector subspace $V_0$ of codimension two in $\bV$. 
For each vector $\vec u$, we find the decomposition of $V$ as $V_+ \oplus V_0 \oplus V_-$
and hence we can form the subbundles $\bV_+, \bV_0, \bV_-$
where $\bV = \bV_+ \oplus \bV_0 \oplus \bV_-$. These are $C^0$-bundles since $\Bd \Omega$ is $C^1$ and hence
the end points of $l$ depends on $l$ in a $C^1$-manner and the supporting subspaces depends on 
the boundary points in $C^0$-manner. 

We can identify $\Bd A^n = {\mathcal{S}}(V)$ where $g$ acts by  ${\mathcal{L}}(g) \in \GL(n, \bR)$. 

If $g \in \Gamma$ acts on $l$, then $V_+$ and $V_-$ are eigenspaces of the largest norm  $\lambda_1(g)$ of the eigenvalues 
and the smallest norm $\lambda_n(g)$ of the eigenvalues of the linear part ${\mathcal{L}}(g)$ of $g$ equal to 
\[\frac{1}{\lambda_{{\tilde E}}(g)^{1+\frac{1}{n}}} \hat h(g).\]
$V_+$ and $V_-$ are one-dimensional by the proximality and strictly convex boundary condition for hyperbolic groups in \cite{Ben1}. 
Hence on $V_+$, $g$ acts by expending by $\lambda_1(g)/\lambda_{{\tilde E}}(g)$
and on $V_-$, $g$ acts by contracting by $\lambda_n(g)/\lambda_{{\tilde E}}(g)$.

There exists a flow $\hat \Phi_t: U\Sigma \ra U\Sigma$ for $t \in \bR$ given 
by sending $\vec v$ to the unit tangent vector to at $\alpha(t)$ where 
$\alpha$ is a geodesic tangent to $\vec v$ with $\alpha(0)$ equal to the base point of $\vec v$.

We define a flow on $\tilde \Phi_t: \bA \ra \bA$ by considering a unit speed geodesic flow line $\vec{l}$ in $U\Omega$ and 
and considering $\vec{l} \times E$ and acting trivially on the second factor as we go from $\vec v$ to $\hat \Phi_t(\vec v)$
(See remarks in the beginning of Section 3.3 and equations in Section 4.1 of \cite{GLM}.)
Each flow line in $U\Sigma$ lifts to a flow line on $\bA$ from every point in it. 

We define a flow on $\tilde \Phi_t: \bV \ra \bV$ by considering a unit speed geodesic flow line $\vec{l}$ in $U\Omega$ and 
and considering $\vec{l} \times V$ and acting trivially on the second factor as we go from $\vec v$ to $\Phi_t(\vec v)$
for each $t$. (This generalizes the flow on \cite{GLM}.)
Also, $\tilde \Phi_t$ preserves $\bV_+$, $\bV_0$, and $\bV_-$ since on the line $l$, the end point $\delta_\pm l$ does not change. 

Since $\Gamma$ acts on $U\Omega$ preserving $\bV_+, \bV_0, \bV_-$, we obtain bundles over $U\Omega/\Gamma$. 
We will use the same notation for these bundles $\bV_+$, $\bV_0$, and $\bV_-$.  

We let $||\cdot||_S$ denote some metric on these bundles over $U\Sigma/\Gamma$ defined as fiberwise inner product.
The construction of such a metric $||\cdot||_S$ is given by choosing one for $A^n$ 
and extending it on $\Omega \times A^n$ by choosing 
a cover of $\Omega/\Gamma$ by compact sets $K_i$ 
and choosing the extension over $K_i \times A^n$. 
Then we use the partition of unity.

As in Section 4.4 of \cite{GLM}, 
$\bV = \bV_+ \oplus \bV_0 \oplus \bV_-$.
By the uniform middle-eigenvalue condition,  
$\bV$ has a fiberwise Euclidean metric $g$ with the following properties: 
\begin{itemize} 
\item the flat linear connection $\nabla^{\bV}$ is bounded with respect to $g$.
\item hyperbolicity: There exists constants $C, k > 0$ so that 
\begin{align} 
 || \tilde \Phi_t ({\vec{v}})||_S \geq \frac{1}{C} \exp(kt) ||{\vec{v}}||_S \hbox{ as } t \ra \infty
\end{align}
for $\vec{v} \in \bV_+$ and 
\begin{align} 
 || \tilde \Phi_t (\vec{v})||_S \leq C \exp(-kt) ||\vec{v}||_S \hbox{ as } t \ra \infty
\end{align}
for $v \in \bV_-$.
\end{itemize} 

Proposition \ref{prop:contr} proves this property by taking $C$ sufficiently large according to $t_1$, 
which is a standard technique. 


\subsection{The proof of the Anosov property.} \label{subsub:Anosov}


We can apply this to $\bV_-$ and $\bV_+$ by possibly reversing the direction of 
the flow. 
The Anosov property follows from the following proposition. 

Let $\bV_{-,1}$ denote the subset of $\bV_-$ of the unit length under $||\cdot||_S$.

\begin{proposition} \label{prop:contr} 
Let $\Omega/\bGamma$ be a closed real projective orbifold with hyperbolic group. 
Then there exists a constant $t_1$ so that 
\[ || \Phi_t(\bv)||_S \leq \tilde C || \bv||_S, \bv \in \bV_- \hbox{ and } || \Phi_{-t}(\bv)||_S \leq \tilde C ||\bv||_S, \bv \in \bV_+\] 
for $t \geq t_1$ and a uniform $\tilde C$, $0 < \tilde C < 1$.
\end{proposition}
\begin{proof} 
It is sufficient to prove the first part of the inequalities since we can substitute $t \ra -t$. 
Let $\bV_{-,1}$ denote the subset of $\bV_-$ of the unit length under $||\cdot||_S$.

By the following Lemma \ref{lem:S_0}, the uniform convergence implies that for given $0< \eps < 1$, 
for every vector $\bv$ in $\bV_{-, 1}$, there exists $T$ so that for $t > T$, 
$\Phi_t(\bv)$ is in $\eps$-neighborhood $U_\eps(S_0)$ of the image $S_0$ of the zero section. 
\end{proof} 

The line bundle $\bV_-$ lifts to $\tilde \bV_-$ where each unit vector $\bu$ on $\Omega$ 
one associates the line $\bV_{-,\bu}$ corresponding to the end point $\Bd \Omega$ of the geodesic tangent to it.
$\Phi$ lifts to a parallel translation or constant flow of $\tilde \bV_-$ fixing each vector $\bv$.

\begin{figure}
\centering
\includegraphics[height=8cm]{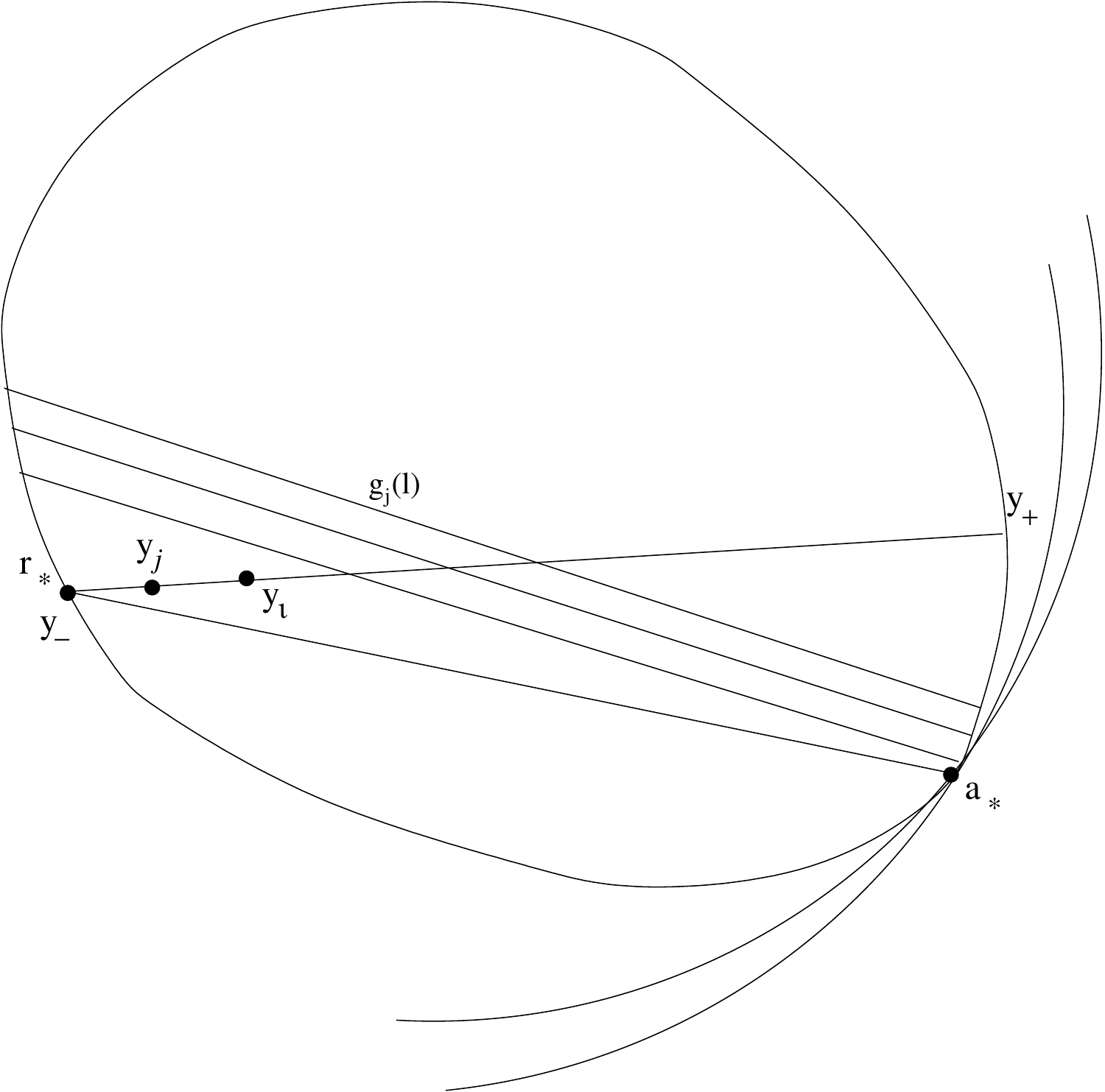}
\caption{The figure for Lemma \ref{lem:S_0}.}

\label{fig:contr}

\end{figure}

By the work in \cite{Ben1}, there exist a unique attracting fixed point 
and a unique repelling fixed point in $\Bd \Omega$ 
and there is no other fixed point 
for each infinite order element $g \in \Gamma$
since $\Gamma$ is hyperbolic and $\Omega/\Gamma$ is compact. 
Hence, we obtain that 
if there is a compact disk $B$ in $\Bd \Omega$ and 
$g \in \Gamma$ so that $g(B) \subset B$, then there exists 
a unique attracting fixed point of $g$ in $B$. 


Let $\Pi:U\Omega \ra \Omega$ be a projection of the unit tangent bundle to the base space. 

\begin{lemma} \label{lem:S_0} 
For each $\bv \in \bV_-$, 
$||\Phi_{t}(\bv)||_S  \ra 0 $ holds as $t \ra \infty$. 
Moreover, this is a uniform convergence. 
\end{lemma}
\begin{proof} 
%
We choose a sequence $\{x_i\}$, $\{x_i\} \ra x$ in a fundamental domain $F$ of $U\Omega$ under $\Gamma$. 
For each $i$, let $\bv_{-, i}$ be a unit vector in  $\bV_{-,1}$ for the unit vector $x_i \in U\Omega$, 
i.e., in the $1$-dimensional subspace in $\bR^{n}$ corresponding to the end point of the geodesic determined by $x_i$ in $\Bd \Omega$. 
Then $x_i$ determines a line $l_i$ in $\Omega$. We will show that $||\Phi_{t_i}(\bv_{-,i})||_S \ra 0$ for any sequence $t_i \ra \infty$. 
This will prove the uniform convergence to $0$ by the compactness of  $\bV_{-,1}$. 
(Here, $[\bv_{-, i}]$ is an end point of $l_i$ in the direction given by $x_i$.)

For this, we just need to show that any sequence of $\{t_i\} \ra \infty$ has a subsequence $\{t_j\}$ so that 
$||\Phi_{t_j}(\bv_{-, j})||_S \ra 0$ converging to $0$
since otherwise we can always extract a subsequence converging to nonzero or to $\infty$. 

Let $y_i := \tilde \Phi_{t_i}(x_i)$ for the lift of the flow $\tilde \Phi$. 
By construction, we recall that each $\Pi(y_i)$ is in the flow line $l_i$. 
Since $x_i \ra x$, we obtain that $l_i$ geometrically converges to a line $l_\infty$ passing $\Pi(x)$ in $\Omega$. 
Let $y_+$ and $y_-$ be the end points of $l_\infty$, where $\{\Pi(y_i)\} \ra y_-$.

Find a deck transformation $g_i$ so that 
$g_i(y_i) \in F$ and $g_i$ acts on the line bundle $\bV_-$ by the linearization of the matrix of form of equation \ref{eqn:bendingm4}: 
\[{\mathcal L}(g_i) := \frac{1}{\lambda_{{\tilde E}}(g_i)^{1+\frac{1}{n}}} \hat h(g_i) : \bV_{-, y_i} \ra \bV_{-, g_i(y_i)}.\] 

Since $g_i(l_i) \cap F \ne \emp$, we choose a subsequence of $g_i$ and relabel it $g_i$ so that
$\{g_i(l_i)\}$ converges to a nontrivial line in $\Omega$.

Since $\Gamma$ is hyperbolic, $y_-$ is a conical limit point as shown in the proof of Theorem 5.7 of \cite{conv}. 
We choose a subsequence of $\{g_i\}$ so that for the attracting fixed point $a_i \in \clo(\Omega)$ 
and the repelling fixed point $r_i \in \clo(\Omega)$ of each $g_i$, the sequences $\{a_i\}$ and $\{r_i\}$
are convergent. Then $\{a_i\} \ra a_\ast$ and $r_i \ra r_\ast$ for $a_\ast, r_\ast \in \Bd \Omega$. 
(See Figure \ref{fig:contr}.) Also, it follows that  for every compact $K \subset \clo(\Omega) - \{r_\ast\}$,
\begin{equation}\label{eqn:giK} 
g_i|K \ra \{a_\ast\} 
\end{equation} 
uniformly as in the proof of Theorem 5.7 of \cite{conv}.

Suppose that $a_{\ast} = r_{\ast}$. Then we choose an element $g \in \Gamma$ so that $g(a_{\ast}) \ne r_{\ast}$
and replace the sequence by $\{g g_i\}$ and replace $F$ by $F \cup g(F)$. 
The above uniform convergence condition still holds. 
Then the new attracting fixed points $a'_i \ra g(a_\ast)$ 
and the sequence $\{r'_i\}$ of repelling fixed point $r'_i$ of $gg_i$ converges to $r_\ast$ also
by Lemma \ref{lem:gatt}.
Hence, we may assume without loss of generality that \[a_\ast \ne r_\ast\]
by replacing our sequence $g_i$.  

We will use the standard metric $|| \cdot ||_E$ on $\bR^{n+1}$.
Suppose that both $y_+, y_- \ne r_\ast$. Then $\{g_i(l_i)\}$ converges to $\{a_\ast\}$ by 
equation \ref{eqn:giK} and this cannot be. 
If \[r_\ast = y_+ \hbox{ and } y_- \in \Bd \Omega - \{r_\ast\},\] then 
$g_i(y_i) \ra a_\ast$. 
Since $g_i(y_i) \in F$, this is a contradiction. 
Therefore \[r_\ast = y_- \hbox{ and } y_+ \in \Bd \Omega-\{r_\ast\}.\] 

Let $d_i$ denote the other end point of $l_i$ from $[\bv_{-, i}]$. 
Since $[\bv_{-, i}] \ra y_-$ and $l_i$ converges to a nontrivial line $l_\infty$, it follows that 
$\{d_i\}$ is in a compact set in $\Bd \Omega -\{y_-\}$.
Then $\{g_i(d_i)\} \ra a_{\ast}$ as $\{d_i\}$ is in a compact set in $\Bd \Omega -\{y_-\}$.
Thus, $\{g_i([\bv_{-, i}])\} \ra y'$ where $a_\ast \ne y'$ holds since 
$\{g_i(l_i)\}$ converges to a nontrivial line in $\Omega$.


Also, $g_i$ has an invariant great sphere $\SI^{n-2}_i \subset \Bd A^n$ containing the attracting fixed point $a_i$ and supporting $\Omega$.
Thus, $r_i$ is uniformly bounded at a distance from $\SI^{n-2}_i$ as $\{r_i\} \ra y_-= r_\ast$.

Since $y_i \ra y_-$, $y_i$ is also uniformly bounded away from $a_i$ and the tangent sphere $\SI^{n-1}_i$ at $a_i$. 
As $l_i \ra l_\infty$, we have $[\bv_{-, i}] \ra y_-$ also. 
The vector $\bv_{-, i}$ has the component $\bv_i^p$ parallel to $r_i$ and the component $\bv_i^S$ 
in the direction of $\SI^{n-2}_i$ where $\bv_{-, i} = \bv_i^p + \bv_i^S$. 
Since $r_i \ra r_\ast= y_-$ and $[\bv_{-, i}] \ra y_-$, we obtain $\bv_i^S \ra 0$ and that $\bv_i^p$ is uniformly bounded in $||\cdot||_E$. 
$g_i$ acts by preserving the directions of $\SI^{n-2}_i$ and $r_i$. 
Since $y' \in \Bd \Omega, y' \ne a_\ast$, $y'$ is bounded away from $\SI^{n-2}_i$. 
Since $\{g_i([\bv_{-, i}])\} (\ra y')$ is bounded away from $\SI^{n-2}_i$ uniformly, we have that 
\begin{itemize} 
\item the Euclidean norm of \[\frac{{\mathcal L}(g_i)(\bv_i^S)}{||{\mathcal L}(g_i)(\bv_i^p)||_E}\] is bounded above uniformly. 
\end{itemize} 
As $r_i$ is a repelling fixed point of $g_i$ and $||\bv_i^p||_E$ is uniformly bounded above, 
we have $\{{\mathcal L}(g_i)(\bv_i^p)\} \ra 0$. 
\[\{{\mathcal L}(g_i)(\bv_i^p)\} \ra 0 \hbox{ implies } \{{\mathcal L}(g_i)(\bv_i^S)\} \ra 0.\] 
Hence, we obtain $\{{\mathcal L}(g_i)(\bv_{-, i})\} \ra 0$ under $||\cdot||_E$. 

This implies $\{||\Phi_{t_i}(\bv_{-, i})||_S \}\ra 0$ since for the fundamental domain $F$, the Euclidean metric and 
the Riemannian metric of $\tilde \bV_-$ are related by a bounded constant on the compact set $F$. 
\end{proof}





\subsection{The neutralized section.}

A section $s:U\Sigma \ra \bA$ is {\em neutralized} if 
\begin{equation}\label{eqn:neu}
\nabla^{\bA}_{\phi} s \in \bV_0 . 
\end{equation}
We denote by $\Gamma(\bV)$ the space of sections 
$U\Sigma \ra \bV$ and by $\Gamma(\bA)$ the space of
sections $U\Sigma \ra \bA$. 

Recall from \cite{GLM} the one parameter-group of bounded 
operators $D\Phi_{t, *}$ on $\Gamma(\bV)$ and 
$\Phi_{t, *}$ on $\Gamma(\bA)$. We denote by $\phi$ the vector field generated by 
this flow on $U\Sigma$.
Recall Lemma 8.3 of \cite{GLM} also 

\begin{lemma}\label{lem:lem83}
 If $\psi \in \Gamma(\bA)$, and 
\[t \mapsto D\Phi_{t, *}(\psi) \] 
is a path in $\Gamma(\bV)$ that is differentiable at $t =0$, then 
\[ \frac{d}{dt}\left|_{t=0}  (D\Phi_t)_* (\psi)  \right. = \nabla^{\bA}_\phi(\psi). \]
\end{lemma} 

Recall that $U\Sigma$ is a recurrent set under the geodesic flow. 
\begin{lemma}\label{lem:exist} 
A neutralized section exists on $U\Sigma$. 
This lifts to a map $\tilde s_0: U\Omega \ra \bA$  
so that $\tilde s_0 \circ \gamma = \gamma \circ \tilde s_0$. 
\end{lemma}
\begin{proof} 
Let $s$ a continuous section $U\Sigma \ra \bA$. 
We decompose 
\[\nabla^{\bA}_\phi(s) = \nabla^{\bA+}_\phi(s) + \nabla^{\bA 0}_\phi(s) + \nabla^{\bA-}_\phi(s) \in \bV \]
so that $\nabla^{\bA\pm}_\phi(s) \in \bV_\pm$ and $\nabla^{\bA_0}_\phi(s)  \in \bV_0$ hold.
Again
\[ s_0 = s + \int_0^\infty (D\Phi_t)_*(\nabla^{\bA-}_\phi(s)) dt - \int_0^\infty (D\Phi_{-t})_*(\nabla^{\bA+}_\phi(s)) dt\] 
is a continuous section and 
$\nabla^{\bA}_\phi(s_0) = \nabla^{\bA_0}_\phi(s_0) \in \bV_0$ as shown in \cite{GLM}. 

Since $U\Sigma$ is connected, there exists a fundamental domain $F$ 
so that we can lift $s_0$ to $\tilde s_0'$ defined on $F$ mapping to $\bA$. 
We can extend $\tilde s_0'$ to $U\Omega \ra \Omega \times E$. 
\end{proof} 

Let $N_2(A^n)$ denote the space of codimension two affine spaces of $A^n$. 
We denote by $G(\Omega)$ the space of maximal oriented geodesics in $\Omega$.
We use the quotient topology on both spaces. 
There exists a natural action of $\Gamma$ on both spaces. 

For each element $g\in \Gamma -\{\Idd\}$, we define $N_2(g)$: 
Now, $g$ acts on $\Bd A^n$ with invariant subspaces corresponding to 
invariant subspace of the linear part ${\mathcal L}(g)$ of $g$. 
Since $g$ and $g^{-1}$ are positive proximal,  a unique fixed point corresponds
to the largest norm eigenvector, an attracting fixed point in $\Bd A^n$, 
and a unique fixed point corresponds to the smallest norm eigenvector,
a repelling fixed point by \cite{Ben1} or \cite{Benasym}.
There exists a ${\mathcal L}(g)$-invariant 
vector subspace $V_g^0$ complementary to the join of the subspace generated 
by these eigenvectors. (This space equals $V_0$ for the unit tangent vector 
tangent to the unique maximal geodesic $l_g$ in $\Omega$ where $g$ acts on.)
It corresponds to a $g$-invariant 
subspace $M(g)$ of codimension two in $\Bd A^n$.

Let $\tilde c$ be the geodesic in $U\Sigma$ that is $g$-invariant for $g \in \Gamma$. 
$\tilde s_0(\tilde c)$ lies on a fixed affine space parallel to $V_0^g$ by the neutrality, i.e., Lemma \ref{lem:exist}. 
There exists a unique affine subspace $N_2(g)$ of codimension two in $A^n$ 
whose containing $\tilde s_0(\tilde c)$. 
One immediate property is $N_2(g) = N_2(g^{-1})$.

\begin{definition}\label{defn:tau} 
We define $S'(\Bd  \Omega)$ the space of $(n-1)$-dimensional hemispheres
with interiors in $A^n$ each of  
whose boundary in $\Bd A^n$ is a supporting hypersphere in $\Bd A^n$ to $\Omega$. 
We denote by $S(\Bd \Omega)$ the space of pairs 
$(x, H)$ where $H \in S'(\Bd \Omega)$ and $x$ is in the boundary of 
$H$ and in $\Bd \Omega$. 
\end{definition}

Define $\Delta$ to be the diagonal set of $\Bd \Omega \times \Bd \Omega$. 
Denote by $\Lambda^* = \Bd \Omega \times \Bd \Omega - \Delta$. 
Let $G(\Omega)$ denote the space of maximal oriented geodesics in $\Omega$. 
$G(\Omega)$ is in one-to-one correspondence with $\Lambda^*$ by 
the map taking the maximal oriented geodesic to the ordered pair of its endpoints. 

\begin{proposition}\label{prop:mapgh}
\begin{itemize}
\item There exists  a continuous function $\hat s: U\Omega \ra N_2(A^n)$ 
equivariant with respect to $\Gamma$-actions.
\item Given $g \in \Gamma$ and for the unique unit speed geodesic $\vec{l}_g$ in 
$U\Omega$ lying over a geodesic $l_g$ where $g$ acts on, $\hat s(\vec{l}_g) = \{N_2(g)\}$. 
\item This gives a continuous map 
\[\tau: \Bd \Omega \times \Bd \Omega - \Delta \ra N_2(A^n)\] 
again equivariant with respect to the $\Gamma$-actions. 
There exists a continuous function 
\[\tau:\Lambda^* \ra S(\Bd \Omega).\]
\end{itemize}
\end{proposition}
\begin{proof} 

Given a vector ${\vec u} \in U\Omega$, we find $\tilde s_0(\vec{u})$.
There exists a lift $\tilde \phi_t: U\Omega \ra U\Omega$ of 
the geodesic flow $\phi_t$.
Then $\tilde s_0(\tilde \phi_t({\vec u}))$ all lies in 
an affine subspace $H^{n-2}$ parallel to $V_0$ for $\vec u$ by the neutrality 
condition equation \ref{eqn:neu}.
We define $\hat s(\vec u)$ to be this $H^{n-2}$. 

For any unit vector $\vec u'$ on the maximal (oriented) geodesic 
in $\Omega$ determined by $\vec u$, we obtain
$\hat s(\vec u') = H^{n-2}$. 
Hence, this determines the continuous map $\bar s: G(\Omega) \ra N_2(A^n)$. 
The $\Gamma$-equivariance comes from that of $\tilde s_0$. 

For $g \in \Gamma$, $\vec u$ and $g(\vec u)$
lie on the $g$-invariant geodesic $l_g$ provided $\vec u$ is tangent to $l_g$. 
Since $g(\tilde s_0(\vec u)) = \tilde s_0(g(\vec u))$ by equivariance, 
$g(\tilde s_0(\vec u))$ lies on $\hat s(\vec u) = \hat  s(g(\vec u))$ by two paragraphs above
and $g(\bar s(l_g)) = \bar s(l_g)$.

$\Bd \Omega \times \Bd \Omega - \Delta$ is in one-to-one correspondence with 
the space $G(\Omega)$.
The last item follows by taking for each pair $(x, y) \in \Lambda^*$
we take the geodesic $l$ with endpoints $x$ and $y$,
and taking the hyperspace in $A^n$ containing $\bar s(l)$ and its boundary containing $x$.  
\end{proof}

\subsubsection{The asymptotic niceness.} \label{sub:asymnice}


\begin{lemma}\label{lem:hdisj} 
Let $U$ be a $\bGamma_{\tilde E}$-invariant 
properly convex open domain in $\bR^n$ so that 
$\Bd U \cap \Bd A^n = \clo(\Omega)$. 
Suppose that $x$ and $y$ are fixed points of an element $g$ of $\Gamma$ in $\Bd \Omega$. 
Then $h(x, y)$ is disjoint from $U$. 
\end{lemma} 
\begin{proof}
Suppose not. 
$h(x, y)$ is a $g$-invariant hemisphere, and $x$ is a fixed point of $g$ in it, 
Then $U \cap h(x, y)$ is a $g$-invariant properly convex open domain containing $x$ in its boundary. 

Suppose first that $h(x, y)$ has a fixed point $z$ of $g$ 
with the smallest eigenvalue in $h(x, y)^o$. 
Then the associated eigenvalue to $z$ 
is strictly less than that of $x$ 
by the uniform middle-eigenvalue condition
and hence $z$ is in the closure of $U\cap h(x, y)$. 
$g$ acts on the $2$-sphere $P$ containing $x, y, z$. 
Then $P\cap U$ cannot be properly convex due to the fact that $z$ is a 
saddle-type fixed point. Hence, there exists no fixed point $z$. 

The alternative is as follows: 
$h(x, y)$ contains a $g$-invariant affine subspace $A^{\prime}$ of codimension at least $2$
and the smallest eigenvalue in $h(x, y)$ is associated with a point of the boundary of $A'$. 
$g| h(x, y)$ has the largest norm eigenvalue at $x, x_{-}$. 
Therefore, acting by  $\langle g \rangle$ on a generic point $z$ of $h(x, y) \cap U$
gives us an arc in $h(x, y)$ with endpoints $x$ or $x_{-}$ and 
an endpoint $y'$ in $\Bd A^{\prime} \subset \Bd A^n$. 
Here $y'$ is a fixed point in $h(x, y)$ different from $y$ as $y \not\in h(x, y)$. 
$x \in \clo(\Omega)$ implies $x_{-}\not\in \clo(\Omega)$ by the proper convexity. 
$x, y' \in \clo(\Omega)$ implies
$\ovl{x y'} \subset \Bd A^n \subset \clo(\Omega)$.
Finally, $\ovl{x y'} \subset \partial h(x, y)$ for 
the supporting subspace $\partial h(x, y)$ of $\clo(\Omega)$
violates the strict convexity of $\Omega$.
(See Benoist \cite{Ben1}.)


\end{proof} 

The proof of the following lemma is the different from one in \cite{afftame}. 
In Theorem \ref{thm:qFuch2}, we will obtain that these are also give us strict lens p-neighborhoods. 

\begin{lemma} \label{lem:inde}
Let $(x, y) \in \Lambda^*$. Then 
\begin{itemize}
\item $\tau(x, y)$ does not depend on $y$ and is unique for each $x$
and 
\item $h(x, y)$ is never a hemisphere in $\Bd A^n$ for every 
$(x, y) \in \Lambda^*$.
\item $\tau: \Bd \Omega \ra S(\Bd \Omega)$ is continuous. 
\end{itemize}
\end{lemma}
\begin{proof} 
We claim that for any $x, y $ in $\Bd \Omega$, $h(x, y)$ is disjoint from $U$: 
By Theorem 1.1, the geodesic flow on $\Omega/\Gamma$ is Anosov, and hence 
closed geodesics in $\Omega/\Gamma$ is dense in the space of geodesics
by the basic property of the Anosov flow. 
Since the fixed points in $\Bd \Omega$ are, we can find a sequence  $x_i \ra x$ and $y_i \ra y$ where 
$x_i$ and $y_i$ are fixed points of an element $g_i \in \Gamma$ for each $i$. 
If $h(x, y) \cap U \ne \emp$, then $h(x_i, y_i) \cap U \ne \emp$ for 
$i$ sufficiently large by the continuity of the map $\tau$.
This is a contradiction by Lemma \ref{lem:hdisj} 

Also $\Bd A^n$ does not contain $h(x, y)$ 
since $h(x, y)$ contains the $\bar s(\ovl{xy})$ while $y$ is chosen 
$y \ne x$. 

Let $H(x, y)$ denote the half-space bounded by $h(x, y)$ containing $U$. 
For each $x$, we define 
\[ H(x) := \bigcap_{y \in \Bd \Omega -\{x\}} H(x, y). \]
Define $h(x)$ as the boundary $(n-1)$-hemisphere of $H(x)$.  
(Note that $\partial H(x, y')$ is supporting $\Bd \Omega$ and hence is independent of $y'$
as $\Bd \Omega$ is $C^1$.)

Let $H(x)$ denote the open half-space bounded by $h(x)$ containing $U$. 
Let $U'$ be defined as the convex open domain 
$\bigcap_{x \in \Bd \Omega} H(x)$ containing $U$. 
Since $\Bd \Omega$ has at least $n+1$ points in general position 
and tangent hemispheres, $U'$ is properly convex.
Let $U''$ be the properly convex open domain $\bigcap_{x \in \Bd \Omega} (E - \clo(H(x)))$. 
It has the boundary $\mathcal{A}(\clo(\Omega))$ for the antipodal map $\mathcal A$ and is a properly convex domain
as the antipodal set of $\Bd \Omega$ has at least $n+1$ points in general position. 
Note that $U' \cap U'' = \emp$.

If for some $x, y$, $h(x, y)$ is different from $h(x)$, then 
$h(x, y) \cap U'' \ne \emp$.
This is a contradiction as the proof of Lemma \ref{lem:hdisj}. 
Thus, we obtain $h(x, y) = h(x)$ for all $y \in \Bd \Omega -\{x \}$. 

We show the continuity of $x \mapsto h(x)$:
Let $x_i \in \Bd \Omega$ be a sequence converging to $x \in \Bd \Omega$. 
Then choose $y_i \in \Bd \Omega$ so that 
$y_i \ra y$ and we have $\{h(x_i) = h(x_i, y_i)\}$ converges to 
$h(x, y) = h(x)$ by the continuity of $\tau$. 
Therefore, $h$ is continuous. 
\end{proof} 

\begin{proof}[{\rm (}Theorem \ref{thm:asymnice}\,{\rm )}] 
For each point $x \in \Bd \Omega$, an $(n-1)$-dimensional hemisphere $h(x)$ passes $A^n$ 
with $\partial h(x) \subset \Bd A^n$ supporting $\Omega$ by Lemma \ref{lem:inde}. 
Then a hemisphere $H(x) \subset A^n$ is bounded by $h(x)$ and contains $\Omega$. 
The properly convex open domain 
$\bigcap_{x \in \Bd \Omega} H(x)$ contains $U$.
The uniqueness of $h(x)$ in the proof of Lemma \ref{lem:inde} gives us the unique asymptotic totally geodesic hypersurfaces. 

\end{proof} 

The following is more useful version of Theorem \ref{thm:asymnice}. 
We don't assume that $\Gamma$ is hyperbolic here. 

\begin{theorem}\label{thm:lensn} 
Let $\Gamma$ be a discrete group in $\SLnp$ acting on $\Omega$, 
$\Omega \subset \Bd A^n$, so that $\Omega/\Gamma$ is 
a compact orbifold. 
\begin{itemize}
\item Suppose that $\Omega$ has a $\Gamma$-invariant properly convex open domain $U$ forming a one-sided neighborhood of $\Omega$ in $A^n$. 
\item Suppose that $\Gamma$ satisfies the uniform middle eigenvalue condition. 
\item Let $P$ be the hyperplane containing $\Omega$. 
\end{itemize}
Then $\Gamma$ acts on a properly convex open domain $L$ in $\SI^n$ containing $\Omega$ and contained in $U$
and having strictly convex boundary in $\SI^n - P$. 
That is, $L$ is a lens-shaped neighborhood of $\Omega$. 
\end{theorem} 
\begin{proof} 
For purpose here, we assume that $\Gamma$ is torsion-free
since we can always add finite order elements later. 
For each $H(x)$, $x \in \Bd \Omega$, above, 
an open hemisphere $H'(x)$ satisfies $\Bd H'(x) = H(x)$. 
Then $V:=\bigcap_{x \in \Bd \Omega} H'(x)$ is a convex open domain containing $\Omega$
as in the proof of Lemma \ref{lem:inde}.

First suppose that $V$ is properly convex. 
Then $V$ has a $\Gamma$-invariant Hilbert metric $d_V$ that is also Finsler. (See \cite{wmgnote} and \cite{Kobpaper}.)
Then
\[N_\eps =\{ x\in V| d_V(x, \Omega)) < \eps\}.\]
 is a convex subset of $V$ by Lemma 1.8 of \cite{CLT2}. 



A compact tubular neighborhood $M$ of $\Omega/\Gamma$ in $V/\Gamma$ is
diffeomorphic to $\Omega/\Gamma \times [-1,1]$. (See Section 4.4.2 of \cite{Cbook}.)
Since $\Omega$ is compact, the regular neighborhood has a compact closure. 
Thus, $\bdd_V(\Omega/\Gamma, \Bd M/\Gamma) > \eps_0$ for some $\eps_0 > 0$. 
If $\eps < \eps_0$, then $N_\eps \subset M$. We obtain that $\Bd N_\eps/\Gamma$ is compact.


Clearly, $\Bd N/\Gamma$ has two components in two respective components of $(V - \Omega)/\Gamma$.
Let $F_1$ and $F_2$ be the fundamental domains of both components. 
We procure 
finitely many open hemispheres $H_i$, $H_i \supset \Omega$, 
so that open sets $(\SI^n - \clo(H_i) )\cap N_\eps$ cover $F_1 \cup F_2$. 
Since any path from $\Omega$ to $\Bd N_\eps$ must meet $(\Bd W -P) \cap V$ first, $N_\eps$ contains
$W := \bigcap_{g \in \Gamma} g(H_i) \cap V$ and $\Bd W$. 
A collection of 
compact totally geodesic polyhedrons meet in angles $< \pi$ and 
comprise $\Bd W/\Gamma$. 
We can smooth $\Bd W$ to obtain a strictly convex lens neighborhood $W'$ of $\Omega$ in $N_\eps$.  

Since the closed set $\Bd U \cap V/\Gamma$ does not meet
the compact $\Sigma$, 
$d_V(\Omega, \Bd U \cap V) > \eps_0$ for a positive number $\eps_0$. 
We choose $\eps > 0$ smaller then this number. 
Then $W$ is a subset of $U$ which we construct for $N_\eps$ as above. 

Suppose that $V$ is not properly convex. Then $\Bd V$ contains $v, v_-$.
$V$ is a tube. 
We take any two open hemispheres $S_1$ and $S_2$ containing $\clo(\Omega)$ so that 
$\{v, v_-\} \cap S_1 \cap S_2 = \emp$. 
Then $\bigcap_{g \in \Gamma} g(S_1 \cap S_2) \cap V$ is a properly convex open domain containing $\Omega$.
and we can apply the same argument as above. 
\end{proof}


\bibliographystyle{plain}

\begin{thebibliography}{99}

\bibitem{Baues} { O. Baues}, 
\newblock{`Deformation spaces for affine crystallographic groups'}, In {\em Cohomology of groups and algebraic K-theory}, 
\newblock{ 55--129, Adv. Lect. Math. (ALM), 12, Int. Press, Somerville, MA, 2010.} 

\bibitem{Ben1} { Y. Benoist},
 \newblock{`Convexes divisibles. I',}
 \newblock{In {\em Algebraic groups and arithmetic}},
  {339--374}, Tata Inst. Fund. Res., Mumbai, 2004.
 
\bibitem{Ben2} { Y. Benoist},
\newblock{`Convexes divisibles. II'},
 \newblock{\em Duke Math. J.}, {120} (2003), 97--120.

\bibitem{Ben3} { Y. Benoist},
\newblock{`Convexes divisibles. III'},
\newblock{{\em Ann. Sci. Ecole Norm. Sup.} (4) 38 (2005), no. 5, 793--832. }

\bibitem{Ben4} { Y. Benoist},
\newblock{ `Convexes divisibles IV : Structure du bord en dimension 3'}, 
\newblock{{\em Invent. math.} 164 (2006), 249--278.}

\bibitem{Ben5} { Y. Benoist},
 \newblock{`Automorphismes des c\^ones convexes'},
\newblock{\em Invent. Math.}, {141} (2000), 149--193.

\bibitem{Benasym} { Y. Benoist}, 
\newblock{`Propri\'et\'es asymptotiques des groupes lin\'eaires',} 
\newblock{ {\em Geom. Funct. Anal.} 7 (1997), no. 1, 1--47.}

\bibitem{BenNil} { Y. Benoist}, 
\newblock{`Nilvari\'et\'es projectives',}
\newblock{ {\em Comm. Math. Helv.} 69 (1994), 447--473.}

\bibitem{Benz} { J.-P., Benz\'ecri}, 
\newblock{`Sur les vari\'et\'es localement affines et localement projectives',}
\newblock{Bull. Soc. Math. France 88 (1960) 229--332.} 

\bibitem{Borel} { A. Borel}, 
\newblock{{\em Linear algebraic group},} 
\newblock{Springer Verlag, 2nd edition p.288 + xi, 1991.}

\bibitem{Car} { Y. Carri\`ere}, 
\newblock{`Feuilletages riemanniens \`a croissance polyn\^omiale'}, 
\newblock{{\em Comment. Math. Helv.} 63 (1988), 1--20.}

\bibitem{ChCh}  { Y. Chae,  S. Choi,  and C. Park},  
\newblock{`Real projective manifolds developing into an affine space'}, 
\newblock{{\em Internat. J. Math.} 4 (1993), no. 2, 179--191.}

  \bibitem{Choi2004}
   { S. Choi},
    \newblock{`Geometric structures on orbifolds and holonomy representations'},
    \newblock{{\em Geom. Dedicata} 104 (2004), 161--199.}

 \bibitem{psconv} S.~Choi, 
\newblock{`The convex and concave decomposition of manifolds with real projective structures'}, 
\newblock{M\'emoires SMF, No. 78, 1999, 102 pp.}

 
  \bibitem{Choi2006}
   { S. Choi},
    \newblock{`The deformation spaces of projective structures on 3-dimensional Coxeter orbifolds',}
    \newblock{{\em Geom. Dedicata} 119 (2006), 69--90.}


\bibitem{conv} { S. Choi}, 
\newblock{`The convex real projective manifolds and orbifolds with radial or totally geodesic ends: the closedness and openness of deformations',} 
\newblock{arXiv:1011.1060}


\bibitem{rdsv} { S. Choi}, 
\newblock{`The decomposition and classification of radiant affine 3-manifolds',}
\newblock{{\em Mem. Amer. Math. Soc.} 154 (2001), no. 730, viii+122 pp.}

\bibitem{cdcr1}
{ S.~Choi}, 
\newblock{`Convex decompositions of real projective surfaces {{\rm {I}:}}
  $\pi$-annuli and convexity'},
\newblock {\em J. Differential Geom.} 40 (1994), 165--208.

\bibitem{cdcr2}
{ S.~Choi},
\newblock{`Convex decompositions of real projective surfaces {{\rm {II}:}}
  {A}dmissible decompositions',}
\newblock {\em J. Differential Geom.}, 40 (1994), 239--283.

\bibitem{cdcr3}
{ S.~Choi}, 
\newblock{`Convex decompositions of real projective surfaces {{\rm {III}:}}
  {F}or closed and nonorientable surfaces'}, 
\newblock {\em J. Korean Math. Soc.}, 33 (1996), 1138--1171.

\bibitem{Cbook} 
{ S.~Choi}, 
\newblock{{\em Geometric structures on 2-orbifolds\,{\rm :} exploration of discrete symmetry}}, 
\newblock{MSJ Memoirs, Vol. 27. 171pp + xii, 2012}

 \bibitem{CG}
   { S. Choi and W.M. Goldman}, 
    \newblock{`The deformation spaces of convex $\mathbb{RP}^2$-structures on 2-orbifolds',}
    \newblock{{\em Amer. J. Math.} 127 (2005), 1019--1102.}


\bibitem{afftame} { S.~Choi and W. M. Goldman}, 
\newblock{`Topological tameness of Margulis spacetimes'}, 
\newblock{arXive 1204.5308} 

\bibitem{CHL} 
{ S.~Choi, C.D.~Hodgson, and G.S.~Lee}, 
\newblock{`Projective deformations of hyperbolic Coxeter 3-orbifolds',} 
\newblock{{\em Geom. Dedicata} 159 (2012), 125--167.}

 \bibitem{Cooper2006}
{    D. Cooper, D. Long, and M. Thistlethwaite},
    \newblock{`Computing varieties of representations of hyperbolic 3-manifolds into ${\rm SL}(4,\mathbb R)$',}
    \newblock{{\em Experiment. Math.} 15 (2006), 291--305.}

  \bibitem{CLT}
{   D. Cooper, D. Long, and M. Thistlethwaite},
    \newblock{` Flexing closed hyperbolic manifolds',}
    \newblock{{\em Geom. Topol.} 11 (2007), 2413--2440.}
    
    \bibitem{CLT2} 
{    D. Cooper, D. Long, and S. Tillmann,} 
    \newblock{`On convex projective manifolds and cusps',} 
    \newblock Preprint, arXiv:1104.0585.

\bibitem{CG2} 
{ J. P. Conze and Y. Guivarch}, 
\newblock{`Remarques sur la distalit\'e dans les espaces vectoriels',} 
\newblock{C. R. Acad. Sci. Paris 278 (1974), 1083--1086.}


\bibitem{CM2} { M. Crampon and L. Marquis}, 
\newblock{`Finitude g\'eom\'etrique en g\'eom\'etrie de Hilbert'}, 
Preprint  arXiv:1202.5442.

\bibitem{GV} 
J. de Groot and H. de Vries, 
\newblock{`Convex sets in projective space',} 
\newblock{Compositio Math., 13 (1958), 113--118.}

\bibitem{Fried86} { D. Fried}, 
\newblock{`Distality, completeness, and affine structure'},
\newblock{ {\em J. Differential Geometry} 24 (1986), 265--273.}

\bibitem{FG} { D. Fried and W. Goldman}, 
\newblock{`Three-dimensional affine crystallographic groups'}, 
\newblock{{\em Adv. Math.} 47 (1983), 1--49.}

\bibitem{FGH} { D. Fried, W. Goldman, and M. Hirsch}, 
\newblock{`Affine manifolds with nilpotent holonomy'}, 
\newblock{{\em Comment. Math. Helv.} 56 (1981), 487--523.}

\bibitem{Gconv} { W. Goldman}, 
\newblock {`Convex real projective structures
    on compact surfaces',} 
\newblock{\em J. Differential Geometry}, 31 (1990), 791--845.

\bibitem{wmgnote} { W. Goldman}, 
\newblock{ `Projective geometry on manifolds',} 
\newblock{Lecture notes available from the author.}

\bibitem{GL} { W. Goldman and F. Labourie},  
\newblock{`Geodesics in Margulis space times'},
\newblock{{\em Ergod. Th. \& Dynamic. Sys.} 32 (2012), 643--651.}

\bibitem{GLM} { W, Goldman, F. Labourie, and G. Margulis},  
\newblock{`Proper affine actions and geodesic flows of
hyperbolic surfaces'},
\newblock{{\em Annals of Mathematics} 170 (2009), 1051--1083. }

\bibitem{Gr} { M. Gromov},
\newblock{ `Groups of polynomial growth and expanding maps',}
\newblock{ {\em Inst. Hautes \'Etudes Sci. Publ. Math.} No. 53 (1981), 53--73. }



\bibitem{Guichard} { O. Guichard}, 
\newblock{`Sur la r\'egularit\'e H\"older des convexes divisibles'}, 
\newblock{{\em Erg. Th. \& Dynam. Sys.} 25 (2005), 1857--1880.}


\bibitem{GW} { O. Guichard and A. Wienhard}, 
\newblock{`Anosov representations: domains of discontinuity and applications'}, 
\newblock{{\em Invent. Math.}} 190 (2012), no. 2, 357--438. 


\bibitem{JM}
{ D. Johnson and J. Millson}, 
\newblock{`Deformation spaces associated to compact hyperbolic manifolds'},
\newblock In {\em Discrete groups in geometry and analysis} (New Haven, Conn., 1984), pp. 48--106,
\newblock Progr. Math., 67, Birkh\"auser Boston, Boston, MA, 1987.



\bibitem{Kac1967}
{  V.G. Kac and  \`{E}.B. Vinberg}, 
    \newblock `Quasi-homogeneous cones',
    \newblock{\em Math. Zamnetki} 1 (1967), 347--354.
    
    \bibitem{Katok} 
 {   A. Katok and B. Hasselblatt}, 
    \newblock{\em Introduction to modern theory of dynamical systems}, 
    \newblock{ Cambridge University Press} 1995.
    
    \bibitem{ink} 
    { Inkang Kim}, 
    \newblock `Compactification of strictly convex real projective structures',
\newblock Geom. Dedicata 113 (2005), 185--195. 

\bibitem{Kobpaper} { S. Kobayashi},
\newblock `Projectively invariant distances for affine and projective structures', 
\newblock In {\em Differential geometry} (Warsaw, 1979), 127--152, 
Banach Center Publ., 12, PWN, Warsaw, 1984. 


\bibitem{BK} { B. Kostant}, 
\newblock{ `On convexity, the Weyl group and the Iwasawa decomposition'}, 
\newblock{{\em Ann. ENS} 4em s\'eree tome 6 no. 4 (1973), 413--455}. 

\bibitem{KS} { B. Kostant and D. Sullivan},
\newblock{ `The Euler characteristic of an affine space form is zero',} 
\newblock{{\em Bull. Amer. Math. Soc.} 81 (1975), no. 5, 937--938. }


\bibitem{Kos} { J. Koszul}, 
\newblock{`Deformations de connexions localement plates', }
\newblock{{\em Ann. Inst. Fourier (Grenoble)} 18 fasc. 1 (1968), 103--114. }

\bibitem{Lab} { F. Labourie}, 
\newblock `Flat Projective Structures on Surfaces and Cubic Holomorphic Differentials', 
\newblock{\em Pure and Applied Mathematics Quaterly} 3 no. 4 (2007), 1057--1099, 
Special Issue: In the Honor of Grisha Margulis, Part 1 of 2. 

\bibitem{Lee} { Jaejeong Lee}, 
\newblock `A convexity theorem for real projective structures',
\newblock arXiv:math.GT/0705.3920.


    \bibitem{Marquis}
{   L. Marquis},
    \newblock `Espace des modules de certains poly\`edres projectifs miroirs',
    \newblock{\em Geom. Dedicata} 147 (2010), 47--86.
    
    \bibitem{Mess}
 {   G. Mess}, 
\newblock{`Lorentz spacetimes of  curvature',} 
\newblock{{\em Geom. Dedicata} 126 (2007), 3--45. }

\bibitem{Mo1} 
{ P. Molino},
\newblock `G\'eom\'etrie global des feuilletages riemanniens', 
\newblock {\em Nederl. Akad. Wetensch. Indag. Math.} 44 (1982), no. 1, 45--76.

\bibitem{Moore} 
{ C. Moore}, 
\newblock Distal affine transformation groups, 
\newblock{\em Amer. J. Math.} 90 (1968) 733--751.

 \bibitem{Mol} { P. Molino}, 
\newblock {\em Riemannian foliations},
\newblock Progress in Mathematics, vol 73, Birkh\"auser, Boston, Basel, 1988.

\bibitem{Rag} { M. S. Raghunathan}, 
\newblock{{\em Discrete subgroups of Lie groups}}, 
\newblock{Ergebnisse der Mathematik und ihrer Grenzgebiete, Band 68}, Springer Verlag, Berlin, 1972.


\bibitem{Shbook} { H. Shima},
\newblock{{\em The geometry of Hessian structures},}
\newblock World Scientific Publishing Co. Pte. Ltd., Hackensack, NJ, xiv+246 pp, 2007. 

\bibitem{Thnote} { W. Thurston}, 
\newblock{\em Geometry and topology of $3$-manifolds,} 
\newblock{available at \url{http://library.msri.org/books/gt3m/}.}

\bibitem{Thbook} { W. Thurston}, 
\newblock{\em Three-dimensional geometry and topology,} 
\newblock{Princeton University Press, Princeton NJ, 1997.}


\bibitem{Var} { V.S. Varadarajan}, 
\newblock{{\em Lie groups, Lie algebras, and their representations}, }
\newblock{GTM Vol 102, Springer, Berlin, 1972.}

 \bibitem{Vey68} { J. Vey},  
  \newblock `Une notion d'hyperbolicit\'e sur les vari\'et\'es localement plates', 
  \newblock {\em C.R. Acad. Sc. Paris}, 266(1968), 622--624.

 \bibitem{Vey} { J. Vey}, 
  \newblock `Sur les automorphismes affines des ouverts convexes saillants',
  \newblock {\em Ann. Scuola Norm. Sup. Pisa} (3) 24(1970), 641--665.



 \bibitem{Vinberg1971}
  { \`{E}.B. Vinberg}, 
    \newblock `Discrete linear groups that are generated by reflections',
    \newblock {\em Izv. Akad. Nauk SSSR} Ser. Mat. 35 (1971), 1072--1112.

  \bibitem{Vinberg1985}
 {  \`{E}.B. Vinberg}, 
    \newblock `Hyperbolic reflection groups',
    \newblock {\em Uspekhi Mat. Nauk} 40 (1985), 29--66.
       
\bibitem{vin63}
{  \`{E}.B. Vinberg},
\newblock `Homogeneous convex cones',
Trans. Moscow Math. Soc. 12 (1963), 340--363.

 \bibitem{Weil1962}
  {  A. Weil}, 
    \newblock `On discrete subgroups of Lie groups II',
    \newblock{\em Ann. of Math.} 75 (1962), 578--602.

   \bibitem{Weil1964}
   { A. Weil}, 
    \newblock `Remarks on the cohomology of groups',
    \newblock{\em Ann. of Math.} 80 (1964), 149--157.


\bibitem{DW} 
{ D. Witte}, 
\newblock{`Superrigidity of lattices in solvable Lie groups'},
\newblock{Inv. Math. 122 (1995), 147--193.}

\end{thebibliography}



\end{document}